\documentclass[leqno,11pt]{amsart}
\usepackage{amsmath,amssymb,amsfonts,amscd}
\usepackage[mathscr]{eucal}
\usepackage{tikz}
\usepackage{verbatim}
\usepackage[colorlinks,pagebackref]{hyperref}
\usepackage{rotating}

\usepackage{color}
\theoremstyle{plain}
\newtheorem{thm}{Theorem}[section]

\newtheorem{cor}[thm]{Corollary}
\newtheorem{prop}[thm]{Proposition}

\newtheorem{sublem}[equation]{Lemma}
\newtheorem{subcor}[equation]{Corollary}
\newtheorem{subprop}[equation]{Proposition}

\theoremstyle{definition}

\newtheorem{cosa}[thm]{}
\newtheorem{subcosa}[equation]{}
\newtheorem{ex}[thm]{Example}

\newtheorem{subex}[equation]{Example}
\newtheorem{subexs}[equation]{Examples}

\theoremstyle{remark}

\newtheorem{subrem}[equation]{Remark}
\newtheorem{subrems}[equation]{Remarks}

\numberwithin{equation}{thm}

\newcommand{\D}{\boldsymbol{\mathsf{D}}}
\newcommand{\bE}{\boldsymbol{\mathsf{E}}}

\newcommand{\LL}{\mathsf L}
\newcommand{\R}{\mathsf R}

\newcommand{\fd}{\boldsymbol{\mathfrak d}}
\newcommand{\ft}{\boldsymbol{\mathfrak t}}
\newcommand{\fp}{\boldsymbol{\mathfrak p}}
\newcommand{\bD}[1]{\boldsymbol{\mathcal D}_{#1}}

\newcommand{\A}[1]{$\mathbf{A_{2#1}}$}

\renewcommand{\SS}{\mathsf{S}}

\newcommand{\Da}{\boldsymbol{\mathsf{a}}}
\newcommand{\Db}{\boldsymbol{\mathsf{b}}}
\newcommand{\Dd}{{\boldsymbol{\mathsf{d}}}}
\newcommand{\Dc}{\boldsymbol{\mathsf{c}}}
\newcommand{\De}{\boldsymbol{\mathsf{e}}}
\newcommand{\Df}{\boldsymbol{\mathsf{f}}}
\newcommand{\Dg}{\boldsymbol{\mathsf{g}}}
\newcommand{\Dh}{\boldsymbol{\mathsf{h}}}

\newcommand{\Dp}{\mathsf{p}}

\newcommand{\Du}{\boldsymbol{\mathsf{u}}}
\newcommand{\Dv}{\boldsymbol{\mathsf{v}}}
\newcommand{\Dw}{\boldsymbol{\mathsf{w}}}

\newcommand{\SSp}{\SS_{\Dp}}

\newcommand{\ZZ}{\mathbb Z}
 \newcommand{\Rf}{\R f^{}_{\<\<*}}

\newcommand{\fsh}{f^{}_{\mkern-1.5mu\sharp}}
\newcommand{\fst}{{f^{}_{\<\<*}}}
\newcommand{\tbw}{\raisebox{2pt}{\scalebox{.9}{$\textstyle\bigwedge$}}}

\newcommand{\cc}{\mathsf{c}}

\newcommand{\qc}{\mathsf{qc}}
\newcommand{\qct}{\mathsf{qct}}
\newcommand{\Dqc}{\D_{\mathsf{qc}}}

\newcommand\Dqcpl{\D_\qc^{\lift.95,\text{\cmt\char'053},}}

\newcommand\Dcpl{\D_\cc^{\lift.95,\text{\cmt\char'053},}}
\newcommand\Dcmi{\D_\cc^{\lift.95,\text{\cmt\char'055},}}
\newcommand\upl{{\lift1,\text{\cmt\char'053},}}
\newcommand\dpl{{\raisebox{-1pt}{\scalebox{.85}{\cmt\char'053}}}}
\newcommand\dpll{{\raisebox{-1pt}{\scalebox{.75}{\cmt\char'053}}}}
\newcommand\ul[1]{\underline{#1}}

\font\cmt=cmtex10
      
\newcommand{\CB}{\mathcal B}

\newcommand{\cd}[1]{\mathcal D_{\!#1}}

\newcommand{\CH}{\mathcal H}
\newcommand{\CI}{\mathcal I}
\newcommand{\CO}{\mathcal O}

\newcommand{\CT}{\mathcal T}

\newcommand{\prj}{\pi}

\newcommand{\bpic}{\begin{tikzpicture}}
\newcommand{\epic}{\end{tikzpicture}}

\newcommand{\Otimes}[1]{\otimes^\LL_{#1}}

\newcommand{\sHom}{\CH om}

\newcommand{\Adot}{E}
\newcommand{\Bdot}{F\>}
\newcommand{\Cdot}{G\>}

\newcommand{\set}{\!:=}

\newcommand{\sX}{{\<\<X}}

\newcommand{\sst}{\scriptstyle}
\newcommand{\sss}{\scriptscriptstyle}

\newcommand{\smallcirc}{{\>\>\lift1,\sst{\circ},\,}}

\newcommand{\ssscirc}{\lift.8,\,{\sss{\circ}\,},}
\newcommand{\<}{\mkern-1mu}
\renewcommand{\>}{\mkern1mu}
\newcommand{\va}[1]{\vspace{#1pt}}

\newcommand{\halfsize}[1]{{\scalebox{.6}{#1}}}

\newcommand{\kf}{\kern.5pt}

\def\lift#1,#2,{\vbox to 0pt{\vskip-#1 ex\hbox{$\scriptstyle #2$}\vss}}

\newcommand{\OX}{\mathcal O_{\<\<X}}
\newcommand{\OY}{\mathcal O_Y}
\newcommand{\OZ}{\mathcal O_{\<Z}}
\newcommand{\OS}{\mathcal O_{\mkern-1.5mu S}}
\newcommand{\OW}{\mathcal O_W}
\newcommand{\OV}{\mathcal O_V}

\newcommand{\couni}[1]{{\smallint}_{#1}}

\newcommand{\bchasterisco}[1]{\theta_{\<#1}}

\newcommand{\bchadmirado}[1]{\mathsf{B}_{\mkern-.5mu#1}}

\newcommand{\defnu}[1]{\phi^{}_{#1}}

\newcommand{\fc}[1]{{\mathsf{c}}_{#1}}
\newcommand{\fundamentalclass}[1]{{\boldsymbol{\mathsf{c}}}_{#1}} 
\newcommand{\fundamentalclassa}[1]{{\boldsymbol{\mathsf{a}}}_{#1}}

\newcommand{\fundamentalclassb}[1]{{\boldsymbol{\mathsf{b}}}_{#1}}

\newcommand{\buE}[4]{{\ul{\operatorname{HH}}}^{#1}({#3} \xto{#2} {#4})}
\newcommand{\bisub}[1]{{\mathsf{#1}}_{\sharp}}

\newcommand{\biup}[1]{{{\mathsf{#1}}}^{\sharp}}

\newcommand{\bbisub}[1]{\mathbf{{\mathsf{#1}}}_{\sharp}}
\newcommand{\bbiup}[1]{\mathbf{{\mathsf{#1}}}^{\sharp}}

\newcommand{\bbbiup}[1]{\boldsymbol{#1}^{\sharp}}

\newcommand{\pf}[1]{{#1}^{}_{\<\<{\star}}}
\newcommand{\pb}[1]{{#1}^{\star}}

\newcommand{\gyf}[1]{{#1}^{}_{{\mkern-1.5mu{\mathsf{c}}}}}
\newcommand{\gyb}[1]{{#1}^{{\>\mathsf{c}}}}
\newcommand{\guf}[1]{{#1}^{}_{{\mkern-1.5mu\ul{\mathsf{c}}}}}
\newcommand{\gub}[1]{{#1}^{{\ul{\mathsf{c}}}}}

\newcommand{\circled}[1]{\textcircled{\scriptsize{#1}}}
\newcommand{\circledd}[1]{\textcircled{\raisebox{.4pt}{\scriptsize{#1}}}}

\newcommand{\Hsch}[1]{\CH_{\mathstrut#1}}
\newcommand{\Hschr}[2]{\CH_{\mathstrut#1|#2}}
\newcommand{\Huch}[1]{\>\ul{\<\CH\<}_{\>\mathstrut#1}}
\newcommand{\Hh}[1]{{\boldsymbol\CH}_{\mathstrut#1}}
\newcommand{\Hhr}[2]{{\boldsymbol\CH}_{\mathstrut#1|#2}}

\newcommand{\lto}{\longrightarrow}
\newcommand{\xto}{\xrightarrow}

\newcommand\iso{{\mkern8mu\longrightarrow \mkern-25.5mu{}^\sim\mkern17mu}}

\DeclareMathOperator{\liso}{\tilde{\lto}}

\DeclareMathOperator{\h}{H}

\DeclareMathOperator{\spec}{Spec}
\DeclareMathOperator{\Hom}{Hom}

\DeclareMathOperator{\ext}{Ext}
\DeclareMathOperator{\stor}{\CT\!\mathit{or}}
\DeclareMathOperator{\id}{id}

\DeclareMathOperator{\HH}{HH}

\DeclareMathOperator{\via}{{\textup{via}}}

\DeclareMathOperator{\ps}{\mathsf{ps}}

\newcommand{\ie}{{i.e.,} }
\newcommand{\cfr}{{\rm see} }

\def\Iso{\vbox to 0pt{\vss\hbox{$\widetilde{\phantom{nn}}$}\vskip-7pt}}
\def\Isoo{\vbox to 0pt{\vss\hbox{$\widetilde{\phantom{m}}$}\vskip-7pt}}

\newcommand{\Gam}[1]{\Gamma_{\!\<#1}}
\newcommand{\Gams}[1]{\Gamma^*{}_{\mkern-14mu#1}\>\>}

\begin{document}

\title[Bivariance, Grothendieck duality, Hochschild homology, II]{Bivariance, Grothendieck duality and  Hochschild homology, II:\\ the fundamental class of a flat scheme-map }

\author[L. Alonso]{Leovigildo Alonso Tarr\'{\i}o}
\address{Departamento de \'Alxebra\\
Facultade de Matem\'a\-ticas\\
Universidade de Santiago de Compostela\\
E\kern1pt-15782  Santiago de Compostela, SPAIN}
\email{leo.alonso@usc.es}
\urladdr{http://webspersoais.usc.es/persoais/leo.alonso/index-en.html}

\author[A. Jerem\'{\i}as]{Ana Jerem\'{\i}as L\'opez}
\address{Departamento de \'Alxebra\\
Facultade de Matem\'a\-ticas\\
Universidade de Santiago de Compostela\\
E\kern1pt-15782  Santiago de Compostela, SPAIN}
\email{ana.jeremias@usc.es}

\author[J. Lipman]{Joseph Lipman}
\address{Department of Mathematics\\
Purdue University\\
West Lafayette IN 47907, USA}
\email{lipman@math.purdue.edu}
\urladdr{http://www.math.purdue.edu/\~{}lipman/}

\thanks{Authors partially supported by 
Spain's MICIIN and E.U.'s FEDER research projects MTM2008-03465 and MTM2011-26088. 
First and second authors partially supported by Xunta de Galicia projects PGIDIT10PXIB207144PR
and GRC2013-045.
Third author partially supported over time by NSF and NSA}


\begin{abstract} Fix a noetherian scheme $S$. For any flat map \mbox{$f\colon X\to Y$} of separated 
essentially-finite\kf-type perfect $S$-schemes we define a canonical derived-category map $\Dc^{}_{\<f}\colon\Hsch{\sX}\to f^!\Hsch{Y}$, the \emph{fundamental class of}~$f\<$, where $\Hsch{Z}$~is the (pre\kf-)Hochschild complex of an $S$-scheme $Z$ and \smash{$f^!$}~is the twisted inverse image coming from Grothendieck duality theory. When $Y=S$ and $f$ is essentially smooth of relative dimension~$n$, this gives an isomorphism $\Omega^n_f[n]=H^{-n}(\Hsch{\sX})[n]\iso f^!\OS$. We focus mainly on \emph{transitivity} of\kf~$\Dc$
vis-\`a-vis compositions $X\to Y\to Z$, and on the compatibility of $\Dc$ with \emph{flat base change}. These properties imply that $\Dc$~orients the flat maps in the bivariant theory of part I \cite{AJL}, compatibly with essentially \'etale base change. Furthermore, $\Dc$ leads to a \emph{dual} oriented bivariant theory, whose homology is the  classical Hochschild homology of  flat $S$-schemes. 
When $Y=S$, $\Dc$ is used to define a duality map 
$\fd_\sX\colon\Hsch\sX\to\R\sHom(\Hsch\sX,f^!\OS$), an isomorphism if $f$ is essentially smooth.
These results apply in~particular to flat 
essentially-finite\kf-type maps of noetherian rings.

\end{abstract}

\maketitle

\tableofcontents

\section*{Introduction} 

\enlargethispage*{2pt}
\begin{cosa}\label{begin intro}
In the prequel \cite{AJL} we developed a \emph{bivariant theory} on the category of separated, essentially finite\kf-type, perfect (i.e., finite tor-dimension) schemes
 $x\colon X\to S$ over a fixed noetherian base scheme $S$. The theory is based on properties of
 the (pre\kf-)Hochschild complex  $\Hsch{x}\set
\LL{\delta}_{\lift1,x,}^{\lift.9,*,}\R{\delta}_{x*}\OX$ where the map \mbox{$\delta_x \colon  X\to X \times_S X$} is the diagonal (\S\ref{preHoch} below), and of the twisted inverse image pseudofunctor $(-)^!$ from Grothendieck duality theory. It associates to a morphism $f\colon (X\xto{\lift.5,x,\>}S)\to (Y\xto{\lift.7,y,\>}S)$ of such $S$-schemes the graded group\va{-2}
\[
\HH^*(f)\!:=\oplus_{i\in\ZZ}\> \ext^i_{\OX}\<(\Hsch{x},f^!\>\Hsch{y})
=\oplus_{i\in\ZZ}\> \Hom_{\D(X)}\!\<\big(\Hsch{x},f^!\>\Hsch{y}[i\>]\big),
\]
\vskip-2pt
\noindent 
so that the associated cohomology groups are\va{-1}
 \[
\HH^i(X|S) \!:= \HH^i(\id_{\<\<X}\<) = \ext^i_{\OX}\<(\Hsch{x},\Hsch{x})
\] 
\vskip-2pt
\noindent
and the associated homology groups are\va{-1}
\[
\HH_i(X|S) \!:= \HH^{-i}(x) = \ext^{-i}_{\OX}\<(\Hsch{x},\>x^!\CO_S).
\]

Before proceeding, let us emphasize that  being able (thanks to \cite{Nk}) to work with \emph{essentially} finite\kf-type maps, one sees, upon consideration of affine schemes, that \emph{the preceding results, and those that follow, hold, in particular, in the context of local commutative algebra.}\va2

\end{cosa}

\begin{cosa}\label{0.2}
 In this paper we prove some basic properties of 
 the \emph{fundamental class of a flat map} $f\colon X\to Y$ of $S$-schemes $x \colon X \to S$, $y\colon Y\to S$   as above.  Having fixed $S$, we'll often set $\delta_{\<\<X}\set\delta_x\>$, $\Hsch{\sX}\set\Hsch{x}\>$. With such notation,
 the fundamental class of $f$ is a natural functorial map, defined in \S\ref{fundclassflat}, 
\begin{equation}\label{Fund}
\boxed{
\fundamentalclass{\<f}\colon \LL{\delta^*_{\<\<X}}\R{\delta^{}_{\<\<X\<\<\lift1,*,}} \LL f^*\to 
f^!\LL{\delta^*_Y}\R\delta^{}_{Y\<\<\lift1,*,}\>\>.}
\end{equation}

The map $\fundamentalclass{\<f}$ entails a natural map
$
\fundamentalclass{\<f}(\OY)\colon \Hsch{\<X}\to f^!\>\Hsch{Y}.
$
Thus one has the canonical element 
$$
c_{\<f}\!:=\fundamentalclass{\<f}(\OY)\in\HH^0(f).
$$

In particular, when $y=\id_S$, one gets a map
in $\HH^0(x)=\HH_0(X|S)$:
\begin{equation}\label{fund2}
c_{x}\colon \Hsch{\<X}\to x^!\OS\>.
\end{equation}
In this case there is a natural (up to sign) $\OX$-isomorphism 
$$
\Omega^1_{X|S}\iso \stor_1^{X\times_S X}(\OX,\OX)=H^{-\<1}\>\Hsch{\<X},
$$
whence, by means of  the standard 
alternating graded $\OX$-algebra structure on $\oplus_{i\ge 0}\stor_i^{X\times_S X}(\OX,\OX)$, 
the universal property of exterior algebras gives rise to natural maps
\begin{equation*}
\Omega^i_{X|S}\to \stor_i^{X\times_S X}(\OX,\OX)=H^{-i}\>\Hsch{\<X}\qquad(i\ge 0).
\end{equation*}
Composing these with the maps $H^{-i}\>\Hsch{\<X}\to H^{-i}\>x^!\mathcal O_S$ induced by \eqref{fund2},
one gets natural maps of coherent sheaves
$$
\Omega^i_{X|S}\to H^{-i}\>x^!\mathcal O_S\qquad(i\ge 0).
$$
In particular, if $x$ is equidimensional of relative dimension $n$, one gets a map\looseness=-1
\begin{equation*}
 \Omega^n_{X|S} \to \omega_{X|S}\!:=H^{-n}x^!\OS,
\end{equation*}
where $\omega_{X|S}$ is the relative dualizing (or canonical) sheaf associated to $x$; or equivalently,
a derived-category map 
\begin{equation}\label{fundomega}
\boxed{C_{X|S} \colon \Omega^n_{X|S}[n] \to x^!\CO_S.}
\end{equation}
This byproduct of \eqref{Fund} is what has usually been regarded in the literature as the fundamental class.
\end{cosa}
\begin{cosa}

The central concern of this paper is with Theorem~\ref{trans fc}, which asserts \emph{transitivity} of the fundamental class vis-\`a-vis composition of flat $S$-maps $X\xto{u\>}Y\xto{v\>} Z$; that is,
\[
\boxed{\fundamentalclass{vu} = u^!\fundamentalclass{v} \smallcirc \>\fundamentalclass{u}v^*.}
\]

\goodbreak
Transitivity gives in particular that 
$\fundamentalclass{vu}(\OZ)=u^!\fundamentalclass{v}(\OZ)\>\smallcirc\>\>\fundamentalclass{u}(\OY)$.
\vspace{1pt}
In terms of the bivariant product $\HH^0(u)\times\HH^0(v)\to \HH^0(vu)$, this says:
$$
c_{vu}=c_{u}\cdot c_{v}.
$$
Thus the family $c_{f}$ is a family of \emph{canonical orientations} for 
the flat maps  in our bivariant theory \cite[p.\,28, 2.6.2]{fmc}.

When $f$ is essentially \'etale, so that $f^!=f^*\<$, $\Dc_{\<f}$ turns out, nontrivially (see Proposition~\ref{c inverse}), to~be inverse to the ``Hochschild localization isomorphism" of Theorem~\ref{caset}.
From this, and transitivity, it follows that \emph{the above orientations are 
compatible with essentially \'etale base change,} see Corollary~\ref{bch-fundamental-class}.

With all this in hand, one can apply the general considerations in \cite{fmc} to obtain, for example,  
\emph{Gysin morphisms,} that provide ``wrong-way" functorialities for homology and 
cohomology (see \S\ref{Gysin}).
\end{cosa}

\begin{cosa}
Some other applications of the fundamental class are given in \S\ref{dual theory}.  We construct an oriented bivariant theory \emph{dual} to the one mentioned above, having the same associated cohomology groups, but whose associated homology groups are the classical Hochschild homology 
groups---given by
the homology of the derived global sections of the Hochschild complex. Also, combining~$c_x$ with
the usual product\va{-1.5} $\Hsch{\sX}\<\Otimes{\sX}\<\Hsch{\sX}\to\Hsch{\sX}$ leads to a map 
\mbox{$\Hsch\sX\to\R\sHom_\sX(\Hsch\sX,x^!\OS)$.} This is an \emph{isomorphism} whenever $x$ is essentially smooth; and if, moreover, the base scheme
$S=\spec H$ with $H$ a Gorenstein artinian ring, 
there results a nonsing\-ular pairing of the classical homology groups into $H$. Presumably (though we have no proof) this pairing is closely related to the Mukai pairing of C\u{a}ld\u{a}raru.
\end{cosa}

\begin{cosa} 
The proof of the transitivity property, Theorem~\ref{trans fc}, is given in \S\ref{provetran}, which occupies more than one third of this paper.
The reason for the length is that the fundamental class is defined to be the composition of a dozen or so maps,
some of which are themselves composed of more elementary maps. Transitivity means roughly that a juxtaposition of two such sequences of maps can be transformed into another such sequence; and this is shown by justifying and combining \emph{many} transformations of subsequences. 

Put differently, the Theorem asserts commutativity of a square whose sides are composed of a dozen or so maps; and the strategy for proving this is to decompose this large diagram
into smaller ones, and then decompose the smaller ones into still smaller ones, and so on, until
the original diagram is decomposed into \emph{many} tiny ones, whose commutativity holds for
elementary reasons.  We found carrying this process to completion extremely tedious, and not at all straightforward, as any reader who sets out to check the details in \S\ref{provetran} will soon see. (And, preliminaries aside,  not all the details appear there: some of the easier ones are left to the reader, and for quite a few others, reference is made to \cite{li}.)

Can the proof be made more palatable?  No doubt some technical and organizational improvements are possible; but we suspect such improvements would not have a major effect. Some kind of \emph{coherence theorem}---beyond those presently available---might guarantee the commutativity of numerous diagrams in the proof, making much of the minute examination superfluous. 
Unearthing such a theorem, or a different conceptual approach, remains a challenge.
\end{cosa}
 
 \enlargethispage*{3pt}
\begin{cosa}
In any case, why bother? To respond, let us give the fundamental class some historical context, and mention a number of problems and potential applications for further study.

The fundamental class links  the concrete and abstract approaches to Grothendieck duality (see, respectively, \cite{Co} and \cite{li}).
The correspondences between these two approaches  are generally taken for granted; but full justifications are not readily available in the literature.
For example, for~\emph{smooth} morphisms of noetherian schemes, in
the concrete approach, the map~\eqref{fundomega}, an isomorphism in this case, exists more or less by definition; the point is then to show that  the top-degree differentials satisfy a suitable generalization of Serre duality, see \cite[Chapter VII, \S4]{RD}. In the abstract approach of Verdier and Deligne, where duality is proved directly, such an isomorphism  comes out of the flat base\kf-change theorem and the fundamental local isomorphism for complete intersections, see \cite[p.\,397, Theorem 3]{V}. Does Verdier's construction, when interpreted in concrete terms, yield the concretely-defined
isomorphism? And does his isomorphism behave pseudo\-functorially with respect to smooth maps?

If $x\colon X\to S$  is essentially smooth of relative dimension~$n$, 
then \emph{using Verdier's isomorphism}, we show in Proposition~\ref{fc+V} that \eqref{fundomega} is
also an isomorphism.  But we don't know whether these two isomorphisms are the same,
even up to sign.

More generally (at least in characteristic zero), 
in \cite{ez} and \cite{aez} El Zein and Ang\'eniol  associate to any 
noetherian $\mathbb Q$-scheme $S$ and any morphism $x \colon X \to S$ 
that is  as above, and also equidimensional of relative dimension~$n$, 
a derived-category map\va{.6}
\(
\gamma^{}_{X|S}\colon\Omega^n_{X|S}[n] \to x^!\CO_S.
\)
In~\cite{ang}, Ang\'eniol uses this map for his  treatment of Chow schemes.

\pagebreak[3]
When  $S = \spec(k)$ with $k$ a perfect field, and $X$ is an integral algebraic scheme over $k$, a map like $\gamma^{}_{X|k}$ is  realized in \cite{blue} as a globalization of the local
residue maps at the points of $X\<$,  leading to explicit versions of local and global duality and the relation between them. These results are generalized to certain maps of noetherian schemes
in \cite{HS}. 

How is $\gamma^{}_{X|S}$\va{-1} related to \eqref{fundomega}? For this, one will have to explicate the relation
between \eqref{fundomega} and the characterizing ``trace property" of $\gamma$ (cf.~ \cite[p.\,114, 7.1.3]{ang}, \cite[p.\,50, 5.2.8 and p.\,55, 6.3.1 ]{anl}.) A small step toward this is taken in Example~\ref{c and trace} below.

In all these treatments, an important role is played---via factorizations  of~$x$ as
smooth$\smallcirc$finite\kf---by the case  $n=0$, where  the notion of  fundamental class is equivalent to   that of traces of differential forms. This leads to a concrete realization of the fundamental class in
terms of \emph{regular differential forms},  an algebraic treatment  of which is given in \cite{KW}.

For more recent developments, in the context of complex spaces, see \cite{Kd}.)\va1

Finally, some vague remarks about possible future projects. One should clarify the connection between the fundamental class and Verdier's isomorphism (see above).  More generally, 
one should explicate some concrete aspects of the fundamental class in terms of differential modules, or perhaps cotangent complexes, via their relation to Hochschild complexes, especially in characteristic zero (see, e.g., \cite{BF2}). 

As indicated above, there is a close relation between the fundamental class and \emph{residues}.
This becomes clearer over formal schemes, where local and global duality merge into a single
duality theory, of which fundamental classes and residues are adjoint aspects. From this viewpoint,
the transitivity theorem for smooth (resp.~finite) maps should be closely related to the properties
(R4) and (R10) of residues given in \cite[pp.\,198--199]{RD}. 

If the theory of the fundamental class could be extended from flat maps to perfect maps, then (R3) could be added to this list. More importantly, such an extension would be desirable for dealing bivariantly with arbitrary finite\kf-type maps between smooth $S$-schemes. It may involve differential-graded and simplicial methods, as in~\cite{BF}, or perhaps cotriples. 

\end{cosa}

\section{The (pre\kf-\kern-.75pt)Hochschild functor}\label{preHoch}

Let $f \colon X \to Y$ be any scheme\kf-map, with associated diagonal  map
$$
\delta=\delta_{\<\<f}: X\to X\times_Y X.
$$

The  \emph{pre\kf -Hochschild functor} of $f$ is 
$$
\Hh{\<f}\set\LL\delta^*\R \delta_{\<*} \colon\D(X)\to\D(X),
$$
where $\D(X)$ is the derived category of (sheaves of) $\OX$-modules.

The \emph{pre\kf -Hochschild complex} of $f$  is
$$
\Hsch{\<f}\set\LL\delta^*\R\delta_*\OX.
$$

When $f$ is flat,  the prefix ``pre\kf-" is omitted, see~\cite[p.\,222, 2.3.1]{BF}.

We'll often use the less precise notations $\Hhr{X}{Y}$ for $\Hh{\<f}$ and 
$\Hschr{X}{Y}$ for $\Hsch{\<f}$.\va1

This section contains some basic facts about $\Hh{\<f}$ and $\Hsch{\<f}$ that are needed in the subsequent treatment of fundamental class maps.
The key points are Corollary~\ref{composedsquare} (transitivity for $\Hh{\<f}$) and\va1
Theorem~\ref{caset} (essentially \'etale localization for $\Hh{\<f}$, generalizing to the present setting a result of Geller and~Weibel \cite[Theorem (0.1)]{gwet}). 

In \S\S\ref{efp}--\ref{upsilon}  we review some necessary preliminaries.\va1
 
 Then in \S\ref{variance} we discuss  the variance\va1 of~$\Hh{\<f}$  with $f\<$, and in
particular, its compatibility with flat base-change
(Corollary~\ref{fibersquareiso}) and its\va{-2} transitivity.
As special cases of variance\va1 one has, for scheme\kf-diagrams $X\xto{f}Y\xto{g}Z$,
homomorphisms\looseness=-1
$$
\Hschr{X}{Z} \to \Hschr{X}{Y}\>,\qquad   
\LL f^*\Hschr YZ \to \Hschr XZ\>\>, \qquad
\Hschr YZ \to   \Rf\Hschr XZ\>\>,
 $$
 the third being adjoint to the second (Example~\ref{functoriality}).

\begin{cosa}\label{efp}
The term \emph{qcqs}, adjectivally modifying ``scheme" or ``map,"  will be used as an abbreviation for   \emph{quasi-compact and quasi-separated} (see \cite[\S6.1]{EGA1}). (In the oft-to\kf-be\kf-used
reference \cite{li}, qcqs is called \emph{concentrated}.)

Any scheme\kf-map with noetherian source is qcqs.

A scheme\kf-map  $f\colon X\to Y$ is \emph{essentially of finite
presentation} (efp) if it is qcqs and if for all $\xi\in X$ there exist affine open neighborhoods 
$\spec L$ of~$\xi$ and $\spec K$ of~$f(\xi)$ such that $L$ is a ring of fractions of a finitely-presentable $K$-algebra. If $f$ is qcqs and for each $\xi$ there are such $K$ and $L$ with $L$ a ring of fractions of $K$ itself, then $f$ is said to be \emph{localizing}.

When $X$ and $Y$ are noetherian, one can use for \mbox{``finitely-presentable"} the equivalent term ``finite\kf-type." \va1

The map $f\colon X\to Y$ is \emph{essentially smooth} (resp.~\emph{essentially \'etale}, resp. \emph{essentially unramified}) if $f$ is efp
and formally smooth (resp.~formally \'etale, resp.~formally unramified), see \cite[\S17.1]{EGA4}. 

When $f$ is essentially smooth, the module of relative differentials $\Omega^1_{\<f}$ is locally free of finite rank, say, $n_{\<f}$,
where $n_{\<f}$, the \emph{relative dimension} of $f$, is a function from~$X$ to~$\mathbb Z$,  constant on connected components. (For local projectivity, see \cite[(16.10.2)]{EGA4}; and for finiteness, see, e.g., the proof of \ref{caset} below.)  Moreover, if $\>Y$  is noetherian then the diagonal map $X\to X\times_Y X$ is a \emph{regular immersion}:  each $\xi\in X\subset X\times_Y X$ has an open neighborhood $U\subset X\times_Y X$ such that 
$\Gamma(U\cap X\<,\>\OX)$ is a quotient of  $\>\Gamma(U,\CO_U)$ by a regular sequence of length $n_{\<f}(\xi)$, see \cite[\S16.9]{EGA4}.

 Let $X\xto{f}Y\xto{g} Z$ be scheme\kf-maps, where $g$ is qcqs (resp.~efp).
One verifies then, via \cite[\S6.1]{EGA1}, that
$f$ is qcqs (resp.~efp) if and only if so is $gf$;
and if~$Z'\to Z$ is any scheme\kf-map then the projection
$Z'\times_Z Y\to Z'$ is qcqs (resp.~efp). It follows that the fiber product, in the category of schemes,
of any two qcqs (resp.~efp) maps with the same target, is also a fiber product in the subcategory of 
qcqs (resp.~efp) maps.

Thus if $f$ and $g$ are qcqs (resp.~efp) then so is the graph $\Gamma_{\!\<f}\colon X\to X\times_Z Y$.\va1

 Similar assertions hold with ``separated" (resp.~``essentially \'etale") in place of ``efp", see \cite[p.\,279, (5.3.1)]{EGA1} (resp.~ \cite[(17.1.3)(ii) and (iii), (17.1.4)]{EGA4}).

 From \cite[Theorem (17.5.1)]{EGA4} it follows that any essentially smooth map---in particular, any essentially \'etale map---is flat.

\end{cosa}

\begin{cosa}\label{derived}

For any scheme $W$ let $\D(W)$ be the derived category of $\OW$-modules; and let 
$$
\D_W\set\Dqc(W)\subset\D(W)
$$  
be the full subcategory with objects the complexes whose homology sheaves are all quasi-coherent.
When $W$ is qcqs, the~natural functor into~$\D_W$ from the derived category of quasi-coherent $\OW$-modules is an equivalence of categories \cite[p.\,230, 5.5]{BN}.

As we will deal almost exclusively with derived functors, we will usually
lighten notation by omitting the symbols $\LL$ and $\R\>$: given a \mbox{scheme\kf-map} 
$f\colon X\to Y$, we'll write $\fst\colon \D(X)\to \D(Y)$ for 
$\R f^{}_{\<\<*}$,  $f^*\colon \D(Y)\to \D(X)$ for~$\LL f^*$,
and $\otimes^{}_{\<\<X} \colon \D(X)\times\D(X)\to \D(X)$ for the left-derived functor $\Otimes{\<\<X}$. In~the presence of such abbreviations, it~should not be forgotten that we are  working with derived functors, unless otherwise indicated.\va2

\emph{Remarks.} 
$\bullet$ Derived functors are determined up to canonical isomorphism, by universal properties. We assume throughout that some specific choice of such functors has been made. As we make use only 
of the characteristic universal properties, our results do not depend on the choice. 

In this vein, we always assume that $\id_X^*$ and $(\OX\otimes^{}_{\<\<X}-)$ are identity functors, and that
for $f\colon X\to Y$ that\va1 $f^*\OY=\OX$.\looseness=-1

$\bullet$ For any scheme\kf-map $f\colon X\to Y$, one has $f^*\D_Y\subset \D_\sX$ \cite[ 3.9.1]{li}, and if $f$ is qcqs (\S\ref{efp}) then $\fst\D_\sX\subset\D_Y$ \cite[3.9.2]{li}. 
Hence if $f$ is qcqs then $\Hh{\<f}\>\D_\sX\subset \D_\sX$.\va1

$\bullet$ For  qcqs $f$, there is a canonical functorial isomorphism (cf.~\eqref{def-of-mu}):
\begin{equation*}\label{Hsch and Hh}
\smash{\zeta(G)\colon \Hsch{\<f}\Otimes{\<\<X}G
\iso 
\Hh{\<\<f}(G)
\qquad(G\in\D_\sX).}
\end{equation*}
As this  will not be used in what follows, we'll say no more about it.\va1

$\bullet$ In this paper, the functors and functorial maps that appear respect
the usual triangulated and graded structures on $\Dqc$ (see, e.g.,
\cite[\S\S5.2, 5.7]{AJL}). That fact will play no role.\va1

\end{cosa}

\goodbreak

\begin{cosa}\label{adj-pseudo}
On the category of schemes there are \emph{adjoint monoidal pseudofunctors} $(-)^*$ and~$(-)_*$
(the first contravariant and the second covariant) assigning to any map $f\colon X\to Y$
the functors $f^*$ and $\fst$   
in~\S\ref{derived} (see \cite[3.6.10]{li}).
Adjoint\-ness   means there are functorial unit and~counit maps
\begin{equation}\label{etaeps}
\eta=\eta^{}_{\<f}\colon\id \to f^{}_{\<\<*}f^*\quad\text{ and }\quad\epsilon=\epsilon^{}_{\<\<f}\colon f^*\<\<f^{}_{\<\<*}\to\id
\end{equation}
such that for $A\in\D_\sX$ and $C\in\D_Y$ the corresponding compositions 
$$
f^{}_{\<\<*}A \xto{\eta^{}_{\<f^{}_{\<\<*}{\<A}}\>\>}  f^{}_{\<\<*}f^*\<\<f^{}_{\<\<*}A \xto{f^{}_{\<\<*}\epsilon^{}_{\!A}\>\>} f^{}_{\<\<*}A,\qquad
f^*\<C \xto{f^*\<\eta^{}_{C}\>} f^*\<\<f^{}_{\<\<*}f^*\<C \xto{\epsilon^{}_{\<\<f^*\<C}} f^*\<C
$$
are identity maps. Pseudofunctoriality\va{-1} of $(-)^*$ and $(-)_*$ entails, for any scheme\kf-maps
$X\xto{f}Y\xto{g\,} Z$,  isomorphisms 
\begin{equation}\label{psdef}
\ps^*\colon f^*\<g^*\iso(gf)^*,\qquad\ \ps_*\colon(gf)_*\iso g_*\fst,
\end{equation}
satisfying a kind of associativity vis-\`a-vis $X\xto{f}Y\xto{g\>\>} Z\xto{h\>\>}W$  
(see, e.g., \cite[p.\,120]{li});
and  pseudofunctoriality\va{-2} of the foregoing adjunction is expressed by commutativity,
for any $X\xto{f} Y\xto{g\>\>} Z$,
of  the diagram
\begin{equation*}\label{adjpseudo}
\CD
 \bpic[xscale=2.7,yscale=1.75]
   \node(11) at (1,-1){$\id$};  
   \node(12) at (2,-1){ $g_*g^*$};
   \node(13) at (3,-1){$g_*(f^{}_{\<*}f^*g^*)$};
  
   \node(21) at (1,-2){$(gf)_*(gf)^*$};
   \node(22) at (2.03,-2){$g_*f^{}_{\<*}(gf)^*$};
   \node(23) at (3,-2){$ g_*f^{}_{\<*}f^*g^*$};

   \draw[->] (11)--(12) node[midway, above, scale=0.75]{$\eta^{}_g$}; 
   \draw[->] (12)--(13) node[midway, above, scale=0.75]{$\text{via}\;\eta^{}_{\<f}$};
  
   \draw [double distance = 2pt](21)--(22) node[midway, below=2.5pt,scale=0.75]{$\ps_*$};
   \draw [double distance = 2pt](23)--(22) node[midway, below=.5pt,scale=0.75]{$\text{via}\;\ps^*$};
  
   \draw [->](11)--(21) node[midway, left=1pt,scale=0.75]{$\eta^{}_{g\<f}$};
   \draw [double distance = 2pt](13)--(23);
  \epic 
 \endCD
 \end{equation*}
For details, work backwards from \cite[p.\,124, 3.6.10]{li}.\va1

These pseudofunctors interact with
the left-derived tensor product $\otimes{}$ via a natural isomorphism
\begin{equation}\label{^* and tensor}
\nu^{}_{\<f\<\<,\>E\<,\>F}\colon f^*(E\otimes_YF)\iso
 f^*\<\<E\otimes_\sX f^*\<\<F
\qquad\big(E,F\in\D(Y)\big),
\end{equation}
see \cite[3.2.4]{li}; via the functorial map
\begin{equation}\label{_* and tensor}
f_{\<*} G\otimes_Yf_{\<*} H \to f_{\<*} (G\otimes_\sX H)
\qquad\big(G,H\in\D(X)\big)
\end{equation}
adjoint to the natural composite map
$$
 f^*(f_{\<*} G\otimes_Y\fst H)\underset{\eqref{^* and tensor}}\iso
 f^*\!\fst G\otimes_\sX  f^*\!\fst H\longrightarrow
G\otimes_\sX H;
$$
and via the functorial \emph{projection isomorphisms,} for \(F\in\D_Y,\;G\in\D_\sX\),
\begin{equation}\label{projection}
f_{\<*} G\otimes_Y F \iso f_{\<*}(G\otimes_\sX  f^*F),
\qquad
F\otimes_Y f_{\<*} G \iso f_{\<*}(f^*F\otimes_\sX G ),
\end{equation}
the first being defined qua map to be  the natural composition
$$
f_{\<*} G\otimes_YF\lto f_{\<*} G\otimes_Yf_{\<*} f^*F\underset{\eqref{_* and tensor}}\lto
f_{\<*}(G\otimes_\sX  f^*F),
$$
and similarly for the second, see \cite[3.9.4]{li}.\va2

The pseudofunctorially adjoint pair $\big((-)^*\<, (-)_*\big)$ is ultimately determined, by categorical properties, only up to unique isomorphism. The pair can be so chosen that for any
scheme\kf-map $f\colon X\to Y\<$, one has that $f^*\OY=\OX$, that  the map $\eta_f$ in \eqref{etaeps}  is the natural composition
$\OY\to f_*\OX\to \Rf\OX$ (where for just this moment, $f_*$ is the \emph{nonderived} direct image functor), that the map $\ps^*(\OZ)$ in \eqref{psdef} is the identity map of $\OX$,  that 
the map~\eqref{^* and tensor} is the obvious one when either $E$ or $F$ is $\OY$, and that 
the map~\eqref{projection} is the obvious one when $F=\OY$ (cf.~e.g., \cite[3.4.7(iii)]{li}). 

\end{cosa}

\begin{cosa}\label{bch*}
In a category, an \emph{orientation} of a relation $f{\smallcirc}\>v=u\>{\smallcirc}g$ among maps is  an ordered pair (right arrow, bottom arrow) whose members are $f$ and~$u$. This can be represented by one of two \emph{oriented commutative squares,} namely $\Dd$ with bottom arrow $u$, and its transpose $\Dd^\prime$ with bottom arrow $f$.
\[
 \begin{tikzpicture}[xscale=.95,yscale=.8]
       \draw[white] (0cm,0.5cm) -- +(0: \linewidth)
      node (E1) [black, pos = 0.21] {$\bullet$}
      node (F1) [black, pos = 0.39] {$\bullet$}
      node (E2) [black, pos = 0.61] {$\bullet$}
      node (F2) [black, pos = 0.79] {$\bullet$};
      \draw[white] (0cm,2.5cm) -- +(0: \linewidth)
      node (G1) [black, pos = 0.21] {$\bullet$}
      node (H1) [black, pos = 0.39] {$\bullet$}
      node (G2) [black, pos = 0.61] {$\bullet$}
      node (H2) [black, pos = 0.79] {$\bullet$};
      \node (C1) at (intersection of G1--F1 and E1--H1) [scale=0.8] {$\Dd$};
      \draw [->] (G1) -- (H1) node[above, midway, sloped, scale=0.75]{$v$};
      \draw [->] (E1) -- (F1) node[below, midway, sloped, scale=0.75]{$u$};
      \draw [->] (G1) -- (E1) node[left, midway, scale=0.75]{$g$};
      \draw [->] (H1) -- (F1) node[right, midway, scale=0.75]{$f$};
      \node (C2) at (intersection of G2--F2 and E2--H2) [scale=0.8] {$\Dd^\prime$};
      \draw [->] (G2) -- (H2) node[above, midway, sloped, scale=0.75]{$g$};
      \draw [->] (E2) -- (F2) node[below, midway, sloped, scale=0.75]{$f$};
      \draw [->] (G2) -- (E2) node[left, midway, scale=0.75]{$v$};
      \draw [->] (H2) -- (F2) node[right, midway, scale=0.75]{$u$};
 \end{tikzpicture}
\]

For any oriented commutative square of scheme\kf-maps 
\begin{equation*}
    \begin{tikzpicture}[xscale=0.9, yscale=0.75]
      \draw[white] (0cm,0.5cm) -- +(0: \linewidth)
      node (E) [black, pos = 0.41] {$Y^\prime$}
      node (F) [black, pos = 0.59] {$Y$};
      \draw[white] (0cm,2.65cm) -- +(0: \linewidth)
      node (G) [black, pos = 0.41] {$X^\prime$}
      node (H) [black, pos = 0.59] {$X$};
      \draw [->] (G) -- (H) node[above, midway, scale=0.75]{$1$};
      \draw [->] (E) -- (F) node[below, midway, scale=0.75]{$3$};
      \draw [->] (G) -- (E) node[left,  midway, scale=0.75]{$2$};
      \draw [->] (H) -- (F) node[right, midway, scale=0.75]{$4$};
      \node (C) at (intersection of G--F and E--H) [scale=0.8]{$\Dd$};
    \end{tikzpicture}
  \end{equation*}
the natural map of functors
       \begin{equation*}
         \label{def-of-bch-asterisco}
         \bchasterisco{\Dd}:  3^*4_*\to 2_*1\<^*,\tag{\ref{bch*}.1}
       \end{equation*}
is defined to be the composition of the following chain of maps of functors
(from $\D(X)$ to $\D(Y^\prime)$):
\begin{equation*}\label{bch.2}
         3^*4_*\xto{\!\!\textup{via}\;\eta^{}_2\>\>} 2_*2^*3^*4_*\overset{\lift1.31,\textup{via\:},\lift1.37,\ps^*,}{=\!=\!=} 2_*1\<^*4^*4_*\xto{\!\!\textup{via}\;\epsilon^{}_4\>\>} 2_*1\<^*,\tag{\ref{bch*}.2}
 \end{equation*}
or equivalently (see \cite[p.\,127, 3.7.2]{li}),
\begin{equation*}\label{bch.3}
         3^*4_*\xto{\!\!\textup{via}\;\eta^{}_1\>} 3^*4_*1_*1\<^*\overset{\textup{via\:}\ps_*}{=\!=\!=} 3^*3_*2_*1\<^*\xto{\!\!\textup{via}\;\epsilon^{}_3\>} 2_*1\<^*.\tag{\ref{bch*}.3}
  \end{equation*}

\addtocounter{equation}{3}
\begin{subcosa}\label{thetaiso}
If $\Dd$ is a  fiber square (\ie the naturally associated map is an isomorphism $X'\iso X\times_Y Y'$) with 4 qcqs (\S\ref{efp}) and 3  flat, then $\bchasterisco{\Dd}(G)$ is
an isomorphism for all $G\in \D_\sX$ (see \cite[p.\,142, Proposition~3.9.5]{li}). 
\end{subcosa}
\end{cosa}

\begin{cosa}\label{upsilon}
Let there be given an oriented commutative square of scheme\kf-maps 
\begin{equation*}
    \begin{tikzpicture}[xscale=0.9, yscale=0.75]
      \draw[white] (0cm,0.5cm) -- +(0: \linewidth)
      node (E) [black, pos = 0.41] {$Y^\prime$}
      node (F) [black, pos = 0.59] {$Y$};
      \draw[white] (0cm,2.65cm) -- +(0: \linewidth)
      node (G) [black, pos = 0.41] {$X^\prime$}
      node (H) [black, pos = 0.59] {$X$};
      \draw [->] (G) -- (H) node[above, midway, scale=0.75]{$1$};
      \draw [->] (E) -- (F) node[below=.5pt, midway, scale=0.75]{$3$};
      \draw [->] (G) -- (E) node[left,  midway, scale=0.75]{$2$};
      \draw [->] (H) -- (F) node[right, midway, scale=0.75]{$4$};
      \node (C) at (intersection of G--F and E--H) [scale=0.8]{$\Dd$};
    \end{tikzpicture}
  \end{equation*}
  With $\ps^*$ the natural isomorphism (\cite[p.\,118, (3.6.1)$^*$]{li}) and $\theta_{\Dd}$  as in \ref{bch*}, define
 \begin{equation}\label{def-of-phi}
\defnu{\<\Dd}\colon   1^{\<*}  4^* 4_*\lto   2^* 2_* 1^{\<*}
 \end{equation}
to be the following composition of functorial maps:
\begin{equation}\label{explicit def}
1^{\<*}  4^* 4_* \overset{\ps^*_{}}{=\!=} 2^*3^* 4_*  \xto{\!\<\<2^* \<\bchasterisco{\Dd}\mkern1.5mu} 2^* 2_* 1^{\<*}.
\end{equation}
\end{cosa}

\begin{subprop}
\label{ayuda0}
If\/ $\Dd$ is a fiber square in which the map\/~$4$  is qcqs and\/ $3$~is flat, then\/ $\defnu{\<\Dd}(G)$~is
an isomorphism for all\/ $G\in \D_\sX$.
\end{subprop}
\begin{proof} This holds because  $\bchasterisco{\Dd}(G)$ is
an isomorphism (see \ref{thetaiso}).
\end{proof}

\pagebreak[3]
Here is a transitivity property of $\phi$.

\begin{subprop}
\label{ayuda}
Let\/ $ \Dd=\Du\smallcirc\Dv$ be the composite oriented commutative square   
\[
   \begin{tikzpicture}
      \draw[white] (0cm,2.5cm) -- +(0: \linewidth)
      node (12) [black, pos = 0.34] {$\bullet$}
      node (13) [black, pos = 0.5 ] {$\bullet$}
      node (14) [black, pos = 0.66] {$\bullet$};
      \draw[white] (0cm,0.5cm) -- +(0: \linewidth)
      node (22) [black, pos = 0.34] {$\bullet$}
      node (23) [black, pos = 0.5 ] {$\bullet$}
      node (24) [black, pos = 0.66] {$\bullet$};
      \node (E) at (intersection of 12--23 and 22--13) [scale=0.8] {$\Dv$};
      \node (B) at (intersection of 13--24 and 23--14) [scale=0.8] {$\Du$};
      \draw [->] (12) -- (13) node[above, midway, scale=0.75]{$1$};
      \draw [->] (13) -- (14) node[above, midway, scale=0.75]{$5$};
      \draw [->] (22) -- (23) node[below, midway, scale=0.75]{$3$};
      \draw [->] (23) -- (24) node[below, midway, scale=0.75]{$7$};
      \draw [->] (12) -- (22) node[left,  midway, scale=0.75]{$2$};
      \draw [->] (13) -- (23) node[left,  midway, scale=0.75]{$4$};
      \draw [->] (14) -- (24) node[right, midway, scale=0.75]{$6$};
   \end{tikzpicture}
  \]
With\/ $\defnu{\<\Dd},$ $\defnu{\>\Dv}$ and\/ $\defnu{\Du}$  as  in\/ \eqref{def-of-phi}$,$ the following diagram commutes.
  \[
   \begin{tikzpicture}
      \draw[white] (0cm,2.8cm) -- +(0: \linewidth)
      node (12) [black, pos = 0.2] {$ (5\smallcirc \<1)^*   6^*  6_*$}
      node (14) [black, pos = 0.8] {$  2^* 2_*  (5\smallcirc \<1)^*$};
      \draw[white] (0cm,0.5cm) -- +(0: \linewidth)
      node (22) [black, pos = 0.2] {$  1^{\<*}  5^*   6^*  6_*$}
      node (23) [black, pos = 0.5] {$  1^{\<*}   4^*  4_*  5^*$}
      node (24) [black, pos = 0.8] {$  2^*  2_*  1^{\<*}  5^*$};
      \draw [->] (12) -- (14) node[above, midway, scale=0.75]{$\defnu{\<\Dd}$};
      \draw [->] (22) -- (23) node[below, midway, scale=0.75]{$1^{\<*}\defnu{\Du}$};
      \draw [->] (23) -- (24) node[below=1pt, midway, scale=0.75]{$ \defnu{\>\Dv}$};
      \draw [double distance=2pt]
                 (12) -- (22) node[left=1pt,  midway, scale=0.75]{$\ps^*$};
      \draw [double distance=2pt]
                 (14) -- (24) node[right=1pt, midway, scale=0.75]{$\>\>2^* 2_* \<\<\ps^*$};
   \end{tikzpicture}
  \]
\end{subprop}
\begin{proof}
Expand the diagram in question,  as follows:
  \[
   \begin{tikzpicture}
      \draw[white] (0cm,4.25cm) -- +(0: \linewidth)
      node (11) [black, pos = 0.1] {$(5\smallcirc \<1)^* 6^*6_*$}
      node (13) [black, pos = 0.5] {$2^*(7\smallcirc 3)^*6_*$}
      node (15) [black, pos = 0.9] {$2^*2_*(5\smallcirc \<1)^*$};
      \draw[white] (0cm,3cm) -- +(0: \linewidth)
      node (22) [black, pos = 0.2] {}
      node (23) [black, pos = 0.5] {$2^*3^*7^*6_*$};
      \draw[white] (0cm,1.75cm) -- +(0: \linewidth)
      node (31) [black, pos = 0.1 ] {$1^{\<*}5^* 6^*6_*$}
      node (32) [black, pos = 0.35] {$1^{\<*} 4^*7^*6_*$}
      node (34) [black, pos = 0.65] {$2^* 3^*4_*5^*$}
      node (35) [black, pos = 0.9 ] {$2^*2_*1^{\<*}5^*$};
      \draw[white] (0cm,0.5cm) -- +(0: \linewidth)
      node (43) [black, pos = 0.5 ] {$1^{\<*}4^*4_*5^*$};
      \draw [double distance=2pt]
                 (11) -- (13) node[auto, midway, scale=0.75]{$\ps^*$};
      \draw [->] (13) -- (15) node[auto, midway, scale=0.75]{$2^*\bchasterisco{\Dd}$};
      \draw [double distance=2pt]
                 (31) -- (32) node[below=1pt, swap, midway, scale=0.75]{$1^{\<*}\!\ps^*$};
      \draw [->] (32) -- (43) node[auto, swap, midway, scale=0.75]{$1^{\<*}4^*\bchasterisco{\Du}$};
      \draw [double distance=2pt]
                 (43) -- (34) node[auto, swap, midway, scale=0.75]{$\ps^*\>$};
      \draw [->] (34) -- (35) node[auto, swap, midway,
	  scale=0.75]{$2^*\bchasterisco{\>\Dv}$};
      \draw [double distance=2pt]
                 (32) -- (23) node[auto, midway, scale=0.75]{$\ps^*$};
      \draw [->] (23) -- (34) node[auto, midway, scale=0.75]{$2^* 3^*\bchasterisco{\Du}$};
      \draw [double distance=2pt]
                 (11) -- (31) node[left=1pt,  midway, scale=0.75]{$\ps^*$};
      \draw [double distance=2pt]
                 (13) -- (23) node[right=1pt, midway, scale=0.75]{$2^*\!\ps^*$};
      \draw [double distance=2pt]
                 (15) -- (35) node[right=1pt, midway, scale=0.75]{$\>\>2^*  2_* \<\<\ps^*$};
      \node (1) at (intersection of 22--23 and 11--32) [scale=1]{\circled{$1$}};
      \node (2) at (intersection of 22--23 and 34--15) [scale=1]{\circled{$2$}};
      \node (3) at (intersection of 23--43 and 32--35) [scale=1]{\circled{$3$}};
   \end{tikzpicture}
  \]

Commutativity of \circled{$1$} follows from associativity of pseudofunctoriality, of~\circled{$2$} follows from  transitivity for  $\bchasterisco{}$  (see \cite[ p.\,128, Proposition~{3.7.2}(iii)]{li}), and of~\circled{$3$} is obvious, whence the conclusion.
\end{proof}

\begin{cosa}\label{variance}
We examine the variance of $\Hh{\<f}$ with respect to $f$.

Given an oriented commutative square of scheme\kf-maps
\begin{equation*}\label{construccion0}
\CD
X'@>h>> X\\
@Vf'V\mkern45mu\scalebox{.9}{$\Dd$} V @VVfV\\
Y'@>>\lift1.2,g,>Y
 \endCD\tag{\ref{variance}.0}
\end{equation*}
 let $ \mathbf{\Dd_{\<\times}}$ be the oriented commutative square
\begin{equation*}\label{construccion1}
\CD
X'@>h>> X\\
@V\delta_{\<\<f'}V\mkern76mu\scalebox{.9}{$\Dd_{\<\times}$} V @VV\delta_{\<\<f}V\\
X'\<\times_{Y'}\<X'@>>\lift1.2,h\times h,>X\<\times_Y\<X
\endCD
\end{equation*}
and let $\biup{\Dd}\colon h^{\<*} \>\Hhr XY\to\Hhr {X'}{Y'}\> h^{\<*}$ be the functorial morphism
$\defnu{{\Dd_{\<\times}}}$:
\begin{equation}
\label{def-biup}
                     \biup{\Dd}\colon h^{\<*} \>\Hhr XY= h^{\<*}\mkern-.6mu  \delta_{\!f}^*  {\delta^{}_{\!f*}}
                     \xto{\lift1.6,\defnu{\<\<\Dd_{\halfsize{$\<\sst\times$}}},}
                         {\delta}_{\!f'}^* \> {\delta^{}_{\!f'\<*}}\mkern1.5mu h^{\<*}\<
                         =\Hhr {X'}{Y'}\> h^{\<*}.
\end{equation}
Also, let 
$$\bisub{\Dd} \colon  \Hhr XY  \to   
h_*  \Hhr {X'}{Y'}\> h^{\<*}
$$ 
be the adjoint of $\biup{\Dd}$, that is, the composition (with $\eta$ as in \eqref{etaeps})

\[
\Hhr XY\xto{\,\eta^{}_{h}\>\>} 
 h_* h^{\<*}\>  \Hhr XY
 \xto{\!h_*\biup{\Dd}\!}
 h_* \Hhr {X'}{Y'}\> h^{\<*}\<.
 \]
We define the map
\begin{equation*}\label{def-bbiup}
      \bbiup{d} \colon  h^{\<*}\>\Hschr{X}{Y} \to \Hschr{X'}{Y'}
\end{equation*}
to be the composition
\[      
h^{\<*}\>\Hschr{X}{Y} =   
h^{\<*}\> \Hhr{X}{Y}\OX
      \xto{\!\biup{\Dd}( \OX\<\<)}   
      \Hhr{X'}{Y'}\> h^{\<*}\OX     
 \underset{\textup{can}}{\iso}
       \Hhr{X'}{Y'}\> \CO_{\<\<X'} = \Hschr{X'}{Y'}\>,
\]
and let 
\begin{equation*}\label{adj bbiup}
\bbisub{d} \colon  \Hschr{X}{Y} \to   h_*\Hschr{X'}{Y'}
\end{equation*} 
be the corresponding adjoint map.
\end{cosa}

\smallskip
\begin{subcor}
\label{fibersquareiso}
If\/ $\Dd$ in\/ \eqref{construccion0} is  an oriented fiber square in which $g$ is~flat and\/~$f$ is qcqs,  then for all\/ $G\in\D_\sX,$ $\biup{\Dd}$ is an isomorphism
$$
h^{\<*}\> \Hhr{X}{Y}\>G\iso \Hhr{X'}{Y'}\> h^{\<*}G.
$$
In particular, $\bbiup{d}:  h^{\<*} \Hschr{X}{Y} \to \Hschr{X'}{Y'}$ is an isomorphism.
\end{subcor}
\begin{proof}
Since $\Dd$ is a fiber square, therefore so is $\Dd_\times$.

Since  $f$ is qcqs therefore so is $\delta_{\<\<f}$  \cite[p.\,294, (6.1.9)(i), (iii), and p.\,291, (6.1.5)(v)]{EGA1}. 

The projection $(X\times_YX)\times_Y Y' \to X\times_Y X$ is flat (since $g$ is), and its composition with the natural isomorphism 
$$
X'\times_{Y'}X'\iso(X\times_YX)\times_Y Y'
$$
is the bottom arrow of $\Dd_\times$, which is therefore  flat. 

So
the assertion results from Proposition~\ref{ayuda0}. 
\end{proof}

\vskip-3pt
(A more general result is given in Corollary~\ref{admissibleiso} below.)

\smallskip
\pagebreak[3]
Here is a transitivity property of $\biup{\Dd}$.
\begin{subcor}
\label{composedsquare} Let\/ $ \Dd$ be the composite oriented commutative square
  \[
   \begin{tikzpicture}[yscale=.98]
      \draw[white] (0cm,0.5cm) -- +(0: \linewidth)
      node (22) [black, pos = 0.3] {$S''$}
      node (23) [black, pos = 0.5] {$S'$}
      node (24) [black, pos = 0.7] {$S$};
      \draw[white] (0cm,2.8cm) -- +(0: \linewidth)
      node (12) [black, pos = 0.3] {$X$}
      node (13) [black, pos = 0.5] {$Y$}
      node (14) [black, pos = 0.7] {$Z$};
      \node (E) at (intersection of 12--23 and 22--13) [scale=0.8] {$\Dv_{\mathstrut}$};
      \node (B) at (intersection of 13--24 and 23--14) [scale=0.8] {$\Du$};
      \draw [->] (12) -- (13) node[above, midway, scale=0.75]{$f$};
      \draw [->] (13) -- (14) node[above, midway, scale=0.75]{$g$};
      \draw [->] (22) -- (23) node[below, midway, scale=0.75]{$$};
      \draw [->] (23) -- (24) node[below, midway, scale=0.75]{$$};
      \draw [->] (12) -- (22) node[left,  midway, scale=0.75]{$x$};
      \draw [->] (13) -- (23) node[left,  midway, scale=0.75]{$y$};
      \draw [->] (14) -- (24) node[right,  midway, scale=0.75]{$z$};
   \end{tikzpicture}
  \]
The following diagram commutes.
\[
   \begin{tikzpicture}[yscale=1.2]
      \draw[white] (0cm,2.8cm) -- +(0: \linewidth)
      node (12) [black, pos = 0.2] {$ (gf)^*\>\Hhr{Z}{S}$}
      node (14) [black, pos = 0.8] {$\Hhr{X}{S''}(gf)^*$};
      \draw[white] (0cm,0.5cm) -- +(0: \linewidth)
      node (22) [black, pos = 0.2] {$   f^* \<  g^*\> \Hhr{Z}{S}$}
      node (23) [black, pos = 0.5] {$   f^* \> \Hhr{Y}{S'}\>g^*$}
      node (24) [black, pos = 0.8] {$  \Hhr{X}{S''}\>f^*\<  g^*$};
      \draw [->] (12) -- (14) node[above, midway, scale=0.75]{$\biup{\Dd}$};
      \draw [->] (22) -- (23) node[below, midway, scale=0.75]{$f^*\biup{\Du}$};
      \draw [->] (23) -- (24) node[below, midway, scale=0.75]{$\biup{\Dv}$};
      \draw [double distance=2pt]
                 (12) -- (22) node[left=1pt,  midway, scale=0.75]{$\ps^*$};
      \draw [double distance=2pt]
                 (14) -- (24) node[right=1pt, midway, scale=0.75]{$\>\via \ps^*$};
   \end{tikzpicture}
\]
In particular, the following diagram commutes.
\[
   \begin{tikzpicture}[yscale=1]
      \draw[white] (0cm,2.8cm) -- +(0: \linewidth)
      node (12) [black, pos = 0.4] {$ (gf)^*\>\Hschr{Z}{S}$}
      node (14) [black, pos = 0.675] {$\Hschr{X}{S''}$};
      \draw[white] (0cm,0.5cm) -- +(0: \linewidth)
      node (22) [black, pos = 0.4] {$   f^* \<  g^*\> \Hschr{Z}{S}$}
      node (23) [black, pos = 0.675] {$   f^* \> \Hschr{Y}{S'}$};
      \draw [->] (12) -- (14) node[above, midway, scale=0.75]{$\biup{d}$};
      \draw [->] (22) -- (23) node[below, midway, scale=0.75]{$f^*\mkern-.5mu \biup{u}$};
      \draw [->] (23) -- (14) node[right, midway, scale=0.75]{$\biup{v}$};
      \draw [double distance=2pt]
                 (12) -- (22) node[left=1pt,  midway, scale=0.75]{$\ps^*$};
   \end{tikzpicture}
\]
\end{subcor}

\begin{proof}
This is just Proposition~\ref{ayuda} applied to the diagram
  \[
   \begin{tikzpicture}[xscale=.9, yscale=1]
      \draw[white] (0cm,2.5cm) -- +(0: \linewidth)
      node (12) [black, pos = 0.24] {$X$}
      node (13) [black, pos = 0.5 ] {$Y$}
      node (14) [black, pos = 0.76] {$Z$};
      \draw[white] (0cm,0.5cm) -- +(0: \linewidth)
      node (22) [black, pos = 0.24] {$X\times_{S''}X$}
      node (23) [black, pos = 0.5 ] {$Y\times_{S'}Y$}
      node (24) [black, pos = 0.76] {$Z\times_{S}Z$};
      \node (E) at (intersection of 12--23 and 22--13) [scale=0.8] {$\Dv_{\<\times}$};
      \node (B) at (intersection of 13--24 and 23--14) [scale=0.8] {$\Du_\times$};
      \draw [->] (12) -- (13) node[above, midway, scale=0.75]{$f$};
      \draw [->] (13) -- (14) node[above, midway, scale=0.75]{$g$};
      \draw [->] (22) -- (23) node[below, midway, scale=0.75]{$ $};
      \draw [->] (23) -- (24) node[below, midway, scale=0.75]{$ $};
      \draw [->] (12) -- (22) node[left,  midway, scale=0.75]{$\delta_x$};
      \draw [->] (13) -- (23) node[left,  midway, scale=0.75]{$\delta_y$};
      \draw [->] (14) -- (24) node[right, midway, scale=0.75]{$\delta_z$};
   \end{tikzpicture}
  \]
where the arrows in the bottom row are the obvious ones.
\end{proof}

\begin{subexs}\label{functoriality}
Given scheme\kf-maps $X\xto{f}Y\xto{g\>\>}Z$, one has, as special cases of the above constructions, canonical associated morphisms
$$
\Hhr{X}{Z} \to \Hhr{X}{Y}\>,\qquad    
f^*\Hhr YZ \to \Hhr XZf^*\<, \qquad 
\Hhr YZ \to   f^{}_{\<\<*}\Hhr XZf^*\<,
$$
the last two adjoint to each other.  

Evaluation at $\OX$ (resp.~$\OY$) yields canonical
homomorphisms
$$
\Hschr{X}{Z} \to \Hschr{X}{Y}\>,\qquad   
f^*\Hschr YZ \to \Hschr XZ\>, \qquad
\Hschr YZ \to   f^{}_{\<\<*}\Hschr XZ\>.
 $$

Here are the details.
Let $i: X\times_Y X\to X\times_{Z} X$ be the canonical immersion. One has then the oriented commutative squares
 \[
   \begin{tikzpicture}[xscale=.95,yscale=.9]
      \draw[white] (0cm,0.5cm) -- +(0: \linewidth)
      node (E) [black, pos = 0.15] {$Y$}
      node (F) [black, pos = 0.36] {$Z$}
      node (I) [black, pos = 0.654] {$X\<\times_{Y}\<\<X$}
      node (J) [black, pos = 0.87] {$X\<\<\times_{Z}\<\<X$};
      \draw[white] (0cm,2.65cm) -- +(0: \linewidth)
      node (G) [black, pos = 0.15] {$X$}
      node (H) [black, pos = 0.36] {$X$}
       node (K) [black, pos = 0.654] {$X$}
      node (L) [black, pos = 0.87] {$X$};
      \node (C) at (intersection of G--F and E--H) [scale=0.8] {$\Dd_{\<f\<\<,\>g}$};
      \draw [->] (G) -- (H) node[above, midway, scale=0.75]{$\id$};
      \draw [->] (E) -- (F) node[below=1pt, midway, scale=0.75]{$g$};
      \draw [->] (G) -- (E) node[left=1pt,  midway, scale=0.75]{$f$};
      \draw [->] (H) -- (F) node[right=1pt, midway, scale=0.75]{$g f$};
      
      \node (M) at (intersection of K--J and I--L) [scale=0.8] {$(\<\Dd_{\<f\<\<,\>g}\<)^{}_{\<\times}$};
      \draw [->] (K) -- (L) node[above, midway, scale=0.75]{$\id$};
      \draw [->] (I) -- (J) node[below, midway, scale=0.75]{$i$};
      \draw [->] (K) -- (I) node[left=1pt,  midway, scale=0.75]{$\delta_{\<\<f}$};
      \draw [->] (L) -- (J) node[right=1pt, midway, scale=0.75]{$\delta_{\<gf}$};
   \end{tikzpicture}
  \]
and one checks that $\bbiup{\smash{(\Dd_{\<f\<\<,\>g})}}=\bbisub{\smash{(\Dd_{\<f\<\<,\>g})}}\colon \Hhr{X}{Z} \to \Hhr{X}{Y}$ is
the composition
\[
\delta_{\<g\<f}^* \>{\delta^{}_{\<\<g\<f*}}\overset{\delta^*_{\!g\<f_{\phantom{i}}}\mkern-7mu\ps_*}{=\!=\!=\!=}
\delta_{\<g\<f}^* \mkern1.5mu i^{}_* {\delta^{}_{\<\<f*}}\overset{\lift1.75,\ps^*,}{=\!=\!=}
\delta_{\<\<f}^*  i^*  i^{}_* {\delta^{}_{\<\<f*}}\xto{\delta_{\!f}^*\epsilon^{}_i \>}
\delta_{\<\<f}^* {\delta^{}_{\<\<f*}}\>.
\]

If, for example, $g$ is \emph{essentially unramified} (\S\ref{efp}) then, since
\begin{equation*}\label{functoriality.1}
\CD
 \bpic[xscale=2.6, yscale=1.5]
  \draw (0,-1) node (11){$X\times_Y X$};
   \draw (1,-1) node (12){$X\times_Z X$};
   
   \draw (0,-2) node (21){$Y$};
   \draw (1,-2) node (22){$\;Y\<\times_Z Y;$};
  
      \draw [->] (11) -- (12) node[above, midway, scale=0.75]{$i$};
      \draw [->] (11) -- (21) node[left=1pt, midway, scale=0.75]{\textup{natural}};
      \draw [->] (12) -- (22) node[right=1pt,  midway, scale=0.75]{$f\times f$};
      \draw [->] (.285,-2) -- (22) node[below=1pt, midway, scale=0.75]{$\delta_{\<g}$};
  \epic
  \endCD
  \tag{\ref{functoriality}.1}
 \end{equation*}
is a fiber square
and, as follows from \cite[17.4.1]{EGA4},  $\delta_g\>$ is a local isomorphism,  therefore  $i$
is a local isomorphism, and so $\varepsilon_i\colon i^*i_*\to \id$ is an isomorphism.\vspace{1pt}
Thus $\bbiup{\smash{(\Dd_{\<f\mkern-1.5mu,\>\>g})}}$ is a functorial isomorphism
$$
\Hhr{X}{Z} \iso \Hhr{X}{Y}. 
$$

For example, if $h\colon X\to Z$ is a qcqs map such that the kernel~$I$ of the associated map 
$\OZ\to h_*\OX$ is of finite type, then for $Y\subset Z$ the schematic image
(defined by $I\<$, see \cite[(6.10.5)]{EGA1}) one has, canonically,
$
\Hhr{X}{Z} \cong \Hhr{X}{Y}.
$

\vskip3pt

One also has oriented commutative squares
 \begin{equation*}\label{d^fg}
  \CD
   \begin{tikzpicture}[yscale=.9]
      \draw[white] (0cm,0.5cm) -- +(0: .85\linewidth)
      node (E) [black, pos = 0.15] {$Z$}
      node (F) [black, pos = 0.37] {$Z$}
      node (I) [black, pos = 0.64] {$X\<\<\times_{Z}\<\<X$}
      node (J) [black, pos = 0.86] {$Y\<\times_{Z}\<Y$};
      \draw[white] (0cm,2.65cm) -- +(0: .85\linewidth)
      node (G) [black, pos = 0.15] {$X$}
      node (H) [black, pos = 0.37] {$Y$}
       node (K) [black, pos = 0.64] {$X$}
      node (L) [black, pos = 0.86] {$Y$};
      \node (C) at (intersection of G--F and E--H) [scale=0.8] {$\Dd^{f\<\<,\mkern1.5mu g}$};
      \draw [->] (G) -- (H) node[above, midway, scale=0.75]{$f$};
      \draw [->] (E) -- (F) node[below=1pt, midway, scale=0.75]{$\id$};
      \draw [->] (G) -- (E) node[left=1pt,  midway, scale=0.75]{$gf$};
      \draw [->] (H) -- (F) node[right=1pt, midway, scale=0.75]{$g$};
      
      \node (M) at (intersection of K--J and I--L) [scale=0.8] 
                        {$(\<\Dd^{f\<\<,\mkern1.5mu g})^{}_{\<\times}$};
      \draw [->] (K) -- (L) node[above, midway, scale=0.75]{$f$};
      \draw [->] (I) -- (J) node[below=1pt, midway, scale=0.75]{$f\<\times\<\< f\,\>$};
      \draw [->] (K) -- (I) node[left=1pt,  midway, scale=0.75]{$\delta_{\<gf}$};
      \draw [->] (L) -- (J) node[right=1pt, midway, scale=0.75]{$\delta_{\<g}$};
   \end{tikzpicture}
  \endCD \tag{\ref{functoriality}.2}
  \end{equation*}
whence the associated  morphism
\[\label{bbbiup}
\smash{\bbbiup{(f,g)}}\set\bbiup{\smash{(\Dd^{f\<\<,\>g})}} \colon   f^*\Hhr YZ \to \Hhr XZf^*
 \tag{\ref{functoriality}.3}
 \]
 and its adjoint 
 $$\bbisub{\smash{(\Dd^{f\<\<,\>g})}} \colon \Hhr YZ \to   f^{}_{\<\<*}\Hhr XZf^*.
 $$
 
 \pagebreak[3]
 
  If, for example,  $f$ is a \emph{flat monomorphism,} then $(\Dd^{f\<\<,\>g})_\times$ is an oriented fiber square
 with flat bottom arrow. If, in addition, $g$ is qcqs, then so is $\delta_g$  (see proof of \ref{fibersquareiso}), and  so by Proposition~\ref{ayuda0},
if $F\in\D_Y$ then $\bbiup{\smash{(\Dd^{f\<\<,\>g})}}(F)$ \emph{is an isomorphism}
$
f^*\Hhr YZ\>\>F \iso \Hhr XZ\>f^*\<F.
$

(See also Theorem~\ref{caset} below).\va3

\end{subexs}

\begin{subcor}\label{trans sharp}
For any\/ $W\xto{e\>\>}X\xto{f}Y\xto{g\>\>}Z$ the next diagram commutes.
 \[
  \bpic[xscale=3.6, yscale=1.8]

   \draw (0,-1) node (11){$(f\<e)^*\>\Hhr{Y}{Z}$};
   \draw (1,-1) node (12){$$};
   \draw (2,-1) node (13){$\Hhr{\<W\<}{Z}\>(f\<e)^*$};

   \draw (0,-2) node (21){$e^*\!f^*\>\Hhr{Y}{Z}$};
   \draw (1,-2) node (22){$e^*\>\Hhr{X}{Z}f^*$};
   \draw (2,-2) node (23){$\Hhr{\<W\<}{Z}\>e^*\!f^*$};
    
   \draw[->] (11)--(13) node[above, midway, scale=0.75]{$\bbbiup{(f\<e,g)}$};
   
   \draw[double distance=2pt] (11)--(21) node[left=1pt, midway, scale=0.75]{$\ps^*$};
   \draw[double distance=2pt] (13)--(23) node[right=1pt,  midway]{$\lift.85,\via \ps^*,$};
   
   \draw[->] (21)--(22) node[below, midway, scale=0.75]{$e^*\bbbiup{(f,g)}$};
   \draw[->] (22)--(23) node[below, midway, scale=0.75]{$\bbbiup{(e,gf)}$};
 \epic
\]

\end{subcor}
\noindent\emph{Proof.}
Apply~\ref{composedsquare} to 
\[\!\!\! \!\! 
   \bpic[xscale=3.2,yscale=2]
      \node (11) at (0,0){$W$};
      \node (12) at (1,0){$X$};
      \node (13) at (2,0) {$Y$};
      \node (21) at (0,-1){$Z$};
      \node (22) at (1,-1){$Z$};
      \node (23) at (2,-1){$Z$};
      \node at (.5,-.5) [scale=0.8] {$\Dd^{e\<,\>g\<f}$};
      \node at (1.5,-.5) [scale=0.8]  {$\Dd^{f\<\<,\>g}$};
      \draw [->] (11) -- (12) node[above, midway, scale=0.75]{$e$};
      \draw [->] (12) -- (13) node[above, midway, scale=0.75]{$f$};
      \draw [double distance=2pt] (21) -- (22);
      \draw [double distance=2pt] (22) -- (23);
      \draw [->] (11) -- (21) node[left,  midway, scale=0.75]{$gfe$};
      \draw [->] (12) -- (22) node[left,  midway, scale=0.75]{$g\<f\<$};
      \draw [->] (13) -- (23) node[ above=-7pt, midway, scale=0.75]{$\quad\  g$};
   \epic
  \]

\begin{thm}\label{caset}
For scheme-maps\/ $X\xto{f}Y\xto{g}Z$ with\/ $f$ essentially \'etale,  
and\/ $F\in\D_Y,$ the map
$$
\bbbiup{(f,g)}(F) \colon   f^*\Hhr YZ\>\>F \to \Hhr XZ\>f^*\<\<F
$$
is an isomorphism.
\end{thm}

 \begin{proof} It suffices to show that for every open immersion $W\xto{e\>}X$ with $W$ affine,
 $e^*\bbbiup{(f,g)}(F)$ is an isomorphism. It results therefore from Corollary~\ref{trans sharp} that  it's enough to prove Theorem~\ref{caset} when $X$ is affine or when $f$~is an open immersion. 
The latter case was disposed of at the end of \S\ref{functoriality}.
 
We may in fact assume that $f$ factors as $X\xto{f^{}_{\<0}\>}Y_0\xto{i\>} Y$ where $Y_0$ is affine and $i$ is an open immersion. Then
an argument like the preceding one, applied to $X\xto{f^{}_{\<0}\>}Y_0\xto{i\>} Y\xto{g\>}Z,$ shows that we can replace $f$ by $f^{}_{\<0\>}$. Thus we may assume that $Y$ as well as $X$ is affine, and so since $f$ is efp, we can set $Y=\spec A$ and \mbox{$X=\spec M^{-1}B$} where $B$ is a finitely presentable $A$-algebra and $M\subset B$ is a multiplicatively closed subset. 

Furthermore, $Y$ and $X$ being affine, the maps
$f$ and~$g$ are both separated, and so the  canonical immersions
 \[
\delta\set\delta_{\<\<f}\colon X\hookrightarrow X\times_Y X=:V,\qquad 
j\colon X\times_YX\hookrightarrow X\times_ZX
\]
are both closed. 
  
The closed immersion $\delta$ is essentially \'etale, hence flat (\S\ref{efp}); so with $\CI$ the kernel of
the natural surjection 
$\OV\to\delta_*\OX$,  $\OV\<\</\>\CI$ is flat over $\OV$, and
$$
\CI/\>\CI^2\cong \mathcal T\!or_1^{\OV}\!(\OV\<\</\>\CI,\>\OV\<\</\>\CI\>)=0.
$$
Moreover, $\CI$ is a \emph{finite-type} $\OV$-ideal. For, with $A$, $M$ and $B$ as before, and\looseness=-1
$$
N\set\{\,m_1\otimes m_2\mid m_1,\>m_2\in M\,\}\subset B\otimes_A B,
$$  
the kernel of the multiplication map $\mu\colon B\otimes_A B\to B$ is finitely  generated (\cite[p.\,301, 6.2.6.2]{EGA1}), as is the kernel of $N^{-1}\mu\colon M^{-1}B\otimes_A M^{-1}B\to M^{-1}B$, giving the assertion.

So by Nakayama's lemma, at any $v\in V\<$ the stalk~$\CI_v$,  being a finitely-generated idempotent $\mathcal O_{V\!,\>v\>}$-ideal, is either~(0) or~$\mathcal O_{V\!,\>v}\>$; and it follows that $\delta$~ \emph{induces an isomorphism of\/~$X$ onto an open-and-closed subscheme of}~$V\<$.\va1

Setting $V'\set j(V\setminus \delta(X))$, one has then the diagram 
\[
  \CD 
   \begin{tikzpicture}
      \draw[white] (3cm,2.8cm) -- +(0: .75\linewidth)
      node (21) [black, pos = 0.2 ] {$\delta(X)$}
      node (22) [black, pos = 0.5] {$X$}
      node (23) [black, pos = 0.65] {$$}
      node (24) [black, pos = 0.8] {$Y$};
      \draw[white] (3cm,0.5cm) -- +(0: .75\linewidth)
      node (11) [black, pos = 0.2]  {$\mkern-45mu(X\times_Z X)\setminus \!V'$}
      node (12) [black, pos = 0.5] {$X\times_Z X$}
      node (14) [black, pos = 0.8] {$Y\times_Z Y$};
      \node (C) at (intersection of 12--24 and 14--22) [scale=0.8] {$ \Du=(\<\Dd^{f\<\<,\mkern1.5mu g})^{}_{\<\times}$};
      \node (D) at (intersection of 12--21 and 11--22) [scale=0.8] {$ \Dv$};
      \draw [->] (21) -- (11) node[auto, swap, midway, scale=0.75]{$j$};
      \draw [->] (21) -- (22) node[above, midway, scale=0.75]{$\delta^{-\<1}$};
      \draw [->] (12) -- (14) node[below=1pt, midway, scale=0.75]{$f \times f$};
      \draw [->] (11) -- (12) node[below, midway, scale=0.75]{$k$};
      \draw [->] (22) -- (24) node[above, midway, scale=0.75]{$f$};
      \draw [->] (22) -- (12) node[left,  midway, scale=0.75]{$\delta_{gf}$};
      \draw [->] (24) -- (14) node[right, midway, scale=0.75]{$\delta_g$};
    \end{tikzpicture}
   \endCD   
\]
where $k$ is an open immersion and each of $\Dd=\Du\smallcirc\Dv$ and $\Dv$ is a  fiber square with flat bottom arrow. Since $X$ is affine,  $\delta_g$ and~$\delta_{gf}$ are qcqs
(use \cite[p.\,291ff, 6.1.4, 6.1.9(iv) and 6.1.9(v)]{EGA1}). 
By  ~\ref{ayuda0}, $\defnu{\Dd}(F)$ and $\defnu{\Dv}(f^*\<\<F\>)$ are both isomorphisms, whence, by~\ref{ayuda}, so is~
 $\defnu{\>\Du}(F)$, which is exactly what Theorem~\ref{caset} asserts.
 \end{proof}

From \ref{caset}, \ref{fibersquareiso} and  \ref{composedsquare} (with $X$ replaced by $X'$, $Y$  by $X\times_Y Y'$,  $Z$ by~$X$, $S$ by $Y$ and  $S''\to S'$ by the identity map of $Y'$), one gets:

\begin{subcor} \textup{(Cf.~\cite[(0.1)]{gwet}.)} For an\/  oriented commutative square of scheme-maps
\label{admissibleiso}
\[
 \begin{tikzpicture}
      \draw[white] (1cm,0.5cm) -- +(0: .75\linewidth)
      node (E) [black, pos = 0.325] {$Y'$}
      node (F) [black, pos = 0.565] {$Y$};
       \draw[white] (1cm,2.65cm) -- +(0: .75\linewidth)
      node (G) [black, pos = 0.325] {$X'$}
      node (H) [black, pos = 0.565] {$X$};
      \node (C) at (intersection of G--F and E--H) [scale=0.8] {$\Dd$};
      \draw [->] (G) -- (H) node[above, midway, scale=0.75]{$h$};
      \draw [->] (E) -- (F) node[below=1pt, midway, scale=0.75]{$g$};
      \draw [->] (G) -- (E) node[left,  midway, scale=0.75]{$f'\>$};
      \draw [->] (H) -- (F) node[right, midway, scale=0.75]{$f$};
 \end{tikzpicture}
\] 
with\/ $g$ flat and\/ $f$ qcqs, whose associated map\/ 
$X'\to X\times_Y Y'$ is \emph{essentially \'etale}, and
for\/ $G\in\D_\sX,$ $\biup{\Dd}$ is an isomorphism
$$
h^{\<*}\> \Hhr{X}{Y}\>G\iso \Hhr{X'}{Y'}\> h^{\<*}G.
$$
In particular, $\bbiup{d}:  h^{\<*} \Hschr{X}{Y} \to \Hschr{X'}{Y'}$ is an isomorphism.
\end{subcor}
 
\section{The fundamental class of a flat map}\label{fundclassflat}

Fix a noetherian scheme $S$. Let $\SS$  be  the category of 
separated, efp (hence noetherian)  $S$-schemes   (see \S\ref{efp}). All maps in~$\SS$ are separated and~efp. The fiber product, in the category of all schemes, of two $\SS$\kf-maps with the same~target is
a fiber product in  $\SS$. 

In this section we describe how to assign to every  \emph{flat} $\SS$\kf-map, \mbox{$f: X\to Y$}, with twisted inverse\kf-image functor $f^!\colon\D_Y\to\D_\sX$ as in \S\ref{! and otimes}, and with $\Hh{\<\<X|S}\>$, $\Hh{Y|S}$
as in \S\ref{preHoch}, a  canonical functorial map
\[
\fundamentalclass{\<f}\colon \Hh{\<\<X|S}f^*\lto f^!\>\Hh{Y|S}\>,
\]
called the \emph{fundamental class of $f\<$}.\va{1.5} 

When $f$ is essentially \'etale, so that $f^!=f^*\<$, $\Dc_{\<f}$ turns out, nontrivially, to~be inverse to the isomorphism in Theorem~\ref{caset}, see Proposition~\ref{c inverse}.

As in \ref{0.2}, we set 
$$
\Hsch{\sX}\set \Hh{\<\<X|S}\>\OX\>.
$$

If $x\colon X\to S$  is essentially smooth of relative dimension~$n$, 
then $\Dc_x$ induces an \emph{isomorphism}
$$
\Omega_f^n[n]\cong (H^{-n}\Hsch{X})[n] \iso (H^{-n}x^!\OS)[n]\cong x^!\OS,
$$
that should be closely related to the well-known one of Verdier \cite[p.\,397, Thm.\,3]{V}, 
see Remark~\ref{c vs V}.

If  $\,x\colon X\to S$  is  finite and flat,  then $\Dc_x(\OX)$
is closely related to the \emph{trace map} $x_*\OX\to \OS$, see Example~\ref{c and trace}.

\begin{cosa}\label{prelims}
We first review a few preliminary considerations.\va2

Let $S$ and $\SS$ be as above. As is customary, we  usually denote an object 
$w\colon W\to S$ in $\SS$ simply by~$W\<$,  with the understanding
that $W$ is equipped with a separated efp ``structure map"~$w\>$. We set 
\[
\D_W\set\Dqc(W)\qquad\textup{(see~\S\ref{derived}).}
\]
 
For $W_1$, $W_2\in\SS$, we denote  $W_1\times_S W_2$ by $W_1\times W_2$.   The diagonal map $W\to W\times W$
will be denoted by $\delta_W$. We set
\[
\Hh{W}\set\Hh{W|S}=\Hh{w}\set(\<\delta_W\<\<)^{\<*}{(\<\delta^{}_W\<\<)}_{\<*}
\qquad(\:\set\LL\delta_W^{*}{(\<\delta^{}_W\<\<)}_{\<*}\>,\textup{ see~\S\ref{derived}).}
\]

\begin{subcosa}\label{upper-lower-*}                                                      
For $f\colon X\to Y$ in $\SS$ one has, with the notational convention of \S\ref{derived},  $f^*\D_Y\subset\D_\sX$ \cite[3.9.1]{li} and
$\fst\D_\sX\subset\D_Y$ \mbox{\cite[3.9.2]{li};} so the adjoint pseudofunctors $(-)^*$ and $(-)_*$ in 
\S\ref{adj-pseudo} can be restricted to take values in the categories~$\D_W$.  \emph{It is assumed henceforth that they are so restricted.}

\end{subcosa}

\begin{subcosa}\label{! and otimes}
For any scheme $W\<$, let $\Dqcpl(W)\subset\D_W$ be the full subcategory with objects those
complexes $G\in\D_W$ such that $H^n(G)=0$ for all $n\ll0$.

\pagebreak[3]

According to \cite[5.3]{Nk},\va{-1} there is a contravariant $\Dqcpl$-valued pseudofunctor~$(-)_{\<\upl}^!$ 
over~$\SS$, uniquely determined up to isomorphism by the properties:\va1

(i) When restricted to proper maps, $(-)_{\<\upl}^!$  is pseudo\-functorially right-adjoint to the
right-derived direct-image\va1 pseudofunctor~$(-)_*\>$.\looseness=-1

(ii) When restricted to essentially \'etale maps, $(-)_{\<\upl}^!$   is equal to the usual
inverse\kf -image pseudofunctor (derived or not).\va1

(iii) For each oriented fiber square $\Dd$ in $\SS$,
\begin{equation*}
    \begin{tikzpicture}[xscale=0.9, yscale=0.8]
      \draw[white] (0cm,0.5cm) -- +(0: \linewidth)
      node (E) [black, pos = 0.41] {$Y^\prime$}
      node (F) [black, pos = 0.59] {$Y$};
      \draw[white] (0cm,2.65cm) -- +(0: \linewidth)
      node (G) [black, pos = 0.41] {$X^\prime$}
      node (H) [black, pos = 0.59] {$X$};
      \draw [->] (G) -- (H) node[above, midway, scale=0.75]{$1$};
      \draw [->] (E) -- (F) node[below, midway, scale=0.75]{$3$};
      \draw [->] (G) -- (E) node[left,  midway, scale=0.75]{$2$};
      \draw [->] (H) -- (F) node[right, midway, scale=0.75]{$4$};
      \node (C) at (intersection of G--F and E--H) [scale=0.8]{$\Dd$};
    \end{tikzpicture}
  \end{equation*}
with $4$ (hence $2$) proper and $3$ (hence $1$) essentially \'etale,  and with $\theta_\Dd$ as in~\ref{thetaiso},
the natural composite isomorphism
$$
1^*4_{\upl}^!\>=1_{\<\upl}^!4_{\upl}^!\>\iso(4\smallcirc\<1)_{\<\upl}^!=(3\smallcirc\<2)_{\<\upl}^!\iso 2_{\upl}^!3_{\<\upl}^!=2_{\upl}^!3^*
$$
is adjoint to  the composition (with $\smallint^{}_{\<\<\dpl}$  the counit map coming from (i) above):
$$ 
2_*1^*4_{\upl}^!\> \underset{\theta_{\<\<\Dd}^{-\<\<1\>}}\iso 3^*4_* 4_{\upl}^!
\underset{\<\int^{}_{\<\dpll}}\lto 3^*.
$$

As in the first Remark in \S\ref{derived}, we fix once and for all  a specific such  pseudofunctor such that
$(\id_X\<)_{\upl}^!$ is the identity functor.
The point of what follows is to extend this to a $\Dqc$-valued pseudofunctor, at least over 
\emph{essentially perfect} (i.e., finite tor-dimension) $\SS$\kf-maps. 

(Henceforth we will abuse terminology by calling $\SS$\kf-maps of
finite tor-dimension ``perfect" instead of ``essentially perfect."
For the purposes of this paper, ``\kf flat" may be substituted throughout for ``\kf perfect.")

Let $\SSp$ be the subcategory of perfect maps in~$\SS$. (Perfection is preserved by composition,  see, e.g., \cite[p.\,243, Cor.\,3.4]{Il}.) As in \cite[\S5.7\kf]{AJL},  there is 
over $\SSp$ a contra\-variant
\emph{twisted inverse-image} pseudofunctor~$(-)^!$, taking values in the categories~$\D_W\<$,
such that 
\begin{equation*}
\mkern66mu f^!\<F=  (f_{\<\<\upl}^!\OY\otimes^{}_{\<\<X} f^*\<\<F\>)\<\in\D_\sX
\qquad (f\colon X\to Y \textup{ in $\SSp\>$;}\; F\in\D_Y).
\end{equation*}

From the assumptions in the first Remark in \S\ref{derived},  one gets then that
$f^!\OY=f_{\<\<\upl}^!\OY$---so that
\begin{equation*}\label{! and otimes1}
f^!\<F=  (f^!\OY\otimes^{}_{\<\<X} f^*\<\<F\>)\in\D_\sX
\qquad (f\colon X\to Y \textup{ in $\SSp\>$;}\; F\in\D_Y).
\tag{\ref{! and otimes}.1}
\end{equation*}
When $X=Y$ and 
$f=\id_\sX$ then $f^!$ is the identity functor on $\D_\sX$.\va1

For $F\in\Dqcpl(Y)$, there is, as in  \cite[5.9]{Nk},
a natural isomorphism 
\begin{equation*}\label{f! and f!+}
f^!\<F\cong f_{\<\<\upl}^!F\<.
\tag{\ref{! and otimes}.2}
\end{equation*}

When $F=\OY$, this is the identity map of $f^!\OY=f_{\<\<\upl}^!\OY$.

When $f$ is essentially \'etale, \eqref{f! and f!+} is the identity map (cf.~\cite[4.9.2.3]{li}, with $E\set\OX$).\va{-2}

Further, for $X\xto{f}Y\xto{g\>} Z$ in $\SSp$, the associated isomorphism
$$
\ps^!\colon f^!g^!\iso(gf)^!
$$
is the natural composition
\begin{equation*}
\begin{aligned}\label{ps-def}
f_{\<\<\upl}^!\OY\otimes^{}_\sX f^*(g_{\upl}^!\OZ\otimes^{}_{Y} g^*)
&\iso 
(f_{\<\<\upl}^!\OY\otimes^{}_\sX f^*\<g_\upl^!\OZ)\otimes_\sX f^*\<g^*\\
&\,=\!=\,
f^!g_\upl^!\OZ\otimes_\sX f^*\<g^*\\
&\iso
f_{\<\<\upl}^!g_\upl^!\OZ\otimes_\sX f^*\<g^*\\
&\iso
(gf)_\upl^!\OZ\otimes_\sX (gf)^*.
\end{aligned}\tag{\ref{! and otimes}.3}
\end{equation*}

In view of (ii) above,  one finds that  the restrictions\- of $(-)^!$ and~$(-)^*$ to  essentially \'etale \mbox{$\SS$\kf-maps} are\va1 
\emph{identical pseudofunctors}.

Over $\SSp$ the isomorphism \eqref{f! and f!+} is \emph{pseudofunctorial}. Thus 
$(-)^!$~may be viewed as an extension of $(-)_\upl^!$ to a $\Dqc$-valued pseudofunctor.\va3

\end{subcosa}

\enlargethispage{-15pt}

\pagebreak[3]

\begin{subcosa}\label{int}
To each \emph{proper} $\SSp$\kf-map $f\colon X\to Y$ is
associated, as in \cite[\S5.9]{AJL}, a functorial map (with id the identity functor of $\D_Y\<$):
$$
\couni{\<\<\!f}\colon f^{}_{\<\<*}f^!\to \id,
$$ 
such that for any $X\xto{f} Y\xto{g\>} Z$ in $\SSp$  with $f$ and $g$ proper, the following diagram commutes
\begin{equation*}\label{transitivity}
\CD
   \begin{tikzpicture}[xscale=3, yscale=1.5]
         \node (41) at (1,-1) {\raisebox{5pt}{$ (gf)_*(gf)^!$}};
    
      \node (53) at (2.5,-1) {\raisebox{5pt}{$ g^{}_* f^{}_{\<\<*}f^!g^!$}};
   
            \node (61) at (1,-3){\raisebox{5pt}{$\id$}};
     \node (63) at (2.5,-3) {\raisebox{5pt}{$\ \,g^{}_*g^!$}};
      \draw [double distance=2pt]
                 (41) -- (53) node[above=1pt,  midway, scale=0.75]{$\textup{via}\;\ps_*\>\>\textup{and}\,\ps^!$};                        
      \draw [->] (41) -- (61) node[left,  midway, scale=0.75]{$ \couni{\!gf}$};
      \draw [->] (53) -- (63) node[right=1pt, midway, scale=0.75]{$g_* \couni{\<\<\!f}$};
         \draw [->] (63) -- (61) node[below=1pt, midway, scale=0.75]{$ \couni{\!g}$};
 \end{tikzpicture}
 \endCD\tag{\ref{int}.1}
\end{equation*}

This $\couni{\<\<\!f}$ is given by the natural functorial composition, with $F\in\D_Y$,
$$
\fst f^!\<F= \fst(f_{\<\<\upl}^!\OY\otimes_\sX f^*\<\<F)\underset{\eqref{projection}}\lto
\fst f_{\<\<\upl}^!\OY\otimes_Y\< F \xto{\<\int^{}_{\<\dpll}\<\<\otimes\id\>\>} \OY\otimes_Y \<F= F,
$$
where $\smallint^{}_{\<\<\dpl}$ arises from (i) in \S\ref{! and otimes}. 

If $f$ is the identity map of $X$ then $\couni{\<\<\!f}$ can be identified with the identity transformation.

More generally, let $X\xto{f} Y\xto{g\>} Z$ be $\SS$\kf-maps with $f$ proper and both 
$g$ and $g\<f$ perfect. The functorial map 
\begin{equation*}\label{specialint}
\couni{\<\<\!f}^{\>g_{\mathstrut}}\<\colon \fst(gf)^!\to g^! 
\tag{\ref{int}.2}
\end{equation*}
is defined to be the  natural composition
\begin{align*}
     \fst\big((gf)_\upl^!\OZ\otimes_\sX (gf)^*\big) 
     &{\iso}  \fst\big(f_{\<\<\upl}^!g_\upl^!\OZ\otimes_\sX f^*\<g^*\big)   \\
     &\underset{\eqref{projection}}\iso  
      \fst f_{\<\<\upl}^!g_\upl^!\OZ\otimes_Y g^*  
    \xto{\<\int^{}_{\<\dpll}\<\<\otimes\id\>\>}    g_\upl^!\OZ\otimes_Y g^*. 
\end{align*}

The next Lemma will be used in the proof of Proposition~\ref{c
  inverse}.

\begin{sublem}\label{proper specint}
For\/ $\SSp$-maps $X\xto{f}Y\xto{g\>}Z$  with\/ $f$ proper, 
$\couni{\<\<\!f}^{\>g_{\mathstrut}}$ factors as
$$
\fst(gf)^! \overset{\fst\<\<\ps^!}{=\!=}\fst f^!g^! \xto{\int^{}_{\<\<f}} g^!\<.
$$
\end{sublem}

\begin{proof} The assertion is that the border of the following natural diagram commutes:
\[\mkern-3mu
\def\1{$\fst(gf)^!$}
\def\2{$\fst\big((gf)_\upl^!\OZ\<\otimes_\sX\<\<(gf)^*\big)$}
\def\3{$\fst f^!g^!$}
\def\4{$\fst(f_{\<\<\upl}^!\OY\<\otimes_\sX\<\< f^*\<g^!)$}
\def\5{$\fst\big(f_{\<\<\upl}^!g_\upl^!\OZ\<\otimes_\sX\<\< f^*\<g^*\big)$}
\def\6{$\fst\big(f_{\<\<\upl}^!\OY\<\otimes_\sX\! f^*(g_\upl^!\OZ\<\otimes_Y\< g^*)\big)$}
\def\7{$\fst f_{\<\<\upl}^!\OY\<\otimes_Y\< g^!$}
\def\8{$\OY\<\otimes_Y\< g^!$}
\def\9{$g^!$}
\def\ten{$\fst f_{\<\<\upl}^!g_\upl^!\OZ\<\otimes_Y\< g^*$}
\def\lvn{$\fst (f_{\<\<\upl}^!\OY\<\otimes_\sX\<\< f^*\<g_\upl^!\OZ)\<\otimes_Y\< g^*$}
\def\twv{$g_\upl^!\OZ\<\otimes_Y\< g^*$}
\def\thn{$\fst f_{\<\<\upl}^!\OY \<\otimes_Y\< g_\upl^!\OZ\<\otimes_Y\< g^*$}
\def\frn{$\OY \<\otimes_Y\< g_\upl^!\OZ\<\otimes_Y\< g^*$}
 \bpic[xscale=4.7,yscale=2.2]
     \node (11) at (1,-1.1)[scale=.97]{\3};
     \node (12) at (1.97,-1.1 )[scale=.97]{\1};
     \node (13) at (3,-1.1)[scale=.97] {\2};
     
     \node (21) at (1,-2)[scale=.97]{\4};
     \node (22) at (1.97,-2)[scale=.97]{\6};
     \node (23) at (3,-2)[scale=.97]{\5};
 
     \node (25) at (2.42,-3.7)[scale=.97]{\lvn};
 
     \node (31) at (1,-2.9)[scale=.97]{\7};
     \node (32) at (1.84,-2.9)[scale=.97]{\thn};
     \node (33) at (3,-2.9)[scale=.97]{\ten};
          
     \node (41) at (1,-4.2)[scale=.97]{\8};
     \node (42) at (1.84,-4.2)[scale=.97]{\frn};

     \node (51) at (1,-5.1)[scale=.97]{\9};
     \node (53) at (3,-5.1)[scale=.97]{\twv};

    \draw[double distance=2pt] (11)--(12) node[above=1pt, midway, scale=.73] {$\fst\!\ps^!$} ;
    \draw[double distance=2pt]  (12)--(13);

    \draw[double distance=2pt]  (21)--(22);
    \draw[->]  (22)--(23) node[above=-5pt, midway, scale=.73] {$\widetilde{\phantom{nn}}\>$} ;

    \draw[double distance=2pt]  (31)--(32);

    \draw[double distance=2pt] (42)--(41);

    \draw[double distance=2pt] (53)--(51);
    
    \draw[double distance=2pt]  (11)--(21);
    \draw[->]  (21)--(31) node[left=1pt, midway, scale=.73] {\eqref{projection}} ;
    \draw[->]  (31)--(41) node[left=1pt, midway, scale=.73] {$\int^{}_{\<\dpl}\<\<\otimes\id$} ;
    \draw[double distance=2pt]  (41)--(51) ;

    \draw[->]  (22)--(32) node[left=1pt, midway, scale=.73] {\eqref{projection}} ;
    \draw[->]  (32)--(42) node[left=1pt, midway, scale=.73] {$\int^{}_{\<\dpl}\<\<\otimes\id$} ;

    \draw[->]  (13)--(23) node[right=1pt, midway, scale=.73] {$\simeq$} ;
    \draw[->]  (23)--(33) node[right=1pt, midway, scale=.73] {\eqref{projection}} ;
    \draw[->]  (33)--(53) node[right=1pt, midway, scale=.73] {$\int^{}_{\<\dpl}\<\<\otimes\id$} ;
    
    \draw[->]  (25)--(32) node[above=-4pt, midway, scale=.73] {\kern48pt\eqref{projection}} ;
    \draw[->]   (33)--(25) node[above, midway, scale=.73] {$\simeq\mkern15mu$} ;
    \draw[double distance=2pt]  (42)--(53) ;

    \node at (1.97,-1.6){\circled1};
    \node at (2.42,-2.75){\circled2};
    \node at (2.42,-4.15){\circled3};

   \epic
\] 
Commutativity of subdiagram \circled1 results from the definition \eqref{ps-def} of $\ps^!$,
of \circled2  from \cite[3.4.7(iv)]{li} with $(A,B,C)\set(g^*, g^!\OZ, f_{\<\<\upl}^!)$, 
\emph{mutatis mutandis},
of \circled3  from the definition of the isomorphism \eqref{f! and f!+} for proper $f$, 
\cite[5.7]{Nk}, and 
of the unlabeled subdiagrams is clear, whence the conclusion.
\end{proof}

The following ``transitivity" property of  $\couni{\<\<\!f}^{\>g_{\mathstrut}}$---that in view of Lemma~\ref{proper specint} generalizes commutativity of
\eqref{transitivity}---will be needed in~\S\ref{provetran}.

\begin{subprop}\label{Transitivity}
  Let\/ $X\xto{f} Y\xto{g\>} Z \xto{h\>} W$ be\/ $\SS$\kf-maps with\/ $f,$ $g$ proper and 
$h,$ $hg,$ $hgf$ perfect. The following diagram commutes.
\[
   \bpic[xscale=3, yscale=1.3]
         \node (41) at (1,-1) {\raisebox{5pt}{$ (gf)_*(hgf)^!$}};
    
      \node (53) at (2.5,-1) {\raisebox{5pt}{$ g^{}_* \fst (hgf)^!$}};
   
            \node (61) at (1,-3){\raisebox{5pt}{$h^!$}};
     \node (63) at (2.5,-3) {\raisebox{5pt}{$\ \,g^{}_*(hg)^!$}};
      \draw [double distance=2pt]
                 (41) -- (53) node[above=1pt,  midway, scale=0.75]{$\ps_*$};                        
      \draw [->] (41) -- (61) node[left,  midway, scale=0.75]{$\smallint_{\mspace{-2mu}\lift.65,gf,}^h$};
      \draw [->] (53) -- (63) node[right=1pt, midway, scale=0.75]{$g_*\! \smallint_{\mspace{-4mu}\lift.65,f,}^{hg}$};
         \draw [->] (63) -- (61) node[below=1pt, midway, scale=0.75]{$ \smallint_{\mspace{-2mu}\lift.65,g,}^h$};
 \epic
\]
\end{subprop}

\begin{proof} The diagram expands naturally as
\[\mkern-5mu
\def\3{$g_*\<\fst\<\<\big((hgf)^!_\upl\OW\<\<\otimes\<\<(hgf)^{\<*}\<\big)$}
\def\4{$(gf)_*\<\<\big((hgf)^!_\upl\OW\<\otimes\<(hgf)^{\<*}\<\big)$}
\def\5{$\mkern-10mu g_*\<\fst\<\<\big(f^!_{\<\<\upl}(hg)^!_\upl\OW\<\<\otimes\<\< f^*(hg)^{\<*}\<\big)\ $}
\def\6{$\ (gf)_*\<\big((gf)^!_\upl h_\upl^!\OW\<\otimes\<(gf)^*h^{\<*}\<\big)$}
\def\7{$g_*\<\big(\fst f^!_{\<\<\upl}(hg)^!_\upl\OW\<\otimes\< (hg)^{\<*}\<\big)$}
\def\8{$(gf)_*(gf)^!_\upl h_\upl^!\OW\<\otimes\< h^{\<*}$}
\def\9{$g_*\<\big((hg)^!_\upl\OW\<\otimes\< (hg)^{\<*}\<\big)$}
\def\ten{$g_*\<\big(g^!_\upl h^!_\upl\OW\<\otimes\< g^*h^{\<*}\<\big)$}
\def\lvn{$g_*g^!_\upl h^!_\upl\OW\<\otimes\< h^*$}
\def\twv{$h^!_\upl\OW\<\otimes\< h^*$}
\def\thn{$g_*\fst\big(f^!_{\<\<\upl}g^!_\upl h^!_\upl\OW\<\otimes\< f^*\<g^*\<h^{\<*}\<\big)$}
\def\frn{$g_*\fst\big((gf)^!_{\upl} h^!_\upl\OW\<\otimes\< f^*\<g^*\<h^{\<*}\<\big)$}
\def\ffn{$g_*\<\big(\fst(gf)^!_{\upl} h^!_\upl\OW\<\otimes\< g^*h^{\<*}\<\big)$}
\def\sxn{$g_*\fst(gf)^!_{\upl} h^!_\upl\OW\<\otimes\< h^*$}
\def\svn{$g_*\fst f^!_{\<\<\upl}g^!_{\upl} h^!_\upl\OW\<\otimes\< h^*$}
\def\egn{$g_*\<\big(\fst f^!_{\<\<\upl}g^!_\upl h^!_\upl\OW\<\otimes\<g^*\<h^{\<*}\<\big)$}
\bpic[xscale=4.2,yscale=1.5]
    \node(11) at (1,-.5)[scale=.9]{\4};
    \node(13) at (3,-.5)[scale=.9]{\3};

    \node(21) at (1,-1.5)[scale=.9]{\6};
    \node(22) at (2,-2.1)[scale=.9]{\frn};
    \node(23) at (3,-1.5)[scale=.9]{\5};
   
    \node(31) at (1.5,-3)[scale=.9]{\ffn};
    \node(32) at (2.5,-3)[scale=.9]{\thn};

   \node(41) at (1.5,-4)[scale=.9]{\sxn};
   \node(42) at (2.5,-4)[scale=.9]{\egn};
   \node(43) at (3,-3.5)[scale=.9]{\7};

   \node(51) at (1,-4.9)[scale=.9]{\8};
   \node(52) at (2,-4.9)[scale=.9]{\svn};
   \node(53) at (3,-4.9)[scale=.9]{\9};

   \node(61) at (1,-5.9)[scale=.9]{\twv};
   \node(62) at (2,-5.9)[scale=.9]{\lvn};
   \node(63) at (3,-5.9)[scale=.9]{\ten};
  
   \draw[double distance=2pt] (11)--(13) node[above=1pt, midway, scale=.7]{$\ps_*$};  

   \draw[double distance=2pt] (51)--(52) node[above=1pt, midway, scale=.7]{$\ps_\upl^!$}
   node[below=1pt, midway, scale=.7]{$\ps_*$};  
   
  \draw[->] (62)--(61) ;
  \draw[->] (63)--(62) ;

    \draw[double distance=2pt] (11)--(21) node[right=1pt, midway, scale=.7]{$\ps^*$}
                                                              node[left=1pt, midway, scale=.7]{$\ps_\upl^!\<$};
    \draw[->] (21)--(51) node[left, midway, scale=.7]{};  
    \draw[->] (51)--(61)node[left=1pt, midway, scale=.7]{$\int^{}_{\<\dpl}$};

    \draw[->] (52)--(62) node[left, midway, scale=.7]{$$};

    \draw[->] (31)--(41) node[left, midway, scale=.7]{$$};
    
    \draw[->] (32)--(42) node[left, midway, scale=.7]{$$};

   \draw[double distance=2pt] (13)--(23) node[left=1pt, midway, scale=.7]{$\ps_\upl^!\<$}
                                                               node[right=1pt, midway, scale=.7]{$\ps^*$};
   \draw[->] (23)--(43) node[left, midway, scale=.7]{};
   \draw[->] (43)--(53) node[left, midway, scale=.7]{$$};
   \draw[double distance=2pt] (53)--(63) node[right=1pt, midway, scale=.7]{$\ps_\upl^!$};
   
     \draw[double distance=2pt] (21)--(22) node[below, midway, scale=.7]{$\ps_*\mkern25mu$}
                                                                   node[above=-1pt, midway, scale=.7]{$\mkern50mu\ps^*$};
     \draw[double distance=2pt] (23)--(32) node[above=1pt, midway, scale=.7]{$\mkern-25mu\ps_\upl^!$}
                                                              node[below=1pt, midway, scale=.7]{$\mkern20mu\ps^*$};
     \draw[->] (22)--(31) node[left, midway, scale=.7]{$$};
     \draw[double distance=2pt] (22)--(32) node[above=-2pt, midway, scale=.7]{$\mkern40mu\ps_\upl^!$};
     \draw[->] (42)--(52) node[left, midway, scale=.7]{$$};

     \draw[double distance=2pt] (51)--(41) node[above=-2pt, midway, scale=.7]{$\ps_*\mkern40mu$};
     \draw[double distance=2pt] (52)--(41) node[above=-2pt, midway, scale=.7]{$\mkern40mu\ps_\upl^!$};
     
    \node at (2, -1.35)[scale=.85]{\circled1};   
    \node at (1.25, -3.52)[scale=.85]{\circled2}; 
    \node at (1.523,-5.45)[scale=.85]{\circled3};   
  
 \epic
\]

Commutativity of subdiagram \circled1 follows from pseudofunctoriality of $(\<-\<)_\upl^!\>$, $(-)^*$ 
and $(-)_*\>$.
That of \circled2 is given by \cite[3.7.1]{li}, \emph{mutatis mutandis}.
As the pseudofunctors $(-)_\upl^!$ and $(-)^!$ agree on $\Dqcpl$ (see \S\ref{! and otimes}), that of \circled3 is~given by that of \eqref{transitivity}.
That of the unlabeled subdiagrams is clear (by~functoriality).
Proposition~\ref{Transitivity} results.
\end{proof}
\end{subcosa}

\begin{subcosa}\label{indsq}

To each oriented fiber square in $\SSp$
\begin{equation*}
    \begin{tikzpicture}[yscale=.85]
      \draw[white] (0cm,0.5cm) -- +(0: \linewidth)
      node (E) [black, pos = 0.41] {$Y'$}
      node (F) [black, pos = 0.59] {$Y$};
      \draw[white] (0cm,2.65cm) -- +(0: \linewidth)
      node (G) [black, pos = 0.41] {$X'$}
      node (H) [black, pos = 0.59] {$X$};
      \draw [->] (G) -- (H) node[above, midway, sloped, scale=0.75]{$v$};
      \draw [->] (E) -- (F) node[below, midway, sloped, scale=0.75]{$u$};
      \draw [->] (G) -- (E) node[left,  midway, scale=0.75]{$g$};
      \draw [->] (H) -- (F) node[right, midway, scale=0.75]{$f$};
      \node (C) at (intersection of G--F and E--H) [scale=0.8] {$\Dd$};
    \end{tikzpicture}
  \end{equation*}
with $u,v$ flat, there is associated a functorial \emph{isomorphism}
\begin{equation*}\label{thetaB}
\bchadmirado\Dd\colon v^*\!f^!\iso g^!u^*,\tag{\ref{indsq}.1}
\end{equation*}
satisfying ``transitivity" with respect to vertical and horizontal composition of such squares
(see \cite[\S5.8]{AJL}).   

With  the ``relative dualizing complexes" $\cd f\set f^!_{\<\<\upl}\OY$,  $\cd g\set g^!_\upl\CO_{Y'}$, and with\looseness=-1
\[
\<\bar{\>\mathsf B}_{\mkern-.5mu\Dd}\colon v^*\cd f= v^*\!f^!_\upl\OY\iso g^{\>!}_\upl u^*\OY
=\cd g
\]
as in \cite[5.8.3]{AJL}, $\bchadmirado\Dd$ is the natural composition
\begin{equation*}\label{def-of-B}
v^*\<\<f^!=v^*(\cd f\otimes f^*) \iso v^*\cd f\otimes v^*\<\<f^*
\xto{\<\bar{\>\mathsf B}_{\mkern-.5mu\Dd}\,\otimes\,\ps^*}
\cd g\otimes g^*u^* = g^!u^*.
\end{equation*}

The next Lemma will be used in the proof of Theorem~\ref{bch-fund-class}.
\begin{sublem}\label{B-theta-int}
For any fiber square diagram in\/ $\SS\colon$
\[ 
 \begin{tikzpicture}[yscale=.75]
      \draw[white] (0cm,0.5cm) -- +(0: \linewidth)
      node (E) [black, pos = 0.4] {$X$}
      node (F) [black, pos = 0.59] {$Y$};
      \draw[white] (0cm,2.65cm) -- +(0: \linewidth)
      node (G) [black, pos = 0.4] {$\bullet$}
      node (H) [black, pos = 0.59] {$\bullet$};
      \draw[white] (0cm,4.8cm) -- +(0: \linewidth)
      node (I) [black, pos = 0.4] {$\bullet$}
      node (J) [black, pos = 0.59] {$\bullet$};
      \draw [->] (G) -- (H) node[above, midway, sloped, scale=0.75]{$h$};
      \draw [->] (E) -- (F) node[below=1pt, midway, sloped, scale=0.75]{$u$};
      \draw [->] (G) -- (E) node[left, midway, scale=0.75]{$k\>$};
      \draw [->] (H) -- (F) node[right, midway, scale=0.75]{$\>g$};
      \draw [->] (I) -- (J) node[above=1pt, midway, sloped, scale=0.75]{$v$};
      \draw [->] (I) -- (G) node[left, midway, scale=0.75]{$j\>$};
      \draw [->] (J) -- (H) node[right, midway, scale=0.75]{$\>f$};
      \node (K) at (intersection of I--H and G--J)    {$\lift.35,\De,$};
      \node (L) at (intersection of G--F and H--E)  {$\lift.85,\Dd,$};
 \end{tikzpicture}
 \]
 with\/  $u$ $($hence $h$ and $v)$ flat, $f$ $($hence $j)$ proper, and $g,$ $gf,$ $k$ and\/ $kj$ perfect, 
the following diagram commutes.
\[
 \bpic[xscale=2.75, yscale =1.5]
  \node(11) at (1,-1) {$j^{}_{\<*}\<v^*\<(gf)^!$};
  \node(12) at (2,-1) {$j^{}_{\<*}\<(kj)^{\<!}u^*$};  
  
  \node(21) at (1,-2) {$h^{\<*}\!\fst(gf)^{\<!}$};  
  
  \node(31) at (1,-3) {$h^{\<*}\<\<g^!$};
  \node(32) at (2,-3) {$k^!u^*$};
  
  \draw[->] (11)--(12) node[above ,midway, scale=.75] {$\mathsf B_{\Dd\De}$}; 
  \draw[->] (31)--(32) node[below ,midway, scale=.75] {$\mathsf B_{\Dd}$};
  
  \draw[<-] (11)--(21) node[left ,midway, scale=.75] {$\theta_{\<\De}$}; 
  \draw[->] (21)--(31) node[left ,midway, scale=.75] {$ \couni{\!\!f}^{\>g_{\mathstrut}}$}; 
  \draw[->] (12)--(32) node[right, midway, scale=.75] {$ \couni{\<\<\!j}^{\>k_{\mathstrut}}$}; 

 \epic
\]
\end{sublem}

\begin{proof} The assertion is that the border of the following diagram commutes---where $\CO\set\OY$ and\/ 
$\CO^*\set u^*\OY=\OX$, and where each map is induced by the natural transformation(s)  specified in its label.
\[
\def\1{$j^{}_{\<*}\<v^*\<((gf)_\upl^!\CO\<\otimes\<(gf)^*\<)$}
\def\2{$j^{}_{\<*}\<(v^*\<(gf)_\upl^!\CO\<\otimes\< v^*\<(gf)^*\<)$}
\def\3{$j^{}_{\<*}\<(\<(kj)_\upl^!\CO^*\!\otimes\<\< (\<kj)\<^*u^*\<)$}
\def\4{$h^{\<*}\!\fst\<((gf)_\upl^!\CO\<\otimes\<(gf)^*\<)$}
\def\5{$j^{}_{\<*}\<(v^*\!f_{\<\upl}^!\>g_\upl^!\CO\<\otimes\< j^*h^{\<*}g^*\<)$}
\def\6{$j^{}_{\<*}\<(j_\upl^!k_\upl^!\CO^*\!\otimes\< j^*\<k^*\<u^*\<)$}
\def\7{$h^{\<*}\!\fst\<(f_{\<\upl}^!\>g_\upl^!\CO\<\otimes\< f^*\<g^*\<)$}
\def\8{$j^{}_{\<*}\<v^*\!f_{\<\upl}^!\>g_\upl^!\CO\<\otimes\< h^{\<*}\<\<g^*$}
\def\9{$j^{}_{\<*}j_\upl^!h^{\<*}\<\<g_\upl^!\CO\<\otimes\< h^{\<*}\<\<g^*$}
\def\ten{$j^{}_{\<*}j_\upl^!k_\upl^!\CO^*\!\otimes\< h^{\<*}\<\<g^*$}
\def\lvn{$h^{\<*}\<(\fst f_{\<\upl}^!\>g_\upl^!\CO\<\otimes\< g^*\<)$}
\def\twv{$h^{\<*}\!\fst f_{\<\upl}^!\>g_\upl^!\CO\<\otimes \<h^{\<*}\<\<g^*$}
\def\thn{$h^{\<*}\<(g_\upl^!\CO\<\otimes\< g^*\<)$}
\def\frn{$h^{\<*}\<\<g_\upl^!\CO\<\otimes\< h^{\<*}\<\<g^*$}
\def\ffn{$k_\upl^!\CO^*\<\otimes\< h^{\<*}\<\<g^*$}
 \bpic[xscale=4.5, yscale=1.45]

   \node(11) at (1,-1){\1};
   \node(12) at (2,-1){\2};   
   \node(13) at (3,-1){\3};
   
   \node(21) at (1,-2){\4};
   \node(22) at (2,-2){\5};
   \node(23) at (3,-2){\6};
  
   \node(31) at (1,-3){\7};
   \node(32) at (2,-3){\8};
   \node(33) at (3,-3){\ten};
   
   \node(42) at (2.5,-4){\9};

   \node(51) at (1,-4.9){\lvn};
   \node(52) at (2,-4.9){\twv};

   \node(61) at (1,-5.9){\thn};
   \node(62) at (2.25,-5.9){\frn};
   \node(63) at (3,-5.9){\ffn};
   
   \draw[->] (11)--(12) node[above=1pt, midway, scale=.75]{\eqref{^* and tensor}\;};
   \draw[->] (12)--(13) node[above=1pt, midway, scale=.75]{$\ps^*$}
                                node[below, midway, scale=.75]{$\<\bar{\>\mathsf B}_{\mkern-.5mu\Dd\De}$};

   \draw[->] (51)--(52) node[above, midway, scale=.75]{\eqref{^* and tensor}};

   \draw[->] (61)--(62) node[below, midway, scale=.75]{\eqref{^* and tensor}};
   \draw[->] (62)--(63) node[below, midway, scale=.75]{$\<\bar{\>\mathsf B}_{\mkern-.5mu\Dd}$};

   \draw[<-] (11)--(21) node[left=1pt, midway, scale=.75]{$\theta_{\De}$};  
   \draw[double distance=2pt] (21)--(31) node[right=1pt, midway, scale=.75]{$\ps^*$}
                                                            node[left=1pt, midway, scale=.75]{\raisebox{4pt}{$\ps_\upl^!$}}; 
   \draw[->] (31)--(51) node[left=1pt, midway, scale=.75]{\eqref{projection}};
   \draw[->] (51)--(61) node[left=1pt, midway, scale=.75]{$\int^{}_{\<\dpl}$};
  
  \draw[double distance=2pt] (12)--(22) node[right=1pt, midway, scale=.75]{$\ps^*$}
                                                            node[left=1pt, midway, scale=.75]{\raisebox{4pt}{$\ps_\upl^!$}};  
   \draw[->] (22)--(32) node[left=1pt, midway, scale=.75]{\eqref{projection}};
   \draw[<-] (32)--(52) node[left, midway, scale=.75]{$\theta_{\De}$};  
 
   \draw[double distance=2pt] (13)--(23) node[right=1pt, midway, scale=.75]{$\ps^*$}
                                                            node[left=1pt, midway, scale=.75]{\raisebox{4pt}{$\ps_\upl^!$}};  
   \draw[->] (23)--(33) node[right=1pt, midway, scale=.75]{\eqref{projection}};
   \draw[->] (33)--(63) node[right=1pt, midway, scale=.75]{$\int^{}_{\<\dpl}$};
 
  \draw[->] (32)--(42)  node[above=-3.5pt, midway, scale=.75]{$\mkern25mu\bar{\>\mathsf B}_{\mkern-.5mu\De}$};
  \draw[->] (42)--(33)  node[above=-3pt, midway, scale=.75]
     {$\bar{\>\mathsf B}_{\mkern-.5mu\Dd}\mkern33mu$};
  \draw[->] (52)--(62) node[below=-6pt, midway, scale=.75]{$\int^{}_{\<\dpl}\mkern32mu$};
  \draw[->] (42)--(62)  node[below=-6pt, midway, scale=.75]{$\mkern25mu\int^{}_{\<\dpl}$};
     
   \node at (1.525,-3.05)[scale=.9]{\circled1};
   \node at (2.5,-2.5)[scale=.9]{\circled2};   
   \node at (2.225,-4.5)[scale=.9]{\circled3};                                                         
                                                      
  \epic
\]

Commutativity of subdiagram \circled1 is a consequence of the mirror image of \cite[3.7.3]{li}, with
 $(f,f'\<\<,g,g'\<\<,P,Q)\set(f,j,h,v, g^*\<\<, f_{\<\upl}^!\>g_\upl^!\CO)$.\va1

Commutativity of \circled2 follows from transitivity of $\>\bar{\>\mathsf B}$ (see \cite[\S5.8.4]{AJL}).\va1

Commutativity of \circled3 is immediate from the definition of $\>\bar{\>\mathsf B}_{\mkern-.5mu\De}$
(see  second paragraph in \cite[\S5.8.2]{AJL}).

The rest is clear.
\end{proof}
\end{subcosa}
\end{cosa}

\begin{cosa}\label{def-fc}

The \emph{fundamental class} of a flat $\SS$\kf-map $f: X\to Y$,   
\looseness=-1
\[
\fundamentalclass{\<f}\colon \Hh{\<\<X}f^*\lto f^!\>\Hh{Y},
\]
is the composition of two functorial maps, with $\Gamma=\Gamma^{}_{\!\!f} \colon X\to X\times Y$ the graph of $f$ (a map in $\SS$): 
\begin{equation}
\label{def-f-c}
\Hh{\<\<X}f^*=\delta^*_{\<\<X}\delta^{}_{\<\<X\<*}f^*
\xrightarrow[\,\fundamentalclassa{\<f}\,]{}
 \Gamma^*  \Gamma_{\<\!*} f^!
\underset{\fundamentalclassb{\<\<f}}\iso
f^!{\delta^*_Y} \delta^{}_{Y\mkern-1.5mu*}=f^!\>\Hh{Y},
\end{equation}
specified as follows.

To define $\fundamentalclassa{\<f}\colon\delta^*_{\<\<X}\delta^{}_{\<\<X\<*}  f^* \to \Gamma^*  \Gamma_{\<\!*} f^!$ consider the commutative diagram
\begin{equation}\label{lambdaf}
\CD
X@>\delta:=\:\delta_{\<\<f}\mkern1.5mu >> X\times_YX @>i>> X\times X @>p^{}_{\<\<X}>> X \\
@. @V p^{}_1 VV  @VV\id_\sX\!\times f V @VV\>fV\\
@. X@>>\Gamma > X\times Y @>>p^{}_{\>Y}> Y
\endCD
\end{equation}
where $p^{}_{\<\<X}$ and $p^{}_{\>Y}$ are the projections onto the second factor, $p^{}_1$ onto the first, and $i$ is the natural map. 

More generally, consider any commutative $\SS$\kf-diagram $\Dd$ 
\begin{equation}
\label{lambdaf2}
\CD
\bullet@>1 >> \bullet @>2>> \bullet @>3>>\bullet \\
@. @V 6 VV  @V7 VV @VV8 V\\
@. \bullet@>>4 > \bullet@>>5> \bullet
\endCD
\end{equation}
in which  the squares are fiber squares, the map 8 is flat, 1 is proper  and both 
$3 \smallcirc 2 \smallcirc\<1$ and $5 \smallcirc 4$ are perfect.
To~ such a~$\Dd$  associate the map $\lambda_{\Dd}$ given by the composition
\begin{align*}
       (2\smallcirc\<1)_* (3 \smallcirc 2 \smallcirc\<1)^!8^*
     \iso &2_* 1_*(3 \smallcirc 2 \smallcirc\<1)^!8^*   \tag{$\ps_*$}\\
     \lto\ \>\>&2_*(3 \smallcirc 2)^!8^*                      \tag{$2_*\!\smallint{}_{\mspace{-8mu}\lift.65,1,}^{\!3\>{\lift.75,\halfsize{$\circ$},} \>2\>\>}$}\\
     {\,\iso}   &2_* 6^*(5 \smallcirc 4)^!                     \tag{$2_*{\bchadmirado{}^{-\<1}}$}\\
     {\,\iso}  &7^*\< 4_{\>*}(5 \smallcirc 4)^!              \tag{${\bchasterisco{}^{-\<1}}$ (\S\ref{thetaiso})}
  \end{align*}
In  \eqref{lambdaf},  $i \smallcirc \delta = \delta_{\<\<X}$, $p^{}_{\<\<X}\smallcirc i \smallcirc \delta = \id_\sX$, and $p^{}_{\>Y} \smallcirc \Gamma = f$.
Thus one has a map
\begin{equation}
\label{def-lambda}
{\delta^{}_{\<\<X\<*}}f^*={\delta^{}_{\<\<X*}}\<\id_\sX^!\< f^* \lto (\id_\sX\! \times f)^{\<*}\> \Gamma_{\<\!*} f^!\<,
\end{equation}
to which one applies $\delta^*_{\<\<X}\cong\delta^*\>i^*$ to produce the natural composite map
\[
 \fundamentalclassa{\<f}\colon \delta^*_{\<\<X}\delta^{}_{\<\<X\<*} f^* \lto  
\delta^*\<i^*(\id_\sX\! \times f)^{\<*}\>  \Gamma_{\<\!*} f^!\iso 
\delta^*\<p^*_1\Gamma^*\Gamma_{\<\!*} f^!\iso
\Gamma^* \Gamma_{\<\!*} f^!\<.\]

\goodbreak

To define the isomorphism $\fundamentalclassb{\<\<f}\colon\Gamma^*  \Gamma_{\<\!*} f^!
\liso f^! \delta_Y^*  \delta^{}_{Y*}$ in (\ref{def-f-c}), consider the fiber square
\begin{equation}
\CD
X @>\!f>> Y\\
@V\Gamma V\mkern65mu\scalebox{.9}{$\Dh$} V @VV\delta_Y V \\
X\times Y @>>f\times\>\id_Y > Y\times Y
\endCD
\end{equation}
Let $p\colon X\times Y \to X$ be the projection, so that $p\> \Gamma=\id$, and $\Gamma^* \<p^*$ is isomorphic to the identity functor of $\>\D_\sX$.  
One has then the functorial
isomorphism
\begin{equation}\label{def-of-mu}
\mu^{}_{\<f}=\mu_{\Gamma\!,\,p}\colon  \Gamma^*  \Gamma_{\<\!*} (A\otimes\<B) \iso A \otimes  \Gamma^* \Gamma_{\<\!*} B \qquad (A, \, B\in \D_{\<X})
\end{equation}
that is defined to be the natural composite isomorphism
  \begin{align*}
        \Gamma^* \Gamma_{\<\!*} (A\otimes\<B)
\underset{\via\ps^*}{=\!=\!=} \Gamma^* \Gamma_{\<\!*} (\Gamma^* \<p^*\! A \otimes\<B)
& \underset{\eqref{projection}^{-\<1}}{\iso} \>  \Gamma^*\< (p^*\! A \otimes \Gamma_{\<\!*}B)   \\
& \underset{\eqref{^* and tensor}}{\iso}
\Gamma^* \<p^*\! A \otimes  \Gamma^* \Gamma_{\<\!*} B
\underset{\via\ps^*}{=\!=\!=} A \otimes \Gamma^* \Gamma_{\<\!*} B. 
  \end{align*}

The map  $\fundamentalclassb{\<\<f}$ is the composite isomorphism (with $\defnu{\Dh}$ as in~\ref{ayuda0})
\begin{equation}\label{def-b}
\begin{split}
\Gamma^*\Gamma_{\<\!*}f^! =\Gamma^*\Gamma_{\<\!*}(f^!\CO_Y\otimes f^*)
&\>\underset{\mu^{}_{\<f}}\iso  f^!\CO_Y\otimes \Gamma^*\Gamma_{\<\!*}f^* \\
&\>\underset{1\>\otimes\> \phi_{\Dh}^{-\<1}}\iso\>\> 
f^!\CO_Y\otimes f^*\<\delta^*_Y \delta^{}_{Y\mkern-1.5mu*} 
= f^!{\delta^*_Y} \delta^{}_{Y\mkern-1.5mu*}\>.
\end{split}
\end{equation}

\vskip3pt
This completes the definition of the fundamental class $\fundamentalclass{\<f}$. 
\end{cosa}

\begin{ex}\label{absolute fc}
In the ``absolute" case, when $Y=S$ and $f$ is the structure map $x\colon X \to S$ (assumed flat), the map
$\fundamentalclassb{x}$ is the identity.
Diagram \eqref{lambdaf} collapses to 
\[
\CD
X @>\delta\>:=\,\delta_{\<\<X}\> >> X\times X @>p^{}_2 >> X \\
@. @V p^{}_1 VV @VVxV \\
@. X @>>\lift1.1,x,> S
\endCD
\]
with $p^{}_1$ and $p^{}_2$  the projections onto the first and second factors, respectively.
The fundamental class $\fundamentalclass{x}=\fundamentalclassa{x}$ is  then the composite map
$$
\boxed{
\delta^*\delta^{}_*\>x^*=\delta^*\delta^{}_*(p^{}_2\smallcirc\delta)^!x^*
  \underset{\!\delta^*\!\!\smallint_{\<\<\delta}^{{\lift.4,\sss p,}^{}_{\halfsize{$\sst2$}}}}{\lto}
\delta^*{p^!_2}\>x^*
 \underset{\delta^*\bchadmirado{}^{\<-\<1}\>}{\iso}
\delta^*p_1^*\>x^! 
  \underset{\ps^*\>\>}{\iso} 
\id^*\<\<x^!=x^!.
}
$$
\end{ex}

\begin{small}
\begin{subrem}  What happens to $\Dc_x$ when   $p^{}_1$ and~$p^{}_2$ in its definition are interchanged? Denoting the resulting map by $\Dc_x'$, one can show that 
$\Dc_x'= \Dc_x\smallcirc e_x$, where, $\sigma\colon X\times X\to
  X\times X$ being the symmetry isomorphism (that is, $p_1\sigma=p_2$ and~$p_2\sigma=p_1$), $e_x\colon\Hh X\to\Hh X$ is the automorphism given by the composition
$$
\delta^*\delta^{}_*=(\sigma\delta)^*(\sigma\delta)_*\overset{\ps^*}{\underset{\ps_*}{=\!=}}
\delta^*\sigma^*\<\sigma_*\delta_* \xto{\delta^*\epsilon_\sigma} \delta^*\delta^{}_*.
$$
The square of this automorphism is easily seen to be the identity, but the automorphism itself need not be. For example, if $x$ is smooth,
then working locally with a Koszul resolution of $\delta_*\OX$, one finds that $e_x(\OX)$ induces multiplication by $(-1)^i$ on $H^i(\Hsch{\sX})$. (See also \cite[Exercise 3.4.4.1]{li}.)

The map $\Dc_x$ is canonical in that though the pair $(\Hh X(E), \Dc_x(E))\ (E\in\D_\sX)$ 
depends on a choice of flat resolution for the 
$\mathcal O_{\<\<X\times X}$-complex $\delta_*E$,  for two such choices there is a canonical isomorphism between the resulting pairs.  What this example illustrates (for instance when $x$ is the natural map $\spec R[T]\to\spec R$ with $R$ a ring  and $T$ an indeterminate) is  that the canonical isomorphism can induce the identity on~$\Hh{\sX}(\OX)$ while not inducing the identity 
on $\Dc_x(\OX)$. 
\end{subrem}
\end{small}

\begin{ex}\label{I/I^2 and H}
If $V\set X\times X$ and $\CI\>$ is the kernel of the natural surjection $\OV\to\delta_*\OX$,
then using any flat resolution of $\delta_*\OX$ one gets a ``natural"%
\footnote{The negative of this isomorphism is equally natural. This leads to some sign issues,
which we will not get into here.}
isomorphism of $\OX$-modules
$$
\Omega^1_x=
\CI/\>\CI^{\>2}\cong \mathcal T\!or_1^{\OV}\!(\OV\<\</\>\CI,\>\OV\<\</\>\CI\>)=H^{-1}\Hsch{x}\>,
$$
whence a map of graded-commutative $\OX$-algebras, with $\Omega^i_x\set \wedge^{\!i}\>\Omega^1_x\>$,
\begin{equation}\label{Omega and H}
\oplus_{i\ge 0} \,\Omega^i_x \to\oplus_{i\ge 0}\,\mathcal T\!or_i^{\OV}\!(\OV\<\</\>\CI,\>\OV\<\</\>\CI\>)
=\oplus_{i\ge 0}\,H^{-i}\Hsch{x}\>.
\end{equation}

\begin{subprop}\label{fc+V} With notation as in ~\textup{\ref{absolute fc},} if $x\colon X\to S$ is essentially smooth of relative dimension $n$ $(\<$see \textup{\S\ref{efp}),} then there is a natural composite 
$\D_\sX$-isomorphism
\begin{equation*}\label{c gives iso}
\Omega_x^n[n]\iso (H^{-n}\Hsch{X})[n] \underset{\textup{via}\:\Dc_x}{\iso} (H^{-n}x^!\OS)[n]\iso x^!\OS.
\end{equation*}

\end{subprop}
\end{ex}

\begin{proof}
Since $x$ is essentially smooth the map \eqref{Omega and H} is an \emph{isomorphism}, as can be checked locally, over affine open sets in~$V$ where 
$\CI\>$ is generated by a regular sequence of length $n$,  whose associated Koszul complex provides a flat resolution of $\>\OV\</\>\CI$ (see \S\ref{efp}).

The complex $x^!\CO_S$ is  concentrated in degree $-n$: there exists an isomorphism 
$\Omega_x^n[n] \cong x^!\OS$
(the proof of~\cite[\!p.\,397,\kern2pt Thm.\,3]{V} holds for essentially smooth maps),\va{.6} whence a natural isomorphism $\,(H^{-n}x^!\OS)[n]\iso x^!\OS$. Likewise for 
the complex $p_2^!\OX$, since $p_2$ is also essentially smooth of relative dimension~$n$
(or by the flat base\kf-change isomorphism $p_2^!\OX\cong p_1^*x^!\OS$).\va{.6} 

In view of Example~\ref{absolute fc} and the definition of \eqref{specialint}, the problem is readily reduced to showing that \emph{the natural map}
\begin{equation}\label{! and RHom}
H^{-n}\delta^*\delta^{}_*\delta_\upl^!\>p_2^!\OX
 \to
H^{-n}\delta^*p_2^!\OX.
\end{equation}  
\emph{is an isomorphism}.

Using that $\delta_\upl^!$ is right-adjoint to $\delta^{}_*\>$, one can
identify $\delta^{}_*\delta_\upl^!\>p_2^!\OX
 \to
p_2^!\OX$ with the map
$
\R\sHom(\delta_*\OX\<,\> p_2^!\OX\<)\to \R\sHom(\OV\<, \>p_2^!\OX\<)
$
induced by the natural map $\OV\to \delta_*\OX$,  and~then check \eqref{! and RHom} locally, where, again,
one can replace
 \mbox{$\delta_*\OX\cong\OV\<\</\>\CI$}
 by  the Koszul complex of a regular sequence.
\end{proof}

\begin{subrem}\label{c vs V} As of this writing, the authors do not know whether the natural isomorphism in \ref{fc+V} coincides (up to sign?) with that of Verdier. Nor do we know whether
\emph{either} of these isomorphisms becomes $\pm$identity when all the data are interpreted as
in \cite{RD} or \cite{Co}. 

The answers might well emerge from the relation of these maps to traces and
residues, a relation to be explored in detail elsewhere.
\end{subrem}

\begin{prop}\label{c inverse} 
If the\/ $\SS$\kf-map\/ $f\colon X\to Y$ is \emph{essentially \'etale} then with\/ 
$y\colon Y\to S$ the structure map and \va{-4}
$$
\boldsymbol f^\sharp\colon f^!\Hh{Y}=f^*\Hh{Y}\xto{\!\!\bbbiup{(f,\>y)}\!} \Hh{X}f^*
$$
 the isomorphism from Theorem~\ref{caset}, 
it holds that\va{-1.5}
$$
\Dc_{\<f}=\big(\boldsymbol f^\sharp\>\big){}^{-1}.
$$
\end{prop}

\begin{proof}
Let us see what $\Dc_{\<f}=\Db_{\<\<f}\mkern-.5mu\smallcirc\Da_{\<f}$ looks like when the
pseudofunctorial identification of $(-)^!$ with
$(-)^*$ for essentially \'etale maps is implemented. 

More specifically,  with reference to \eqref{d^fg}, and notation as in~\eqref{lambdaf},
consider the following decomposition  $(\<\Dd^{f\<\<,\mkern1.5mu y})^{}_{\<\times}=\Du\Dv\Dw$:\va{-2} 
\begin{equation}
\CD\label{decomp}
 \bpic[xscale=2.5,yscale=1.5]
  \node(11) at (1,-1) {$X$};
  \node(12) at (2,-1) {$X$};
  \node(13) at (3,-1) {$Y$};
    
  \node(21) at (1,-2) {$X\!\times_Y\! X$};
  \node(22) at (2,-2) {$X$};
      
   \node(31) at (1,-3) {$X\!\times \!X$};
   \node(32) at (2,-3) {$X\!\times \!Y\!$};
   \node(33) at (3,-3) {$Y\!\times \!Y$};

   \draw[double distance=2pt] (11)--(12);
   \draw[->] (12)--(13) node[above=1pt, midway, scale=.75]{$f$};
   
   \draw[->] (1.3, -2)--(1.78, -2) node[above, midway, scale=.75]{$\!\<p^{}_{\mkern-.5mu1}$};
   
   \draw[->] (31)--(32) node[below=1pt, midway, scale=.75]{$\id_X\!\times f\>$};
   \draw[->] (32)--(33) node[below=1pt, midway, scale=.75]{$f\<\<\times\<\<\id_Y$};
   
  \draw[->] (11)--(21) node[left=1pt, midway, scale=.75]{$\delta$};
  \draw[double distance=2pt](12)--(22);
  \draw[->] (13)--(33) node[right=1pt, midway, scale=.75]{$\delta_Y$};
  
  \draw[->] (21)--(31) node[left=1pt, midway, scale=.75]{$i$};
  \draw[->] (22)--(32) node[right=1pt, midway, scale=.75]{$\Gamma$};
  
  \node at (1.5, -1.5)[scale=.8]{$\ \Dw$}; 
  \node at (1.5, -2.5)[scale=.8]{$\ \Dv$};
  \node at (2.5, -2)[scale=.8]{$\Du$};
 \epic
\endCD
\end{equation} 
By the definition of $\bbbiup{(f,\>y)}$ (see \eqref{bbbiup} and \eqref{def-biup}), 
and  Proposition~\ref{ayuda},
$$
\boldsymbol f^\sharp\set\phi_{\Du\Dv\Dw}=\phi_{\Dv\Dw} f^*\<\<\smallcirc\phi_{\Du}.
$$
Since $\phi_{\Du}$ is an isomorphism
(see Proposition~\ref{ayuda0}),   therefore so is $\phi_{\Dv\Dw} f^*$.
To prove Proposition~\ref{c inverse}  it will suffice then to show that
\begin{equation}\label{a_f}
\Da_{\<f}=(\phi_{\Dv\Dw} f^*)^{-1}
\end{equation}
and
\begin{equation}\label{b_f}
\Db_{\<f}=\phi_{\Du}^{-1}.
\end{equation}

\smallskip
We first treat some constituent parts of the definition of $\Da_{\<f}$.
\va1

\begin{subcosa}\label{(A)}
Since $f$ is essentially \'etale, therefore so too are the diagonal map 
\mbox{$\delta\colon X\to X\times_Y X$} and the projection 
$p_j\colon X\times_Y X\to X$ to the
$j$-th factor ($j=1,2$), see \S\ref{efp}. Since the identification  $(-)^!=(-)^*$ for essentially \'etale maps is pseudofunctorial, therefore   \emph{for each\/ $j,\ $the two isomorphisms}\va{-1}
\begin{equation*}\label{two isos}
\id=(p_j\delta)^!\overset{\ps^!}{=\!=} \delta^!p_j^! \quad \textup{and} \quad\id=(p_j\delta)^*\overset{\ps^*}{=\!=} \delta^*\<p_j^*
\end{equation*} 
\emph{are identical.}\va1
\end{subcosa}

\begin{subcosa}\label{(B)} 
In \eqref{thetaB}, \emph{if\/ $f$ is essentially \'etale} (whence so is $g$, 
see~\S\ref{efp}) \emph{then the following diagram commutes.} \va{-3}
\[\mkern20mu
 \bpic[xscale=2.2,yscale=1.45]
  \node(11) at (1,-1) {$v^*\<\<f^!$};
  \node(12) at (2,-1) {$g^!u^*$};
  \node(21) at (1,-2) {$v^*\<\<f^*$};
  \node(22) at (2,-2) {$g^*u^*$};
  
  \draw[->] (11)--(12) node[above, midway, scale=.75]{$\bchadmirado{\Dd}$};
  \draw[double distance=2pt](11)--(21);
  \draw[double distance=2pt](12)--(22);
  \draw[double distance=2pt](21)--(22) node[below=1pt, midway, scale=.75]{$\ps^*$};
 \epic
\]  
\vskip-3pt
To see this, one uses \eqref{! and otimes1}, monoidality of the
pseudofunctor $(-)^*$ (see \cite[p.\,121, 3.6.7(b)]{li}) and the \emph{dual} (see \cite[p.\,105, (3.4.5)]{li}) of diagram~\circled2 in \cite[p.\,109, 3.4.7.1]{AJL} to reduce to
showing commutativity of the diagram \emph{after} it is applied to $\OY$; and that
follows from \cite[p.\,208, Theorem 4.8.3(ii)]{li}, by 
a straightforward extension of \cite[Remark 4.8.5.2]{li} with ``finitely presentable" (resp.~``\'etale")
replaced by ``efp" (resp.~``essentially \'etale").
Details---routine, but somewhat tedious---are left to the reader. 
\end{subcosa}

\begin{subcosa}
Using Lemma~\ref{proper specint}, and applying
\ref{(A)} and~\ref{(B)}, one checks now that the map $\Da_{\<f}$ is 
the composition
\begin{align*}
\delta^*_{\<\<X}\delta^{}_{\<\<X\<*}f^*                          
 & \overset{\ps_*}{=\!=}
\delta^*_{\<\<X}i^{}_{\<*}\delta^{}_{\<*}f^*                       
\overset{\ps^*}{=\!=}
\delta^*_{\<\<X}i^{}_{\<*}\delta^{}_{\<*}\delta^*p_1^*\<\<f^*  =\!=
\delta^*_{\<\<X}i^{}_{\<*}\delta^{}_{\<*}\delta^!p_1^*\<\<f^*\\
&  \xto{\textup{via}\,\smallint^{}_{\<\<\delta}\>}
\delta^*_{\<\<X}i^{}_{\<*}p^*_1f^*                                   
\xto{\delta^*_{\<\<X}\theta_{\Dv}^{-\<1}}
\delta^*_{\<\<X}\<(\id_\sX\<\<\times\<f)^{\<*}\Gamma_{\<\!*}\< f^* 
\overset{\ps^*}{=\!=}
\delta^*\<p_1^*\Gamma^*\Gamma_{\<\!*} f^*  \overset{\ps^*}{=\!=}
\Gamma^*  \Gamma_{\<\!*} f^*\<.
\end{align*}

It follows that $\phi_{\Dv\Dw} f^*\<\<\smallcirc \Da_{\<f}$ is obtained by going around the following diagram clockwise from $\delta^*_{\<\<X}\delta^{}_{\<\<X\<*}f^*$ back to itself.\va{-2} 
\begin{equation*}
\def\1{$\delta^*_{\<\<X}\delta^{}_{\<\<X\<*}f^*$}
\def\4{$\delta^*_{\<\<X}i^{}_{\<*}\delta^{}_{\<*}f^*$}
\def\8{$\delta^*_{\<\<X}i^{}_{\<*}p^*_1f^*$}
\def\9{$\delta^*_{\<\<X}\<(\id_\sX\<\<\times\<f)^{\<*}\Gamma_{\<\!*}\< f^*$}
\def\twy{$\delta^*_{\<\<X}i^{}_{\<*}\delta^{}_{\<*}\delta^*p^*_1f^*$}
\def\twt{$\delta^*_{\<\<X}i^{}_{\<*}\delta^{}_{\<*}\delta^!p^*_1f^*$}
 \bpic[xscale=4.2,yscale=1.9]
   \node(11) at (1,-1){\1};
   \node(12) at (1.75,-1){\4};
   \node(13) at (3.2,-1){\twy};
   
   \node(21) at (1,-2){\9};
   \node(22) at (1.75,-2){\8};
   \node(23) at (3.2,-2){\twt};
  
   \draw[double distance=2pt] (11)--(12) 
           node[above=1pt, midway, scale=.75]{$\ps_*$}          
           node[below=1pt, midway, scale=.7]{1};   
                                                                  
    \draw[double distance=2pt] (12)--(13) 
                    node[above=1pt, midway, scale=.75]{$\ps^*$}
                    node[below=1pt, midway, scale=.7]{2};   

   \draw[->] (22)--(21) node[below=1pt, midway, scale=.75]{$\delta_{\<\<X}^*\theta_{\Dv}^{-\<1}$}
                                                                     node[above, midway, scale=.7]{5};   
   \draw[->] (23)--(22) node[below=1pt, midway, scale=.75]{$\delta^*_{\<\<X}i^{}_{\<*}\!\smallint^{}_{\<\<\delta}$}
                                                                     node[above, midway, scale=.7]{4};   

   \draw[->] (21)--(11)node[left, midway, scale=.7]{$\delta^*_{\<\<X}\theta_{\mkern-.5mu\Dv\Dw}$}
                                                                        node[right, midway, scale=.7]{6};
                                                                
  \draw [->] (22)--(12) node[left, midway, scale=0.75]{$\delta^*_{\<\<X}i^{}_{\<*}\theta_{\Dw}\!$}
                                                                       node[right, midway, scale=.7]{7};   

   \draw[->] (13)--(23)  node[left, midway, scale=.7]{3}
                                                                          node[right, midway, scale=.8]{\raisebox{-4.5pt}{=}};

    \draw [->] (22) --(13)
                               node[above=1.2pt, midway, scale=0.75]{$\delta^*_{\<\<X}i^{}_{\<*}\eta^{}_{\<\delta}\,$}                                                                     node[above=-11pt, midway, scale=.7]{8\kern6.5pt};

   \node at (1.35,-1.5)[scale=.9]{\circled{\raisebox{.5pt}{a}}};  
   \node at (2.05,-1.35)[scale=.9]{\circled{\raisebox{-.35pt}{\kern.25pt b}}};  
   \node at (2.8,-1.65)[scale=.9]{\circled{\raisebox{.5pt}{c}}}; 
  \epic
\end{equation*}

Commutativity of subdiagram \circled{a} results from transitivity of $\theta$ (see \cite[p.\,128, (iii)]{li}); that of \circled{\kern.25pt b} is the definition \eqref{bch.2} of $\theta_{\Dw}$; and that
of \circled{c} is given by the next Lemma. 
Thus  $\phi_{\Dv\Dw} f^*\<\<\smallcirc \Da_{\<f}=6\smallcirc 5\smallcirc 4\smallcirc 3\smallcirc 2\smallcirc\<1$ is the composition
$$
1^{-\<1}\smallcirc 7\smallcirc 4\smallcirc 3\smallcirc 2\smallcirc\<1= 
1^{-\<1}\smallcirc 2^{-\<1}\smallcirc8\smallcirc 4\smallcirc 3\smallcirc 2\smallcirc\<1=
1^{-\<1}\smallcirc 2^{-\<1}\smallcirc3^{-\<1}\smallcirc 3\smallcirc 2\smallcirc\<1,
$$
which is the identity map of~$\delta^*_{\<\<X}\delta^{}_{\<\<X\<*}f^*$; and this  proves \eqref{a_f}.

\pagebreak[3]

\begin{sublem}
The following composite map is the identity.
\begin{equation*}
 \delta_{\<\<X}^*i_*\xto{\delta_{\<\<X}^*i_*\<\eta^{}_{\delta}\>}  \delta_{\<\<X}^*i_*\delta_*\delta^*=\!= \delta_{\<\<X}^*i_*\delta_*\delta^!
 \xto{\delta_{\<\<X}^*i_*\!\smallint^{}_{\<\<\delta}}  \delta_{\<\<X}^*i_*.
\end{equation*}
\end{sublem}

\begin{proof}
As in the proof of Theorem~\ref{caset},
$\delta$ is an isomorphism of\/ $X$ onto an open-and-closed subscheme
of~$V\set X\times_Y X\<$. Let $i\colon V\to X\times X$ be the natural map,
and set $V'\set i(V\setminus \delta(X))$. We have then a fiber square, with
$k$ an open immersion,
\begin{equation*}
\mkern35mu
\CD
X @>\delta >> X\times_Y X \\
@VjVV @VViV\\
\mkern-45mu (X\times X)\setminus V' @>>k > X\times X
\endCD
\end{equation*}
and hence  isomorphisms (see~\ref{ayuda0}), with $\delta_{\<\<X}\set i\delta$,
\begin{equation}\label{j and delta}
\delta^*_{\<\<X}i_*\underset{\ps^*}{=\!=}\delta^*i^*i^{}_* 
 \underset{\eqref{def-of-phi}}{\iso} j^*\mkern-1.5mu j_*\>\delta^*\<.
\end{equation}
So it's enough to show that 
\emph{the following composite map is the identity.}
\begin{equation}\label{ijk}
 \delta^*\xto{\delta^*\<\eta^{}_{\delta}} \delta^*\delta_*\delta^*=\!= \delta^*\delta_*\delta^!
 \xto{\delta^*\!\!\smallint^{}_{\<\delta}} \delta^*\<.
\end{equation}

But $\delta^*\!\!\smallint^{}_{\<\delta}$ is the same map as
\[
\delta^*\<\<\delta_*\delta^! \xto{\epsilon^{}_{\<\delta}} \delta^! = \delta^*\<.
\]

To see this, one can use the isomorphisms \eqref{! and otimes1} and \eqref{f! and f!+}, and the description of~$\smallint^{}_{\<\<\delta}$ in \cite[\S5.9]{AJL}, to reduce to showing that the diagram commutes \emph{after} it is applied to $\mathcal O_{\<\<X\times_Y X}\>$---which follows from \cite[p.\,168, Exercise 4.9.1(c),
and~p.\,204, Theorem 4.8.1(iii)]{li}. Details are left to the reader.

Thus the  composite map \eqref{ijk} is the same as 
\[
 \delta^*\xto{\<\<\delta^*\<\eta^{}_{\delta}\>} \delta^*\delta_*\delta^*=\!= \delta^*\delta_*\delta^!
\xto{\>\>\epsilon^{}_{\<\delta }\>} \delta^! =\!= \delta^*\<,
\]
\vskip-3pt
\noindent that is,\va{-3}
\[
 \delta^*\xto{\<\<\delta^*\eta^{}_{\delta}\>} \delta^*\delta_*\delta^*=\!= \delta^*\delta_*\delta^!
=\!=\delta^*\delta_*\delta^*\xto{\>\>\epsilon^{}_\delta\>}  \delta^*\<,
\]
which is indeed the identity map.
\end{proof}
\end{subcosa}

\begin{subcosa}
As for \eqref{b_f}, recall from  \eqref{! and otimes1} and the first Remark in \S\ref{derived} the
equalities
\[\mkern24mu
 \bpic[xscale=3,yscale=1.5]
  \node(11) at (2,-2) {$f^*\OY\!\otimes f^*$};
  \node(12) at (1,-2) {$ f^!\OY\!\otimes f^*$};
  \node(21) at (2,-1) {$f^*$};
  \node(22) at (1,-1) {$f^!$};
  
  \draw[double distance=2pt] (21)--(11) ;
  \draw[double distance=2pt](11)--(12);
  \draw[double distance=2pt](22)--(12);
2  \draw[double distance=2pt](22)--(1.9,-1);
 \epic
\]  
Using this, one checks that  \eqref{b_f} asserts commutativity of the outer border of the following diagram of natural maps, where $\CO\set f^*\OY=f^!\OY$, and
$\CO_{\<\times}\set\CO_{\<\<X\<\times Y}$, so that $\CO=\Gamma^*\CO_{\<\times}$:
\[
\def\1{$\Gamma^*\Gamma_{\!\<*}f^*$}
\def\2{$\Gamma^*\Gamma_{\!\<*}(\CO\<\<\otimes f^*)$}
\def\3{$\Gamma^*\Gamma_{\!\<*}\Gamma^*\Gamma_{\!\<*} f^*$}
\def\4{$\Gamma^*\<(\CO_{\<\times}\<\<\otimes\< \Gamma_{\!\<*}f^*\<)$}
\def\5{$\Gamma^*\<(f\!\times\!\id_Y)^*\delta^{}_{Y\mkern-1.5mu*}$}
\def\6{$\Gamma^*\<\CO_{\<\times}\<\<\otimes\< \Gamma^*\Gamma_{\!\<*}f^*$}
\def\7{$f^*\<\delta_Y^*\delta^{}_{Y\mkern-1.5mu*}$}
\def\8{$\CO\<\<\otimes\! f^*\delta_Y^*\delta^{}_{Y\mkern-1.5mu*}$}
\def\9{$\CO\<\<\otimes \<\Gamma^*\<(f\!\times\!\id_Y)^{\<*}\delta^{}_{Y\mkern-1.5mu*}$}
\def\ten{$\ \>\CO\<\<\otimes\< \Gamma^*\Gamma_{\!\<*}f^*$}
\def\lvn{$\Gamma^*\Gamma_{\!\<*}\Gamma^*\<(\CO_{\<\times}\<\<\otimes\< \Gamma_{\!\<*}f^*\<)$}
\def\twv{$\Gamma^*\Gamma_{\!\<*}\<(\Gamma^*\<\CO_{\<\times}\<\<\otimes\< \Gamma^*\Gamma_{\!\<*}f^*\<)$}
\def\thn{$\Gamma^*\Gamma_{\!\<*}\<(\CO\<\<\otimes\< \Gamma^*\Gamma_{\!\<*}f^*\<)$}
 \bpic[xscale=2.75,yscale=1.5]
    \node(10) at (1,2)[scale=1]{\thn};
    \node(02) at (3.93,2)[scale=1]{\twv};

    \node(13) at (1,1)[scale=1]{\3};
    \node(03) at (3.93,1)[scale=1]{\lvn};
   
    \node(12) at (1,0)[scale=1]{\1};
    \node(14) at (3.93,0)[scale=1]{\4};

   \node(11) at (1,-1)[scale=1]{\1};
   \node(34) at (2.5,-1)[scale=1]{\ten};
   \node(24) at (3.93,-1)[scale=1]{\6};

   \node(21) at (1,-2)[scale=1]{\5};
   \node(33) at (2.5,-2)[scale=1]{\9};

   \node(31) at (1,-3)[scale=1]{\7};
   \node(32) at (2.5,-3)[scale=1]{\8};
  
   \draw[double distance=2pt] (02)--(10); 
  
   \draw[double distance=2pt] (03)--(13);

   \draw[double distance=2pt] (14)--(12) ;
   \draw[double distance=2pt] (34)--(11);
   \draw[double distance=2pt] (34)--(24);

   \draw[double distance=2pt] (33)--(21);

   \draw[double distance=2pt] (32)--(31);
   
    \draw[double distance=2pt] (10)--(13);
    \draw[->] (03)--(02) node[right, midway, scale=.75]{$\simeq$}
                                   node[left, midway, scale=.75]{\eqref{^* and tensor}};  
   \draw[->] (1.03,.183)--(1.03, .815)node[right, midway, scale=.75]{$\Gamma^*\<\eta^{}_{\>\Gamma}$};
   \draw[->] (.97,.815)--(.97,.183) node[left, midway, scale=.75]{$\Gamma^*\Gamma_{\!\<*}\epsilon^{}_\Gamma\<$};
   \draw[->] (14)--(03) node[right, midway, scale=.75]{$\Gamma^*\<\eta^{}_{\>\Gamma}$};

   \draw[double distance=2pt] (12)--(11);
   \draw[->] (24)--(14) node[right, midway, scale=.75]{$\simeq$}
                                 node[left, midway, scale=.75]{\eqref{^* and tensor}};

   \draw[->] (21)--(11) node[left, midway, scale=.75]{$\Gamma^*\theta_{\Du}$};
   \draw[->] (33)--(34) node[right, midway, scale=.75]{$1\<\<\otimes\<\Gamma^*\theta_{\Du}$};
   
   \draw[double distance=2pt] (21)--(31) node[left=1pt, midway, scale=.7]{$\ps^*$};
   \draw[double distance=2pt] (33)--(32) node[right=1pt, midway, scale=.7]{$1\<\<\otimes\<\ps^*$};

   \node at (2.5, 1.5)[scale=.9]{\circled4};   
   \node at (2.5,-.53)[scale=.9]{\circled5};

 \epic
\]

Commutativity of the unlabeled subdiagrams is clear.
That of subdiagrams \circled4 and \circled5 is straightforward to check, either directly or by
\emph{dualizing}  the first commutative square in \cite[p.\,103, (3.4.2)]{li}  (see \cite[\S(3.4.5)]{li}). 

Since $\Gamma_{\!\<*}\epsilon^{}_\Gamma\smallcirc \eta^{}_{\>\Gamma}$ is the identity map,
one finds then that the outer border does indeed commute.
\end{subcosa}

This completes the proof of \eqref{b_f} and of Proposition~\ref{c inverse}.
\end{proof}

\medskip
Here is one more concrete illustration.

\begin{ex} \label{c and trace} 
Let $A$ be a noetherian ring, and $B$ a finite\kf-rank projective \mbox{$A$-algebra,}
with corresponding 
scheme\kf-map \mbox{$x\colon X=\spec B\to\spec A= S$.} From adjointness of the functors $x^!$ and $x_*$
one gets a canonical isomorphism
$
x^!\OS\cong \Hom_A(B, A)^\sim,
$ 
via which $\fundamentalclass{x}(\OS)$ can be identified with an $\OX$-homomorphism
$$
H^0(\delta^*\delta^{}_*\OX\<) = \OX\to \Hom_A(B, A)^\sim,
$$
the sheafification of a $B$-homomorphism $c\colon B\to\Hom_A(B,A)$. 

This identification being made, one finds, following through definitions, that $c$ factors as 
\begin{align*}
B\xto{\ \,\,\,} \Hom_A(B,B)\otimes_{B\otimes_AB} B 
&\iso(\Hom_A(B,A)\otimes_A B)\otimes_{B\otimes_AB} B \\
&\iso \Hom_A(B,A)\otimes_B B\\
&\iso \Hom_A(B,A),
\end{align*}
where the first map takes $b\in B$ to $\id\otimes \>\>b\>$, and the isomorphisms are the natural ones.
Hence $c(1)$ is the \emph{trace map} $B\to A$. 

We won't use this example further, so details are left to the reader.

This is a simple case of a fundamental relation, to be
treated \mbox{elsewhere,}  between $f_{\<*}\Dc_{\<f}$ ($f$ any finite
$\SS$\kf-map) and a certain trace map for~Hochschild complexes (cf.~\cite[\S\S4.5--4.6]{white},
\cite[p.\,55, Proposition 6.3.1]{anl}). 
\end{ex}

\section{Transitivity of the fundamental class; bivariant interpretation}\label{bivorient}
After stating the central ``transitivity" result of this paper, Theorem~\ref{trans fc}---whose proof will be given 
in~\S\ref{provetran}---we interpret it in terms of orientations for flat maps in a bivariant
Hochschild theory (\S\ref{orientations}), orientations that are compatible with essentially \'etale base change
(Corollary~\ref{bch-fundamental-class}); and, in \S\ref{Gysin}, illustrate by a brief discussion of the resulting Gysin maps for bivariant 
homology and cohomology.\va2

Notation remains as in \S\ref{prelims}. For any $W$ in $\SS$, 
$\delta_W\colon W\to W\times W$ 
denotes the diagonal. For any flat $\SS$\kf-map $f$, 
the fundamental class $\Dc_{\<f}$ is as in~\S\ref{def-fc}.
\va1

\begin{thm}
\label{trans fc}

Let $X\xto{\,u\,} Y\xto{\,v\,} Z$ be flat $\SS$\kf-maps. The following functorial diagram commutes.
\[
 \begin{tikzpicture}
      \draw[white] (0cm,2cm) -- +(0: \linewidth)
      node (H) [black, pos = 0.2] {$\delta^*_{\<\<X}\delta^{}_{\<\<X\<*}u^*v^*$}
      node (I) [black, pos = 0.5] {$u^!\delta^*_Y \delta^{}_{Y\mkern-1.5mu*}v^*$}
      node (J) [black, pos = 0.8] {$u^!v^!\delta^*_{\<\<Z} \delta^{}_{\<\<Z*}$};
      \draw[white] (0cm,0.5cm) -- +(0: \linewidth)
      node (E) [black, pos = 0.2] {$\delta^*_{\<\<X}\delta^{}_{\<\<X\<*} (v u)^*$}
      node (G) [black, pos = 0.8] {$(v u)^!\delta^*_{\<\<Z} \delta^{}_{\<\<Z*}$};
      \draw [double distance=2pt]
                 (H) -- (E) node[left, midway, scale=0.75]{$\via\,\ps^*$};
      \draw [double distance=2pt]
                 (J) -- (G) node[right=1pt, midway, scale=0.75]{$\via\,\ps^!$};
      \draw [->] (H) -- (I) node[above, midway, scale=0.75]{$\fundamentalclass{u}v^*$};
      \draw [->] (I) -- (J) node[above, midway, scale=0.75]{$u^!\fundamentalclass{v}$};
      \draw [->] (E) -- (G) node[below, midway, scale=0.75]{$\fundamentalclass{v u}$};
\end{tikzpicture}
\]

\end{thm}

\smallskip
\begin{subrems}  (A) When $u$ and $v$ are both essentially \'etale, the assertion results, in view of Proposition~\ref{c inverse}, from Corollary~\ref{composedsquare}.

(B)  If $u$ (but not necessarily $v$)  is essentially \'etale, then 
Proposition~\ref{c inverse} and Theorem~\ref{trans fc} provide a canonical identification of  $\Dc_{vu}$ with $u^*\Dc_v\>$.
\end{subrems}
\vskip2pt

\begin{cosa}\label{orientations}
 In view of Theorem~\ref{trans fc},  fundamental classes are \emph{orientations} for the flat maps in a suitable \emph{bivariant Hochschild theory,}  as follows.
 
 The setup for the theory is constructed in \cite[\S5]{AJL}.  The underlying category is 
 $\SSp\subset\SS$, the category of \emph{perfect} 
$\SS$\kf-maps, the \emph{confined maps} being the proper $\SSp$\kf-maps, and the 
\emph{independent squares}
being the oriented fiber squares with essentially \'etale bottom arrow, cf.~\cite[\S5.1.5(a)]{AJL}.
Coefficients 
are provided by the \emph{pre-Hochschild complexes}~$\Hsch{\<X}\ (X\in\SS)$ in~\S\ref{preHoch} above.
That these $\Hsch{\<X\<}$  satisfy the conditions  at~the beginning of
\cite[\S3.2]{AJL} is seen as follows.


First, for any $\SS$\kf-map $f\colon X\to Y\<$, letting $y\colon Y\to S$ be the structure map define
$f^\sharp\colon f^*\Hsch{Y}\to\Hsch{\<X}$ to be the map \smash{$\bbbiup{(f,y)}(\OY)$} (see \eqref{bbbiup}).

If $Y=X$ and $f$ is the identity map, then $f^\sharp$ is the identity map of $\Hsch{\<X}$.

Next, in Corollary \ref{composedsquare}, suppose $S''=S'=S$, both maps $S''\to S'\to S$ being the identity.
With the preceding notation, the conclusion is that the~following diagram commutes:
 \begin{equation}\label{trans^sharp}
 \CD
  \bpic[xscale=2.9, yscale=1.8]

   \draw (0,-1) node (11){$(g\<f)^*\>\Hsch{Z}$};
   \draw (1,-1) node (13){$\Hsch{\<X\<}$};

   \draw (0,-2) node (21){$f^*\!g^*\>\Hsch{Z}$};
   \draw (1,-2) node (23){$f^*\>\Hsch{\<Y}$};

   \draw[->] (11)--(13) node[above, midway, scale=0.75]{$(gf)^\sharp$};
   
   \draw[double distance=2pt] (11)--(21) node[left=1pt, midway, scale=0.75]{$\ps^*$};
   \draw[->] (23)--(13) node[right=1pt, midway, scale=0.75]{$f^\sharp$};
   
   \draw[->] (21)--(23) node[below, midway, scale=0.75]{$f^*\<\<g^\sharp$};
 \epic
 \endCD
\end{equation}  
This is the diagram (3.2.1) in \cite{AJL}, a diagram whose commutativity is required for the bivariant
theory constructed there.

The remaining requirement in \cite[\S3.2]{AJL}, 
that $f^\sharp$ be an isomorphism when $f$ is essentially \'etale,
is given by Theorem~\ref{caset}.\va2

To any $\SSp$\kf-map $f\colon X\to Y$ this bivariant theory assigns the graded group
\[
\HH^*(f)\set \oplus_{j\in\mathbb Z}\,\ext^j(\Hsch{\<X},f^!\Hsch{Y}).
\]

So for flat $f$ the fundamental class $\Dc_{\<f}$ induces a canonical element 
\begin{equation}\label{csubf}
c_{\<f}\set\fundamentalclass{\<f}(\OY\<)\in\HH^0(f);
\end{equation}
and in terms of the \emph{bivariant product} $\HH^0(u)\times\HH^0(v)\to \HH^0(vu)$ 
\cite[3.3.2]{AJL}, Theorem~\ref{trans fc} says:\va{-2}
\begin{equation}\label{fcl and dot}
\boxed{c_{\>vu}=c_{\>u}\!\cdot\< \<c_{\>v}\>\>.}
\end{equation}
Together with the easily-checked fact that if $X=Y$ and $u$ is the identity map, then $c_{\>u}$
is the identity map of $\Hsch{\<X}$, this shows that the family $c_{\<f}$ is a family of \emph{canonical orientations} for 
the flat maps  in our bivariant theory, see \cite[p.\,28, 2.6.2]{fmc}.

Remark~\ref{bch-compatible}(A) below shows that these orientations are 
compatible with essentially \'etale base change.\va2

\end{cosa}

\medskip
The next Corollary is also a special case of Theorem~\ref{bch-fund-class}.

\begin{cor}\label{bch-fundamental-class}
If in the oriented fiber square of flat\/ $\SS$\kf-maps
\[\mkern20mu
 \bpic[xscale=2.9,yscale=2.3]
  \node(11) at (1,-1) {$X'$};
  \node(12) at (2,-1) {$X$};
  \node(21) at (1,-2) {$Y'$};
  \node(22) at (2,-2) {$Y$};
  
  \draw[->] (11)--(12) node[above=1pt, midway, scale=.75]{$v$};
  \draw[->] (21)--(22) node[below=1pt, midway, scale=.75]{$u$};
  
  \draw[->] (11)--(21) node[left=1pt, midway, scale=.75]{$g$};
  \draw[->] (12)--(22) node[right=1pt, midway, scale=.75]{$f$};

  \node at (1.5,-1.5) [scale=.85]{$\Dd$};
 \epic
\]  
$u$ $($hence\/ $v)$ is essentially \'etale, then, with notation as in Proposition~\ref{c inverse}, $\Dc_gu^*$ factors as
$$
\Hh{\<X'}g^*\<u^*
 \! \underset{\ps^*}{=\!=}\<\<
\Hh{\<X'}v^*\<\<f^*
  \!\xrightarrow[\!(\<\<\bbbiup{v}){}^{-1}\!]{\lift-.25,\Isoo,}\<\<
v^*\Hh{\<X}f^*
  \!\xrightarrow[\!\<v^*\<\Dc_{\<f}\>]{}\<\<
v^*\<\<f^! \Hh{Y}
  \!\xrightarrow[\<\!\bchadmirado{\Dd}]{\lift-.25,\Isoo,}\>\>
g^!u^*\Hh{Y}\<
  \xrightarrow[\!g^!\bbbiup{u}\!]{\lift-.25,\Isoo,}\<\<
g^!\Hh{Y'}u^*\<\<.
$$
\end{cor}

\smallskip

\begin{proof} We have $u^*=u^!\<$, $v^*=v^!$; and 
the following diagram commutes: 
\[\mkern20mu
 \bpic[xscale=2.9,yscale=2.3]
  \node(11) at (1,-1) {$v^*\<\<f^!$};
  \node(12) at (2,-1) {$g^!u^*$};
  \node(21) at (1,-2) {$v^!\<\<f^!$};
  \node(22) at (2,-2) {$g^!u^!$};
  
  \draw[->] (11)--(12) node[above, midway, scale=.75]{$\bchadmirado{\Dd}$};
  \draw[double distance=2pt](11)--(21);
  \draw[double distance=2pt](12)--(22);
  \draw[double distance=2pt](21)--(22) node[below=1pt, midway, scale=.75]{$\ps^!$};
 \epic
\]  
This looks like \cite[p.\,208, Theorem 4.8.3(iii)]{li}; but that theorem applies only to the full subcategory
$\Dqcpl\subset\Dqc$ of homologically bounded-below complexes. To treat all of $\Dqc$ one must expand the diagram according to the definitions of 
$\bchadmirado{\Dd}$ and $\ps^!$ (see \eqref{ps-def} and \eqref{thetaB},  which agree on~$\Dqcpl$ with the usual definitions), and then check that the expanded diagram commutes. The cited Theorem 4.8.3(iii) enters into this verification, but only as 
applied to~$\OY$. Details---routine, though tedious---are left to the reader.\looseness=-1

In view of Proposition~\ref{c inverse}, we need then to show that subdiagram \circled1 in the next diagram commutes.
\[
\def\1{$\Hh{\<X'}g^*u^*$}
\def\2{$\Hh{\<X'}(ug)^*$}
\def\3{$(ug)^!\>\Hh{Y}$}
\def\4{$\Hh{\<X'}v^*\<\<f^*$}
\def\5{$\Hh{\<X'}(fv)^*$}
\def\6{$(fv)^!\>\Hh{Y}$}
\def\7{$v^!\>\Hh{\<X}f^*$}
\def\8{$v^!\<\<f^!\>\Hh{Y}$}
\def\9{$g^!\>\Hh{Y'}u^*$}
\def\0{$g^!u^!\>\Hh{Y}$}
 \bpic[xscale=5, yscale=1.2]
  \node(11) at (1,-1){\1};
  \node(12) at (2,-1){\2};
  \node(13) at (3,-1){\3};
  
  \node(21) at (1.3,-2){\4};
  \node(22) at (2,-2){\5};
  \node(23) at (2.7,-2){\6};
  
  \node(31) at (1.3,-3){\7};
  \node(32) at (2.7,-3){\8};
  
  \node(41) at (1,-4){\9};
  \node(42) at (3,-4){\0};

  \draw[-,double distance=2pt] (11)--(12) node[above=1pt, midway, scale=.75]{$\ps^*$};
  \draw[->] (12)--(13) node[above, midway, scale=.75]{$\Dc_{ug}$};

  \draw[-,double distance=2pt] (21)--(22) node[above, midway, scale=.75]{$\ps^*$};
  \draw[->] (22)--(23) node[above, midway, scale=.75]{$\Dc_{\<fv}$};
  
  \draw[->] (31)--(32) node[below, midway, scale=.75]{$v^!\Dc_{\<f}$};
 
  \draw[->] (41)--(42) node[below, midway, scale=.75]{$g^!\Dc_{u}$};
 
  \draw[->] (11)--(41) node[left, midway, scale=.75]{$\Dc_{g}u^*$};
  \draw[-,double distance=2pt] (12)--(22) ;
  \draw[->] [-,double distance=2pt] (13)--(42) node[right=1pt, midway, scale=.75]{$\ps^!$};

  \draw[->] (21)--(31) node[left, midway, scale=.75]{$\Dc_{v}$};
  \draw[-,double distance=2pt] (23)--(32) node[left=1pt, midway, scale=.75]{$\ps^!$};

  \draw[-,double distance=2pt] (11)--(21) node[right=2pt, midway]{$\lift1.5,\ps^*,$};
  \draw[-,double distance=2pt] (13)--(23) ;
  \draw[-,double distance=2pt] (32)--(42) node[left, midway]{$\lift0,\ps^!,$};

  \node at (1.6,-3.57){\circled1};
  \node at (2,-2.57){\circled2};

 \epic
\]

It is clear that the unlabeled subdiagrams commute. By Theorem~\ref{trans fc}, the outer border and subdiagram \circled2 both commute. The conclusion follows.
\end{proof}

\begin{subrems}\label{bch-compatible}
(See Remark following \cite[p.\,28, 2.6.2]{fmc}.)\va1

(A) When applied to $\OY$, Corollary~\ref{bch-fundamental-class}  says, in bivariant terms, that for any independent square 
$\Dd$ as above,
$$
\boxed{c_{\>g} = u^\star\< c_{\<f}}
$$
where $u^\star\colon\HH^0(f)\to\HH^0(g)$ is the \emph{pullback,} see \cite[3.3.4]{AJL}.\va1

(B) As for \emph{pushforward} (see \cite[3.3.3]{AJL}), for 
$\SSp$\kf-maps $X\xto{f\>}Y\xto{g\>}Z$ with $f$~proper, it holds that
$$
\boxed{f_{\<\star}\>c_{gf}=f_{\<\star}\>(c_{\<f}\<\<\cdot\<\<c_g) = (f_{\<\star}\> c_{\<f})\<\<\cdot\<\<c_g\>.}
$$
where the first equality is given by \eqref{fcl and dot}, and the second by \cite[4.4]{AJL}.
Of~course this doesn't convey much without further information on $f_{\<\star}\> c_{\<f}\>$. 

Note that by definition, $f_{\<\star}\> c_{g\<f}\>$ factors as
$$ 
\Hsch{Y}\xto{f^{}_{\<\sharp}\>\>}\fst\Hsch{\<X} \xto{\widetilde c^{}_{g\<f}} g^!\Hsch{Z}
$$
where $f^{}_{\<\sharp}$ is adjoint to $f^\sharp$ (see \S\ref{orientations})
and $\widetilde c_{gf}$ corresponds under duality to $c_{g\<f}\>$.

As indicated at the end of Example~\ref{c and trace},  further study of $f_{\<\star}\> c_{\<f}\>$
will be carried out elsewhere.  

\end{subrems}

\begin{cosa}\label{Gysin} For any $S$-scheme $W$ and $i\in\ZZ$, set $\HH_i(W)\set\HH_i(W|S)$ and
$\HH^i(W)\set\HH^i(W|S)$ (see \S\ref{begin intro}).

The orientations $c_{\<f}$ of  flat $\SS$\kf-maps $f\colon X \to Y$
give rise  to ``wrong~way"  \emph{Gysin homomorphisms} 
$$
\gyb{f} \colon \HH_j(Y) \to \HH_j(X)\ (j\in\ZZ)
$$ 
and, if $f$ is also proper, 
\[\gyf{f} \colon \HH^j(X) \to \HH^j(Y).
\]
As in \cite[\S\S2.5, 2.6.2]{fmc}, these homomorphisms are defined by 
\[
\gyb{f}(\beta) = c_{\<f} \cdot \beta,\qquad
\gyf{f}(\alpha) = \pf{f}(\alpha \cdot c_{\<f})\>.
\]

More explicitly, if $x\colon X\to S$ and $y\colon Y\to S$\va1 are the structure maps, $\beta\colon\Hsch{Y}\to y^!\OS[-j\>]$ is in $\HH_j(Y)$,\va1 and $\alpha\colon \Hsch{\<X}\to \Hsch{\<X}[\>\>j\>]$ is in $\HH^j(X)$, then
$\gyb{f}(\beta)$ and $\gyf{f}(\alpha)$ are given, respectively, by the compositions
\[
\Hsch{\<X}\xto{c_{\<f}\>}f^!\Hsch{Y}\xto{f^!\beta\>\>}f^!y^!\OS[-j\>]\overset{\ps^!}{=\!=} x^!\OS[-j\>]
\]
and (with $f^{}_{\<\sharp}$ adjoint to $f^\sharp$, see \S\ref{orientations})
\[
\Hsch{Y}\xto{f_\sharp\>}\fst\Hsch{\<X}\xto{\fst\alpha\>\>}\fst\Hsch{\<X}[\>\>j\>]\xto{\fst c_{\<f}}\fst f^!\Hsch{Y}[\>\>j\>]\xto{\smallint_{\!f}}\Hsch{Y}[\>\>j\>].
\]
\end{cosa}

The basic properties of Gysin homomorphisms are listed in \cite[p. 26]{fmc}.  As noted there, they are
all immediate consequences of the bivariant axioms.  Let us briefly review the interpretation of these properties and their derivations in the present context, for which purpose we will need the transitivity of the fundamental class (Theorem~\ref{trans fc}) and the base\kf-change Corollary~\ref{bch-fundamental-class}. 

\emph{For the remainder of this section,  when a pushforward like\/ $h_\star$ appears, the\/ 
$\SSp$\kf-map $h$ is assumed to be proper; and when a pullback like\/ $h^\star$ appears, the\/  
$\SSp$\kf-map\/ $h$ is assumed to be essentially \'etale.} \va2

First, Gysin maps are functorial:
\begin{subprop}
For flat\/  $\SS$\kf-maps\/ $X\xto{f\>}Y\xto{g\>}Z,$ one has
\[\gyb{(gf)\<}= \gyb{f}\<\gyb{g}\quad\textup{and}\quad 
(gf)\gyf{}= \gyf{g\mkern.5mu}\>\gyf{f}\>.
\]
\end{subprop}

\begin{proof}
In view of Theorem~\ref{trans fc}, the first equality results from associativity of the $\cdot$  product \cite[Proposition 4.1]{AJL}, and the second from that associativity plus functoriality of pushforward \cite[Proposition~4.2]{AJL} plus commutativity of pushforward with product  \cite[Proposition~4.4]{AJL}.
\end{proof}

\pagebreak[3]
Gysin maps behave well with respect to essentially-\'etale base change:

\begin{subprop}\label{Gysin and etale base change}
For any oriented fiber square in\/ $\SS$
\[
    \begin{tikzpicture}[yscale=.9]
      \draw[white] (0cm,0.5cm) -- +(0: \linewidth)
      node (E) [black, pos = 0.41] {$Y^\prime$}
      node (F) [black, pos = 0.59] {$Y$};
      \draw[white] (0cm,2.65cm) -- +(0: \linewidth)
      node (G) [black, pos = 0.41] {$X^\prime$}
      node (H) [black, pos = 0.59] {$X$};
      \draw [->] (G) -- (H) node[above, midway, sloped, scale=0.75]{$g'$};
      \draw [->] (E) -- (F) node[below=1pt, midway, sloped, scale=0.75]{$g$};
      \draw [->] (G) -- (E) node[left=1pt,  midway, scale=0.75]{$f'$};
      \draw [->] (H) -- (F) node[right, midway, scale=0.75]{$f$};
    \end{tikzpicture}
\]
with\/ $f$ $($hence $f'\>)$ flat and\/ $g$ $($hence $g'\>)$ essentially \'etale, one has
\[
\gyb{f}\<\pf{g\>} = g_\star'f'{}^{\mathsf c}\>;
\]
and if\/ in addition $f$ $($hence $f'\>)$ is proper, then
\[
   \pb{g}\!\gyf{f^{}\!} = \gyf{f'\!\!}\>\>g'{}^\star.
\]
\end{subprop}

\begin{subrem} In Proposition~\ref{Gysin and flat base change}, we will prove that 
if $f$ and $g$ are flat and $f$ is proper, then $\gyb{g}\<\<\pf{f\<\<} = f_{\<\<\star}'\>g'{}^{\mathsf c}.$

\end{subrem}

\begin{proof}
By Remark~\textup{\ref{bch-compatible}(A),} $\fc{f'}=\pb{g}\fc{f}$. Hence the first equality results from the projection formula  \cite[Proposition 4.7\kern1pt]{AJL}, and the second from commutativity of pullback with product, see  \cite[Proposition 4.5]{AJL} as applied to the following diagram (where $\alpha\in\HH^*(X)$):
\begin{equation*} 
 \begin{tikzpicture}scale=1.3
      \draw[white] (0cm,0.5cm) -- +(0: \linewidth)
      node (E) [black, pos = 0.4] {$Y'$}
      node (F) [black, pos = 0.59] {$Y$};
      \draw[white] (0cm,2.65cm) -- +(0: \linewidth)
      node (G) [black, pos = 0.4] {$X'$}
      node (H) [black, pos = 0.59] {$X$};
      \draw[white] (0cm,4.8cm) -- +(0: \linewidth)
      node (I) [black, pos = 0.4] {$X'$}
      node (J) [black, pos = 0.59] {$X$};
      \draw [->] (G) -- (H) node[above, midway, sloped, scale=0.75]{$g'$};
      \draw [->] (E) -- (F) node[below=1pt, midway, sloped, scale=0.75]{$g$};
      \draw [->] (G) -- (E) node[left=1pt, midway, scale=0.75]{$f'$};
      \draw [->] (H) -- (F) node[right=1pt, midway, scale=0.75]{$f$}
                                     node[left, midway]{\lift-.3,{$\circled{\lift1.5,\displaystyle \>c_{\<\<f},}$},};
      \draw [->] (I) -- (J) node[above, midway, sloped, scale=0.75]{$g'$};
      \draw [double distance=2pt] (I) -- (G) node[left, midway, scale=0.75]{$$};
      \draw [-,double distance=2pt] (J) -- (H) node[right=1pt, midway, scale=0.75]{$$}
                                 node[left, midway]{\lift-.5,$\circled{\lift1.2,\displaystyle\alpha,}$,};
 \end{tikzpicture}
\end{equation*}
\end{proof}

\emph{Notation}: for  an
essentially \'etale $\SS$\kf-map $f\colon X\to Y$ and $\alpha\in\HH^*(Y)$,  $f^\star\<\alpha$ is the pullback of 
$\alpha$ by $f$, through the independent square
\[
\CD
X @>f>> Y\\
@| @|\\
X @>>f> Y
\endCD
\]

The relation of Gysin maps and pushforward is shown in the next result. 
\begin{subprop}
For  flat\/ $\SS$\kf-maps\/ $X\xto{f\>}Y\xto{g\>}Z,$ with\/ $f$ proper, one has
\begin{equation*}    
 \pf{f}\big(\gyb{(gf)\<}\beta\big) = (\pf{f}\fc{gf}) \cdot \beta
 =(\pf{f} c_{\<f})\cdot\gyb{g}\beta
   \qquad(\beta\in\HH_*(Z)),\tag{1}
 \end{equation*}
and if, moreover, $f$ is essentially \'etale, 
\begin{equation*}
 \gyf{(gf)}(\pb{f}\<\alpha) =   \pf{g\>}(\alpha \cdot \<\pf{f}\fc{gf}) 
=\gyf{g^{}\<}(\alpha\cdot\< \pf{f} c_{\<f})
    \qquad(\alpha\in\HH^*(Y)).\tag{2}
\end{equation*}
\end{subprop}

\begin{proof} 
The first equality in (1) follows at once from commutativity of pushforward with product  
\cite[Proposition~4.4]{AJL}; and in (2) from functoriality of pushforward \cite[Proposition~4.2]{AJL}
and the projection formula \cite[Proposition~4.7]{AJL}.  The second equality in both (1) and (2)
results, in view of associativity of the $\cdot$ product
\cite[Proposition 4.1]{AJL}, from \ref{bch-compatible}\textup{(B)}.
\end{proof}

Finally, there are two projection-like properties. 
\begin{subprop}
For a flat\/  $\SS$\kf-map\/ $X\xto{f\>}Y,$ $\alpha \in \HH^*(Y),$ $\beta \in \HH_*(Y)$ and 
$\alpha' \in \HH^*(X),$ one has
\begin{equation*}    
\pf{f}\big(\alpha' \cdot (\gyb{f}\beta)\big) = \big(\gyf{f}\>\alpha'\big)\cdot \beta,\tag{1}
 \end{equation*}
and if, moreover, $f$ is essentially \'etale, 
\begin{equation*}
 \gyf{f}\big((\pb{f}\<\alpha) \cdot \alpha'\big) = \alpha \cdot \big(\gyf{f}\>\alpha'\big).\tag{2}
\end{equation*}
\end{subprop}

\noindent\emph{Remark.} The products in (1) and (2) can be interpreted, respectively,
as cap and cup,  see \cite[\S3.6]{AJL}.

\begin{proof}
The equality (1) follows at once from commutativity of product and pushforward \cite[Proposition 4.4]{AJL},
and (2) from the projection formula \cite[Proposition~4.7]{AJL}.
\end{proof}

\section{The dual oriented bivariant theory.} \label{dual theory}
Together with Theorem~\ref{trans fc}, the constructions in
\cite{AJL} give, at least for flat maps, an oriented bivariant
theory~$\CB$, see \S\ref{orientations}. In this context there is an
order-2 symmetry taking $\CB$ to a \emph{dual} oriented bivariant theory
$\ul{\CB\<}\>$, as detailed (in a more abstract situation) in Theorem~\ref{bivdual} below. 
While the cohomology groups $\ul{\<\HH\<}^{\>i}(X)\ (i\in\ZZ)$ associated to a flat
$\SS$\kf-scheme $X$ by $\ul{\CB\<}\>$ are isomorphic to those coming from $\CB$,
the homology groups $\ul{\<\HH\<}_{\>i}(X)$ are the classical Hochschild homology groups
$\h^{-i}(X\<,\Hsch{\sX})$.

In the specialization of Theorem~\ref{bivdual} to the just-mentioned context, the fundamental class of a
flat map plays a key role,  illustrated in Example~\ref{concrete}.

After proving Theorem~\ref{bivdual}, we define (for any flat $\SS$\kf-map $x\colon X\to S$) a pairing 
\[
\fp_\sX=\fp_x\colon\Hsch{\sX}\<\otimes_{\<X}\<\Hsch{\sX}\to x^!\OS
\]
by composing the fundamental class 
$c_x\colon\Hsch{\sX}\to x^!\OS$ with a natural product map 
$\Hsch{\sX}\otimes \Hsch{\sX}\to \Hsch{\sX}$. Correspondingly, there is a \emph{duality map} 
\[
\fd_\sX\colon \Hsch{\sX}\to \R\sHom_\sX(\Hsch{\sX}, x^!\OS)
\]
that turns out to be compatible with \'etale localization.  Whenever $x$ is essentially
smooth, $\fd_\sX$~is an \emph{isomorphism} (Theorem~\ref{dualityiso}).
There results an isomorphism between the bivariant groups associated
to any flat map of essentially smooth $\SS$\kf-schemes by~$\CB$ and by~$\ul{\CB\<}\>$.
In particular, one gets (again, cf.~\cite[\S\S6.4, 6.5]{AJL}) that the bivariant Hochschild
homology groups $\HH_i(X)\set\HH_i(X|S)$ of an essentially smooth
$\SS$\kf-scheme~$X$  (see \S\ref{begin intro}) are isomorphic to the classical Hochschild homology groups
$\h^{-i}(X\<,\Hsch{\sX})$.

One deduces directly from Theorem~\ref{dualityiso} that if $S=\spec H$ with\/ $H$ a  Gorenstein 
artinian ring and if $x\colon X\to S$ is  proper and smooth, then there is a 
\emph{non-singular} pairing on classical  Hochschild homology 
\[
\h^{-i}(X\<, \Hsch{\sX}\<)\otimes_H\h^{i}(X\<, \Hsch{\sX}\<)\to H,
\]
see Corollary~\ref{Mukai?}. 
We haven't figured out the precise relation of this pairing to the Mukai pairing of
\cite[\S5]{c1}.

Also left open is the relation of $\>\fd_\sX\<$ to some other duality
isomorphisms that appear in the literature in connection e.g., 
with proving Riemann-Roch theorems via Hochschild homology (see \S\ref{othermaps}).

\begin{cosa}
Let there be given a setup 
$$
\Sigma\set\big(\SS, H\<, (\D_W)_{W\in\SS}, (-)^*\<, (-)^!\<,\dots\big)
$$ 
as in \cite[\S3.1.1]{AJL}, but modified slightly as specified in \S\ref{dual setup} below; and a family of 
degree\kf-0 $\D_\sX$-maps 
$$
(f^\sharp\colon f^*\Hsch{Y}\to\Hsch{\sX})_{f\colon \!X\to Y\, \lift1.1,\in,\,\SS}
$$ 
as in \cite[\S3.2]{AJL}. 
Let there also be given a family of degree\kf-0 $\D_\sX$-maps
$$
(c_{\<f}\colon \Hsch{\sX}\to f^!\Hsch{Y})_{f\colon \!X\to Y\, \lift1.1,\in,\,\SS}
$$
such that for any $\SS$\kf-maps $\<X\<\<\xto{\,\lift.5,u,\,}\<Y\<\<\xto{\,\lift.5,v,\,} \mkern-.5mu Z$, 
$c_{vu}$ factors as
\[\Hsch{\sX}\xto{c_u\,}u^!\Hsch{Y}\xto{\<\<u^!c_{v}\mkern1.5mu}u^!v^!\Hsch{Z}\overset{\ps^!}{=\!=}(v u)^!\Hsch{Z}\>,
\]
and such that $c_{\<f}$ is an isomorphism whenever $f$ is the bottom or top arrow of an independent square.  

\begin{subex} In the bivariant theory of \S\ref{orientations}, restricting to flat maps---a~
restriction which we hope eventually to eliminate---one gets such data from Proposition~\ref{c inverse} and Theorem~\ref{trans fc}.
\end{subex}
\end{cosa}

\begin{thm}\label{bivdual}
Under the preceding conditions, there is a bivariant theory\/~$\ul{\CB\<}\>$  assigning to
an $\SS$\kf-map $f\colon X\to Y$ the symmetric graded $H$-module
\[
\buE* {f}{X}{Y}\set 
\oplus_{i\in\ZZ}\, \D_\sX^i(f^*\>\Hsch{Y}, \Hsch{\<X})
\]
$($so that the family\/ $(f^\sharp)$ orients\/~$\ul{\CB\<}\>\>),$ and having the following operations.
\end{thm} 
  
  \pagebreak[3]
  
\begin{subcosa}\label{Product}\emph{Product.}
   Let $f:X\to Y$ and $g:Y\to Z$ be in $\SS$.\va1 

For $i,j\in\ZZ$ and ${\alpha} \in \buE{\>i}{f}{X}{Y}$, ${\beta} \in \buE{\>j}{g\>}{Y}{Z},$  the {product}
 \[
    {\alpha} \<\cdot\<\< {\beta} \in \buE{i+j}{g f}{X}{Z}
  \]
is $(-1)^{ij}$ times the  composite ${\D}_\sX$-map
\[
(gf)^*\Hsch{\<Z}\overset{\ps^*}{=\!=} f^*\<g^*\Hsch{\<Z}\xto{f^*\beta\>} f^*\Hsch{Y} \xto{\,\,\alpha\,\,}\Hsch{\sX}.
\]
\end{subcosa}

\begin{subcosa}\label{Pushforward}
 \emph{Pushforward.}
 Let $f \colon X \to Y$ and $g \colon Y \to Z$ be $\SS$\kf-maps, $f$~confined. 
The {pushforward  by} {$f$}
\[
f^{}_{\<\<\underline\star}\colon \buE* {g f}{X}{Z} \to \buE* {\>\>g\,}{Y}{Z}
\]
is the graded $H\<$-linear ${\D}_Y$-map such that for $i\in\ZZ$ and ${\alpha} \in \buE{i}{g f}{X}{Z}$, the image
$f^{}_{\<\<\underline\star}\>{\alpha} \in \buE{i}{\>\>g\,}{Y}{Z}$
is  the $\D_Y$-composition
\[
  \bpic[xscale=2.5]
   \draw (-.07,-1) node (11){$g^*\Hsch{Z}$};
   \draw (0.78,-1) node (12){$\fst f^*\<g^*\Hsch{\<Z}$};
   \draw (1.84,-1) node (13){$ \fst(gf)^*\>\Hsch{Z}$};
   \draw (2.75,-1) node (14){$\fst\Hsch{\<X}$};
   \draw (3.55,-1) node (15){$\fst f^!\Hsch{Y}$};
   \draw (4.25,-1) node (16){$\Hsch{Y}.$};
  
   \draw [->] (15) -- (16) node[above=-1pt, midway, scale=0.75]{$\int^{}_{\!f}$};
   \draw [->] (14) -- (15) node[above=-1pt, midway, scale=0.75]{$\fst c_{\<f}$};
   \draw [->] (13) -- (14) node[above, midway, scale=0.75]{$ f^{}_{\<\<*}{\alpha}\ $};
   \draw [-, double distance=2pt] (12) -- (13) node[above=0.5mm,midway,scale=0.75]
          {$\fst\<\<\ps^*$};
   \draw [->] (11) -- (12) node[above, midway, scale=0.75]{$\eta^{}_{\<\<f}\,$};
   
  \epic
\]
\end{subcosa}
\begin{subcosa}\label{Pullback}
\emph{Pullback.}
Given an independent square in $\SS$
  \[
   \begin{tikzpicture}[yscale=.84]
      \draw[white] (0cm,0.5cm) -- +(0: \linewidth)
      node (21) [black, pos = 0.41] {$Y'$}
      node (22) [black, pos = 0.59] {$Y$};
      \draw[white] (0cm,2.65cm) -- +(0: \linewidth)
      node (11) [black, pos = 0.41] {$X'$}
      node (12) [black, pos = 0.59] {$X$};
      \node (C) at (intersection of 11--22 and 12--21) [scale=0.9] {$\Dd$};
      \draw [->] (11) -- (12) node[above, midway, sloped, scale=0.75]{$v$};
      \draw [->] (21) -- (22) node[below, midway, sloped, scale=0.75]{$u$};
      \draw [->] (11) -- (21) node[left, midway, scale=0.75]{$g$};
      \draw [->] (12) -- (22) node[right, midway, scale=0.75]{$f$};
   \end{tikzpicture}
  \]
let $\pi\colon (-)^*\iso (-)^!$ be the pseudofunctorial isomorphism in \eqref{B=ps^!} below,\va1 and let $\mathsf B'_{\<\Dd}$  be the\va{-2} composed isomorphism
\[
g^*u^! \xto{\!g^*\pi^{-\<1\!}} g^*u^* \overset{\ps^*}{=\!=} v^*\!f^* \xto{\;\pi\;} v^!\<\<f^*\<.
\]

The { pullback by $u,$ through $\Dd,$}
\[
u^{\>\underline{\<\star\<}}\>\colon \buE* {f}{X}{Y}\lto \buE* {g\>}{X'}{Y'}
\]
is the graded $H\<$-linear ${\D}_Y$-map such that for $i\in\ZZ$ and ${\alpha} \in \buE{i}{ f}{X}{Y}$, 
the image $u^{\>\underline{\<\star\<}}\>\>{\alpha} \in \buE{i}{g}{X'}{Y'}$
 is the $\D_Y$-composition \va{-2}  
\[
  \bpic[xscale=2.5, yscale=.9]
   \draw (-.07,-1) node (11){$g{}^*\Hsch{Y'}$};
   \draw (0.9,-1) node (12){$g{}^*\<u^!\Hsch{Y}$};
   \draw (1.84,-1) node (13){$ v{}^!\<\<f^*\Hsch{Y}$};
   \draw (2.75,-1) node (14){$ v{}^!\Hsch{\<X}$};
   \draw (3.45,-1) node (15){$\Hsch{\sX'}.$};
  
   \draw [->] (14) -- (15) node[above=.5pt, midway, scale=0.75]{$c_{v}^{-\<1}$};
   \draw [->] (13) -- (14) node[above=1.5pt, midway, scale=0.75]{$ v{}^!{\alpha}\ $};
   \draw [-, double distance=2pt] (12) -- (13) node[above=.5pt,midway,scale=0.75]
          {$\bchadmirado{\<\Dd}'$};
   \draw [->] (11) -- (12) node[above, midway, scale=0.75]{$g{}^*\<c_u$};
   
  \epic
  \]
\end{subcosa}

\begin{subrem} In the specific circumstances dealt with in \S\ref{fundclassflat}, in particular 
Proposition~\ref{c inverse}, the pullback of ${\alpha} \in \buE{i}{ f}{X}{Y}$ is\va{-2} the $\D_Y$-composition
\[
  \bpic[xscale=2.5, yscale=.9]
   \draw (-.07,-1) node (11){$g{}^*\Hsch{Y'}$};
   \draw (0.9,-1) node (12){$g{}^*\<u^*\Hsch{Y}$};
   \draw (1.84,-1) node (13){$ v{}^*\<\<f^*\Hsch{Y}$};
   \draw (2.75,-1) node (14){$ v{}^*\Hsch{\<X}$};
   \draw (3.45,-1) node (15){$\Hsch{\sX'}.$};
  
   \draw [->] (14) -- (15) node[above=.5pt, midway, scale=0.75]{$\lift2.4,\displaystyle v^\sharp,$};
   \draw [->] (13) -- (14) node[above=1.5pt, midway, scale=0.75]{$ v{}^*{\alpha}\ $};
   \draw [-, double distance=2pt] (12) -- (13) node[above=.5pt,midway,scale=0.75]
          {$\ps^*$};
   \draw [->] (11) -- (12) node[above, midway, scale=0.75]{$g{}^*u^\sharp{}^{-\<1}$};
   
  \epic
\]

\end{subrem}
 \begin{subrem}  The homology groups associated to $\ul{\CB\<}\>$ are
\[
\ul{\<\HH\<}_{\>i}(X)  
\set\D^{-i}_{\sX}(\OX, \Hsch{\<X})  \qquad(i\in\ZZ).
\]
For the example just before Theorem~\ref{bivdual}, these are the classical Hochschild 
(hyper)homology groups $\h^{-i}(X,\Hsch{X})$ of the flat $S$-scheme $X$. 

These  groups are
covariant for confined $\SS$\kern.5pt-maps, via pushforward (take $Z=S$ in \ref{Pushforward}); and contravariant for all 
$f\colon X\to Y$, via Gysin maps
$$
\gub{f} \colon \ul{\HH}_j(Y) \to \ul{\HH}_j(X)\ (j\in\ZZ)
$$ 
such that for any $\beta\colon \OY\to \Hsch{Y}[-j\>]$, $\gub{f}\beta$ is the composition\va{-8}

\[
\OX=f^*\OY\xto{f^*\beta\,}f^*\Hsch{Y}[-j\>]\xto{f^\sharp} \Hsch{\sX}[-j\>].
\]

The  cohomology $H$-algebras
\[
\ul\HH^i(X) 
\set \D^i_{\sX}(\Hsch{\<X},\Hsch{\<X}\<),
\] 
are  contravariant via pullback for co\kf-confined maps. (This pullback functor, associated with~ 
$\ul{\CB\<}\>$, actually coincides with the one in \cite[3.6.1]{AJL}.) They are also covariant, as groups, for confined maps $f\colon X\to Y\<$, via Gysin maps 
\[
\guf{f} \colon \ul{\HH}^j(X) \to \ul{\HH}^j(Y)
\]
that are such that for any $\alpha\colon \Hsch{\<X}\to \Hsch{\sX}[\>j\>]$,
$\guf{f}\>\alpha$ is the composition\va{-8}
 
\[
\Hsch{Y} \xto{\eta^{}_{\<\<f}} \fst f^*\Hsch{Y} \xto{\fst f^\sharp}
\fst\Hsch{\<X}\xto{\fst\alpha\,}\fst\Hsch{\<X}[\>\>j\>] \xto{\fst c_{\<f}}\fst f^!\Hsch{Y}[\>\>j\>]
\xto{\int_{\!f}}\Hsch{Y}[\>\>j\>].
\]
This $\guf{f}$ coincides with the Gysin map  $\gyf{f^{}\!}$ in \S\ref{Gysin}.
\end{subrem}

\begin{subrem}\label{Hodge}
For essentially smooth maps, where sometimes the Hochschild homology groups can be interpreted
in terms of Hodge homology (via so-called Hochschild-Kostant-Rosenberg isomorphisms), one would  like to~know more about the relation between the preceding operations on Hochschild homology groups and those on Hodge groups that play an important role in~\cite{CR}.  For this, our constructions are at present too limited,
in that they do not apply to arbitrary perfect maps, such as local complete intersection maps, nor to homology with supports.\par
\end{subrem}

\begin{cosa}\label{dual setup}
\noindent\emph{Proof of Theorem \emph{\ref{bivdual}}.} The following modification of the given setup does not compromise the validity of \cite[Theorem 3.4]{AJL}. Accordingly, there exists a bivariant theory 
$\mathcal B$ oriented by the family $c_{\<f}$ (cf.~\S\ref{bivorient}).  In this situation we are going to \emph{dualize} all the given data, and then find that in the dualized situation all the assumptions needed for \cite[Theorem 3.4]{AJL} are satisfied. The resulting bivariant theory $\ul{\mathcal B}$ is the one referred to in Theorem~\ref{bivdual}.

Furthermore, $\ul{\mathcal B}$ is oriented by the family of maps $f^\sharp$, and its setup conforms to the following modifications. So $\ul{\mathcal B}$ itself can be dualized; and it will
result easily from the detailed description of $\ul{\mathcal B}$ that $\ul{\ul{\mathcal B}}=\mathcal B$.\va2


To begin with, let us weaken slightly the conditions defining a ``setup"
upon which bivariant theories are built (see \cite [\S3.1]{AJL}). 

\pagebreak[3]
To wit, in \cite{AJL},  
require only that in \S2.4  the pseudofunctor $(-)_*$ and the pseudo\-functorial adjunction 
$(-)^*\dashv(-)_*$ exist over the subcategory of \emph{confined} $\SS$\kf-maps; that in \S2.5, $\theta_{\Dd}\colon u^*\<\<\fst\iso g_*v^*$ be defined only for those independent squares
\begin{equation}\label{diagram d}
  \CD
    \begin{tikzpicture}[yscale=.85]
      \draw[white] (.25cm,0.5cm) -- +(0: .5\linewidth)
      node (E) [black, pos = 0.32] {$\bullet$}
      node (F) [black, pos = 0.68] {$\bullet$};
      \draw[white] (.25cm,2.65cm) -- +(0: .5\linewidth)
      node (G) [black, pos = 0.32] {$\bullet$}
      node (H) [black, pos = 0.68] {$\bullet$};
      \draw [->] (G) -- (H) node[above, midway, sloped, scale=0.75]{$v$};
      \draw [->] (E) -- (F) node[below, midway, sloped, scale=0.75]{$u$};
      \draw [->] (G) -- (E) node[left,  midway, scale=0.75]{$g$};
      \draw [->] (H) -- (F) node[right, midway, scale=0.75]{$f$};
      \node (C) at (intersection of G--F and E--H) [scale=0.75] {$\Dd$};
    \end{tikzpicture}
   \endCD 
  \end{equation}
in which $f$ and $g$ are \emph{confined,} to be the following composition:%
\footnote{This definition, rather than \cite[(2.5.1)]{AJL} was used in \cite[(2.6.1)]{AJL}, where the equivalence 
of the two definitions should have been noted---see~\cite[3.7.2(i)]{li}.}
\begin{equation}\label{def-theta}
u^*\!\fst \xto{\eta^{}_g\>}g_*g^*u^*\!\fst \overset{\ps^*}{=\!=} g_*v^*\!f^*\!\fst 
\xto{\epsilon^{}_{\!f}\>} g_*v^*;
\end{equation}
and that in \S3.2, $\fsh$ be defined when $f$ is confined.
It is then straightforward to check that the definitions in \cite[\S3.3]{AJL} of product, pushforward and pullback,  and the verification in \cite[\S4]{AJL} of the bivariant axioms remain valid without change.\va1

Also, for the sake of the symmetry\va{.6} about to be described we assume that the map 
$\int_{\!f}\colon f_*f^!\to\id$ in \cite[(2.4.3)]{AJL} is the counit \va{-1} for a pseudofunctorial adjunction
$(-)_*\dashv (-)^!$ holding over confined maps.\va{.6} (\emph{Pseudofunctoriality} of the adjunction means that the diagram \cite[(2.4.4)]{AJL} commutes.) This condition holds for the setup constructed\va1 in \cite[\S5]{AJL} (see \cite{AJL}, proof of 5.9.3).\looseness=-1

From these assumptions it follows that commutativity of diagram (2.6.2) in \cite{AJL} is implied by 
that of (2.6.1). Indeed, 
for an independent square~$\Dd$ as above, with $f$ and $g$ confined, the latter commutativity means that the map
$\bchadmirado{\Dd}\colon v^*\mkern-2.5mu f^!\to g^!u^*$ is adjoint to 
$$
g_*v^*\mkern-2.5mu f^!\xto{\theta^{-\<1}_{\Dd}} u^*\mkern-2.5mu \fst f^!\xto{\int^{}_{\!f}} u^*\<,
$$
so that $\bchadmirado{\Dd}$ is the composite map\va{-3}
$$
v^*\mkern-2.5mu f^!\to g^!g_*v^*\mkern-2.5mu f^!\xto{\theta^{-\<1}_{\Dd}} g^!u^*\mkern-2.5mu \fst f^!\xto{\int^{}_{\!f}} g^!u^*\<.
$$
Knowing this, one proves commutativity of (2.6.2) just as in \cite[\S5.10.2]{AJL}.

\smallbreak
One further assumption: there is a category $\mathsf{CC}\subset \SS$ that contains the top and  bottom arrows of every independent square, and 
an isomorphism $\pi$ from the restriction to $\mathsf{CC}$ of the pseudofunctor~ 
$(-)^*$ to that of~$(-)^!$,  such that for any independent $\Dd$ as above ($f$ and $g$ not necessarily confined), the next diagram commutes:\looseness=-1
\begin{equation}\label{B=ps^!}
 \CD
 \bpic[xscale=2,yscale=1.5]
  \node(11) at (1,-1){$v^*\!f^!$};
  \node(12) at (2,-1){$g^!u^*$};
  \node(21) at (1,-2){$v^!\<\<f^!$};
  \node(22) at (2,-2){$g^!u^!$};
 
  \draw[->] (11)--(12) node[above, midway, scale=.75]{$\bchadmirado{\Dd}$};
  \draw[double distance=2pt] (21)--(22) node[below=1pt, midway, scale=.75]{$\ps^!$};
 
  \draw[->] (11)--(21) node[right=1pt, midway, scale=.75]{$\simeq$}
                                 node[left=1pt, midway, scale=.75]{$\pi$};
  \draw[->] (12)--(22) node[left=1pt, midway, scale=.75]{$\simeq$}
                                 node[right=1pt, midway, scale=.75]{$g^!\pi$};

 \epic
 \endCD
\end{equation}

This assumption too holds for the setup 
in \cite[\S5]{AJL}---see \mbox{\cite[p.\,541]{Nk}} and~\eqref{! and otimes1} above.
\end{cosa}

\begin{cosa}\label{dual setup2} Let there be given a setup
\[
\Sigma\set\big(\SS, H\<, (\D_W)_{W\in\SS}, (-)^*\<, (-)^!\<,\dots\big), 
\]
modified as in \S\ref{dual setup}. Referring to \cite[\S2]{AJL},
we now construct a \emph{dual setup}
\[
\underline{\Sigma}=(\SS, H\<, ({\ul \D_W})_{\lift{.95},W\in\SS,}, (-)^{\>\ul{\<*\<}\>}\>\<,(-)^{\ul{!\<}}\>,\dots\big).
\]

First, the category $\SS$ and its confined maps and independent squares, as well as the graded-commutative ring $H\<$, remain the same. (See \cite[\S2.1]{AJL}.)

Next, to each object $W\in\SS$ associate the category $\ul \D_W$ \emph{dual to} the
$H$-graded category $\D_W$ originally associated to $W\<$. By definition, the dual $\ul \bE$ of an $H$-graded category~$\bE$ has the same objects as $\bE$; for each object $A\in\bE$ let $\ul A$ be $A$ considered as an object of $\underline \bE$.  For any $A$, $B\in\bE$, let  $\ul\bE(\ul A,\ul B)$ be the graded $H$-module
$\bE(B,A)$;  for each  $\alpha \in\bE(B,A)$ let $\ul \alpha$ be $\alpha$ considered as an element of $\ul \bE(\ul A, \ul B)$. Finally, define composition  in $\ul \bE$, 
\[
\ul \bE(\ul B,\ul C)\times \ul \bE(\ul A,\ul B)\xto{\ \ul{\<\smallcirc\<}\ }
\ul \bE(\ul A,\ul C),
\]
to be the unique $\ZZ$-bilinear map taking $(\ul\beta,\ul\alpha)\in\ul \bE^p(\ul B,\ul C)\times \ul \bE^q(\ul A,\ul B)$ to 
$$
\ul\beta\,\>{\lift-.2,\ul{\!\smallcirc\!},}\,\>\ul\alpha\>\set(-1)^{pq}\>\>\alpha{\lift-.2,\smallcirc,}\<\beta
\in\bE^{p+q}(C,A)=\ul \bE^{p+q}(\ul A,\ul C).
$$
One checks that this $\ul\bE$ is an $H$-graded category. (See \cite[\S1.1]{AJL}.)

\smallskip
Any $H$-graded functor $F$ between $H$-graded categories can be regarded in the obvious way as
an $H$-graded functor, denoted $\ul{F\<}\>\>$, between the dual $H$-graded categories. A functorial map 
$\xi\colon F\to G$ of degree $n$ is then the same as a functorial map $\ul\xi\colon\ul G\to \ul{F\<}\>$
of degree $n$. (See \cite[\S1.2]{AJL}.)

For any $\SS$\kf-map $f\colon X\to Y$, define the functors
\[
f^{\>\ul{\<*\<}\>}\>\set \ul{f^!\mkern-1.5mu}\>\colon\ul\D_{\>Y}\to\ul\D_{\<X},\qquad 
f^{\ul{!\<}}\>\set \ul{f^*\<\<}\>\colon\ul\D_{\>Y}\to\ul\D_{\<X},\qquad 
f^{}_{\<\<\>\ul{\<*\<}\>}\set \ul{f^{}_{\<\<*}\<}\>\colon\ul\D_{\<X}\to\ul\D_{\>Y}.
\]
There result $H$-graded pseudofunctors with, for a second $\SS$\kf-map $Y\xto{g\,}Z$,
\begin{align*}
\ps^{\>\ul{\<*\<}\>}&=(\>\ul{\<\ps^!\<}\>){}^{-1}\colon f^{\>\ul{\<*\<}\>}g^{\>\ul{\<*\<}\>}\iso(gf)^{\>\ul{\<*\<}\>} ,\\
\ps^{\ul{!\<}}&=(\>\ul{\<\ps^*\<}\>){}^{-1}\colon f^{\ul{!\<}}\>g^{\>\ul{!}}\iso(gf)^{\ul{!\<}},\\
\ps_{\>\ul{\<*\<}\>}&=(\>\ul{\<\ps_*\<}\>)^{-1}\colon (gf)^{}_{\ul{\<*\<}\>}\iso  g^{}_{\>\ul{\<*\<}\>}\>f^{}_{\<\<\>\ul{\<*\<}}.
\end{align*}

\pagebreak[3]
It is easily seen that over confined maps, the pseudofunctorial adjunctions 
$(-)^*\dashv(-)_*$ and $(-)_*\dashv (-)^!$   
(\S\ref{dual setup}) give rise to\va{.6} pseudofunc\-torial adjunctions 
$(-)_{\ul{\<*\<}\>}\dashv (-)^{\ul{!\<}}$ and $(-)^{\>\ul{\<*\<}\>}\dashv(-)_{\ul{\<*\<}\>}$ respectively. For confined $f\<$, let \va{-1} $\ul{\int\!}_{\<f}\colon f^{}_{\<\<\ul{\<*\<}}\>f^{\ul{!\<}}\< \to\id$ be the counit map associated with the first of these adjunctions.\va1

Over co\kf-confined maps, one has the pseudofunctorial isomorphism 
\[\ul{\pi}\colon (-)^{\ul{\<*\<}\>}\iso(-)^{\ul{!\<}}\>.
\] 

 For an independent  $\Dd$ as in \eqref{diagram d}, let $\mathsf B'_{\<\Dd}$  be the\va{-2} composite isomorphism\looseness=-1
\[
g^*u^! \xto{\!g^*\pi^{-\<1\!}} g^*u^* \overset{\ps^*}{=\!=} v^*\<\<f^* \xto{\;\pi\;} v^!\<\<f^*;
\]
\vskip-2pt
\noindent and set\va{-2}
\begin{equation}\label{ulb}
 \>\ul{\<\mathsf B\<}_{\>\>\Dd}\set \ul{\mathsf B'_{\<\Dd}}\colon 
 v^{\ul{\<*\<}}\<f^{\ul{!\<}} \iso g^{\ul{!\<}}\>\>u^{\ul{\<*\<}}\>.
 \end{equation}
 Horizontal and vertical transitivity of $\>\ul{\<\mathsf B\<}_{\>\>\Dd}$ (cf.~\cite[(2.3.1) and (2.3.2)]{AJL}) are straightforward consequences of $\ul\pi$ being an isomorphism of pseudofunctors.
 The required commutativity of the diagram (cf.~\eqref{B=ps^!} (dualized))
 \[
 \bpic[xscale=2.1,yscale=1.6]
  \node(11) at (1,-1){$v^{\ul{\<*\<}}\!f^{\ul{!\<}}$};
  \node(12) at (2,-1){$g^{\ul{!\<}}\>\>u^{\ul{\<*\<}}$};
  \node(21) at (1,-2){$v^{\ul{!\<}}\<\<f^{\ul{!\<}}$};
  \node(22) at (2,-2){$g^{\ul{!\<}}\>\>u^{\ul{!\<}}$};
 
  \draw[->] (11)--(12) node[above, midway, scale=.75]{$\>\ul{\<\mathsf B\<}_{\>\>\Dd}$};
  \draw[double distance=2pt] (21)--(22) node[below=1pt, midway, scale=.75]{$\ps^!$};
 
  \draw[->] (11)--(21) node[right=1pt, midway, scale=.75]{$\simeq$}
                                 node[left=1pt, midway, scale=.75]{$\ul\pi$};
  \draw[->] (12)--(22) node[left=1pt, midway, scale=.75]{$\simeq$}
                                 node[right=1pt, midway, scale=.75]{$g^{\ul{!\<}}\>\>\ul\pi$};

 \epic
\]
follows easily from \eqref{B=ps^!} (dualized) and the definition of $\>\ul{\<\mathsf B\<}_{\>\>\Dd}$.

Also, when $f$ and $g$ in $\Dd$ are confined,  let $\theta'_{\<\Dd}$ be the natural composition \va{-2}
\[
g_*v^!\to 
g_*v^!\<\<f^!\<\<\fst
\overset{\ps^!}{=\!=} 
g_*g^!u^!\<\<\fst
\to
u^!\!\fst;
\]
\vskip-2pt
\noindent and set\va{-2}
\[
 \>\ul{\<\theta\<}_{\>\>\Dd}\set \ul{\theta'_{\<\Dd}}\colon 
u^{\ul{\<*\<}}\<f^{}_{\<\<\ul{\<*\<}}\to g_{\ul{\<*\<}}\>v^{\ul{\<*\<}}. 
\]
In other words, $ \>\ul{\<\theta\<}_{\>\>\Dd}$ 
is the natural composition (cf.~\eqref{def-theta})
\[
u^{\ul{\<*\<}}\<f^{}_{\<\<\ul{\<*\<}} \to
g_{\ul{\<*\<}}\>\>g^{\ul{\<*\<}}\>\>u^{\ul{\<*\<}}\<f^{}_{\<\<\ul{\<*\<}} 
\overset{\ps^{\>\ul{\<*\<}}}{=\!=} 
g_{\ul{\<*\<}}v^{\ul{\<*\<}}\<f^{\ul{\<*\<}}\<f^{}_{\<\<\ul{\<*\<}} \to
g_{\ul{\<*\<}}v^{\ul{\<*\<}}.
\]
 
It needs to be shown that $\>\ul{\<\theta\<}_{\>\>\Dd}$ is an isomorphism, or equivalently, 
that $\theta'_{\<\Dd}$~is an isomorphism.  Since $\theta_{\Dd}$ and $\pi$ are isomorphisms, 
this results from:

\begin{sublem} The following diagram  commutes.
 \[
 \bpic[xscale=2.1,yscale=1.6]
  \node(11) at (2,-1){$g_*v^!$};
  \node(12) at (2,-2){$u^!\<\<\fst$};
  \node(21) at (1,-1){$g_*v^*$};
  \node(22) at (1,-2){$u^*\!\fst$};
 
  \draw[->] (11)--(12) node[right, midway, scale=.75]{$\theta'_{\<\<\Dd}$};
  \draw[->] (21)--(22) node[left, midway, scale=.75]{$\theta^{-\<1}_{\<\<\Dd}$}
                                 node[right, midway, scale=.75]{$\simeq$};
 
  \draw[<-] (11)--(21)  node[above=-.8pt,midway]{$\Iso$}
                                  node[below, midway, scale=.75]{$g_*\pi$};
  \draw[<-] (12)--(22) node[above=-.8pt,midway]{$\Iso$}
                                 node[below=1pt, midway, scale=.75]{$\pi$};

 \epic
\]
\end{sublem}

\begin{proof} Embed the diagram in question as subdiagram \circled1 of the following diagram, where 
$\varpi\colon\id \to f^!\<\<\fst$ is the unit map of the adjunction\/
$\fst\<\<\dashv f^!$ (see paragraph a few lines after \eqref{def-theta}).
\[
\def\1{$g_*v^*$}
\def\2{$g_*v^!$}
\def\3{$g_*v^!\<\<f^!\<\<\fst$}
\def\4{$u^*\!\fst$}
\def\5{$u^!\!\fst$}
\def\6{$g_*\>g^!u^!\!\fst$}
\def\7{$u^*\!\fst f^!\<\<\fst$}
\def\8{$g_*\>g^!u^*\!\fst$}
\def\9{$g_*v^*\<\< f^!\<\<\fst$}
 \bpic[xscale=3.2, yscale=2]
  \node(11) at (1,-1){\1};
  \node(12) at (2,-1){\2};
  \node(14) at (4,-1){\3};
  
  \node(21) at (1,-2){\4};
  \node(22) at (2,-2){\5};
  \node(23) at (3,-2){\6};
  
  \node(32) at (2,-3){\8};
  
  \node(41) at (1,-4){\7};
  \node(44) at (4,-4){\9};

   \draw[->] (11)--(12) 
                                  node[above, midway, scale=.75]{$g_*\pi$};
   \draw[->] (12)--(14)  node[above, midway, scale=.75]{$g_*v^!\varpi$}; 
  \draw[->] (21)--(22) 
                                  node[below, midway, scale=.75]{$\pi$};
  \draw[<-] (22)--(23) node[below, midway, scale=.75]{$\smallint^{}_{\<g}$};

  \draw[->] (41)--(44) node[below=1pt, midway, scale=.75]{$\theta_{\<\Dd}$};
 
  \draw[<-] (21)--(11) node[left, midway, scale=.75]{$\theta^{-\<1}_{\<\<\Dd}$};
  \draw[->] (.975,-2.2)--(.97,-3.78) node[left=1pt,midway,scale=.75]{$u^{\<*}\!\fst\>\varpi$};
  \draw[<-] (1.025,-2.2)--(1.025,-3.78) node[right=1pt,midway]{$\lift.8,u^{\<*}\!\!\<\smallint^{}_{\!f},$};
  \draw[<-] (22)--(12) node[right, midway, scale=.75]{$\theta'_{\<\<\Dd}$};

  \draw[->] (14)--(44) 
                                  node[right=1pt, midway, scale=.75]{$g_*\pi^{-\<1}$};

  \draw[->] (32)--(21) node[below, midway, scale=.75]{$\smallint^{}_{\<\<g}$};
  \draw[double distance=2pt] (14)--(23) node[above, midway, scale=.75]{$\ps^!\mkern15mu $};
  \draw[<-] (32)--(44) node[above=-.8pt, midway,scale=.75]{$\mkern45mu g_*{\mathsf B}_{\<\Dd}$};
  \draw[->] (23)--(32) node[below=-2pt,midway,scale=.75]{$\mkern40mu g_*g^!\pi^{-\<1}$};  

  \node at (1.5,-1.5){\circled1};
  \node at (2.7,-1.5){\circled2};
  \node at (2, -2.55){\circled3};
  \node at (3.25,-2.55){\circled4};
  \node at (1.5,-3.5){\circled5};

 \epic
\]
The outer border of this diagram commutes: going clockwise around the border from $g_*v^*$ to
$g_*v^*\<\<f^!\<\<\fst$ clearly gives the map $g_*v^*\varpi$, as does going around counterclockwise.

Commutativity of subdiagram \circled2 is the definition of $\theta'_{\<\<\Dd}$; that of \circled3 is clear;
that of \circled4 results from that of diagram~\eqref{B=ps^!}; and that of \circled5 from that of \cite[(2.6.1)]{AJL}.

With these commutativities in view, and since\va{.5} $\int^{}_{\!f}\<\smallcirc f_*\varpi$ 
is the identity map of
$\fst$, diagram chasing yields the desired commutativity of \circled1.
\end{proof}

The remaining nontrivial condition for $\ul{\Sigma}$ to be a setup (as in \S\ref{dual setup}) is commutativity of the analog of \cite[(2.6.1)]{AJL}.

\begin{sublem} The following diagram commutes.
 \[
 \bpic[xscale=2.5,yscale=2]
  \node(11) at (1,-1){$u^{\ul{\<*\<}}\>\!f^{}_{\<\<{\ul{\<*\<}}\>}f^{\ul{!\<}}$};
  \node(12) at (2,-1){$g_{\ul{\<*\<}}\>v^{\ul{\<*\<}}\>\!f^{\ul{!\<}}$};
  \node(21) at (1,-2){$u^*$};
  \node(22) at (2,-2){$g_{\ul{\<*\<}}\>g^{\ul{!\<}} \>u^{\ul{\<*\<}}\>$};
 
  \draw[->] (11)--(12) node[above, midway, scale=.75]{$\ul{\theta\<}_{\>\>\Dd}$};
  \draw[<-] (21)--(22) node[below, midway, scale=.75]{$\ \ul{\int\!}_{\>g}$};
 
  \draw[->] (11)--(21) node[left=1pt, midway, scale=.75]{$u^{\ul{\<*\<}}\>\<\<\ul{\int\!}_{\<f}$};
  \draw[->] (12)--(22) node[right=1pt, midway, scale=.75]{$g_{\ul{\<*\<}}\>\>\ul{\<\mathsf B\<}_{\>\>\Dd}$};

 \epic
\]
\end{sublem}

\begin{proof}
After expansion of the ``dualized" diagram, the assertion becomes that the border of the 
following natural diagram commutes.
\[
\def\1{$u^!\mkern-2.5mu \fst f^*$}
\def\2{$g_*g^!u^!\mkern-2.5mu \fst f^*$}
\def\3{$g_*v^!\mkern-2.5mu f^!\mkern-2.5mu \fst f^*$}
\def\4{$g_*v^!\mkern-2.5mu f^!$}
\def\5{$u^*\mkern-2.5mu \fst f^*$}
\def\6{$g_*g^!u^*\mkern-2.5mu \fst f^*$}
\def\7{$g_*v^*\mkern-2.5mu f^!\mkern-2.5mu \fst f^*$}
\def\8{$g_*v^*\mkern-2.5mu f^*$}
\def\9{$u^*$}
\def\ten{$g_*g^*\<u^*$}
\def\lvn{$u^!$}
\def\twv{$g_*g^!u^!$}
 \bpic[xscale=3.5,yscale=1.5]
  
  \node(11) at (1,-1){\1};
  \node(12) at (2,-1){\2};
  \node(13) at (3,-1){\3};
  \node(14) at (4,-1){\4};
    
  \node(21) at (1,-2){\5};
  \node(22) at (2,-2){\6};
  \node(23) at (3,-2){\7};
  \node(24) at (4,-2){\8};
  
  \node(31) at (1,-3){\9};
  \node(34) at (4,-3){\ten};
  
  \node(41) at (1,-4){\lvn};
  \node(44) at (4,-4){\twv};

  \draw[->] (12)--(11) ;
  \draw[-,double distance=2pt] (12)--(13) node[above, midway, scale=.75]{$\via\ps^!$};
  \draw[->] (14)--(13) ;

  \draw[->] (22)--(21) ;
  \draw[->] (23)--(22) node[below, midway, scale=.75]{$\via\bchadmirado{\Dd}$};
  \draw[->] (24)--(23) ;

  \draw[->] (31)--(34) ;
 
  \draw[->] (41)--(44) ;
  
  \draw[<-] (11)--(21) node[left=1pt, midway, scale=.75]{$\pi$};
  \draw[<-] (21)--(31) ;
  \draw[<-] (31)--(41) node[left=1pt, midway, scale=.75]{$\pi^{-\<1}$};

  \draw[<-] (12)--(22) node[left, midway, scale=.75]{$\via\pi$};

  \draw[<-] (13)--(23) node[right, midway, scale=.75]{$\via\pi$};

  \draw[<-] (14)--(24) node[right, midway, scale=.75]{$\via\pi$};
  \draw[-,double distance=2pt] (24)--(34) node[right=1pt, midway, scale=.75]{$\ps^*$};  
  \draw[<-] (34)--(44) node[right=1pt, midway, scale=.75]{$\via\pi^{-\<1}$};

  \node at (2.5,-1.55)[scale=.9]{\circled1};
  \node at (2.5,-2.55)[scale=.9]{\circled2};

 \epic
\]
Commutativity of the unlabeled subdiagrams is clear; and that of \circled1 is given by \eqref{B=ps^!}.
Subdiagram \circled2 expands as 

\[
\def\1{$g_*g^!u^*\mkern-2.5mu \fst f^*$}
\def\2{$g_*v^*\mkern-2.5mu f^!\mkern-2.5mu \fst f^*$}
\def\3{$u^*\mkern-2.5mu\fst  f^!\mkern-2.5mu \fst f^*$}
\def\4{$u^*\mkern-2.5mu \fst f^*$}
\def\5{$g_*v^*\mkern-2.5mu f^*$}
\def\6{$g_*g^*\<u^*\mkern-2.5mu \fst f^*$}
\def\7{$g_*v^*\mkern-2.5mu f^*\mkern-2.5mu \fst f^*$}
\def\8{$u^*$}
\def\9{$g_*g^*\<u^*$}
 \bpic[xscale=3.5,yscale=1.5]
  
  \node(11) at (1,-1){\1};
  \node(12) at (3,-1){\2};
     
  \node(23) at (2,-2){\3};

  \node(31) at (1,-3){\4};
  \node(33) at (3,-3){\5};
  
  \node(41) at (1.5,-4){\6};
  \node(42) at (2.5,-4){\7};
  
  \node(51) at (1,-5){\8};
  \node(53) at (3,-5){\9};
 
  \draw[->] (12)--(11) node[above, midway, scale=.75]{$\via\bchadmirado{\Dd}$};
 
  \draw[->] (31)--(33) node[above, midway, scale=.75]{$\theta$};

  \draw[-,double distance=2pt] (41)--(42) node[above=1pt, midway, scale=.75]{$\ps^*$};

  \draw[->] (51)--(53) ;
  
  \draw[->] (11)--(31);
  \draw[<-] (31)--(51) ;

  \draw[<-] (13)--(33) ;
  \draw[-,double distance=2pt] (33)--(53) node[right=1pt, midway, scale=.75]{$\ps^*$};

  \draw[->] (22)--(13) node[above, midway, scale=.75]{$\theta$};
  \draw[->] (31)--(41) ;
  
  \draw[<-] (41)--(53) node[left=1pt, midway, scale=.75]{$$};

  \draw[<-] (1.2, -2.78)--(1.7, -2.26) ;
  \draw[->] (1.23, -2.87)--(1.73, -2.35) ;

  \draw[->] (2.58, -3.76)--(2.84, -3.25) ;
  \draw[<-] (2.65, -3.79)--(2.91, -3.28) ;

  \node at (2,-1.5)[scale=.9]{\circled3};

 \epic
\]
Here, commutativity of the unlabeled subdiagrams is clear. The oblique arrow-pairs compose to the identity maps of $u^*\mkern-2.5mu \fst f^*$ and
$g_*v^*\mkern-2.5mu f^*$, respectively. Further, subdiagram \circled3 commutes by \cite[(2.6.1)]{AJL}. 
It follows then by diagram chasing that the outer border commutes, proving the Lemma.
\end{proof}

\emph{In conclusion.} To each setup $\Sigma$ as in \S\ref{dual setup} we have associated a setup~$\ul{\Sigma}$ satisfying the same conditions. Moreover, one verifies that $\ul{\ul{\Sigma}}=\Sigma$.
\end{cosa}

\begin{cosa} With reference to the situation described just before ~\ref{bivdual},
and denoting the image  of $\Hsch{-}$ in the dual category
$\ul{\D\<}_{\>-}$ by $\Huch{-}$, assign to $f\colon X\to Y$ in~$\SS$ the $\ul{\D\<}_{X}$-map\va{-3} 
\begin{equation}\label{ulsharp}
f^{\ul\sharp}=\ul{c_{\<\<f}}\colon f^{\ul{\<*\<}}\>\Huch{Y}\to\Huch{\sX}.
\end{equation}
\vskip3pt
The adjoint map $f^{}_{\<\ul{\sharp}}\colon\Huch{Y}\to f^{}_{\<\ul{\<*\<}}\>\Huch{\sX}$ is defined whenever
$f$ is confined.

Transitivity for the family $(f^{\ul\sharp})$, the property
that $f^{\ul\sharp}$ is an isomorphism
if $f$~is the bottom or top arrow of some independent square, and that $f^{\ul\sharp}$ is an identity map if $f$ is, all result from the corresponding properties of the family $(c_{\<f})$.

By \cite[Theorem 3.4]{AJL}, we get a bivariant theory $\ul{\mathcal B\<}\>$ over~$\SS$, assigning to
each $\SS$\kf-map $f\colon X\to Y$ the symmetric graded $H$-module
\begin{align*}
\buE* {f}{X}{Y}\set \ul{\D\<}_{X}(\Huch{\<X},\>f^{\ul{!\<}}\>\Huch{Y})
&=
\oplus_{i\in\ZZ}\, \ul{\D\<}_{X}^i(\Huch{\<X},\>f^{\ul{!\<}}\>\Huch{Y})\\
&= 
\oplus_{i\in\ZZ}\, \D_\sX^i(f^*\>\Hsch{Y}, \Hsch{\<X}).
\end{align*}

This is the theory referred to in Theorem~\ref{bivdual}. 

The descriptions in Theorem~\ref{bivdual} of product, pushforward and pullback 
for~$\ul{\mathcal B\<}\> $ are obtained by 
dualizing those in \cite[\S3.3]{AJL}---as applied to $\ul{\mathcal B\<}\>$, taking into account the relations  $(-)^{\ul\sharp}=\ul{c_{(-)}}$ (see~\ref{ulsharp})
and $ \>\ul{\<\mathsf B\<}^{}_{\>\>\Dd}\set \ul{\mathsf B'_{\<\Dd}}$ (see~\ref{ulb}).

\smallskip
This completes the proof of Theorem~\ref{bivdual}.
\end{cosa}

\begin{ex} \label{concrete}
Assume $S$ has a dualizing complex $\cd{}$. Then for  \mbox{$x\colon X\to S$} in~$\SS$, $\cd x\set x^!_{\<\upl}\cd{}$ (see \ref{! and otimes}) is a dualizing complex on~$X$:  localization over noetherian rings preserves injectivity of modules, and hence for a localizing immersion $f$ the functor $f^!=f^*$
preserves finiteness of injective dimension, so that the proof of \cite[Proposition 4.10.1(i)]{li} extends to the efp context. As before, we abuse notation by writing $\cd{X}$ for $\cd x\>$.

Note in particular the case where $x$ is essentially smooth of relative dimension $n$, so that by Proposition~\ref{fc+V}, or otherwise, $\cd X\cong\Omega^n_x[n]\otimes x^*\cd{}\>$. 

Let $\Dcpl(X)$ (resp.~$\Dcmi(X)$) be that full subcategory of $\D(X)$ whose objects are the $\OX$-complexes~$E$ such that the cohomology sheaf $H^n(E)$ is coherent for all~$n\in \ZZ$ and vanishes for $n\ll0$ (resp.~$n\gg0$).  
The functor
$$
\boldsymbol{\mathcal D}_\sX(-)\set\R\sHom(-, \cd{X})
$$ 
induces quasi-inverse anti-equivalences between $\>\Dcpl(X)$ and $\>\Dcmi(X)$, 
and so corresponds to an equivalence of each of $\>\Dcpl(X)$ and $\>\Dcmi(X)$ with the dual category of the other---even as graded $H$-categories,\va1 see the beginning of~
\S\ref{dual setup2}.

Suppose now that $f\colon X\to Y$ is an $\SSp$\kf-map, i.e., 
a perfect $\SS$\kf-map. Then  $f^*\Dcpl(Y)\subset\Dcpl(X)$ and $f^*\Dcmi(Y)\subset\Dcmi(X)$
(cf.~\cite[Remark 2.1.5]{AIL}).\va1 

Moreover, $f^!\Dcpl(Y)\subset\Dcpl(X)$: indeed, for $F\in\Dcpl(Y)$,
the property that $f^!\<F\in\Dcpl(X)$ is local on $X\<$, so to verify that
property one may assume that  $f=pi$ where $p\colon Z\to Y$ is
essentially smooth of relative dimension $n$ (say)
and $i\colon X\to Z$ is a closed immersion in $\SSp\>$, with 
$i_*\OX$ perfect over $\OZ$ (\cite[Remark~2.1.2]{AIL});  and then 
$p^!F\cong\Omega^n_p[n]\otimes_{\<Z}\>\> p^*\<\<F\in\Dcpl(Z)$ (see \ref{fc+V}),\va{-1}  so it suffices to note that since the 
$\OZ$-complex $i_*i^!\OZ \cong \R\sHom_Z(i_*\OX,\OZ)$ is perfect,\va{.5} therefore
$$
i_*f^!\<F\cong i_*i^!p^!\<F= i_*(i^!\OZ\otimes_{\sX} i^*p^!\<F)\underset{\eqref{projection}}\cong
i_*i^!\OZ\otimes_{\<Z} p^!\<F\in\Dcpl(Z).
$$

Similarly, $f^!\Dcmi(Y)\subset\Dcmi(X)$.

If $f$ is also proper, then 
$f_*\Dcpl(X)\subset\Dcpl(Y)$ and $f_*\Dcmi(X)\subset\Dcmi(Y)$ (see, e.g., \cite[3.9.2 and~4.3.3.2]{li}); and  as $\cd X\cong f^!_{\upl}\cd{\,Y}=f^!\cd{\,Y}$, Grothendieck duality gives an isomorphism of functors
$$f_*\boldsymbol{\mathcal D}_\sX\cong \boldsymbol{\mathcal D}_Y\< f_*\>.
$$

Finally, there are functorial isomorphisms 
\begin{align*}
f^*\boldsymbol{\mathcal D}_Y M &\cong\boldsymbol{\mathcal D}_\sX f^!M\qquad(M\in\Dcmi(Y)),\\
\qquad f^! \boldsymbol{\mathcal D}_YN & \cong\boldsymbol{\mathcal D}_\sX f^*\<\<N\qquad(N\in\Dcpl(Y)).
\end{align*}
Indeed, for $M\in\Dcmi(Y)$, \cite[Proposition 4.6.7]{li} and \cite[Lemma 2.1.10]{AIL}---in which, using the present definition of $f^!$ (see \eqref{! and otimes1}), one need only assume $M\in\Dcmi(Y)$---give functorial isomorphisms
\begin{align*}
f^*\bD YM=f^*\R\sHom_Y(M\<,\cd{\,Y}\<)&\iso\R\sHom_\sX(f^*\<\<M\<,f^*\cd{\,Y}\<)\\
&\iso\R\sHom_\sX(f^!\<M\<,f^!\cd{\,Y}\<)
\cong\bD\sX f^!\<M\>;
\end{align*}
and there results a sequence of functorial isomorphisms
\[
f^!\bD YN\iso\bD\sX\bD\sX f^!\bD YN\iso \bD\sX f^*\bD Y\bD YN\iso\bD\sX f^*\<\<N.
\]
(One could also imitate the proof of \cite[Proposition 4.10.1(ii)]{li}---without ignoring, as that proof does, the question of dependence of the constructed isomorphism on the choice of the implicitly used
compactification.) 
\vskip1pt

It follows (details left to the reader) that for the bivariant theory described in~ \S\ref{orientations},
but restricted to flat $\SS$\kf-maps and  to complexes in $\Dcmi$, one
can regard the dual bivariant theory as arising from the same setup, except that $\Dcmi$ is replaced throughout by $\Dcpl$, and $\Hsch{}$ by $\Hsch{W}'\set\bD W\Hsch{W}$,\va{-2} 
and  for any flat $\SS$-map $f\colon X\to Y$, 
$f^\sharp\>{}'\colon f^*\Hsch{Y}'\to \Hsch{\sX}'$  is defined to be the dual of the fundamental 
class $c_{\<f}$, i.e., the natural composition\va{-5}
\[
 f^*\Hsch{Y}'=f^*\bD Y\Hsch{Y}\iso \bD{\sX}f^!\Hsch{Y}\xto{\<\<\bD{\sX}c^{}_{\<\<f}\>}
\bD{\sX}\Hsch{\sX}=\Hsch{\sX}'.
\]

\end{ex}

\begin{cosa}
Next, for a flat $x\colon X\to S$ in~$\SS$, we discuss a  product 
$\Hsch{\sX}\otimes\Hsch{\sX}\to \Hsch{\sX}$, and then combine it with the fundamental class 
$c_x$  to define
the duality map~\eqref{dualitymap}
\[
\fd_\sX\colon \Hsch{\sX}\to \R\sHom_\sX(\Hsch{\sX}, x^!\OS).
\]
As mentioned at the beginning of \S\ref{dual theory}, Theorem~\ref{dualityiso} says that when $x$ is essentially smooth, $\fd_\sX$ is an \emph{isomorphism}. \va2

Let $f\colon X\to Y$ be a scheme-map, with diagonal $\delta\colon X\to X\times_YX$ and pre\kf-Hochschild functor 
$$
\Hh{\<f}\set \delta^* \delta_{\<*} \colon\D(X)\to\D(X),
$$
see beginning  of~\S\ref{preHoch} (with, as usual, the notational convention of \S\ref{derived}). 

\pagebreak[3]
Define the bifunctorial map
\begin{equation}\label{def-of-t}
\ft_{\<f}(A,B)\colon\Hh{\<f}A\otimes_\sX\Hh{\<f}B\lto \Hh{\<f}(A\otimes_\sX B)\qquad(A,B\in\D(X))
\end{equation}
to be the natural composition
$$
\delta^*\delta_*A\otimes_\sX\delta^*\delta_*B\iso\delta^*(\delta_*A\otimes_{\sX\times_{\<Y}X}\delta_*B)
\lto \delta^*\delta_*(A\otimes_\sX B).
$$
In particular, for $x\colon X\to S$ in $\SS$ one has the map
\begin{equation}\label{def-of-t0}
\ft_x(\OX,\OX)\colon\Hsch{\sX}\otimes_\sX\Hsch{\sX}\to\Hsch{\sX}.
\end{equation}

Corresponding to $\ft_x(\OX,\OX)$  under hom$-\otimes$ adjunction, there is a $\D_\sX$-map 
$\Hsch{\sX}\to \R\sHom_\sX(\Hsch{\sX},\Hsch{\sX})$, whence for each $i\in\ZZ$ a natural map
\begin{equation}\label{hom to cohom}
\h^{-i}(X\<, \Hsch{\sX})\to \ext^{-i}_{\OX}(\Hsch{\sX},\Hsch{\sX})=\HH^{-i}(X|S)
\end{equation}
from the $i$-th classical Hochschild homology of $X$ to the $({-i})$-th bivariant cohomology (see \S\ref{begin intro}).\va2

The map $\ft$ is also functorial on the category of schemes over a fixed~$Z\>$:

\begin{sublem}\label{Hprod and *}
For scheme-maps\/ $X\xto{\,f\,}Y\xto{\,g\,}Z$ and\/ $E,F\in\D(Y),$ the following natural diagram commutes.
\[
\def\1{$f^*(\Hh{g}E\otimes_Y\Hh{g}F)$}
\def\2{$f^*\Hh{g}(E\otimes_Y F)$}
\def\3{$f^*\Hh{g}E\otimes_\sX\< f^*\Hh{g}F$}
\def\4{$\Hh{gf}f^*(E\otimes_Y F)$}
\def\5{$\Hh{gf}f^*\<\<E\otimes_\sX \Hh{gf}f^*\<\<F$}
\def\6{$\Hh{gf}(f^*\<\<E\otimes_\sX\< f^*\<\<F)$}
 \bpic[xscale=3.5,yscale=2]
  
  \node(11) at (1,-1){\1};
  \node(12) at (3,-1){\2};
    
  \node(21) at (1,-2){\3};
  \node(22) at (3,-2){\4};

  \node(31) at (1,-3){\5};
  \node(32) at (3,-3){\6};

  \draw[->] (11)--(12) node[above, midway, scale=.75]{$f^*\ft_g$};
     
  \draw[->] (31)--(32) node[below=1pt, midway, scale=.75]{$\ft_{gf}$};
  
  \draw[->] (11)--(21) node[left=1pt, midway, scale=.75]{$\simeq$};
  \draw[->] (21)--(31) node[left=1pt, midway, scale=.75]{\eqref{bbbiup}};

  \draw[->] (12)--(22) node[right=1pt, midway, scale=.75]{\eqref{bbbiup}};
  \draw[->] (22)--(32) node[right=1pt, midway, scale=.75]{$\simeq$};
  
 \epic
\]

\end{sublem}
\begin{proof}  Let $\delta\colon X\to X\times_Z X$ and $\bar\delta\colon Y\to Y\times_Z Y$ be the diagonal maps, and $h\set f\times_Z f\colon X\times_Z X\to Y\times_Z Y$, so that $\bar\delta f=h\delta$: 
\[
\bpic[xscale=1.3,yscale=2]

  \node(11) at (1,-1){$X$};
  \node(12) at (3,-1){$Y$};
    
  \node(21) at (1,-2){$X\times_Z X$};
  \node(22) at (3,-2){$Y\times_Z Y$};

  \draw[->] (11)--(12) node[above, midway, scale=.75]{$f$};
  \draw[->] (21)--(22) node[below=1pt, midway, scale=.75]{$h$};
  
   \draw[->] (11)--(21) node[left=1pt, midway, scale=.75]{$\delta$};
   \draw[->] (12)--(22) node[right=1pt, midway, scale=.75]{$\bar\delta$};
   
  \epic 
\]
The diagram in ~\ref{Hprod and *} expands naturally, via definitions and multiple uses of~\eqref{^* and tensor}, as follows (where the  obvious subscripts for $\otimes$
have been omitted).
\[
\def\1{$f^*(\bar\delta^*\bar\delta_*E\otimes\bar\delta^*\bar\delta_*F)$}
\def\2{$f^*\bar\delta^*(\bar\delta_*E\otimes\bar\delta_*F)$}
\def\3{$f^*\bar\delta^*\bar\delta_*(E\otimes F)$}
\def\5{$f^*\bar\delta^*\bar\delta_*E\otimes f^*\bar\delta^*\bar\delta_*F$}
\def\6{$\delta^*h^*(\bar\delta_*E\otimes\bar\delta_*F)$}
\def\7{$\delta^*h^*\bar\delta_*(E\otimes F)$}
\def\9{$\delta^*h^*\bar\delta_*E\otimes\delta^*h^*\bar\delta_*F$}
\def\ten{$\delta^*(h^*\bar\delta_*E\otimes h^*\bar\delta_*F)$}
\def\lvn{$\delta^*\delta_*f^*(E\otimes F)$}
\def\thn{$\delta^*\delta_*f^*\<\<E\otimes \delta^*\delta_*f^*\<\<F$}
\def\frn{$\delta^*(\delta_*f^*\<\<E\otimes \delta_*f^*\<\<F)$}
\def\ffn{$\delta^*\delta_*(f^*\<\<E\otimes f^*\<\<F)$}
\def\sxn{$\delta^*\delta_*f^*\<\<E\otimes \delta^*\delta_*f^*\<\<F$}
\def\svn{$\delta^*(\delta_*f^*\<\<E\otimes \delta_*f^*\<\<F)$}
\def\egn{$\delta^*\delta_*(f^*\<\<E\otimes f^*\<\<F)$}
 \bpic[xscale=4.5,yscale=1.75]
  
  \node(11) at (1,-1){\1};
  \node(12) at (2,-1){\2};
  \node(13) at (3,-1){\3};
    
  \node(21) at (1,-2){\5};
  \node(22) at (2,-2){\6};
  \node(23) at (3,-2){\7};

  \node(31) at (1,-3){\9};
  \node(32) at (2,-3){\ten};
  \node(33) at (3,-3){\lvn};

  \node(41) at (1,-4){\sxn};
  \node(42) at (2,-4){\svn};
  \node(43) at (3,-4){\egn};

  \draw[->] (11)--(12) ;
  \draw[->] (12)--(13) ;
     
  \draw[->] (22)--(23) ;
     
  \draw[->] (32)--(31) ;
  
  \draw[->] (41)--(42) ;
  \draw[->] (42)--(43) ;
  
  \draw[->] (11)--(21) ;
  \draw[double distance=2pt] (21)--(31) node[left=1pt, midway, scale=.75]{$\ps^*$};
  \draw[->] (31)--(41) node[left=1pt, midway, scale=.75]{\eqref{def-of-bch-asterisco}};
  
  \draw[double distance=2pt] (12)--(22) node[left=1pt, midway, scale=.75]{$\ps^*$};
  \draw[->] (22)--(32) ;
  \draw[->] (32)--(42) node[right=1pt, midway, scale=.75]{\eqref{def-of-bch-asterisco}};

  \draw[double distance=2pt] (13)--(23) node[right=1pt, midway, scale=.75]{$\ps^*$};
  \draw[->] (23)--(33) node[right=1pt, midway, scale=.75]{\eqref{def-of-bch-asterisco}};
  \draw[->] (33)--(43) ;
  
 \node at (1.5,-2)[scale=.9]{\circled1};
 \node at (2.5,-3)[scale=.9]{\circled2};

  \epic
\]
Commutativity of  the unlabeled subdiagrams is obvious. For commutativity of \circled1
argue as in the last half of \cite[\S3.6.10]{li}. For commutativity of~\circled2, note that since 
$\delta^*$ and~$\delta_*$ are adjoint, therefore for any
 $\D(X\times_Z X)$-maps $\alpha,\,\beta\colon C\to\delta_*D\>$ the $\D(X)$-maps $\delta^*\<\alpha$
and $\delta^*\<\beta$ are equal if the same holds after composition with the natural map $\delta^*\delta_*D\to D$---so that it suffices to show commutativity of the next natural diagram: 
\[
\mkern-10mu
\def\1{$\delta^*h^*\<(\bar\delta_*E\<\otimes\<\bar\delta_*F)$}
\def\2{$\delta^*h^*\<\bar\delta_*(E\<\otimes\< F)$}
\def\3{$\delta^*\<(h^*\<\bar\delta_*E\<\otimes\< h^*\<\bar\delta_*F)$}
\def\4{$\delta^*\<\delta_*f^*\<(E\<\otimes\< F)$}
\def\5{$\delta^*\<(\delta_*f^*\<\<E\<\otimes\< \delta_*f^*\<\<F)$}
\def\6{$\delta^*\<\delta_*(f^*\<\<E\<\otimes\<\< f^*\<\<F)$}
\def\7{$f^*\<\<E\<\otimes\< f^*\<\<F$}
\def\8{$f^*\<\bar\delta^*\<\bar\delta_*(E\<\otimes\< F)$}
\def\9{$\mkern-12mu f^*\<(E\<\otimes\< F)$}
\def\ten{$f^*\<\bar\delta^*\<\bar\delta_*E\<\otimes\< f^*\<\bar\delta^*\<\bar\delta_*F^{\mathstrut}$}
\def\lvn{$\delta^*h^*\<\bar\delta_*E\<\otimes\< \delta^*h^*\<\bar\delta_*F$}
\def\twv{$\delta^*\<\delta_*f^*\<\<E\<\otimes\< \delta^*\<\delta_*f^*\<\<F$}
\def\thn{$f^*\<(\bar\delta^*\<\bar\delta_*E\<\otimes\<\bar\delta^*\<\bar\delta_*F)$}
\def\frn{$f^*\<\bar\delta^*\<(\bar\delta_*E\<\otimes\<\bar\delta_*F)$}
\ \bpic[scale=.835, xscale=4.25,yscale=1.5]
  
  \node(11) at (1,-1){\1};
  \node(14) at (3.75,-1){\2};
  
  \node(21) at (1,-2){\3};
  \node(22) at (2,-2.65){\thn};
  \node(23) at (2.975,-1.8){\frn};

  \node(32) at (2,-3.65){\ten};
  \node(33) at (2.975,-3){\8};

  \node(42) at (1,-3){\lvn};
  \node(43) at (2.975,-4){\9};
  \node(44) at (3.75,-4){\4};
  
  \node(52) at (1,-5){\twv};
  \node(53) at (2.975,-5){\7};
  
  \node(61) at (1,-6){\5};
  \node(64) at (3.75,-6){\6};
  
  \draw[->] (11)--(14) ;
  
   \draw[->] (21)--(42) ;
    
  \draw[<-] (43)--(44) ;
     
  \draw[->] (52)--(53) ;
  
  \draw[->] (61)--(64) ;
  
  \draw[->] (11)--(21) ;
  
  \draw[double distance=2pt] (32)--(42) node[below=-.5pt, midway, scale=.75]{$\ps^*\mkern23mu$};
  \draw[->] (42)--(52) node[left=1pt, midway, scale=.75]{\eqref{def-of-bch-asterisco}};

  \draw[->] (23)--(33) node[left=1pt, midway, scale=.75]{$$};
  \draw[->] (33)--(43) ;
  \draw[->] (43)--(53) node[left=1pt, midway, scale=.75]{};
  
  \draw[->] (14)--(44) node[right=1pt, midway, scale=.75]{\eqref{def-of-bch-asterisco}};
  \draw[->] (44)--(64) ;
 
    \draw[double distance=2pt] (1.42,-1.18)--(2.56,-1.64) node[below, midway, scale=.75]{$\ps^*$};
   \draw[->] (23)--(22);
   \draw[->] (22)--(32);
   \draw[->] (22)--(43);
   \draw[->] (32)--(53);
   \draw[double distance=2pt] (3.65,-1.25)--(3.3,-2.7) node[below=-2pt, midway, scale=.75]{$\mkern30mu\ps^*$};
    \draw[<-] (61)--(52) ;
    \draw[->] (64)--(53) ;
    
 \node at (1.9,-2.05)[scale=.9]{\circled1};
 \node at (2.65,-2.6)[scale=.9]{\circled3};
 \node at (3.525,-3.03)[scale=.9]{\circled4};
 \node at (1.9,-4.4)[scale=.9]{\circled5};
 \node at (1.9,-5.5)[scale=.9]{\circled6};

  \epic
\]

Subdiagram \circled1 is the same as the  commutative subdiagram \circled1 in the preceding diagram.

Commutativity of \circled3 and \circled6 is given by \cite[3.6.7d(iv)]{li}---as realized in \cite[3.6.10]{li}.

Commutativity of \circled4 and \circled5 results from the definition in \eqref{bch.2} of the map \eqref{def-of-bch-asterisco}.\looseness=-1

Commutativity of the remaining three subdiagrams is obvious.
\end{proof}

For any flat map  $x\colon X\to S$ in $\SS$, one has then the map
\[
\fp_\sX=\fp_x\colon\Hsch{\sX}\otimes_\sX\Hsch{\sX}\xto{\<\ft_{x}(\OX\<,\mkern1.5mu\OX\<)\>} \Hsch{\sX}
\xrightarrow[\!\eqref{csubf}\!\!]{c_x} x^!\OS.
\]

From Proposition~\ref{c inverse}, Theorem~\ref{trans fc} and Lemma~\ref{Hprod and *}, one deduces:
\looseness=-1
\begin{subcor} \label{p & etale}
Let\/ $x\colon X\to S$ and\/ $y\colon Y\to S$ be flat maps in\/ $\SS,$ and let\/
$f\colon X\to Y$ be an essentially \'etale\/ $\SS$\kf-map. The following diagram commutes.
\[
\def\1{$f^*(\Hsch{Y}\otimes_Y \Hsch{Y})$}
\def\2{$f^*y^!\OS$}
\def\3{$f^*\Hsch{Y}\otimes_\sX\< f^*\Hsch{Y}$}
\def\4{$f^!y^!\OS$}
\def\5{$\Hsch{\sX}\otimes_\sX \Hsch{\sX}$}
\def\6{$x^!\OS$}
 \bpic[xscale=3.5,yscale=1.5]
  
  \node(11) at (1,-1){\1};
  \node(12) at (3,-1){\2};
    
  \node(21) at (1,-2){\3};
  \node(22) at (3,-2){\4};

  \node(31) at (1,-3){\5};
  \node(32) at (3,-3){\6};

  \draw[->] (11)--(12) node[above=1pt, midway, scale=.75]{$f^*\fp_Y$};
     
  \draw[->] (31)--(32) node[below=1pt, midway, scale=.75]{$\fp_{\sX}$};
  
  \draw[->] (11)--(21) node[left=1pt, midway, scale=.75]{$\simeq$};
  \draw[->] (21)--(31) node[left=1pt, midway, scale=.75]{\eqref{bbbiup}};

  \draw[double distance=2pt] (12)--(22) ;
  \draw[double distance=2pt] (22)--(32) node[right=1pt, midway, scale=.75]{$\ps^!$};
  
 \epic
\]

\end{subcor}

Corresponding to $\fp_\sX$ under hom$-\otimes$ adjunction, there is  a \emph{duality map}
\begin{equation}\label{dualitymap}
\fd_\sX\colon \Hsch{\sX}\to \R\sHom_\sX(\Hsch{\sX}, x^!\OS).
\end{equation}
whence for each $i\in\ZZ$ a natural map
\begin{equation}\label{hom to hom}
\h^{-i}(X\<, \Hsch{\sX})\to \ext^{-i}_{\OX}(\Hsch{x},x^!\OS)=\HH_i(X|S)
\end{equation}
from the $i$-th classical Hochschild homology of $X$ to the $i$-th bivariant homology (see \S\ref{begin intro}). 

Also, for any scheme\kf-map $f\colon X\to Y$ there is a bifunctorial map
\begin{equation}\label{* and hom}
f^*\R\sHom_Y(E,F)\to \R\sHom_\sX(f^*\<\<E,f^*\<\<F)\qquad (E,F\in\D(Y)),
\end{equation}
corresponding under hom$-\otimes$ adjunction to the natural composition
\[
f^*\R\sHom_Y(E,F)\otimes_\sX f^*\<\<E\to f^*(\R\sHom_Y(E,F)\otimes_\sX E)\to f^*\<\<F,
\]
see \cite[3.5.6(a)]{li}, or \cite[3.5.6(g)]{li}, with $(C,D,E)\set (\R\sHom_Y(E,F), E, F).$

If $X$ is noetherian, $f$ is perfect, $E$ is cohomologically bounded-above, with coherent homology, 
and $F$ is cohomologically bounded below, then the map~\eqref{* and hom}~is an \emph{isomorphism} \cite[4.6.7]{li}. \va2

The duality map $\fd$ is compatible with essentially \'etale localization:

\begin{subprop} Let\/ $x\colon X\to S$ and\/ $y\colon Y\to S$ be flat\/ $\SS$\kf-maps and let\/
\mbox{$f\colon X\to Y$} be an essentially \'etale\/ $\SS$\kf-map. The following diagram commutes.
\[
\mkern-3mu
\def\1{$f^*\Hsch{Y}$}
\def\2{$f^*\R\sHom_Y(\Hsch{Y}, y^!\OS)$}
\def\3{$\R\sHom_Y(f^*\Hsch{Y}, f^*y^!\OS)$}
\def\4{$\R\sHom_Y(f^*\Hsch{Y}, f^!y^!\OS)$}
\def\5{$\Hsch{\sX}$}
\def\6{$\R\sHom_\sX(\Hsch{\sX}, x^!\OS)$}
\def\7{$\R\sHom_Y(f^*\Hsch{Y}, x^!\OS)$}
 \bpic[xscale=4.65,yscale=1.45]
  
  \node(11) at (1,-1){\1};
  \node(12) at (1.92,-1){\2};
  \node(13) at (3,-1){\3};
    
  \node(23) at (3,-2){\4};

  \node(31) at (1,-3){\5};
  \node(32) at (1.92,-3){\6};
  \node(33) at (3,-3){\7};

  \draw[->] (11)--(12) node[above=1pt, midway, scale=.75]{$f^*\fd_Y$};
  \draw[->] (12)--(13) node[below=1pt, midway, scale=.75]{$\eqref{* and hom}$}
                          node[above=-5pt, midway, scale=.75]{$\widetilde{\phantom{mm}}$};
     
  \draw[->] (31)--(32) node[below=1pt, midway, scale=.75]{$\fd_{\sX}$};
  \draw[->] (32)--(33) node[below=1pt, midway, scale=.75]{\eqref{caset}}
                                  node[above=-5pt, midway, scale=.75]{$\widetilde{\phantom{mm}}$};
  
   \draw[->] (11)--(31) node[left=1pt, midway, scale=.75]{$\eqref{caset}$}
                                  node[right=1pt, midway, scale=.75]{$\simeq$}  ;

   \draw[double distance=2pt] (13)--(23) ;
   \draw[double distance=2pt] (23)--(33) node[right=1pt, midway, scale=.75]{$\ps^!$};
  
 \epic
\]
\end{subprop}

\begin{proof}
It suffices to show that the adjoint diagram---that is, the border of the following natural diagram, where $\R\CH\set\R\sHom$, and the obvious subscripts for $\otimes$ are omitted---commutes:
\[
\mkern-4mu
\def\1{$f^*\Hsch{Y}\<\otimes\< f^*\Hsch{Y}$}
\def\2{$ f^*\R\CH_Y(\Hsch{Y}, y^!\OS)\<\otimes\< f^*\Hsch{Y}$}
\def\3{$\mkern-35mu\R\CH_Y(f^*\Hsch{Y}, f^*y^!\OS)\mkern-1.5mu\otimes\<\< f^*\Hsch{Y}$}
\def\4{$f^*\<(\Hsch{Y}\<\otimes\< \Hsch{Y})$}
\def\5{$f^*y^!\OS$}
\def\6{$f^!y^!\OS$}
\def\7{$\Hsch{\sX}\<\otimes\< f^*\Hsch{Y}$}
\def\8{$\R\CH_\sX(\Hsch{\sX}, x^!\OS)\<\otimes\< f^*\Hsch{Y}$}
\def\9{$x^!\OS$}
\def\ten{$\Hsch{\sX}\<\otimes\< \Hsch{\sX}$}
\def\lvn{$\R\CH_\sX(f^*\Hsch{Y}, x^!\OS)\<\otimes\< f^*\Hsch{Y}$}
\def\twv{$f^*\<(\R\sHom_Y(\Hsch{Y}, y^!\OS)\<\otimes\< \Hsch{Y})$}
\def\thn{$\R\CH_\sX(\Hsch{\sX}, x^!\OS)\<\otimes\< \Hsch{\sX}$}
 \bpic[scale=.835, xscale=5,yscale=1.65]
  
  \node(11) at (1,0){\2};
  \node(12) at (2,-1){\twv};
  \node(13) at (3,0){\3};
  
  \node(21) at (1,-2){\1};
  \node(22) at (2,-2){\4};
  \node(23) at (3,-2){\5};

  \node(31) at (1,-3){\7};
  \node(32) at (2.175,-3){\ten};
  \node(33) at (3,-3){\6};
  
  \node(41) at (1,-4){\8};
  \node(42) at (2.175,-4){\thn};
  \node(43) at (3,-4){\9};

  \node(52) at (2.175,-5){\lvn};

  \draw[->] (11)--(12) ;
  \draw[->] (11)--(2.3,0) ; 
  
   \draw[->] (21)--(22) node[above=-7.5pt, midway]{$\widetilde{\phantom{nn}}$};
   \draw[->] (22)--(23) node[below=1pt, midway, scale=.75]{$f^*\fp_Y$};
   
   \draw[->] (31)--(32) ;
    
  \draw[->] (41)--(52) ;
  \draw[->] (41)--(42) ;
  \draw[->] (42)--(43) ;
 
  \draw[->] (52)--(43) ;

  \draw[<-] (11)--(21) node[left=1pt, midway, scale=.75]{$\via\fd_Y$};
  \draw[->] (12)--(23) ;
  \draw[->] (21)--(31) ;
  \draw[->] (31)--(41) node[left=1pt, midway, scale=.75]{$\via\fd_\sX$};

  \draw[->] (13)--(23) ;    
  \draw[->] (22)--(12) node[left=1pt, midway, scale=.75]{$\via\fd_Y$};
  \draw[->] (32)--(42) node[left=1pt, midway, scale=.75]{$\via\fd_\sX$};    
  
  \draw[->] (32)--(43) node[above=1pt, midway, scale=.75]{$\fp_\sX$};
  
  \draw[double distance=2pt] (23)--(33) ;
  \draw[double distance=2pt] (33)--(43) node[right=1pt, midway, scale=.75]{$\ps^!$};
    
  \node at (2.175,-.5)[scale=.9]{\circled1};
  \node at (2.175,-2.52)[scale=.9]{\circled2};
  \node at (2.175,-4.5)[scale=.9]{\circled3};

 \epic
\]

Commutativity of subdiagram \circled1 is just the definition following \eqref{* and hom}; that of \circled2 is given by Corollary~\ref{p & etale};
that of \circled3 is given by \cite[3.5.3(h)]{li};
and that of the unlabeled subdiagrams is obvious.
\end{proof}
\end{cosa}

\begin{thm} \label{dualityiso}
If\/ $x\colon X\to S$ is  essentially smooth then the  \mbox{duality map\/} 
$\fd_\sX\mkern-1.5mu$ in~\eqref{dualitymap} is an isomorphism.
\end{thm}

\begin{proof} Corollary~\ref{p & etale} shows  the assertion is local on $X$; so we may assume
that $X$ and $S$ are affine, say $S=\spec A$ and $X=\spec R$, with $R$ an essentially smooth $A$-algebra such that the kernel of the multiplication~map $T\set R\otimes_AR\to R$ is generated by a regular sequence $\mathbf t=(t_1, \dots,t_n)$ in~$T$ (see~\S\ref{efp}). 
Then $\mathbf t/\mathbf t^2=\Omega^1_{R|A}$, and the Koszul complex 
$\mathscr K_\bullet(\mathbf t)$\va{-1} is a flat resolution of the $T$-module 
$R=T/\mathbf t T$. 

\pagebreak[3]
 Omitting the  obvious subscripts for $\otimes$, we have that
the map 
$$
\delta_*\OX\otimes \delta_*\OX\to\delta_*(\OX\otimes\OX)=\delta_*\OX
$$
that forms part of the definition of $\ft_x(\OX,\OX)$ (see \eqref{def-of-t0}) is, by \cite[3.4.5.2]{li}, the unique $\xi$ such that the 
following natural diagram commutes:
\[
 \bpic[xscale=2,yscale=1.3]
  \node(11) at (1,-1){$\delta^*(\delta_*\OX\otimes \delta_*\OX)$};
  \node(13) at (3,-1){$\delta^*\delta_*\OX$};

  \node(21) at (1,-2){$\delta^*\delta_*\OX\otimes \delta^*\delta_*\OX$};
  \node(23) at (3,-2){$\OX$};
  
  \draw[->] (11)--(13) node[above, midway, scale=.75]{$\delta^*\xi$};
  
  \draw[->] (21)--(23) node[below=1pt, midway, scale=.75]{$\epsilon\otimes\epsilon$};
  
  \draw[->](11)--(21);
  \draw[->](13)--(23) node[right=1pt, midway, scale=.75]{$\epsilon$};
 \epic
\]

 The counit map $\epsilon\colon\delta^*\delta_*\OX\to\OX$ can be identified with the sheafification of the natural map of complexes (concentrated in negative degrees, and with vanishing differentials)
$$
\mathscr K_\bullet(\mathbf t)\otimes_T R=\tbw_\bullet\,\mathbf t/\mathbf t^2=\oplus_{i=-n}^0\,\tbw^{\<-i}\,\mathbf t/\mathbf t^2\to 
\tbw^{\<0}\,\mathbf t/\mathbf t^2=R.
$$
It results that the map $\xi$ can be identified with the usual multiplication map 
$\mathscr K_\bullet(\mathbf t)\otimes_T \mathscr K_\bullet(\mathbf t)\to\mathscr K_\bullet(\mathbf t)$
(a map of complexes, $\mathscr K_\bullet(\mathbf t)$ being a differential graded algebra).
Hence  the~map
$\ft_x(\OX,\OX)$ can be identified with the exterior multiplication map 
$$
\tbw_\bullet\,\Omega^1_x\otimes \tbw_\bullet\,\Omega^1_x\lto\tbw_\bullet\, \Omega^1_x\>. 
$$

Furthermore, Proposition~\ref{fc+V} gives an identification of $c_x\colon\Hsch{\sX}\to x^!\OX$
with the natural map of complexes
\begin{equation}\label{exterior projection}
\tbw_\bullet\, \Omega^1_x\to\Omega^n_x[n].
\end{equation}
The assertion results now from the well-known isomorphisms
$$
 \Omega^i_{R|A}\iso \Hom_R(\Omega^{n-i}_{R|A}\>, \>\Omega^n_{R|A}).
$$
arising from exterior multiplication followed by \eqref{exterior projection}.
\end{proof}

Recall from \cite[\S6.5]{AJL} that for essentially smooth $x$ the C\u{a}ld\u{a}raru\kf-Willerton version of Hochschild homology, $\HH_i^{\text{cl}}\!\big(X)$, is isomorphic to the bivariant homology  
$$
\HH_i(X|S)\set\ext^{-i}_{\OX}(\Hsch{\sX}, x^!\OS)=
\Hom_{\D(X)}\!\big(\OX[i], \R\sHom_\sX(\Hsch{\sX}, x^!\OS)\big).
$$
\begin{subcor} \label{homologies}
If\/ $x\colon X\to S$ is an essentially smooth\/ $\SS$\kf-map then for each\/ $i\in\ZZ$ the map \eqref{hom to hom} is an isomorphism. Hence there is a natural isomorphism
$$
\h^{-i}(X\<,\Hsch{\sX})\iso\HH_i^{\textup{cl}}\!\big(X).
$$
\end{subcor}

More generally:
\begin{subcor} \label{iso to dual}
Let\/ $f\colon X\to Y$ be a flat map of essentially smooth\/ $S$-schemes. The bivariant groups associated to $f$  by\/ $\ul{\CB\<}\>$ 
and\/ $\CB$ are isomorphic.
\end{subcor}

\begin{proof} One has, for $i\in\ZZ$,
natural isomorphisms (the last one induced by $\fd_\sX$ and $\fd_Y$):
\begin{align*}
\ext^i_{\OX}( f^*\Hsch{Y}, \Hsch{\sX})
&\cong
\ext^i_{\OX}(\bD\sX\Hsch{\sX}, \bD\sX\< f^*\Hsch{Y})\\
&\cong
\ext^i_{\OX}(\bD\sX\Hsch{\sX}, f^!\bD Y\Hsch{Y})
\cong
\ext^i_{\OX}(\Hsch{\sX},  f^!\Hsch{Y}).
\end{align*}
\end{proof}

When $Y=S$,  one gets from~\ref{iso to dual} homology isomorphisms
\[
\ul{\<\HH\<}_{\>i}(X) 
=\ext^{-i}_{\sX}(\OX, \Hsch{\<X}) \cong \h^{-i}(X,\Hsch{X})
\!\iso\!\ext^{-i}_{\sX}(\Hsch{\<X}, x^!\OS)=\HH_{\>i}(X)
\]
that, one checks, coincide with those in \eqref{hom to hom}.

When $f$ is the identity map of $X=Y$, one gets cohomology isomorphisms 
\[
\ul\HH^i(X) 
= \ext^i_{\sX}(\Hsch{\<X},\Hsch{\<X}\<)\iso
\ext^i_{\sX}(\Hsch{\<X},\Hsch{\<X}\<)= \HH^i(X) 
\] 
that are in fact identity maps.\va2

For proper $x$ there is a natural pairing on classical Hochschild homology, with $H\set\h^0(S,\OS)$:
\begin{multline*}
\h^{-i}(X\<, \Hsch{\sX}\<)\otimes_H\h^{i}(X\<, \Hsch{\sX}\<)\to \h^0(X\<, 
   \Hsch{\sX}\otimes_\sX \Hsch{\sX}\<)\\
\xto{\<\via\fp_{\sX}} \h^0(X\<, x^!\OS)\cong \h^0(S, x_*x^!\OS)\to \h^0(S,\OS)=H,
\end{multline*}
where the first map is the case $j=-i$, $k=i$ of the map
\[
\ext^j(\OX,\Hsch{\sX}\<)\otimes_H\ext^k(\OX,\Hsch{\sX}\<)\to \ext^{j+k}(\OX,  \Hsch{\sX}\otimes_\sX \Hsch{\sX}\<)
\qquad(j,k\in\ZZ)
\]
that takes $\alpha\<\otimes_{\<H}\<\<\beta$ to the $\D(X)$-map 
\[
\OX=\OX\otimes_\sX\OX\xto{\<\alpha\>\otimes\>\beta\>\>}\Hsch{\sX}[\>j\>]\otimes_\sX\Hsch{\sX}[k]
\cong\big(\Hsch{\sX}\otimes_\sX \Hsch{\sX}\big)[\>j+k\>].
\]

\begin{subcor} \label{Mukai?} If\/ $S=\spec H$ with\/ $H$ a self-injective\/ $($i.e., Gorenstein\/$)$ artinian ring, and\/ $x\colon X\to S$ is  proper and smooth, then the above pairing
is\/ \emph{non-singular,} that is, the associated\/ $H$-linear map is an isomorphism
\[
\h^{-i}(X\<, \Hsch{\sX}\<)\iso \Hom_H(\h^{i}(X\<, \Hsch{\sX}\<),H).
\]
\end{subcor}

\begin{proof} Theorem~\ref{dualityiso}, the assumption on $H$, and the 0-dimensionality of $S$ entail natural isomorphisms
\begin{align*}
\h^{-i}(X\<, \Hsch{\sX}\<)&\iso\h^{-i}(X\<, \R\sHom_\sX(\Hsch{\sX}, x^!\OS))\\
&\iso\h^{-i}(S, x_*\R\sHom_\sX(\Hsch{\sX}, x^!\OS))\\
&\iso\h^{-i}(S, \R\sHom_S(x_*\Hsch{\sX}, \OS))\\
&\iso\h^{-i} \R\Hom_S(x_*\Hsch{\sX}, \OS)\\ 
&\iso\h^{-i} \Hom_S(x_*\Hsch{\sX}, \OS)\\ 
&\iso\h^{-i} \Hom_H(\Gamma(S,x_*\Hsch{\sX}), H)\\ 
&\iso \Hom_H(\h^{i}\<\Gamma(S,x_*\Hsch{\sX}), H)
\iso \Hom_H(\h^{i} (X,\Hsch{\sX}), H).
\end{align*}
\end{proof}

\begin{cosa}\label{othermaps} (Unfinished business.) 

$\bullet$ See Remark~\ref{Hodge}.

$\bullet$ The duality isomorphism $\fd$ calls to mind other maps in the literature. For one, when 
$x\colon X\to S$ is essentially smooth there is an isomorphism, attributed to C\u{a}ld\u{a}raru,
$$
\delta_*\Hsch{\sX}\iso \delta_*\R\sHom_\sX(\Hsch{\sX},x^!\OS),
$$
described in \cite[p.\,648]{R}.  For another, there is an isomorphism first defined by Kashiwara
$$
\textup{td}\colon \Hsch{\sX}\iso \delta^!\delta_* x^!\OS,
$$
see \cite[p.\,122, (5.2.2)]{KS} (which with $\omega_\sX\set x^!\OS$ makes sense, as a map,  for any flat
$\SS$\kf-map $x$),
whence an isomorphism
$$
\delta_*\Hsch{\sX}\xto{\!\delta_*\textup{td}\,} \delta_*\delta^!\delta_* x^!\OS
\cong\delta_*\R\sHom_\sX(\Hsch{\sX},x^!\OS).
$$
These two isomorphisms have interesting connections to Todd  classes and Riemann-Roch theorems.
Are they the same? How do they relate to~$\delta_*\fd_\sX$? How is the isomorphism \eqref{homologies} related to those in \cite[\S\S4.2, 5]{c1}?\va2

$\bullet$ How does the pairing in~\ref{Mukai?} relate to the Mukai pairing of
\cite[\S5]{c1}?\va2

$\bullet$ One might ask whether the isomorphisms in~\ref{iso to dual} respect the orientations and 
the bivariant operations in $\ul{\CB\<}\>$ and\/ $\CB$. For this, one needs commutativity---which we haven't yet been able to prove or disprove---of the diagram 
\[
 \bpic[xscale=2.75,yscale=1.5]
  \node(11) at (1,-1){$\Hsch{\sX}$};
  \node(13) at (3,-1){$f^!\Hsch Y$};

  \node(21) at (1,-2){$\bD\sX\Hsch{\sX}$};
  \node(22) at (2,-2){$\bD\sX f^*\Hsch Y$};
  \node(23) at (3,-2){$f^!\bD Y\Hsch Y$};

  \draw[->] (11)--(13) node[above, midway, scale=.75]{$c^{}_{\<f}$};
  
  \draw[->] (21)--(22) node[below=1pt, midway, scale=.75]{$\<\bD\sX\< f^\sharp$};
  \draw[->] (22)--(23) node[above, midway, scale=.75]{$\Iso$};
  
  \draw[->] (11)--(21) node[left=1pt, midway, scale=.75]{$\fd_\sX$};
  \draw[->] (13)--(23) node[right=1pt, midway, scale=.75]{$f^!\fd_Y$};
  
 \epic
\]

\end{cosa}

\section{Fundamental class and base-change}

The fundamental class $\Dc_f=\Db_{\<\<f}\smallcirc\Da_{\<\<f}$ of a flat $\SS$\kf-map $f\colon X\to Y$,
\[
\Hh{\<\<X}f^*=\delta^*_{\<\<X}\delta^{}_{\<\<X\<*}f^*
\xrightarrow[\,\fundamentalclassa{\<f}\,]{}
 \Gamma^*  \Gamma_{\<\!*} f^!
\underset{\fundamentalclassb{\<\<f}}\iso
f^!{\delta^*_Y} \delta^{}_{Y\mkern-1.5mu*}=f^!\>\Hh{Y},
\]
is as in \S\ref{def-fc}. The next Theorem describes its behavior under flat base change.
There results a flat-base-change property for contravariant Gysin maps (Proposition~\ref{Gysin and flat base change}).\looseness=-1

\begin{thm}
\label{bch-fund-class}
For any oriented 
fiber square of flat\/ $\SS$\kf-maps
  \[
   \begin{tikzpicture}[yscale=.9]
      \draw[white] (0cm,0.5cm) -- +(0: \linewidth)
      node (21) [black, pos = 0.41] {$Y'$}
      node (22) [black, pos = 0.59] {$Y$};
      \draw[white] (0cm,2.65cm) -- +(0: \linewidth)
      node (11) [black, pos = 0.41] {$X'$}
      node (12) [black, pos = 0.59] {$X$};
      \draw [->] (11) -- (12) node[above=1pt, midway, sloped, scale=0.75]{$g'$};
      \draw [->] (21) -- (22) node[below=1pt, midway, sloped, scale=0.75]{$g$};
      \draw [->] (11) -- (21) node[left=1pt, midway, scale=0.75]{$f'$};
      \draw [->] (12) -- (22) node[right=1pt, midway, scale=0.75]{$f$};
   \end{tikzpicture}
  \]
the following diagram, with\/ $\mathsf B$ as in\/ \eqref{thetaB}$,$  commutes:
\begin{equation*}\label{bch-diag}
\CD
 \bpic[xscale=3, yscale=1.75]

   \node(11) at (1,-1) {$g'^*\Hh{\sX}f^*$};
   \node(12) at (2,-1) {$g'^*\<\<f^!\Hh{Y}$};
   \node(13) at (2,-2) {$f'{}^!g^*\Hh{Y}$};

   \node(21) at (1,-2) {$\Hh{\<\<X'}g'^*\<\<f^*$};
   \node(22) at (1,-3) {$\Hh{\<\<X'}\<f'^*\<g^*$};
   \node(23) at (2,-3) {$f'{}^!\>\Hh{Y'}g^*$};
   
   \draw[->] (11) -- (12) node[above=1pt, midway, scale=.75] {$g'^*\Dc^{}_{\<f}$};
   \draw[->] (12) -- (13) node[left=1pt, midway, scale=.75] {$\simeq$}
                                    node[right=1pt, midway, scale=.75] {$\bchadmirado{}$};

   \draw[double distance=2pt] (21) -- (22) node[left=1pt, midway, scale=.75] {$\via\ps^*$};
   \draw[->] (22) -- (23) node[below=1pt, midway, scale=.75] {$\Dc^{}_{\<f'}g^*$};

   \draw[->] (11) -- (21) node[left=1pt, midway, scale=.75] {$\eqref{bbbiup}$};

   \draw[->] (13) -- (23) node[right=1pt, midway, scale=.75] {$\eqref{bbbiup}$};

  \epic
 \endCD\tag{\ref{bch-fund-class}.1}
\]

\end{thm}

\begin{proof} 
Notation will be as in the following commutative diagram,
in which $\nu\set\delta_{\<f}$, $\nu'\set\delta_{\<f'}$, $\delta\set\delta_\sX=i\nu$ and $\delta'\set\delta_{\<\<X'}=i'\nu'$  are diagonal maps, $\Gamma=\Gam{f}$ (resp.~$\Gamma'=\Gam{f'}$) is the graph of $f$ (resp.~$f'$),  $i$ and $i'$ are the natural immersions, $t$ and $t'$ are the projections onto the first factor,
$p$, $p'$, $q$ and $q'$ are the projections onto the second factor, and $h$~ is the composite map
\[
X'\<\times_{Y'}\< X' \xto{\textup{natural}\,} X'\<\times_{Y}\< X' \xto{g'\times_{\<Y}\> g'} X\<\times_{Y}\< X.
\]

 \begin{center}
    \begin{tikzpicture}[xscale=1.1]
      \draw[white] (0cm,0.5cm) -- +(0: \linewidth)
      node (11) [black, pos = 0.2] {$Y'$}
      node (14) [black, pos = 0.52] {$Y$};
      
      \draw[white] (0cm,1.9cm) -- +(0: \linewidth)
      node (33) [black, pos = 0.4] {$X'$}
      node (36) [black, pos = 0.7] {$X$};
      
      \draw[white] (0cm,3.5cm) -- +(0: \linewidth)
      node (51) [black, pos = 0.2] {$X'\<\times \<Y'$}
      node (54) [black, pos = 0.52] {$X\<\times \<Y$};
      
      \draw[white] (0cm,4.9cm) -- +(0: \linewidth)
      node (73) [black, pos = 0.4] {$X'\<\times\< X'$}
      node (76) [black, pos = 0.7] {$X\<\times\< X$};
      
      \draw[white] (0cm,6.5cm) -- +(0: \linewidth)
      node (91) [black, pos = 0.2] {$X'$}
      node (94) [black, pos = 0.52] {$X$};
      
      \draw[white] (0cm,7.9cm) -- +(0: \linewidth)
      node (103) [black, pos = 0.4] {$X'\<\times_{Y'}\< X'$}
      node (106) [black, pos = 0.7] {$X\<\times_{Y}\< X$};
      
      \draw[white] (0cm,10.9cm) -- +(0: \linewidth)
      node (113) [black, pos = 0.4] {$X'$}
      node (116) [black, pos = 0.7] {$X$};
      
      \node (C) at (intersection of  73--33 and 51--54) { };
      \node (D) at (intersection of  33--36 and 54--14) { };
      \node (B) at (intersection of 103--33 and 91--94) { };
      \node (E) at (intersection of  54--94 and 73--76) { };
   
      \draw [->] (91) -- (51)   node[auto, swap,  midway, scale=0.75]{$\Gamma'$};
      \draw [->] (51) -- (11)   node[auto, swap,  midway, scale=0.75]{$q'$};
      \draw [->] (94) -- (54)   node[auto, near start, scale=0.75]{$\Gamma$};
      \draw [->] (54) -- (14)   node[auto, near start, scale=0.75]{$q$};
      \draw [->] (113) -- (103) node[auto, swap,  midway, scale=0.75]{$\nu\>':=\delta_{\<\<f'}$};
      \draw [-]  (103) -- (B)   node[auto, swap,  midway, scale=0.75]{ };
      \draw [->] (B) -- (73)    node[auto, swap,  midway, scale=0.75]{$i'$};
      \draw [-]  (73) -- (C)    node[auto, swap,  midway, scale=0.75]{ };
      \draw [->] (C) -- (33)    node[auto, swap,  midway, scale=0.75]{$p'$};
      \draw [->] (116) -- (106) node[right=1pt, midway, scale=0.75]{$\nu:=\delta_{\<\<f}$};
      \draw [->] (106) -- (76)  node[auto, midway, scale=0.75]{$i$};
      \draw [->] (76) -- (36)   node[auto, midway, scale=0.75]{$p$};
      \draw [->] (51) -- (54)   node[auto, midway, scale=0.75]{$g'\<\times\< g\mkern20mu$};
      \draw [-]  (73) -- (E)    node[auto, midway, scale=0.75]{$ $};
      \draw [->] (E) -- (76)    node[auto, midway, scale=0.75]{$g'\<\times g'$};
      \draw [->] (91) -- (94)   node[auto, midway, scale=0.75]{$g'$};
      \draw [->] (103) -- (106) node[auto, midway, scale=0.75]{$h$};
      \draw [->] (113) -- (116) node[above=1pt, midway, scale=0.75]{$g'$};
      \draw [-]  (33) -- (D)    node[auto, midway, scale=0.75]{ };
      \draw [->] (D) -- (36)    node[auto, midway, scale=0.75]{$g'$};
      \draw [->] (11) -- (14)   node[auto, swap, midway, scale=0.75]{$g$};
      \draw [<-] (51) -- (73)  node[auto,       midway, scale=0.75]{$\id_{\<\<X'}\<\<\times f'\mkern-15mu$};
      \draw [<-] (54) -- (76)  node[auto, swap, midway, scale=0.75]{$\mkern-8mu\id_\sX\!\times\< f$};
      \draw [<-] (11) -- (33)  node[auto,       midway, scale=0.75]{$f'$};
      \draw [<-] (14) -- (36)  node[auto, swap, midway, scale=0.75]{$f$};
      \draw [<-] (91) -- (103) node[auto,       midway, scale=0.75]{$t'$};
      \draw [->] (91) -- (73) node[above=-1pt,       midway, scale=0.75]{$\mkern15mu\delta'$};
      \draw [->] (94) -- (76) node[above=-1pt,       midway, scale=0.75]{$\mkern15mu\delta$};
      \draw [<-] (94) -- (103) node[above=-2.5pt,       midway, scale=0.75]{$\mkern27mu j$};
      \draw [<-] (94) -- (106) node[auto, swap, midway, scale=0.75]{$t$};
      \draw [<-] (54) -- (73) node[above=-2.5pt,       midway, scale=0.75]{$\mkern27mu k$};
      \draw [<-] (14) -- (33) node[above=-2.5pt,       midway, scale=0.75]{$\mkern27mu\ell$};
    \end{tikzpicture}
\end{center}

Diagram \eqref{bch-diag} (transposed) expands as follows, with $\phi$ as in \eqref{def-of-phi}:

\[
\def\1{$g'^*\<\delta^*\<\delta_{\<*}f^*$}
\def\2{$g'^*\Gamma^*\Gam{*}f^!$}
\def\3{$g'^*\<\<f^!\delta_Y^*\delta^{}_{Y\<*}$}
\def\4{$f'{}^!\<g^*\<\delta_Y^*\delta^{}_{Y\<*}$}
\def\5{${\delta'}^*\<\delta'_{\<*}\>g'{}^*\<\<f^* $}
\def\6{${\delta'}^*\<\delta'_{\<*}{\smash{f'}}^*\<\<g\<^*$}
\def\7{${\Gamma'}^*\Gam{*}'{f'}{}^!\<g^*$}
\def\8{$f'{}^!\delta_{Y'}^*\>\delta^{}_{Y'\<*}\>g^*$}
\def\9{$\Gamma'{}^*(g'\times g)^*\>\Gam{*}f^!$}
\def\ten{$\Gamma'{}^*\Gam{*}'\>g'{}^*\<\<f^!$}
 \bpic[xscale=5.1,yscale=1.9]
      \node (11) at (1,-1) {\1};
      \node (12) at (2,-1) {\5};
      \node (13) at (3,-1) {\6};
 
      \node (21) at (1,-2) {\2};
      \node (222) at (2,-2){\ten};
      \node (23) at (3,-2) {\7};

      \node (31) at (1,-3) {\3};
      \node (32) at (2,-3) {\4};
      \node (33) at (3,-3) {\8};

      \draw [->] (11) -- (12) node[above=1pt, midway, scale=0.75] {$\phi$};
      \draw [double distance=2pt] (12) -- (13) node[above=1pt, midway, scale=0.75] {$\via\ps^*$};

      \draw [->] (21) -- (222) node[above=1pt, midway, scale=0.75] {$\phi$};
      \draw [->] (222) -- (23) node[above=1pt,  midway, scale=0.75]{$\via \mathsf B$};

      \draw [->] (31) -- (32) node[below=1pt, midway, scale=0.75] {$\bchadmirado{ }$};
      \draw [->] (32) -- (33) node[below=1pt, midway, scale=0.75] {$\via\phi$};


      \draw [->] (11) -- (21) node[left=1pt, midway, scale=0.75] {$g'^*\fundamentalclassa{f}$};
      \draw [->] (21) -- (31) node[left=1pt, midway, scale=0.75] {$g'^*\fundamentalclassb{f}$};
 
      \draw [->] (13) -- (23) node[right=1pt, midway, scale=0.75] {$\fundamentalclassa{f'}$};
      \draw [->] (23) -- (33) node[right =1pt, midway, scale=0.75] {$\fundamentalclassb{f'}$};
      
      \node (label10) at (intersection of 11--23 and 13--21)
                             [black, scale=0.9] {$(\#_{\>\fundamentalclassa{}{}})$};
      \node (label11) at (intersection of 21--33 and 23--31)
                             [black, scale=0.9] {$(\#_{\>\fundamentalclassb{}{}})$};
 \epic
\]

We first prove  commutativity of subdiagram $(\#_{\>\fundamentalclassa{}{}})$,
expanded as follows, with  $s\set pi$, 
$s'\set p'i'$ (see \S\ref{def-fc}, and also, recall that both $(s\nu)^!=\id_\sX^!$ and~ 
$(s'\nu')^!=\id_{\<\<X'}^!$ are identity functors).
Each map in this diagram is induced by the natural transformation specified in its label.
The commutative
$\SS$\kf-square to which any  label $\bchadmirado{}$, $\bchadmirado{}^{-1}$ or $\theta^{-1}$ is associated is in each case easily verified to be an oriented fiber square with flat bottom arrow. In particular, one sees from the fiber square 
 \[
   \begin{tikzpicture}[xscale=1.1, yscale=.9]
      \draw[white] (0cm,0.5cm) -- +(0: \linewidth)
      node (21) [black, pos = 0.41] {$X'$}
      node (22) [black, pos = 0.59] {$X$};
      \draw[white] (0cm,2.65cm) -- +(0: \linewidth)
      node (11) [black, pos = 0.41] {$X'\<\<\times_{Y'}\<\<X'$}
      node (12) [black, pos = 0.59] {$X\<\<\times_Y\<\< X$};
      \draw [->] (11) -- (12) node[above=1pt, midway, sloped, scale=0.75]{$h$};
      \draw [->] (21) -- (22) node[below=1pt, midway, sloped, scale=0.75]{$g'$};
      \draw [->] (11) -- (21) node[left=1pt, midway, scale=0.75]{$s'$};
      \draw [->] (12) -- (22) node[right=1pt, midway, scale=0.75]{$s$};
   \end{tikzpicture}
  \]
that the map $h$ is flat.
 \[\mkern-20mu
    \begin{tikzpicture}[scale=.99]
      \draw[white] (0cm,13.5cm) -- +(0: \linewidth)
      node (21) [black, pos = 0.09, scale=.87] {$g'^*\<\delta^*\<\delta_{\<*} f^*$}
      node (22) [black, pos=.42 , scale=.87] {${\delta'}^*\<({g'}\<\<\times\< {g'})^{\<*} \delta_{\<*} f^*$}
      node (24) [black, pos=.627, scale=.87] {${\delta'}^*\<\delta'_{\<*}\>g'{}^*\<\<f^*$}
      node (27) [black, pos = 0.9, scale=.87] {${\delta'}^*\<\delta'_{\<*}{\smash{f'}}^*\<g^*$};
 
      \draw[white] (0cm,12.2cm) -- +(0: \linewidth)
      node (42) [black, pos=.275 , scale=.87] {$\mkern-12mu{\delta'}^*\<(g'\<\<\times\< g')^{\<*} {i}_*\>\nu_* (s\>\nu)^{\<!}\<f^* $}
      node (44) [black, pos=.505 , scale=.87] {${\mkern6mu\delta'}^*\<i'_*\>\nu'_{\<*}\>g'{}^*\<(\<s\>\nu)^{\<!}\<f^*$}
      node (46) [black, pos=.75, scale=.87] {$\,{\delta'}^*\<i'_*\>\nu'_{\<*}(s'\nu')^!g'{}^*\<\<f^*$};
      
      \draw[white] (0cm,10.6cm) -- +(0: \linewidth)
      node (451) [black, pos=.09, scale=.87] {$\;g'^*\<\delta^*i_*\>\nu_*(s\>\nu)^!\<f^*$}
      node (33) [black, pos=.505 , scale=.87] {${\delta'}^*\<i'_*h^{\<*}\nu_* (s\>\nu)^!\<f^*\;$}
      node  [black, pos = 0.625, scale=.87] {\raisebox{-10pt}{\circled3}}
      node (457) [black, pos = 0.9, scale=.87] {$\mkern15mu{\delta'}^*\<i'_*\>\nu'_{\<*} (s'\nu')^!\<{\smash{f'}}^*\<g^*$};
      
      \draw[white] (0cm,9cm) -- +(0: \linewidth)
      node (91) [black, pos=.09, scale=.87] {$ $}
      node (72) [black, pos=.275 , scale=.87] {${\delta'}^*\<(g'\times g')^{\<*}{i}_*
                                                   s^!\<\<f^*$}
      node (85) [black, pos=.51 , scale=.87] {${\delta'}^*\<i'_{\<*}h^{\<*}\<\<s^!\<\<f^*$}
      node (76) [black, pos=.75 , scale=.87] {${\delta'}^*i'_*s'{}^!g'{}^*\<\< f^*$}
      node (77) [black, pos = 0.9, scale=.87] {$ $};
      
      \draw[white] (0cm,7.7cm) -- +(0: \linewidth)
      node (81) [black, pos=.09, scale=.87] {$g'^*\<\delta^*i_*s^!\<\<f^*$}
      node (87) [black, pos = 0.9, scale=.87] {${\delta'}^*\<i'_*s'{}^!\<\<f'{}^*\<g^*$};
         
      \draw[white] (0cm,6.25cm) -- +(0: \linewidth)
      node (92) [black, pos=.275 , scale=.87] {${\delta'}^*\<(g'\<\<\times\< g')^{\<*}{i}_*t^*\<\<f^!$}
      node (a3) [black, pos=.51, scale=.87] {${\delta'}^*\<i'_*h^{\<*}t^*\<\<f^!$}
      node (94) [black, pos=.75 , scale=.87] {${\delta'}^*i'_*s'{}^!\ell^*$};
     
      \draw[white] (0cm,5.05cm) -- +(0: \linewidth)
      node (a1) [black, pos=.09, scale=.87] {$g'^*\<\delta^*i_*t^*\!f^!$}
      node[black, pos=.43, scale=.87]{\circled5$_1$}
      node (93) [black, pos=.63 , scale=.87] {${\delta'}^*i'_*j^*\<\<f^!$}
      node  [black, pos=.765, scale=.87] {{\circled4$_2$}};
     
      \draw[white] (0cm,4cm) -- +(0: \linewidth)
             node (a2) [black, pos=.275 , scale=.87] {${\delta'}^*\<(g'\<\<\times\< g')^{\<*}(\id_\sX\!\times\>f)^*\Gam{*}f^!$}      
             node (e3) [black, pos=.63, scale=.87] {${\delta'}^*\<i'_{\<*}\>t'{}^*\<g'{}^*\!f^!$}
             node (e7) [black, pos = 0.9  , scale=.87] {${\delta'}^*\<i'_*\>t'{}^*\<\<{f'}^!\<g^*$};

      \draw[white] (0cm,2.65cm) -- +(0: \linewidth)
      node (i1) [black, pos=.09, scale=.87] {$\ \ \;g'^*\<\delta^*\<(\id_\sX\!\times f)^{\<*}\Gam{*}f^!$}
      node (i2) [black, pos=.39, scale=.87] {${\delta'}^*\<\<k^*\Gam{*}f^!$}
      node[black, pos=.53, scale=.87]{\circled5$_2$};

      \draw[white] (0cm,1.5cm) -- +(0: \linewidth)
                  node (k1) [black, pos=.275 , scale=.87] {${\delta'}^*\<(\id_{\<\<X'}\!\times \>f'\>)^{\<*}(g'\<\<\times\< g)^{\<*}\Gam{*}f^!$}
        node (i5) [black, pos=.63 , scale=.87] {$\<{\delta'}^*\<(\id_{\<\<X'}\!\times \>f'\>)^{\mkern-1.5mu*}\Gam{*}'\>g'{}^*\!f^!$}
        node (i7) [black, pos = 0.9, scale=.87] {${\delta'}^*\<\<(\<\id_{\<\<X'}\!\times\< f')^{\<*}\Gam{*}'{f'}^!\mkern-1.75mu g^{\mkern-.5mu*}$};
        
      \draw[white] (0cm,.35cm) -- +(0: \linewidth)
      node (u1) [black, pos=.09, scale=.87] {$g'^*\Gamma^*\Gam{*}f^!$}
      node (o2) [black, pos=.275 , scale=.87] {${\Gamma'}^*\<(g'\<\<\times\< g)^{\<*}\Gam{*}f^!$}
      node (u3) [black, pos=.63 , scale=.87] {${\Gamma'}^*\Gam{*}'\>g'{}^*\!f^!$}
      node (u7) [black, pos = 0.9, scale=.87] {${\Gamma'}^*\Gam{*}'{f'}^!\<g^*$};

      \draw [double distance=2pt] (21) -- (451) node[left=.5pt, midway, scale=0.7]{$\ps_*$};
      \draw [->] (451) -- (81) node[left, midway, scale=0.7]{$\int_{\raisebox{-1pt}{$\sst\<\nu$}}^s$};
      \draw [->] (81) -- (a1) node[left, midway, scale=0.7]{$\bchadmirado{}^{-\<1}$};
      \draw [->] (a1) -- (i1) node[left, midway, scale=0.7]{$\bchasterisco{}^{-\<1}$};
      \draw [double distance=2pt] (i1) -- (u1) node[left=.5pt, midway, scale=0.7]{$\ps^*$};
       
      \draw [->] (92) -- (a2) node[left, midway, scale=0.7]{$\bchasterisco{}^{-\<1}$};
      \draw [double distance=2pt] (22) -- (42) node[below=1pt, midway, sloped, scale=0.7]{$\mkern15mu\ps_*$};
      \draw [->] (42) -- (72) node[left, midway, scale=0.7]{$\int_{\raisebox{-1pt}{$\sst\<\nu$}}^s$};
      \draw [->] (72) -- (92) node[left, midway, scale=0.7]{$\bchadmirado{}^{-\<1}$};
      
      \draw [double distance=2pt] (24) -- (44) node[below=1pt, midway, sloped, scale=0.7]{$\mkern15mu\ps_*$};
      \draw [->] (33) -- (85) node[left, midway, scale=0.7]{$\int_{\raisebox{-1pt}{$\sst\<\nu$}}^s$};
      \draw [->] (85) -- (a3) node[right=1pt, midway, scale=0.7]{$\bchadmirado{}^{-\<1}$};
      \draw [double distance=2pt](a2) -- (i2)  node[above=.5pt, midway, sloped, scale=0.7]{$\mkern15mu\ps^*$};
       \draw [double distance=2pt](k1) -- (i2)  node[below, midway, sloped, scale=0.7]
                   {$\mkern10mu\ps^*$};   
      \draw [->](5.3cm,2.9cm) -- (93) node[below=1pt, midway, scale=0.7]{$\theta$};
      \draw [->] (e3) -- (i5) node[right=1pt, midway, scale=0.7]{$\bchasterisco{}^{-\<1}$};
      \draw [double distance=2pt](i5) -- (u3) node[right=1pt, midway, scale=0.7]{$\ps^*$};
      
      \draw [->] (46) -- (76) node[right, midway, scale=0.7]{$\int_{\raisebox{-1pt}{$\sst\<\nu'_{\mathstrut}{}^{\mathstrut}$}}^{s'}$};
      \draw [double distance=2pt] (76) -- (94) node[right=1pt, midway, scale=0.7]{$\ps_*$};

      \draw [double distance=2pt] (27) -- (457) node[right=1pt, midway, scale=0.7]{$\ps_*$};
      \draw [->] (457) -- (87) node[right, midway, scale=0.7]{$\int_{\raisebox{-1pt}{$\sst\<\nu'_{\mathstrut}{}^{\mathstrut}$}}^{s'}$};
      \draw [->] (87) -- (e7) node[right=1pt, midway, scale=0.7]{$\bchadmirado{}^{-\<1}$};
      \draw [->] (e7) -- (i7) node[right=1pt, midway, scale=0.7]{$\bchasterisco{}^{-\<1}$};
      \draw [double distance=2pt] (i7) -- (u7) node[right=1pt, midway, scale=0.7]{$\ps^*$};
      
      \draw [double distance=2pt] (21) -- (22)  node[above=1pt, midway, scale=0.7]{$\ps^*$};
      \draw [->] (22) -- (24)  node[above=.5pt, midway, scale=0.7]{$\theta$}
                                           node[below=11.5pt, midway, scale=.85]{\kern-70pt\circled1};
      \draw [double distance=2pt] (24) -- (27)  node[above=1pt, midway, scale=0.7]{$\ps^*$};
      
      \draw [double distance=2pt] (44) -- (46)  node[auto,  midway, scale=0.7]{$\bchadmirado{}$}
                                                                       node[above=12pt, scale=.85]{\circled2\kern104pt};
      \draw [double distance=2pt] (85) -- (85)  node  [below=33pt, midway, scale=.87] {\kern50pt\circled4$_1$};
      \draw [->] (93) -- (94)  node[below=-2pt, midway, scale=0.7]{$\mkern20mu\bchadmirado{}$};
      \draw [->] (e3) -- (e7)  node[below, midway, scale=0.7]{$\bchadmirado{}$};
      \draw [->] (i5) -- (i7)  node[above, midway, scale=0.7]{$\bchadmirado{}$};
      \draw [->] (u3) -- (u7)  node[below, midway, scale=0.7]{$\bchadmirado{}$};
      
      \draw [double distance=2pt] (451) -- (42) node[below, midway, sloped, scale=0.7]
                   {$\mkern15mu\ps^*$};
      \draw [->] (42) -- (33) node[below=-1pt, midway, scale=0.7]{$\bchasterisco{}\mkern15mu$};
      \draw [->] (33) -- (44) node[right=1pt, midway, scale=0.7]{$\bchasterisco{}$};
      \draw [double distance=2pt](46) -- (457) node[below, midway, sloped, scale=0.7]{$\ps^*$};
      \draw [double distance=2pt](46) -- (24) node[below=1pt, midway, sloped, scale=0.7]{$\ps_*$};
      \draw [double distance=2pt] (81)  -- (72) node[above=.5pt, midway, sloped, scale=0.7]{$\mkern-5mu\ps^*$};
      \draw [double distance=2pt] (a3) -- (93) node[above=1pt, midway, sloped, scale=0.7]{$\mkern10mu\ps^*$};
            \draw [double distance=2pt] (e3) -- (93) node[right=2pt, midway,  scale=0.7]{\raisebox{10pt}{$\ps^*$}};
      \draw [->] (72)  -- (85) node[above, swap, midway, scale=0.7]{$\bchasterisco{}$};
      \draw [->] (85) -- (76)  node[above, midway, scale=0.7]{$\bchadmirado{}$};
      \draw [double distance=2pt] (76) -- (87) node[above=1pt, midway, sloped, scale=0.7]{$\mkern13mu\ps^*$};
      \draw [double distance=2pt] (87) -- (94) node[below=1pt, midway, sloped, scale=0.7]{$\ps^*$};
      \draw [double distance=2pt] (a1) -- (92) node[above=.5pt, midway, sloped, scale=0.7]{$\mkern-5mu\ps^*$};
      \draw [double distance=2pt] (i1) -- (a2) node[above=.5pt, midway, sloped, scale=0.7]{$\mkern-5mu\ps^*$};
      \draw [double distance=2pt] (k1) -- (a2) node[left, midway,  scale=0.7]{$\ps^*$};
      \draw [double distance=2pt] (k1) -- (o2) node[left, midway,  scale=0.7]{$\ps^*$};
      \draw [->] (92)  -- (a3) node[above, swap, midway, scale=0.7]{$\bchasterisco{}$};
      \draw [<-] (i5) -- (k1) node[below, midway, scale=0.7]{$\bchasterisco{}$};
      \draw [double distance=2pt]
                 (u1)  -- (o2) node[below=1pt, midway, scale=0.7]{$\ps^*$};
      \draw [->] (o2)  -- (u3) node[below, midway, scale=0.7]{$\bchasterisco{}$};

  \end{tikzpicture}
\]

In this diagram, commutativity of the unlabeled subdiagrams is easy to check, via (pseudo)functoriality of the maps involved. So it suffices to show commutativity of the labeled subdiagrams.

Commutativity of \circled1 results from vertical transitivity of $\theta$ \cite[3.7.2(ii)]{li};
and commutativity of \circled5$_1$ and \circled5$_2$ from horizontal transitivity 
\cite[3.7.2(iii)]{li}.

For \circled2, it's enough to note that the map $\bchadmirado{}\colon g'{}^*\<(\<s\>\nu)^{\<!}\to (s'\nu')^!g'{}^*$ is the identity, see  \cite[4.8.1(iii)]{li}.

Commutativity of \circled3 is given by Lemma~\ref{B-theta-int}, with  $(f\<,g,h,j,k,u,v)\set(\nu,s,h,\nu'\<, s'\<, g'\<,g'\>)$.

Commutativity of \circled4$_1$ and \circled 4$_2$ follows from horizontal transitivity of $\bchadmirado{}$ \cite[\S5.8.4]{AJL}.

Next, we prove commutativity of diagram~$(\#_{\>\fundamentalclassb{}{}})$. The morphisms used to define 
$\fundamentalclassb{\<\<f}$ and $\fundamentalclassb{\<\<f'}$ fit into the  commutative cube
  \begin{center}
    \begin{tikzpicture}[scale=.95]
      \draw[white] (0cm,5.2cm) -- +(0: \linewidth)
      node (13) [black, pos = 0.45] {$X'$}
      node (16) [black, pos = 0.73] {$X$};
      \draw[white] (0cm,3.7cm) -- +(0: \linewidth)
      node (21) [black, pos = 0.27] {$X'\<\times Y'$}
      node (24) [black, pos = 0.55] {$X\times Y$};
      \draw[white] (0cm,2cm) -- +(0: \linewidth)
      node (33) [black, pos = 0.45] {$Y'$}
      node (36) [black, pos = 0.73] {$Y$};
      \draw[white] (0cm,0.5cm) -- +(0: \linewidth)
      node (41) [black, pos = 0.27] {$Y'\<\times Y'$}
      node (44) [black, pos = 0.55] {$Y\times Y$};
      
      \node (C) at (intersection of 13--33 and 21--24) { };
      \node (D) at (intersection of 33--36 and 24--44) { };
      
      \draw [->] (21) -- (24) node[auto, midway, scale=0.75]{$g'\<\<\times\< g\mkern17mu$};
      \draw [->] (13) -- (16) node[auto, midway, scale=0.75]{$g'$};
      \draw [-]  (33) -- (D)  node[auto, midway, scale=0.75]{ };
      \draw [->] (D)  -- (36) node[auto, near start, swap, scale=0.75]{$g$};
      \draw [->] (41) -- (44) node[auto, swap, midway, scale=0.75]{$g\times g$};
      \draw [<-] (21) -- (13) node[auto, midway, scale=0.75]{$\Gamma'\mkern-10mu$};
      \draw [<-] (24) -- (16) node[auto, midway, scale=0.75]{$\Gamma\mkern-7mu$};
      \draw [<-] (41) -- (33) node[auto, near end, swap, scale=0.75]{$\mkern-7mu\delta'\<\set\<\delta_{Y'}$};
      \draw [<-] (44) -- (36) node[auto, swap, near end, scale=0.75]{$\mkern-8mu\delta\set\<\delta_{Y}$};
      \draw [<-] (41) -- (21) node[auto, midway, scale=0.75]{$f'\<\<\times\<\< \id_{Y'}$};
      \draw [-]  (C)  -- (13) node[auto, midway, scale=0.75]{ };
      \draw [<-] (33) -- (C)  node[auto, midway, scale=0.75]{$f'$};
midway      \draw [->] (24) -- (44) node[auto, midway, scale=0.75]{\raisebox{30pt}{$f\<\<\times\<\< \id_Y $}};
      \draw [->] (16) -- (36) node[auto, midway, scale=0.75]{$f$};
      
      \node at (intersection of 16--44 and 36--24)[scale=.8] {$\Du$}; 
      \node at (intersection of 16--21 and 13--24)[scale=.8] {$\Dv$}; 
      \node at (intersection of 13--41 and 33--21)[scale=.8] {$\Dv'$}; 
      \node at (intersection of 33--44 and 36--41)[scale=.8] {$\Du'$}; 

    \end{tikzpicture}
  \end{center}

\smallskip  
  
Diagram $(\#_{\>\fundamentalclassb{}{}})$ expands as follows,  where $\CO\set\OY$, $\CO^*\set g^*\OY=\CO_{\<Y'}$, $\mu$ is as in \eqref{def-of-mu}, and each arrow is labeled with the natural transformation that induces it.

  \begin{center}
    \begin{tikzpicture}[xscale=0.92, yscale=1.1]
      
      \draw[white] (0cm,6.25cm) -- +(0: \linewidth)
      node (41) [black, pos = 0.1] {$g'^*\Gamma^*\Gam{*}(f^!\CO\<\otimes\<\< f^*)$}
      node (42) [black, pos = 0.3] { }
      node (44) [black, pos = 0.5] {${\Gamma'}^*\Gam{*}'(g'{}^*\<\<f^!\CO\<\otimes\< 
      g'{}^*\<\<f^*)$}
      node (46) [black, pos = 0.7] { }
      node (47) [black, pos = 0.9] {${\Gamma'}^*\Gam{*}'({f'}{}^!\CO^*\!\otimes\<\< {\smash{f'}}^*{g}^*)$};
      
      \draw[white] (0cm,5cm) -- +(0: \linewidth)
      node (34) [black, pos = 0.3] {${\Gamma'}^*\Gam{*}'g'{}^*(f^!\CO\<\otimes\<\< f^*)$}
      node (54) [black, pos = 0.5] { }
      node (56) [black, pos = 0.7] {${\Gamma'}^*\Gam{*}'(g'{}^*\<\<f^!\CO\<\otimes\< {\smash{f'}}^*\<\<g^*)$}
      node (57) [black, pos = 0.9] { };
      
      \draw[white] (0cm,3.75cm) -- +(0: \linewidth)
      node (61) [black, pos = 0.1] {$g'^*(f^!\CO\<\otimes\< \Gamma^*\Gam{*}f^*)$}
      node (62) [black, pos = 0.3 ] { }
      node (64) [black, pos = 0.5 ] {$g'{}^*\<\<f^!\CO\<\otimes\< {\Gamma'}^*\Gam{*}'\>g'{}^*\<\<f^*$}
      node (66) [black, pos = 0.7] { }
      node (67) [black, pos = 0.9] {${f'}{}^!\CO^*\!\otimes\< {\Gamma'}^*\Gam{*}'{\smash{f'}}^*\<{g}^* $};
      
      \draw[white] (0cm,2.5cm) -- +(0: \linewidth)
      node (71) [black, pos = 0.1] { }
      node (72) [black, pos = 0.3] {$g'^*\<\<f^!\CO\<\otimes\< g'^*\Gamma^*\Gam{*}f^*$}
      node (74) [black, pos = 0.5] { }
      node (76) [black, pos = 0.7] {$g'{}^*\<\<f^!\CO\<\otimes\< {\Gamma'}^*\Gam{*}'{\smash{f'}}^*\<g^*$}
      node (77) [black, pos = 0.9] { };
            
      \draw[white] (0cm,1.25cm) -- +(0: \linewidth)
      node (91) [black, pos = 0.1] { }
      node (92) [black, pos = 0.3] {$g'^*\<\<f^!\CO\<\otimes\< g'^*\<\<f^*\delta^*\delta_*$}
      node (94) [black, pos = 0.5] { }
      node (96) [black, pos = 0.7] {$g'^*\<\<f^!\CO\<\otimes\<\< f'^*\delta'^*\delta'_*g^*$}
      node (97) [black, pos = 0.9] { };
      
      \draw[white] (0cm,0cm) -- +(0: \linewidth)
      node (v1) [black, pos = 0.1] {$g'^*(f^!\CO\<\otimes\<\< f^*\delta^*\delta_*)$}
      node (v2) [black, pos = 0.3] { }
      node (v4) [black, pos = 0.5] {$g'{}^*\<\<f^!\CO\<\otimes\<\< {\smash{f'}}^*\<g^*\delta^*\delta_*$}
      node (v6) [black, pos = 0.7] { }
      node (v7) [black, pos = 0.9] {$f'^!\CO^*\!\otimes\<\< f'^*\delta'^*\delta'_*g^*$};
      
      \node at (intersection of 41--64 and 34--61) [scale=0.85]{\raisebox{-60pt}{\circled1}};
      \node at (intersection of 72--96 and 76--92) [scale=0.85]{\circled{$2$}};
      
      \draw [->] (41) -- (34) node[above, midway, scale=0.7]{$\defnu{}$};
      
      \draw [->] (41) -- (61) node[left,  midway, scale=0.7]{$\mu$};
      \draw [->] (61) -- (v1) node[left,  midway, scale=0.7]{$\defnu{}{}^{-\<1}$};
      \draw [->] (72) -- (92) node[left,  midway, scale=0.7]{$\defnu{}{}^{-\<1}$};    
      \draw [->] (34) -- (44) node[above=-3pt,  midway,   scale=0.7]{\eqref{^* and tensor}\kern40pt}; 

      \draw [->] (44) -- (64) node[left,  midway, scale=0.7]{$\mu$}; 
      \draw [->] (56) -- (76) node[left,  midway, scale=0.7]{$\mu$};
      \draw [->] (76) -- (96) node[left,  midway, scale=0.7]{$\defnu{}{}^{-\<1}$}; 
      \draw [->] (47) -- (67) node[right, midway, scale=0.7]{$\mu$};
      \draw [->] (67) -- (v7) node[right, midway, scale=0.7]{$\defnu{}{}^{-\<1}$};

      \draw [-, double distance=2pt]
                 (44) -- (56) node[auto, midway, scale=0.7]{$\ps^*$};
      \draw [->] (56) -- (47) node[auto, midway, scale=0.7]{$\bchadmirado{}$};
      \draw [->] (61) -- (72) node[above=-3pt,  midway,   scale=0.7]{\kern40pt\eqref{^* and tensor}};
      \draw [->] (72) -- (64) node[auto, swap, midway, scale=0.7]{$\defnu{} $};
      \draw [->] (v1) -- (92) node[above=-4pt,  midway,   scale=0.7]{\eqref{^* and tensor}\kern45pt};
      \draw [-, double distance=2pt]
                 (92) -- (v4) node[auto, midway, scale=0.7]{$\ps^*$};
      \draw [->] (96) -- (v7) node[auto, midway, scale=0.7]{$\bchadmirado{}$};
      \draw [-, double distance=2pt]
                 (64) -- (76) node[auto, swap, midway, scale=0.7]{$\ps^*$};
      \draw [->] (76) -- (67) node[auto, midway, scale=0.7]{$\bchadmirado{}$};
      \draw [->] (v4) -- (96) node[auto, midway, scale=0.7]{$\defnu{}$};
     \end{tikzpicture}
  \end{center}

Subdiagram \circled1 commutes by Lemma~\ref{muynu} (with $u\set g'$, etc.).

Commutativity of \circled2 is given by
Lemma~\ref{ayuda}
applied to each of the two decompositions $\Du\Dv$ and $\Du'\Dv'$ of the  diagram
\[
 \bpic[xscale=6,yscale=1.6]
  \node(11) at (1,-1) {$X'$};
  \node(12) at (2,-1) {$Y$};

  \node(21) at (1,-2) {$X'\<\times Y'$};
  \node(22) at (2,-2) {$Y\times Y$};

  \draw[->] (11) -- (12) node[above=1pt,  midway,   scale=0.75]{$f\<g'=gf'$};
  \draw[->] (21) -- (22) node[below=1pt,  midway,   scale=0.75]{$(f\<\<\times\<\<\id_Y)(g'\!\times\<g)=(g\<\times\!g)(f'\!\times\<\<\id_{Y'})$};
  
  \draw[->] (11) -- (21) node[left=1pt,  midway,   scale=0.75]{$\Gamma'$};
  \draw[->] (12) -- (22) node[right=1pt,  midway,   scale=0.75]{$\delta$};

 \epic
\]

Commutativity of the remaining subdiagrams is clear.

Thus $(\#_{\>\fundamentalclassb{}{}})$ commutes, as well as $(\#_{\>\fundamentalclassa{}{}})$,
and the proof of Theorem~\ref{bch-fund-class} is~complete.
\end{proof}

For an $\SS$\kf-map $h\colon V\to W$, the Gysin map $\gyb{h}$ is as in \S\ref{Gysin}. If $h$ is proper, and $v\colon V\to S$, $w\colon W\to S$ are the structure maps, then for any $j\in\ZZ$,
the pushforward 
$$
\pf{h}\colon \HH_j(V|S)= \ext^{-j}_{\OV}\<(\Hsch{V},v^!\CO_S)
\to \ext^{-j}_{\OW}\<(\Hsch{W},w^!\CO_S)=\HH_j(W|S)
$$ 
is as in \S\ref{bch-compatible}(B): it takes $\beta\colon \Hsch{V}\to v^!\OS[-j\>]$ to the composite map
$$
\Hsch{W}\xrightarrow{h_\sharp\>}h_*\Hsch{V}\xto{\<h_*\beta\>}h_*v^!\OS[-j\>]\overset{\ps^!}{=\!=}
h_*h^!w^!\OS[-j\>]\xto{\smallint^{}_{\<\<h}\>}w^!\OS[-j\>]
$$
where $h^{}_{\<\sharp}$ is adjoint to $h^\sharp$ (see \S\ref{orientations}).

\begin{prop}\label{Gysin and flat base change}
For any oriented fiber square of flat\/ $\SS$\kf-maps 
\[
    \begin{tikzpicture}[yscale=.9]
      \draw[white] (0cm,0.5cm) -- +(0: \linewidth)
      node (E) [black, pos = 0.41] {$Y^\prime$}
      node (F) [black, pos = 0.59] {$Y$};
      \draw[white] (0cm,2.65cm) -- +(0: \linewidth)
      node (G) [black, pos = 0.41] {$X^\prime$}
      node (H) [black, pos = 0.59] {$X$};
      \draw [->] (G) -- (H) node[above, midway, sloped, scale=0.75]{$g'$};
      \draw [->] (E) -- (F) node[below=1pt, midway, sloped, scale=0.75]{$g$};
      \draw [->] (G) -- (E) node[left=1pt,  midway, scale=0.75]{$f'$};
      \draw [->] (H) -- (F) node[right, midway, scale=0.75]{$f$};
    \end{tikzpicture}
\]
with\/ $f$ $($hence $f'\>)$ proper, one has
\[
\gyb{g}\<\<\pf{f\<\<} = f_{\<\<\star}'\>g'{}^{\mathsf c}.
\]
\end{prop}

\begin{proof} Let $x\colon X\to S$ and $y\colon Y\to S$ be the structure maps.
By  definition of $(-)^{\<\mathsf c}$ and $(-)_\star$, the assertion is that for any 
$$
\alpha\colon \Hsch{\sX}\to x^!\OS[i\>]\overset{\ps^!}{=\!=}f^!y^!\OS[i\>]\qquad(i\in\ZZ),
$$
the outer border of the following diagram---where each arrow is labeled with the natural transformation that induces it---commutes.

\[
\def\1{$\Hsch{Y'}$}
\def\2{$f'_{\<\<*}\>\Hsch{\sX'}$}
\def\3{$f'_{\<\<*}g'{}^!\>\Hsch{\sX}$}
\def\4{$f'_{\<\<*}g'{}^!x^!\OS[i\>]$}
\def\5{$g^!\>\Hsch{Y}$}
\def\6{$f'_{\<\<*}g'{}^!\<\<f^*\>\Hsch{Y}$}
\def\7{$f'_{\<\<*}x'{}^!\OS[i\>]$}
\def\8{$g^!\<\<\fst f^*\>\Hsch{Y}$}
\def\9{$f'_{\<\<*}g'{}^!\<\<f^!y^!\OS[i\>]$}
\def\ten{$g^!\<\<\fst\Hsch{\sX}$}
\def\lvn{$f'_{\<\<*}g'{}^!\<\<f^*\<\<\fst f^!y^!\OS[i\>]$}
\def\twv{$f'_{\<\<*}f'{}^*\<g^!\<\<\fst f^!y^!\OS[i\>]$}
\def\thn{$g^!\<\<\fst f^!y^!\OS[i\>]$}
\def\frn{$g^!y^!\OS[i\>]$}
\def\ffn{$f'_{\<\<*}f'{}^!\<g^!y^!\OS[i\>]$}
\def\sxn{$f'_{\<\<*}f'{}^*\<g^!\<\<\fst \Hsch{\sX}$}
\def\svn{$f'_{\<\<*}g'{}^!\<\<f^*\<\<\fst \Hsch{\sX}$}
\def\egn{$f'_{\<\<*}f'{}^*\<g^!\<\<\fst f^*\>\Hsch{Y}$}
\def\ntn{$f'_{\<\<*}g'{}^!\<\<f^*\<\<\fst f^*\>\Hsch{Y}$}
\def\twy{$f'_{\<\<*}f'{}^*\>\Hsch{Y'}$}
\def\twn{$f'_{\<\<*}f'{}^*\<g^!\Hsch{Y}$}
 \bpic[xscale=3.4, yscale=1.4]
  \node(11) at (1,0){\1};
  \node(122) at (1.944,0){\twy};
  \node(12) at (3.03,0){\2};
  \node(13) at (4,0){\3};
  
  \node(21) at (1,-1){\5};
  \node(22) at (1.944,-1){\twn};
  \node(23) at (3.03,-1){\6};
  \node(24) at (4,-4){\9};
  
  \node(31) at (1,-3){\ten};
  \node(32) at (1,-2){\8};
  \node(33) at (3.03,-4){\lvn};
  \node(322) at (1.944,-2){\egn};
  
  \node(323) at (3.03,-2){\ntn};  
  \node(332) at (1.944,-3){\sxn};
  
  \node(333) at (3.03,-3){\svn};  
  \node(334) at (4,-3){\3};
  
  \node(41) at (1,-4){\thn};
  \node(42) at (1.944,-4){\twv};
  \node(43) at (1,-5){\frn};
  \node(44) at (4,-5){\ffn};

   \draw[->] (11)--(122) node[above, midway, scale=.75]{$\eta^{}_{\<\<f'}$};
   \draw[->] (122)--(12) node[above, midway, scale=.75]{$ f'{}^\sharp$};
   \draw[->] (12)--(13) node[above, midway, scale=.75]{$ \Dc_{g'}$};

  \draw[->] (21)--(22) node[below=1pt, midway, scale=.75]{$\eta^{}_{\<\<f'}$};
  \draw[->] (22)--(23) node[below, midway, scale=.75]{$ \bchadmirado{}$};

  \draw[->] (32)--(322) node[below=1pt, midway, scale=.75]{$\eta^{}_{\<\<f'}$};      
  \draw[->] (322)--(323) node[below, midway, scale=.75]{$ \bchadmirado{}$};  
  
  \draw[->] (31)--(332) node[below=1pt, midway, scale=.75]{$\eta^{}_{\<\<f'}$};
  \draw[->] (332)--(333) node[below, midway, scale=.75]{$ \bchadmirado{}$};
  \draw[->] (333)--(334) node[above, midway, scale=.75]{$ \epsilon^{}_{\<\<f}$};

  \draw[->] (41)--(42) node[above, midway, scale=.75]{$\eta^{}_{\<\<f'}$};  
  \draw[->] (42)--(33) node[above=-1pt, midway, scale=.75]{$ \bchadmirado{}$};
  \draw[->] (33)--(24) node[above, midway, scale=.75]{$ \epsilon^{}_{\<\<f}$};

  \draw[<-] (43)--(44) node[below=1pt, midway, scale=.75]{$\smallint^{}_{\<\<f'}$};
 
  \draw[->] (11)--(21) node[left, midway, scale=.75]{$\Dc_g$};
  \draw[->] (21)--(32) node[left, midway, scale=.75]{$ \eta^{}_{\<\<f}$};
  \draw[->] (31)--(41) node[left, midway, scale=.75]{$\alpha$};
  \draw[<-] (31)--(32) node[left, midway, scale=.75]{$ f^\sharp$};    
  \draw[->] (41)--(43) node[left, midway, scale=.75]{$\<\smallint^{}_{\<\<f}$};
 
  \draw[->] (122)--(22) node[left, midway, scale=.75]{$\Dc_g$};
  \draw[->] (22)--(322) node[left, midway, scale=.75]{$ \eta^{}_{\<\<f}$};
  \draw[->] (322)--(332) node[left, midway, scale=.75]{$ f^\sharp$};
  \draw[->] (332)--(42) node[left, midway, scale=.75]{$\alpha$};

  \draw[->] (323)--(333) node[right, midway, scale=.75]{$ f^\sharp$};
  \draw[->] (3.05,-1.73)--(3.05, -1.27) node[right, midway, scale=.75]{$  \epsilon^{}_{\<\<f}$};
  \draw[<-] (3.01,-1.73)--(3.01, -1.27) node[left, midway, scale=.75]{$  \eta^{}_{\<\<f}$};
  \draw[->] (333)--(33) node[right, midway, scale=.75]{$\alpha$};
  
  \draw[-, double distance=2pt] (13)--(334) ;
  \draw[->] (334)--(24) node[right=1pt, midway, scale=.75]{$\alpha$};
  \draw[-,double distance=2pt] (24)--(44) node[right=1pt, midway, scale=.75]{$\ps^!$};

  \draw[->] (23)--(13) node[below, midway, scale=.75]{$f^\sharp$};

  \node at (2.475,-.55){\circled1};
  \node at (2.475,-4.5){\circled2};

 \epic
\]

Since commutativity of the unlabeled subdiagrams is clear, and
$\epsilon^{}_{\<\<f}\smallcirc\eta^{}_{\<\<f}$ is the identity map, it's enough to prove 
commutativity of subdiagrams \circled1 and~\circled2.\looseness=-1

Commutativity of \circled1 is given
by application of the commutative func\-torial diagram in Theorem~\ref{bch-fund-class} to $\OY$, 
after transposition of the fiber square in that theorem (i.e., make the interchange 
$(f,f'\<,X)\leftrightarrow (g,g'\<,Y')$.)

As for \circled2, we can replace $f^!$ by $f_{\<\<\upl}^!$, and similarly for $f'\<$, $g$ and $g'\<$, interpreting $\int$ as the counit map given by \ref{! and otimes}(i). This is because by definition the isomorphism \eqref{f! and f!+} corresponds via \ref{! and otimes}(i) to $\smallint_{\!f}$, and because that isomorphism is pseudofunctorial (last paragraph in \S\ref{! and otimes}) and compatible with the base\kf-change map $\bchadmirado{}$ (see \cite[Exercise 4.9.3(c)]{li}).  We will now show that the resulting diagram commutes, \emph{even when $g$ and $g'$ are not flat}. 

By \cite[2.8.1 and Theorem 4.1]{Nk}, there exists a fiber square diagram $\bar \Dd\smallcirc\Dd$,
\[
    \bpic[xscale=2, yscale=1.5]
      \node (D) at (1,-1){$X'$};
      \node (E) at (2,-1) {$\>\overline{\<X\<}\>'$};
      \node (F) at (3,-1) {$X$};
 
      \node (G) at (1,-2){$Y'$};
      \node (H) at (2,-2) {$\overline{Y\<}\>'$};
      \node (K) at (3,-2) {$Y$};

      \draw [->] (D) -- (E) node[above, midway,  scale=0.75]{$v$};
      \draw [->] (E) -- (F) node[above, midway,  scale=0.75]{$\bar g'$};
  
      \draw [->] (G) -- (H) node[below=1pt, midway,  scale=0.75]{$u$};
      \draw [->] (H) -- (K) node[below, midway,  scale=0.75]{$\bar g$};
    
      \draw [->] (D) -- (G) node[left=1pt,  midway, scale=0.75]{$f'$};
      \draw [->] (E) -- (H) node[right, midway, scale=0.75]{$\bar f$};
      \draw [->] (F) -- (K) node[right, midway, scale=0.75]{$f$};
      
      \node at (1.5, -1.5) [scale=.8]{$\Dd$};
      \node at (2.5, -1.5) [scale=.8]{$\bar\Dd$};
    \epic
\]
where $u$ (hence $v$) is a localizing immersion, $\bar g$ (hence $\bar g')$ is  proper, 
$g=\bar gu$ and $g'=\bar g'v$, cf.~\cite[\S5.8.2]{AJL}. Among other things, \cite[5.3]{Nk}  gives 
$u^!_\upl=u^*\<$, $v^!_\upl=v^*\<$, and 
$\bchadmirado{\Dd}=\ps^!\colon v^*\<\bar{\<f}_{\<\<\upl}^!\iso 
f_{\<\<\upl}'{}^{\<!} u^{\<*}$. So
subdiagram \circled2, without~ $y^!_\upl\OS[i\>]$,  expands as follows, with $\phi\colon \bar f^{}_{\<\<*}\bar g'_\upl{}^{\!\<!}\iso \bar g^!_\upl\fst$ as in \cite[3.10.4]{li} (see~\ref{! and otimes}(i)):

\[
\def\1{$u^{\<*}\<\bar g^{\>\>!}_\upl\fst f^!_{\<\<\upl}$}
\def\2{$f'_{\<\<*}f'{}^*u^{\<*}\<\bar g^{\>\>!}_\upl\fst f^!_{\<\<\upl}$}
\def\3{$f'_{\<\<*}v^*\<\bar g'_\upl{}^{\!\<!}\<\<f^*\!\fst f^!_{\<\<\upl}$}
\def\4{$f'_{\<\<*}f'{}^*u^{\<*}\!\bar f^{}_{\<\<*}\>\bar g'_\upl{}^{\!\<\<!}\<f^!_{\<\<\upl}$}
\def\5{$f'_{\<\<*}v^*\<\<\bar f{}^*\<\<\bar f^{}_{\<\<*}\>\bar g'_\upl{}^{\!\<!}\<\<f^!_{\<\<\upl}$}
\def\6{$u^{\<*}\!\bar f^{}_{\<\<*}\bar g'_\upl{}^{\!\<!}\<\<f^!_{\<\<\upl}$}
\def\7{$f'_{\<\<*}v^*\<\bar g'_\upl{}^{\!\<!}\<\<f^!_{\<\<\upl}$}
\def\8{$u^{\<*}\!\bar f^{}_{\<\<*}{\>\bar{\<f}{}^!_{\!\<\upl}}\<\bar g^{\>\>!}_\upl$}
\def\9{$f'_{\<\<*}v^*\<\<{\>\bar{\<f}{}^!_{\!\<\upl}}\<\bar g^{\>\>!}_\upl$}
\def\ten{$u^{\<*}\<\bar g^{\>\>!}_\upl$}
\def\lvn{$f'_{\<\<*}f_{\<\<\upl}'{}^{\<!} u^{\<*}\<\bar g^{\>\>!}_\upl$}
\def\twv{$f'_{\<\<*}v^*\<\<\bar f{}^*\bar g^{\>\>!}_\upl\fst f^!_{\<\<\upl}$}
 \bpic[xscale=4.5, yscale=1.5]
  \node(11) at (1,-1){\1};
  \node(12) at (1.9,-1){\2};
  \node(13) at (3,-1){\3};
 
  \node(21) at (1.9,-2){\4};
  \node(22) at (2.6,-2){\twv};
  
  \node(32) at (1.9,-3){\5};
  \node(33) at (3,-3){\7};

  \node(41) at (1.3,-4){\6};
  \node(42) at (1.9,-4){\8};
  \node(43) at (3,-4){\9};
  
  \node(51) at (1,-5){\ten};
  \node(52) at (1.9,-5){\ten};
  \node(53) at (3,-5){\lvn};

   \draw[->] (11)--(12) node[above, midway, scale=.75]{$\eta^{}_{\<\<f'}$};
   \draw[->] (12)--(13) node[above, midway, scale=.75]{$ \bchadmirado{\mkern.5mu\bar\Dd\<\ssscirc\mkern-.5mu\Dd}$};

  \draw[->] (32)--(33) node[above, midway, scale=.75]{$\epsilon{}_{\<\<\bar f}$};      

   \draw[-,double distance=2pt] (41)--(42) node[below, midway, scale=.75]{$\ps^!$};
   \draw[->] (42)--(43) node[above, midway, scale=.75]{$\theta$};
  
   \draw[-,double distance=2pt] (51)--(52) ; 
   \draw[<-] (52)--(53) node[below=1pt, midway, scale=.75]{$\smallint_{\!f'}$};
 
  \draw[->] (11)--(51) node[left, midway, scale=.75]{$\smallint^{}_{\!f}$};
  
  \draw[->] (12)--(21) node[right, midway, scale=.75]{$\phi^{-\<1}$}; 
  \draw[-,double distance=2pt] (21)--(32) node[right=1pt, midway, scale=.75]{$\ps^*$}; 
  \draw[<-] (21)--(41) node[left, midway, scale=.75]{$ \eta^{}_{\<\<f'}\mkern5mu$};
  \draw[->] (42)--(52) node[right, midway, scale=.75]{$\smallint^{}_{\!\bar f}$}; 
 
  \draw[->] (13)--(33) node[right, midway, scale=.75]{$\epsilon^{}_{\<\<f}$};
  \draw[-,double distance=2pt] (33)--(43) node[right=1pt, midway, scale=.75]{$\ps^!$};
  \draw[->] (43)--(53) node[right=1pt, midway, scale=.75]{$\bchadmirado{\Dd}$};
  
  \draw[->] (11)--(41) node[above=-4pt, midway, scale=.75]{$\mkern45mu\phi^{-\<1}$}; 
  \draw[->] (41)--(33) node[above, midway, scale=.75]{$\theta$}; 
  \draw[->] (22)--(13) node[below=-3.5pt, midway, scale=.75]{$\mkern40mu\bchadmirado{\bar\Dd}$};
  \draw[-,double distance=2pt] (12)--(22) node[above=-2pt, midway, scale=.75]{$\mkern45mu\ps^*$}; 
  \draw[->] (22)--(32) node[below=-3.5pt, midway, scale=.75]{$\mkern40mu\phi^{-\<1}$};

  \node at (2.55,-1.45){\circled3};
  \node at (2.55,-2.45){\circled4};
  \node at (1.65,-3.5){\circled5};
  \node at (1.45,-4.55){\circled6};
  \node at (2.475,-4.55){\circled7};

 \epic
\]
Using \cite[p.\,208, Theorem 4.8.3(ii) and Remark 4.8.5.2]{li}\va{.6}
as in \ref{(B)}, with the replacement $(f,g,u,v)\mapsto(u,v,\bar f, f')$, one gets $\bchadmirado{\Dd}=\ps^*\colon f'{}^*u^*\iso v^*\<\<\bar f^*\<$.
Consequently, commutativity of \circled3 results from horizontal transitivity of $\bchadmirado{}$.

Commutativity of \circled4\va{.6} is given by \cite[3.10.4(b)]{li}, with the replacement $(f,g,u,v)\mapsto(\bar g,\bar g'\<,f,\bar f\>)$.

Commutativity of \circled5 is given by \cite[3.7.2(i)(c)]{li}.

 From\va{-1} the adjunction 
$\bar f^{}_{\<\<*}\<\dashv\bar{\<f}{}^!_{\!\<\upl}$ in~\ref{! and otimes}(i), with unit
$\varpi^{}_{\!f}$, one deduces that commutativity of \circled6 results from 
that of \circled9 below, with $\tilde\phi$~the adjoint of $\phi$.
\[
\def\1{$\bar g'_\upl{}^{\!\<!}\<\<f^!_{\<\<\upl}\fst f^!_{\<\<\upl}$}
\def\2{$\bar g'_\upl{}^{\!\<!}\<\<f^!_{\<\<\upl}$}
\def\3{$\bar{\<f}{}^!_{\!\<\upl}\<\bar g^{\>\>!}_\upl\fst f^!_{\<\<\upl}$}
\def\4{$\bar{\<f}{}^!_{\!\<\upl}\<\bar g^{\>\>!}_\upl$}
 \bpic[xscale=4.5, yscale=1.5]
  \node(11) at (1,-1){\1};
  \node(12) at (2,-1){\2};
   
  \node(21) at (1,-2){\3};
  \node(22) at (2,-2){\4};
  
  \draw[<-] (1.185, -.96)--(1.89,-.96) node[above=-.5pt, midway, scale=.75]{$\varpi^{}_{\!f}$}; 
  \draw[->] (1.185, -1.04)--(1.89,-1.04) node[below=-.5pt, midway, scale=.75]{$\smallint^{}_{\!f}$}; 
  
  \draw[->] (21)--(22) node[below, midway, scale=.75]{$\smallint^{}_{\!f}$};

  \draw[-,double distance=2pt] (11)--(21) node[left=1pt, midway, scale=.75]{$\ps_\upl^!$};
  \draw[-,double distance=2pt] (12)--(22) node[right=1pt, midway, scale=.75]{$\ps_\upl^!$};
  
  \draw[->] (12)--(21) node[below=-3.5pt, midway, scale=.75]{$\mkern40mu\tilde\phi$}; 

  \node at (1.25,-1.45){\circled8};
  \node at (1.8,-1.6){\circled9};
 \epic
\]
 Subdiagram \circled8
commutes,  by \cite[3.10.4(c)]{li}  applied to the above diagram~$\bar\Dd$; and since
$\int^{}_{\!f}\<\<\smallcirc \varpi^{}_{\!f}$ is the identity map, it follows that \circled9, hence \circled6, commutes.

Commutativity of \circled7 is given by the definition of $\bchadmirado{\Dd}$ \cite[(5.7.2), (5.8.5)]{AJL}.

Thus \circled2 commutes, and the proof is complete.
\end{proof}

\section{Proof of transitivity}\label{provetran}
\begin{cosa}
Referring to the statement of transitivity, Theorem~\ref{trans fc}, let $\Gam u$, $\Gam v$ and~
 $\Gam {v u}$ be the graphs of $u$, $v$ and $v u$ respectively. According to (\ref{def-f-c}) we~can expand the diagram in the statement  as follows:\va3
\begin{equation}
\label{sostenidos}
 \begin{tikzpicture}
     \draw[white] (0cm,2cm) -- +(0: \linewidth)
      node (11) [black, pos = 0.1] {$\delta^*_{\<\<X}\delta^{}_{\<\<X\<*}u^*v^*$}
      node (12) [black, pos = 0.4] {$\Gams u\Gam {u*} u^!v^*$}
      node (13) [black, pos = 0.7] {$u^!\delta^*_Y \delta^{}_{Y\mkern-1.5mu*}v^*$}
      node (14) [black, pos = 0.9] {$u^!v^!\delta^*_{\<\<Z} \delta^{}_{\<\<Z*}$};
      \draw[white] (0cm,0.5cm) -- +(0: \linewidth)
      node (21) [black, pos = 0.1] {$\delta_{\<\<X}^*{\delta_{\<\<X}}_* (v  u)^*$}
      node (22) [black, pos = 0.4] {$\Gams{v  u}\Gam{v  u*}(v  u)^!$}
      node (24) [black, pos = 0.9] {$(v  u)^!\delta^*_{\<\<Z} \delta^{}_{\<\<Z*}$};
      \draw [double distance=2pt]
                 (11) -- (21) node[left, midway, scale=0.75]{$\via \ps^*$};
      \draw [->] (12) -- (22) node[left, midway, scale=0.75]{$?$};
      \draw [double distance=2pt]
                 (14) -- (24) node[right, midway, scale=0.75]{$\via\,\ps^!$};
      \draw [->] (11) -- (12) node[above, midway, scale=0.75]{$\fundamentalclassa{u}$};
      \draw [->] (12) -- (13) node[above, midway, scale=0.75]{$\fundamentalclassb{u}$};
      \draw [->] (13) -- (14) node[above, midway, scale=0.75]{$u^!\Dc_v$};
      \draw [->] (21) -- (22) node[below, midway, scale=0.75]{$\fundamentalclassa{v u}$};
      \draw [->] (22) -- (24) node[below, midway, scale=0.75]{$\fundamentalclassb{v u}$};
      \node (C) at (intersection of 11--22 and 12--21) [scale=0.75] {($\#$)};
      \node (D) at (intersection of 14--22 and 12--24) [scale=0.75] {($\# \#$)};
\end{tikzpicture}
\end{equation}
where the map labeled $?$ is defined just below. It suffices then to show that the two 
subdiagrams $(\#)$ and $(\# \#)$ are commutative.
\end{cosa}

To define the map ? in \eqref{sostenidos},  consider the  diagram of fiber squares
\begin{equation}\label{pasoI}
\CD
X@>\delta_u>> X\<\times_YX @>j>> X\<\times_ZX @>l>> X\times X @>p^{}_{\sX}>> X \\
@. @Vs^{}_{\<1}VV @V\id\<\times_{\<Z}\>\>u VV @V\id\<\times \>u VV @VV u V \\
@. X @>g>> X\<\times_Z Y @>k>> X\times Y @> p^{}_{Y}>> Y \\
@.@. @VrVV @V\id\<\times \>v VV @VVvV\\
@.@. X @>>\lift1.1,\Gam{vu},> X\times Z @>>\lift1,p^{}_{\<Z},> Z
\endCD
\end{equation}
where $j$, $k$ and $l$ are the natural closed immersions, $s^{}_{\<1}$ and $r$ are the projections onto the first factor, $p^{}_{\sX}$, $p^{}_{\>Y}$ and $p^{}_{\<Z}$ are the projections onto the second factor, and $g$ is the unique closed immersion such that $k\smallcirc g=\Gam{u}$. 

The subdiagram 
$\De$ formed by the bottom two rows is an instance of the diagram~\eqref{lambdaf2},  and so has associated to it the map 
\[
\lambda_{\De} \colon \Gam {u*} u^!v^* \to (\id\times v)^* {\Gam{v u*}}(v u)^!,
\]
from which we get the map ? in (\ref{sostenidos}) as the composition

\begin{align*}
\Gams u \Gam{u*} u^!v^*\xto{\<\<\Gams u\lambda_{\De}\>\>} \Gams u (\id\times v)^{\<*} \Gam{v u*} (v  u)^! 
&\overset{\ps^*}{=\!=}g^*\<k^*(\id\times v)^{\<*} \Gam{v u*} (v  u)^! \\
&\overset{\ps^*}{=\!=}g^*r^*\Gams{v u}\Gam{v u*} (v  u)^! 
\overset{\ps^*}{=\!=}\Gams{v u}\Gam{v u*} (v  u)^! .
\end{align*}

\begin{cosa}[\emph{Step I\kern2pt}]
For showing that $(\#)$ commutes 
consider more generally a diagram of fiber squares in the category $\SS$
\begin{equation}\label{pasoIgeneral}
\CD
\bullet@>16>> \bullet@>0>> \bullet @>2>> \bullet@>3>> \bullet \\
@. @V15VV @V6VV @V7VV @VV8 V \\
@.\bullet @>14>> \bullet@>4>>\bullet @> 5>> \bullet \\
@.@. @V9VV @V10 VV @VV11V\\
@.@. \bullet @>>\lift1.1,12,> \bullet @>>\lift1.1,13,> \bullet
\endCD
\end{equation}
where $14$ (hence $0$) and $16$ are proper,  8 and 11 are flat (whence so are 7, 6, 15, 10 and 9),
and $3\smallcirc2\smallcirc0\smallcirc 16$, $5\smallcirc4\smallcirc14$ and $13\smallcirc12$ are perfect (whence so are $3\smallcirc2\smallcirc0$, $5\smallcirc 4$ and $3\smallcirc2$, see \cite[p.\,245, Cor.\,3.5.2]{Il}).

\pagebreak[3]

From this we extract the following five subdiagrams, all of which satisfy the conditions imposed on the diagram~$\Dd$ in \eqref{lambdaf2}, and thus have associated $\lambda$ maps:

\vfill
\begin{equation*}
\CD
\bullet@>0\smallcirc\<\<16>> \bullet@>2>> \bullet @>3>> \bullet \\
@. @V6VV @V7VV @VV8V \\
@.\bullet @>>4> \bullet@>>5>\bullet  
\endCD\tag{$\Dd^+$}
\end{equation*}

\vfill
\begin{equation*}
\CD
\bullet@>0\smallcirc\<\<16>> \bullet@>2>> \bullet @>3>> \bullet \\
@. @V9\smallcirc6VV @V10\smallcirc7VV @VV11\smallcirc8V \\
@.\bullet @>>12> \bullet@>>13>\bullet  
\endCD\tag{$\Dd''$}
\end{equation*}

\vfill
\begin{equation*}
\CD
\bullet@>16>> \bullet@>2\smallcirc0>> \bullet @>3>> \bullet \\
@. @V15VV @V7VV @VV8V \\
@.\bullet @>>4\smallcirc\<\<14> \bullet@>>5>\bullet    
\endCD\tag{$\Dd'$}
\end{equation*}

\vfill
\begin{equation*}
\CD
\bullet@>14>> \bullet@>4>> \bullet @>5>> \bullet \\
@. @V9VV @V10VV @VV11V \\
@.\bullet @>>12> \bullet@>>13>\bullet  
\endCD\tag{$\De$}
\end{equation*}

\vfill
\begin{equation*}
\CD
\bullet@= \bullet@>4>> \bullet @>5>> \bullet  \\
@. @V9VV @V10VV @VV11V \\
@.\bullet @>>12> \bullet@>>13>\bullet  
\endCD\tag{$\De^-$}
\end{equation*}

\vfill
It is straightforward, if demanding of patience, to
verify  that commutativity of~$(\#)$ in (\ref{sostenidos}) is obtained, upon specialization of~ 
\eqref{pasoIgeneral} to~ \eqref{pasoI}, 
by application of the functor~$(2\smallcirc 0\smallcirc\<\< 16)^*$ to the diagram in the next lemma.

\end{cosa}
  
\begin{sublem}
\label{lemstepI}
The following diagram of\/ $\Dqc$-valued functors  commutes. 
\begin{equation*}
  \mkern-8mu
  \begin{tikzpicture}[xscale=1.087,yscale=1.35] 
         
     \draw[white] (0cm,9.6cm) -- +(0: \linewidth)
      node (11) [black, pos = 0.18 ][scale=.99] {$(2\smallcirc 0\smallcirc \<\<16)_*(3\smallcirc 2\smallcirc 0\smallcirc\<\<16)^!8^* 11^{\<*}$}
      node (13) [black, pos = 0.785][scale=.99] {$\!7^*\<(\<4\smallcirc\<\<14)_{\<*}
      (5\smallcirc\< 4\smallcirc\!14)\<^!11^{\<*}$};
      
      \draw[white] (0cm,8.2cm) -- +(0: \linewidth)
      node (23) [black, pos = 0.58][scale=.99] {$\,7^*\<4_*\< 14_*\<(5\smallcirc 4\smallcirc\<\<14)^!11^{\<*}$};

      \draw[white] (0cm,6.8cm) -- +(0: \linewidth)
       node (2) [black, pos = 0.35] {\circled{$2$}}
       node (C) [black, pos = 0.58][scale=.99]{$7^*4_*(5\smallcirc 4)^!11^{\<*}$};

      \draw[white] (0cm,5.4cm) -- +(0: \linewidth)
      node (31) [black, pos = 0.18][scale=.99] {$\,(2\smallcirc 0\smallcirc\<\<16)_*(3\smallcirc 2\smallcirc 0\smallcirc\<\<16)^!(11\smallcirc 8)^*$}
      node (32) [black, pos = 0.525][scale=.99] {$(10\smallcirc\< 7)\<^*\<12_*\<(\<13\smallcirc\!12\<)\<^!$}
      node (33) [black, pos = 0.785][scale=.99] {$7^*\<10^*\<12_*\<(\<13\smallcirc\!12)\<^!$};
           
      \draw [->] (11) -- (13) node[above, midway, scale=0.75]{$\lambda_{\Dd'}$};
      \draw [->] (13) -- (33) node[right, midway, scale=0.75]{$7^*\<\lambda_{\De}$}
           node[left, midway]{\circled3$\mkern32mu$};
      \draw [double distance=2pt]
                 (11) -- (31) node[left=1pt, midway, scale=0.75]{$\via\,\ps^*$};
      \draw [double distance=2pt]
                 (33) -- (32) node[below=1pt, midway, scale=0.75]{$\>\ps^*$};
      \draw [double distance=2pt]
                 (13) -- (23) node[left, midway]{$\circled1\mkern58mu\lift1.3,7^*\!\ps_*\,\,\,\,,$};
      \draw [->] (31)  --  (32) node[below, midway, scale=0.75]{$\lambda_{\Dd''}$};
      \draw [->] (11)  --   (C) node[left, midway, scale=0.75]{$\lambda_{\Dd^+}\,$};
      \draw [->] (23) --   (C) node[right=1pt, midway, scale=0.75]{$\!\via
      \smallint_{\mspace{-2mu}\lift.65,14,}^{5\>{\lift.75,\halfsize{$\circ$},} \>4}$};
      \draw [->] (C)   --  (33) node[left=1pt, midway, scale=0.75]{$7^*\<\lambda_{\De^-}$};

    \end{tikzpicture}
\end{equation*}
\end{sublem}

\begin{proof}
Subdiagram \circled{${1}$} without $11^{\<*}$ expands, by the definition of $\lambda$, to
\begin{equation*}\mkern-3mu
    \begin{tikzpicture}[xscale=.99,yscale=1.35]
    
      \draw[white] (0cm,7.8cm) -- +(0: \linewidth)
      node (02) [black, pos = 0.13 , scale=0.8] {$\mkern-6mu 2_*0_* 16_*(3\smallcirc 2\smallcirc 0\smallcirc\<\<16)^!8^*$}
      node (12) [black, pos = 0.41 , scale=0.8] {$2_* 0_* (3\smallcirc
                                                  2\smallcirc 0)^!8^*$}
      node (13) [black, pos = 0.65 , scale=0.8] {$2_* 0_* 15^*(5\smallcirc
                                                   4\smallcirc\<\<14)^{\<!}$}
      
      node (14) [black, pos = 0.89 , scale=0.8] {$(2\smallcirc 0)_* 15^*\<(5\smallcirc
                                                   4\smallcirc\<\<14)^{\<!}$};
                                                   
      \draw[white] (0cm,6.55cm) -- +(0: \linewidth)
         node (01) [black, pos = 0.13 , scale=0.8] {$(2\smallcirc 0)_* 16_*(3\smallcirc 2\smallcirc 0\smallcirc\<\<16)^!8^*$}
      node (23) [black, pos = 0.65, scale=0.8] {}
      node (24) [black, pos = 0.89 , scale=0.8] {$\ 7^*\<(4\smallcirc\<\<14)_* (5\smallcirc
                                                   4\smallcirc\<\<14)^{\<!}$};
                                                   
      \draw[white] (0cm,5.3cm) -- +(0: \linewidth)
         node (21) [black, pos = 0.13 , scale=0.8] {$(2\smallcirc 0\smallcirc\<\<16)_*
                                                   (3\smallcirc 2\smallcirc 0\smallcirc\<\<16)^!8^*$}
      node (33) [black, pos = 0.65, scale=0.8] {$2_* 6^*14_*(5\smallcirc 4\smallcirc\<\<14)^{\<!}$}
      node (34) [black, pos = 0.89 , scale=0.8] {$7^*4_* 14_* (5\smallcirc
                                                   4\smallcirc\<\<14)^{\<!}$};
                                                   
      \draw[white] (0cm,4.05cm) -- +(0: \linewidth)
         node (51) [black, pos = 0.13 , scale=0.8] {$2_*(0\smallcirc\<\<16)_*
                                                   (3\smallcirc 2\smallcirc 0\smallcirc\<\<16)^!8^*$}    
      node (52) [black, pos = 0.41 ,  scale=0.8] {$2_* (3\smallcirc 2)^!8^*$}
      node (53) [black, pos = 0.65 , scale=0.8] {$2_* 6^*(5\smallcirc 4)^{\<!}$}
      node (54) [black, pos = 0.89 , scale=0.8] {$7^*4_* (5\smallcirc 4)^{\<!}$};
    
      \node (4) at (intersection of 02--52 and 12--51) [scale=0.85]{\ \circled{4}\;}; 
      \node (5) at (intersection of 12--53 and 13--52) [scale=0.85]{\circled{5}\;};
      \node (6) at (intersection of 13--34 and 14--33) [scale=0.85]{\circled{6}$^{\phantom{g}}$\;};
      \node (7) at (intersection of 33--54 and 34--53) [scale=0.85]{\circled{7}};
         \draw [double distance=2pt]
                 (01) -- (02) node[left=1pt, midway, scale=.7]{$\ps_*$};

      \draw [->] (02) -- (12) node[above=1pt, midway, scale=.7]{$\!\via
      \smallint_{\mspace{-2mu}\lift.65,16,}^{3\>{\lift.75,\halfsize{$\circ$},} \>2\>{\lift.75,\halfsize{$\circ$},} \>0}$};
      \draw [->] (12) -- (13) node[above, midway, scale=.7]{$\via\>\>\bchadmirado{}^{-\<1}$};
      \draw [->] (14) -- (24) node[right=1pt, midway, scale=.7]{${\bchasterisco{}}^{-\<1}$};
      \draw [->] (52) -- (53) node[below, midway, scale=.7]{$\via\>\>\bchadmirado{}^{-\<1}$};
      \draw [->] (53) -- (54) node[below, midway, scale=.7]{$\bchasterisco{}^{-\<1}$};
      \draw [double distance=2pt]
                 (13) -- (14) node[above=1pt, midway, scale=.7]{$\ps_*$};
      \draw [double distance=2pt]
                 (24) -- (34) node[right=1pt, midway, scale=.7]{$7^*\<\<\ps_*$};
      \draw [double distance=2pt]
                 (01) -- (21) node[left=1pt, midway, scale=.7]{$\ps_*$};
      \draw [double distance=2pt]
                 (21) -- (51) node[left=1pt, midway, scale=.7]{$\ps_*$};
       \draw [->] (51) -- (52) node[below=1pt, midway, scale=.7]{$\!\via
      \smallint_{\mspace{-2mu}\lift.65,{0\>{\lift.75,\halfsize{$\circ$},} 16},}^{3\>{\lift.75,\halfsize{$\circ$},}    \>2}$};      
      \draw [->] (12) -- (52) node[left=1pt, midway, scale=.7]{$\!\via
      \smallint_{\mspace{-2mu}\lift.65,0,}^{3\>{\lift.75,\halfsize{$\circ$},} \>2}$};
      \draw [->] (13) -- (33) node[left=1pt, midway, scale=.7]{$2_*{\bchasterisco{}}^{-\<1}$};
      \draw [->] (33) -- (53) node[left=1pt, midway, scale=.7]{$\!\via
      \smallint_{\mspace{-2mu}\lift.65,14,}^{5\>{\lift.75,\halfsize{$\circ$},} \>4}$};
      \draw [->] (34) -- (54) node[right=1pt, midway, scale=.7]{$\!\via
      \smallint_{\mspace{-2mu}\lift.65,14,}^{5\>{\lift.75,\halfsize{$\circ$},} \>4}$};
      \draw [->] (33) -- (34) node[above , midway, scale=.7]{${\bchasterisco{}}^{-\<1}$};
    \end{tikzpicture}
\end{equation*}

The commutativity of \circled{7} is obvious, of \circled{6} follows from transitivity of $\bchasterisco{}$ (\cfr \cite[Proposition~3.7.2(iii)]{li}), and of \circled{4} is given by Proposition~\ref{Transitivity}.

As for commutativity of \circled{5}, with regard to the fiber square $\SS$\kf-diagram~$\Du\Dv$:
\[
\CD
\bullet@>\ \;\>a\:\set\:0\;\ >>\bullet@>c\:\set\:3\smallcirc2\;>>X\\
@Ve\:\set\:15V\mkern50mu\Dv V @Vf\!V\<\set\:6V@V\Du\mkern50muVg\:\set\:8V\\
\bullet@>>\;\,b\:\set\:14\,\ >\bullet@>>d\:\set\:5\smallcirc4\;>Y
\endCD
\]
(where  $b$ and $a$ are proper,  $g$, $f$, $e$ are flat, and $d$, $db$, $c$, $ca$ are perfect),
it's enough to show commutativity of the next diagram, in which $\prj$ stands for projection maps
as in \eqref{projection}, and the unlabeled maps  are the obvious ones.\looseness=-1
\[\mkern-2mu
\def\1{$a_*(c a)^{\<!}g^*$}
\def\2{$a_*e^*\<(d b)^!$}
\def\3{$f^*\<b_*(d b)^!$}
\def\4{$a_*\<\big(\<(ca)^!_\upl\OX\<\otimes\<(ca)^{\<*}g^*\<\big)$}
\def\5{$f^*\<b_*\<\<\big(\<(db)^!_\upl\OY\<\<\otimes\<(db)^{\<*}\<\big)$}
\def\6{$a_*\<\big(a^!_\upl c^!_\upl\OX\<\otimes\<a^*\<c^*\<g^*\<\big)\>$}
\def\7{$f^*\<b_*\<\<\big(b^!_\upl d^{\>!}_\upl\OY\<\<\otimes\<b^*\<d^*\<\big)$}
\def\8{$a_*a^!_\upl c^!_\upl\OX\<\otimes\<c^*\<g^*$}
\def\9{$f^*\<\big(b_*b^!_\upl d^{\>!}_\upl\OY\<\<\otimes\<d^*\<\big)$}
\def\ten{$c^!_\upl\OX\<\otimes\<c^*\<g^*$}
\def\lvn{$f^*\<d^{\>!}_\upl\OY\<\<\otimes\<f^*\<d^*$}
\def\twv{$f^*\<\big(d^{\>!}_\upl\OY\<\<\otimes\<d^*\big)$}
\def\thn{$a_*e^*\<\<\big(\<(db)^!_\upl\OY\<\<\otimes\<(db)^{\<*}\<\big)$}
\def\frn{$\!\!a_*e^*\!\big(b^!_\upl d^{\>!}_\upl\OY\!\otimes\<b^*\<d^*\<\big)$}
\def\ffn{$a_*e^*\<\big((db)^!_\upl\OY\<\<\otimes\<(db)^{\<*}\<\big)$}
\def\sxn{$a_*\<\big(e^*\<b^!_\upl d^{\>!}_\upl\OY\<\<\otimes\<e^*\<b^*\<d^*\<\big)$}
\def\svn{$\>a_*\<\big(e^*\<b^!_\upl d^{\>!}_\upl\OY\<\<\otimes\<a^*\<\<f^*\<d^*\<\big)$}
\def\egn{$a_*\<e^*\<b^!_\upl d^{\>!}_\upl\OY\<\<\otimes\<f^*\<d^*$}
\def\ntn{$a_*a^!_\upl f^*\<d^{\>!}_\upl\OY\<\<\otimes\<f^*\<d^*$}
\def\twy{$\mkern8mu f^*\<b_*b^!_\upl d^{\>!}_\upl\OY\<\<\otimes\<f^*\<d^*$}
\def\twn{$\ a_*\<\big(a_\upl^!\<f^* d^{\>!}_\upl\OY\<\<\otimes\<a^*\<f^*\<d^*\<\big)$}
\def\twt{$\ \,a_*\<\big(e^*\<(db)^!_\upl\OY\<\<\otimes\<e^*\<(db)^{\<*}\<\big)$}
\bpic[xscale=4.4,yscale=1.85]
 
    \node(11) at (1,-2){\1};
    \node(12) at (2.03,-2){\2};
    \node(13) at (3.02,-2){\3};
    
    \node(22) at (2.03,-3){\ffn};
    \node(23) at (3.02,-3){\5};

    \node(31) at (1.5,-4){\twt};
    \node(32) at (2.5,-4){\frn};

    \node(41) at (1,-5){\4};
    \node(42) at (2.03,-5){\sxn};
    \node(43) at (3.02,-5){\7};

    \node(61) at (1,-6){\6};
    \node(62) at (2.03,-6){\svn};

    \node(71) at (1.5,-7){\twn};
  
    \node(81) at (1,-8){\8};
    \node(82) at (2.03,-8){\egn};
    \node(83) at (3.02,-8){\9};

    \node(91) at (1.5,-9){\ntn};
    \node(92) at (2.5,-9){\twy};

    \node(01) at (1,-10){\ten};
    \node(02) at (2.03,-10){\lvn};
    \node(03) at (3.02,-10){\twv};
   
   \draw[->] (12)--(11) node[above=1pt, midway, scale=.75]{$a_*\bchadmirado{\Du\Dv}$};  
   \draw[->] (13)--(12) node[above=1pt, midway, scale=.75]{$\theta_{\Dv}$};  

   \draw[->] (23)--(22) node[above=1pt, midway, scale=.75]{$\theta_{\Dv}$};  

   \draw[->]  (02)--(01) node[below=1pt, midway, scale=.75]{$\bchadmirado{\Du}$}
                                     node[above=1pt, midway, scale=.75]{$\ps^*$};  
   \draw[->]  (03)--(02) ;  
  
    \draw[double distance=2pt] (11)--(41) node[left, midway, scale=.75]{$$};
    \draw[double distance=2pt] (41)--(61) node[right=1pt, midway, scale=.75]{$\ps^*$}
                                                      node[left, midway, scale=.75]{$\ps_\upl^!$};  
    \draw[->] (61)--(81)node[left, midway, scale=.75]{$\prj$};
    \draw[->] (81)--(01)node[right, midway, scale=.75]{$$};

    \draw[->] (71)--(91)node[left, midway, scale=.75]{$\prj$};
    \draw[double distance=2pt] (12)--(22) node[left, midway, scale=.75]{$$};
    \draw[->] (22)--(32) node[left, midway, scale=.75]{$$};
    \draw[->] (32)--(42) node[left, midway, scale=.75]{$$};
    \draw[double distance=2pt] (42)--(62) node[right=1pt, midway, scale=.75]{$\ps^*$};
    \draw[->] (62)--(82) node[right, midway, scale=.75]{$\prj$};
    
    \draw[double distance=2pt] (13)--(23) node[left, midway, scale=.75]{$$};
    \draw[double distance=2pt] (23)--(43) node[right=.5pt, midway, scale=.75]{$\ps^*$}
                                      node[left=-1pt, midway, scale=.75]{$\ps^!_\upl$};    
    \draw[->] (43)--(83) node[right, midway, scale=.75]{$\prj$};
    \draw[->] (83)--(03);
    
     \draw[->] (22)--(31) node[left, midway, scale=.75]{$$};
     \draw[double distance=2pt] (22)--(32) node[right=1pt, midway, scale=.75]{$\mkern15mu\ps^*$}
                                                      node[left, midway, scale=.75]{$\ps_\upl^!\mkern15mu$};
     \draw[->] (31)--(41) node[above=-4pt, midway, scale=.75]{$\<\bar{\>\mathsf B}_{\mkern-.5mu{\Du\Dv}}\mkern50mu$}
                                node[right=2pt, midway, scale=.75]{$\ps^*$};
     \draw[double distance=2pt] (31)--(42) node[right=1pt, midway, scale=.75]{$\mkern13mu\ps^*_{\mathstrut}$}
                                                      node[left, midway, scale=.75]{$\ps_\upl^!\mkern15mu$};
     \draw[->] (43)--(32) node[above=-3pt, midway, scale=.75]{$\mkern25mu\theta_{\Dv}$};  
     \draw[->] (62)--(71) node[above=-9pt, midway, scale=.75]{$\mkern38mu\bar{\>\mathsf B}_{\mkern-.5mu{\Dv}}$};
     \draw[->] (71)--(61) node[right=1pt, midway, scale=.75]{$\mkern6mu\ps^*_{\mathstrut}$}
                             node[left, midway, scale=.75]{$\bar{\>\mathsf B}^{\mathstrut}_{\mkern-.5mu{\Du}}\mkern2mu$};
     \draw[->] (82)--(91) node[above=-9pt, midway, scale=.75]{$\mkern38mu\bar{\>\mathsf B}_{\mkern-.5mu{\Dv}}$};
     \draw[->] (91)--(81) node[right=1pt, midway, scale=.75]{$\mkern7mu\ps^*$}
                             node[left, midway, scale=.75]{$\bar{\>\mathsf B}^{\mathstrut}_{\mkern-.5mu{\Du}}\mkern2mu$};
     \draw[->] (83)--(92) node[left, midway, scale=.75]{$$};
     \draw[->] (92)--(82) node[above=-3pt, midway, scale=.75]{$\mkern25mu\theta_{\Dv}$}; 
     \draw[->] (91)--(02) node[left, midway, scale=.75]{$$};
     \draw[->] (92)--(02) node[left, midway, scale=.75]{$$};
     
    \node at (1.45, -3){\circled5$_{\<1}$};   
    \node at (1.5, -5.5){\circled5$_2$}; 
    \node at (2.5,-6.5){\circled5$_3$};   
    \node at (2.03,-9){\circled5$_4$};   
  
 \epic
\]

Commutativity of \circled5$_{\<1}$ results directly from the definition of $\bchadmirado{\Du\Dv}$ (\S\ref{indsq}). Commutativity of \circled5$_2$ is given by pseudofunctoriality of $(-)^*$ and transitivity of $\<\bar{\>\mathsf B}$ (see \cite[\S5.8.4]{AJL}). Commutativity of \circled5$_3$ is given by \cite[3.7.3]{li},
in which one makes the substitution $(f,g,f'\<\<,g'\<\<,P,Q)\mapsto(b,f,a,e,d^*\<\<,b^!_\upl d^{\>!}_\upl\OY)$ (and harmlessly reverses the order of the factors in the tensor products).
Since $b$ and $a$ are proper, commutativity
of~\circled5$_4$ holds by the definition of $\bar{\>\mathsf B}_{\mkern-.5mu{\Dv}}$ \cite[\S5.8.2]{AJL}.
Commutativity of the unlabeled subdiagrams is clear. 

Thus \circled1 does indeed commute.\va2

\pagebreak[3]
We deal next with \circled2, which expands, by the definition of $\lambda$,  to
\begin{equation*}\mkern-4mu
    \begin{tikzpicture}[xscale=.99]
      \draw[white] (0cm,4.4cm) -- +(0: \linewidth)
      node (11) [black, pos = 0.1 , scale=0.8] {$(2\smallcirc 0 \smallcirc\<16)_*(3\smallcirc 2\smallcirc 0\smallcirc\<16)^{\<!}8^*11^{\<*}$}
      node (12) [black, pos = 0.39 , scale=0.8] {$2_* (3\smallcirc 2)^{\<!}8^*11^{\<*}$}
      node (13) [black, pos = 0.62, scale=0.8] {$2_* 6^*(5\smallcirc 4)^{\<!}11^{\<*}$}
      node (14) [black, pos = 0.855, scale=0.8] {$7^*4_* (5\smallcirc 4)^{\<!}11^{\<*}$};
      \draw[white] (0cm,3.1cm) -- +(0: \linewidth)
      node (23) [black, pos = 0.62 , scale=0.8] {$2_* 6^*9^*(13\smallcirc\<12)^{\<!}$}
      node (24) [black, pos = 0.855 , scale=0.8] {$7^*4_* 9^*(13\smallcirc\<12)^{\<!}$};
      \draw[white] (0cm,1.8cm) -- +(0: \linewidth)
      node (31) [black, pos = 0.1 , scale=0.8] {$(2\smallcirc 0\smallcirc\<\<16)_*(3\smallcirc 2\smallcirc0\smallcirc\<\<16)^{\<!}(\<11 \smallcirc 8)^*$}
      node (32) [black, pos = 0.39 , scale=0.8] {$2_* (3\smallcirc 2)^{\<!}(11\< \smallcirc 8)^*\ $}
      node (33) [black, pos = 0.62 , scale=0.8] {$\ 2_* (9 \smallcirc 6)^*(13\smallcirc\<12)^{\<!}$}
      node (34) [black, pos = 0.77 , scale=0.8] {\circled9$^{\phantom{e}}$};

      \draw[white] (0cm,0.5cm) -- +(0: \linewidth)
      node (43) [black, pos = 0.62 , scale=0.8] {$(10\smallcirc 7)^{\<*}12_* (13\smallcirc\<\<12)^{\<!}$}
      node (44) [black, pos = 0.855 , scale=0.8] {$7^*\<10^*\<12_*\< (13\smallcirc\<\<12)^{\<!}$};

      \node (label8) at (intersection of 12--33 and 32--13) [scale=0.8]{\circled{8}$^{\phantom{e}}\mkern5mu$};
      
      \draw [->] (11) -- (12) node[above, midway, scale=0.67]{};
      \draw [->] (12) -- (13) node[above, midway, scale=0.67]{$2_*\bchadmirado{}^{-\<1}$};
      \draw [->] (13) -- (14) node[above, midway, scale=0.67]{$\bchasterisco{}^{-\<1}$};
      \draw [->] (31) -- (32) node[above, midway, scale=0.67]{};
      \draw [->] (32) -- (33) node[above, midway, scale=0.67]{$2_*\bchadmirado{}^{-\<1}$};
      \draw [->] (33) -- (43) node[left, midway, scale=0.67]{$\via\, \bchasterisco{}^{-\<1}$};
      \draw [double distance=2pt]
                 (11) -- (31) node[left=1pt,  midway, scale=0.67]{$\via\, \ps^*$};
      \draw [double distance=2pt]
                 (12) -- (32) node[left=1pt,   midway, scale=0.67]{$\via\, \ps^*$};
      \draw [->] (13) -- (23) node[left,   midway, scale=0.67]{$\via\, \bchadmirado{}^{-\<1}$};
      \draw [->] (23) -- (24) node[above,  midway, scale=0.67]{$\bchasterisco{}^{-\<1}$};
      \draw [->] (14) -- (24) node[right=1pt,  midway, scale=0.67]{$\via\, \bchadmirado{}^{-\<1}$};
      \draw [->] (24) -- (44) node[right=1pt,  midway, scale=0.67]{$7^*\bchasterisco{}^{-\<1}$};
      \draw [double distance=2pt]
                 (23) -- (33) node[left, midway, scale=0.67]{$\via\, \ps^*$};
      \draw [double distance=2pt]
                 (43) -- (44) node[below, midway, scale=0.67]{$\ps^*$};
    \end{tikzpicture}
\end{equation*}
The unlabeled maps are induced by  $\ps_*$ and $\smallint_{\mspace{-2.5mu}\lift.65,3\>{\lift.75,\halfsize{$\circ$},} \>2,}^{\>0\>{\lift.75,\halfsize{$\circ$},} \>16}$.\va1 Commutativity of the two unlabeled diagrams is obvious, that of \circled{8} follows from 
transitivity of~$\bchadmirado{}$ (\S\ref{indsq}), and that of 
\circled{9} from transitivity of $\bchasterisco{}$\va{.6} \cite[Proposition~3.7.2(ii)]{li}. 
Thus \circled{2} commutes.

Commutativity of \circled{3} results directly from the definitions of $\lambda_{\De}$ and $\lambda_{\De^-}\>$.\va2

 This completes the proof of Lemma~\ref{lemstepI}, and of Step I (commutativity of subdiagram ($\#$) in \eqref{sostenidos}).
\end{proof}

\begin{cosa}[\emph{Step II\kern2pt}]
Let us now check that diagram $(\# \#)$ in (\ref{sostenidos}) commutes. With $\Da$ and~$\Db$ as in \S\ref{def-fc}, and with $\chi$ given by the composite isomorphism
  \[
   \begin{tikzpicture}
      \draw[white] (0cm,0.5cm) -- +(0: \linewidth)
      node (11) [black, pos = 0.1] {$u^!v^!$}
      node (12) [black, pos = 0.27 ] {$(v u)^!$}
      node (13) [black, pos = 0.51  ] {$(v u)^!\CO_{\<Z}\otimes (v u)^*$}
      node (14) [black, pos = 0.83 ] {$u^!v^! \CO_{\<Z}\otimes u^*\<v^*,$};
      \draw [double distance=2pt]
                 (11) -- (12) node[above, midway, scale=0.75]{$\ps^!$};
      \draw [double distance=2pt] (12) -- (13) ;
      \draw [double distance=2pt]
                 (13) -- (14) node[above, midway, scale=0.75]{$\ps^!\otimes \ps^*$};
   \end{tikzpicture}
  \]
the diagram expands as \va3
\begin{equation}\label{expand sharps}
\mkern 16mu
    \begin{tikzpicture}
      \draw[white] (0cm,5.3cm) -- +(0: \linewidth)
      node (11) [black, pos = 0.15 ] {$\Gams u\Gam {u*} u^!v^*$}
      node (12) [black, pos = 0.5 ] {$u^!\CO_Y\otimes u^*\delta^*_Y \delta^{}_{Y\mkern-1.5mu*}v^*$}
      node (13) [black, pos = 0.85] {$u^!\delta^*_Y \delta^{}_{Y\mkern-1.5mu*} v^*$};
      \draw[white] (0cm,4.1cm) -- +(0: \linewidth)
      node (21) [black, pos = 0.15] {$\Gams{v u}\Gam{v u*}(v u)^!$};
      \draw[white] (0cm,2.9cm) -- +(0: \linewidth)
      node (31) [black, pos = 0.15 ] {$\Gams{v u}\Gam{v u*}u^!v^!$}
      node (32) [black, pos = 0.5 ] {$u^!\CO_Y\otimes u^*\Gams v\Gam {v*} v^!$}
      node (33) [black, pos = 0.85 ] {$u^!\Gams v\Gam {v*}  v^!$};
      \draw[white] (0cm,1.7cm) -- +(0: \linewidth)
      node (41) [black, pos = 0.15 ] {$\Gams{v u}\Gam{v u*}(u^!v^!\CO_{\<Z}\otimes u^*\<v^*)$};
      \draw[white] (0cm,0.5cm) -- +(0: \linewidth)
      node (51) [black, pos = 0.15 ] {$u^!v^!\CO_{\<Z}\otimes\Gams{v u}\Gam{v u*}u^*\mkern-.5mu v^*$}
      node (52) [black, pos = 0.5 ] {$u^!\CO_Y\otimes u^*\mkern-.5mu v^!\delta^*_{\<\<Z} \delta^{}_{\<\<Z*}$}
      node (53) [black, pos = 0.85 ] {$u^!v^!\delta^*_{\<\<Z} \delta^{}_{\<\<Z*}$};
      \node (labelA) at (intersection of 11--32 and 31--12) [scale=0.85]{$\qquad\mathbf{A}$};
      \node (labelB) at (intersection of 31--52 and 32--51) [scale=0.85]{$\qquad\mathbf{B}$};
      \draw [->] (11) -- (12) node[above=1pt, midway, scale=0.75]{$4$};
      \draw [double distance=2pt] (12) -- (13) node[above=1pt, midway, scale=0.75]{$5$};
      \draw [->] (31) -- (32) node[below=1pt, midway, scale=0.75]{$6$};
      \draw [double distance=2pt] (32) -- (33) node[below=1pt, midway, scale=0.75]{$7$};
      \draw [->] (51) -- (52) node[below=1pt, midway, scale=0.75]{$8$};
      \draw [double distance=2pt] (52) -- (53) node[below=1pt, midway, scale=0.75]{$9$};
      \draw [->] (11) -- (21) node[left=1pt, midway, scale=0.75]{$?$}
                              node[right, midway, scale=0.75]{in $(\ref{sostenidos})$};
      \draw [double distance=2pt]
                 (21) -- (31) node[left=1pt, midway, scale=0.75]{$0$}
                              node[right=1pt, midway, scale=0.75]{$\via\>\>\ps^!$};
      \draw [->] (31) -- (41) node[left=1pt,  midway, scale=0.75]{$2$}
                                         node[right=1pt, midway, scale=0.75]{$\via\>\>\chi$};
      \draw [->] (41) -- (51) node[left=1pt,  midway, scale=0.75]{$3$};
      \draw [->] (12) -- (32) node[left, midway, scale=0.75]{$\via\>\>\Da_v$};
      \draw [->] (32) -- (52) node[left, midway, scale=0.75]{$\via\>\>\Db_v$};
      \draw [->] (13) -- (33) node[right=1pt, midway, scale=0.75]{$10$}
                                         node[left, midway, scale=0.75]{$u^!\Da_v$};
      \draw [->] (33) -- (53) node[right=1pt,  midway, scale=0.75]{$11$}
                                         node[left, midway, scale=0.75]{$u^!\Db_v$};
    \end{tikzpicture}
\end{equation}
Here $3$ is an instance of the isomorphism $\mu_{vu}$ (see \ref{def-of-mu}); 
4 is the composition of the first three maps in \eqref{def-b} 
(with $u$ in place of $f$), so that $11\smallcirc\<10 \smallcirc 5 \smallcirc 4$ is the composition of the two arrows in the first row of~$(\# \#)$;
$6$ is the composite isomorphism

 \begin{align*}
\Gams{v u}\Gam{v u*}u^!v^! 
=
 \Gams{v u}\Gam{v u*}(u^!\CO_Y\otimes u^*v^!)
 &\xto{\mu_{vu}\>}
 u^!\CO_Y \otimes  \Gams{v u}\Gam{v u*}u^*\<v^!\\
 &\xto{\via\>\phi_{\Du}^{-\<1}\<}
 u^!\CO_Y \otimes u^*\Gams v\Gam {v*} v^!,
 \end{align*}
where, relative to the next diagram, $\defnu{}$ is as in  (Proposition~\ref{ayuda0}) 
\begin{equation}\label{Du and Dv}
   \CD
   \begin{tikzpicture}[xscale=2.3,yscale=1.7]
      \node (12) at (1,-1) {$X$};
      \node (13) at (2,-1) {$Y$};
      \node (14) at (3,-1) {$Z$};
      \node (22) at (1,-2) {$X\times Z$};
      \node (23) at (2,-2){$Y\times Z$};
      \node (24) at (3,-2) {$\:Z\times Z\>;$};
      \node (E) at (intersection of 12--23 and 22--13) [scale=0.8] {$\Du$};
      \node (B) at (intersection of 13--24 and 23--14) [scale=0.8] {$\Dv$};
      \draw [->] (12) -- (13) node[above=1pt, midway, scale=0.75]{$u$};
      \draw [->] (13) -- (14) node[above=1pt, midway, scale=0.75]{$v$};
      \draw [->] (22) -- (23) node[below=1pt, midway, scale=0.75]{$u\<\<\times\<\<\id_{\<Z}$};
      \draw [->] (23) -- (24) node[below=1pt, midway, scale=0.75]{$v\<\<\times\<\<\id_{\<Z}$};
      \draw [->] (12) -- (22) node[left,  midway, scale=0.75]{$\Gam{vu}$};
      \draw [->] (13) -- (23) node[left,  midway, scale=0.75]{$\Gam v$};
      \draw [->] (14) -- (24) node[right=1pt, midway, scale=0.75]{$\delta_{\<Z}$};
   \end{tikzpicture}
   \endCD
  \end{equation}
and with $\ps^!$ and $\Db_{vu}$ as in $(\# \#)$, $8\set 9^{-1}(\ps^!\!\smallcirc \Db_{vu})0^{-1}2^{-1}3^{-1}\<$, so  that 
$$
9\smallcirc 8\smallcirc 3\smallcirc 2\smallcirc 0=\ps^!\smallcirc \Db_{vu}.
$$

It is clear that the unlabeled subdiagrams in \eqref{expand sharps} commute; so it suffices to show that the subdiagrams 
$ \mathbf{A}$ and $ \mathbf{B}$ commute. 
\end{cosa}

\begin{cosa}[\emph{Step IIB\kern2pt}]
We deal first with $\mathbf{B}$.  
Let $\bar\phi$ be the composite isomorphism
\[
\Gams{vu}\Gam{vu*}u^*\<v^*
\xto{\phi^{-\<1}_{\Du}} 
u^*\Gams v\Gam{v*}v^*
\xto{u^*\<\phi^{-\<1}_{\Dv}}
u^*\<v^*\delta^*_{\<\<Z} \delta^{}_{\<\<Z\<*}
\qquad(\Du,  \Dv\textup{  as above}).
\]
The map $8\set 9^{-1}(\ps^!\!\smallcirc \Db_{vu})0^{-1}2^{-1}3^{-1}\<$  in $\mathbf B$ factors as
\begin{align*}
u^!v^!\CO_{\<Z}\otimes \Gams{vu}\Gam{vu*}u^*\<v^*
&\xto{\via\>\bar\phi\,}
u^!v^!\CO_{\<Z}\otimes u^*\<v^*\<\delta^*_{\<\<Z} \delta^{}_{\<\<Z\<*}\\[-1pt]
&\xto{\chi^{\<-1}}u^!v^!\delta^*_{\<\<Z} \delta^{}_{\<\<Z\<*}\\[1pt]
&=\!=\, u^!\CO_Y\<\otimes\< u^*\<v^!\delta^*_{\<\<Z} \delta^{}_{\<\<Z\<*}.
\end{align*}
that is, as
\begin{align*}
u^!v^!\CO_{\<Z}\otimes \Gams{vu}\Gam{vu*}u^*\<v^*
&\xto{\via\>\bar\phi\,}
u^!v^!\CO_{\<Z}\otimes u^*\<v^*\<\delta^*_{\<\<Z} \delta^{}_{\<\<Z\<*}\\
&\overset{\via\ps^!\<\<,\,\ps^*}{=\!=\!=\!=\!=}
(v u)^!\CO_{\<Z}\otimes (v u)^*\<\delta^*_{\<\<Z} \delta^{}_{\<\<Z\<*}\\
&=\!=(v u)^!\delta^*_{\<\<Z} \delta^{}_{\<\<Z\<*}
\overset{\ps^!}{=\!=}
u^!v^!\delta^*_{\<\<Z} \delta^{}_{\<\<Z\<*}\\
&=\!=\, u^!\CO_Y\<\otimes\< u^*\<v^!\delta^*_{\<\<Z} \delta^{}_{\<\<Z\<*}.
\end{align*}
This results from commutativity of all the subdiagrams of the following diagram, where the
subdiagram \circled0 commutes by Proposition~\ref{ayuda}, 
and the rest by the definitions of the maps involved.

\[
\def\1{$u^!v^!\CO_{\<Z}\otimes \Gams{vu}\Gam{vu*}u^*\<v^*$}
\def\2{$\Gams{vu}\Gam{vu*}(u^!v^!\CO_{\<Z}\otimes u^*\<v^*\<)$}
\def\3{$\Gams{vu}\Gam{vu*}u^!v^!$}
\def\4{$(v u)^!\CO_{\<Z}\<\otimes \<\Gams{vu}\Gam{vu*}(v u)^{\<*\<}$}
\def\5{$\Gams{vu}\Gam{vu*}(\<(v u)^!\CO_{\<Z}\<\otimes \<(v u)^*)$}
\def\6{$\Gams{vu}\Gam{vu*}(vu)^!$}
\def\7{$(v u)^!\CO_{\<Z}\otimes (v u)^*\<\delta^*_{\<\<Z} \delta^{}_{\<\<Z\<*}$}
\def\8{$(v u)^!\delta^*_{\<\<Z} \delta^{}_{\<\<Z\<*}$}
\def\9{$u^!v^!\CO_{\<Z}\otimes u^*\<v^*\<\delta^*_{\<\<Z} \delta^{}_{\<\<Z\<*}$}
\bpic[xscale=4.3,yscale=.9]
  \node(02) at (2,0){\9}; 
 
  \node(11) at (1,-1){\1};
  \node(12) at (2,-1){\circled0};
  \node(13) at (3,-1){\7};
  
  \node(22) at (2,-2){\4};
  
  \node(31) at (3,-3){\8};
  
  \node(42) at (2,-4){\5};
 
  \node(51) at (1,-5){\2};
  \node(53) at (3,-5){\6};

  \node(62) at (2,-6){\3};

  \draw[->] (11)--(51) node[left=1pt, midway, scale=.75]{$3^{\<-1}$};
  \draw[->] (51)--(62) node[below=1pt, midway, scale=.75]{$2^{-\<1}$};

  \draw[->] (42)--(22) node[left=1pt, midway, scale=.75]{$\eqref{def-of-mu}$};
  \draw[double distance=2pt] (53)--(42);
  
  \draw[double distance=2pt] (13)--(31) ;
  
  \draw[double distance=2pt] (11)--(22) node[below=-3pt,midway, scale=.75]{$\via \ps^!\<\<,\ps^*\mkern110mu$};
  \draw[double distance=2pt] (51)--(42) ;
  \draw[double distance=2pt] (62)--(53) node[below=1pt, midway, scale=.75]{$\quad0^{-\<1}$};
  
  \draw[->] (22) -- (13) node[below=-3pt, midway, scale=.75]{$\mkern50mu\via\phi_{\Dv\Du}^{-\<1}$};
  \draw[->] (53)--(31) node[right=1pt, midway, scale=.75]{$\Db_{\Dv\Du}$};
  
 \draw[->] (11)--(02) node[above=-1pt, midway, scale=.75]{$\via\bar\phi\mkern25mu$}; 
 \draw[-,double distance=2pt] (02)--(13) node[above=-1pt,midway, scale=.75]{$\mkern90mu\via \ps^!\<,\ps^*$};
                                                                
 \epic
\]

So $\mathbf{B}$ expands as follows,\va1 with $\Gamma=\Gam v$, $\Gamma^{\prime}=\Gam{v u}$,
$\delta\set\delta_{\<Z}\>$, and $\nu$ standing for  natural isomorphisms of the form\va{-6}  $u^{\<*}\<(E\otimes F)\iso u^{\<*}\<\<E\otimes u^{\<*}\<\<F$ (see \eqref{^* and tensor}).\looseness=-1

\begin{equation*}\label{expand B}\mkern-4mu
    \begin{tikzpicture}[xscale=0.92,yscale=1.08]
      \draw[white] (0cm,14.9cm) -- +(0: \linewidth)
      node (11) [black, pos=0.09 , scale=.95] {$\Gamma^{{\prime}*}\Gam{*}'u^!v^!$}
      node (12) [black, pos = .36 , scale=.95] {$\Gamma^{{\prime}*}\Gam{*}'(u^!\CO_Y\<\otimes u^*\<v^!)$}      
      node (14) [black, pos = 0.84, scale=.95] {$u^!\CO_Y\<\otimes u^*\Gamma^*\Gamma_{\<\!*} v^!$};
      \draw[white] (0cm,13.7cm) -- +(0: \linewidth)
      node (13) [black, pos = .6 , scale=.95] {$u^!\CO_Y\<\otimes \Gamma^{{\prime}*}\Gam{*}'u^*\<v^!$};
      \draw[white] (0cm,12.5cm) -- +(0: \linewidth)
      node (21) [black, pos=0.09 , scale=.95] {$$}
      node (22) [black, pos = .36 , scale=.95] {$\Gamma^{{\prime}*}\Gam{*}'(u^!\CO_Y\<\otimes u^*(v^!       \CO_{\<\<Z}\otimes v^*))$}
      node (24) [black, pos = 0.84 , scale=.95] {$u^!\CO_Y\<\otimes u^*\Gamma^*\Gamma_{\<\!*} (v^!\CO_{\<\<Z}\otimes v^*)$};
      \draw[white] (0cm,11.3cm) -- +(0: \linewidth)
      node (03) [black, pos = .21 , scale=.8] {$\mathbf{B_1}$}
      node (23) [black, pos = .6 , scale=.95] {$u^!\CO_Y\<\otimes \Gamma^{{\prime}*}\Gam{*}'u^*(v^!\CO_{\<\<Z}\otimes v^*)$};
      \draw[white] (0cm,10.1cm) -- +(0: \linewidth)
      node (31) [black, pos=0.09 , scale=.95] {}
      node (32) [black, pos = .36 , scale=.95] {$\Gamma^{{\prime}*}\Gam{*}'(u^!\CO_Y\<\otimes (u^*\<v^!\CO_{\<\<Z}\otimes u^*\<v^*))$}
      node (34) [black, pos = 0.84, scale=.95] {$u^!\CO_Y\<\otimes u^*(v^!\CO_{\<\<Z}\otimes \Gamma^* \Gamma_{\<\!*} v^*)$};
      \draw[white] (0cm,8.9cm) -- +(0: \linewidth)
            node (41) [black, pos=0.09 , scale=.95] {$\Gamma^{{\prime}*}\Gam{*}'(u^!v^!\CO_{\<\<Z}\otimes u^*\<v^*)$}
      node (33) [black, pos = .6, scale=.95] {$u^!\CO_Y\<\otimes \Gamma^{{\prime}*}\Gam{*}'(u^*\<v^!\CO_{\<\<Z}\otimes u^*\<v^*)$};
      \draw[white] (0cm,7.7cm) -- +(0: \linewidth)
      node (42) [black, pos = .36 , scale=.95] {$\<\<\Gamma^{{\prime}*}\Gam{*}'(\<(u^!\CO_Y\<\<\otimes \<\<u^*\<v^!\CO_{\<\<Z}\>)\<\<\otimes\< u^*\<v^*\<)$}      
      node (44) [black, pos = 0.84, scale=.95] {$u^!\CO_Y\<\otimes u^*\<v^!\CO_{\<\<Z}\otimes u^*\Gamma^*\Gamma_{\<\!*} v^*$};
      \draw[white] (0cm,6.5cm) -- +(0: \linewidth)
            node (51) [black, pos=0.09 , scale=.95] {$u^!v^!\CO_{\<\<Z}\otimes \Gamma^{{\prime}*}\Gam{*}'u^*\<v^*$}
      node (43) [black, pos = .6, scale=.95] {$u^!\CO_Y\<\otimes u^*\<v^!\CO_{\<\<Z}\otimes \Gamma^{{\prime}*}\Gam{*}'u^*\<v^*$};
      \draw[white] (0cm,5.3cm) -- +(0: \linewidth)
      node (52) [black, pos = .09 , scale=.95] {$u^!v^!\CO_{\<\<Z}\otimes u^*\<v^*\<\delta^*\< \delta_*$}
      node (54) [black, pos = 0.84, scale=.95] {$u^!\CO_Y\<\otimes u^*\<v^!\<\CO_{\<\<Z}\<\otimes \<u^*\<v^*\<\delta^*\< \delta_*$};
      \draw[white] (0cm,4.1cm) -- +(0: \linewidth)
      node (93) [black, pos = .09 , scale=.95] {$u^!v^!\delta^*\< \delta_*$}
      node (94) [black, pos = 0.465, scale=.95] {$u^!\CO_Y\<\otimes u^*\<v^!\delta^*\< \delta_*$}
      node (64) [black, pos = 0.84, scale=.95] {$u^!\CO_Y\<\otimes u^*(v^!\CO_{\<\<Z}\otimes v^*\<\delta^*\< \delta_*\<)$};  
      \draw [->] (11) -- (12) node[above=1pt, midway, scale=.7125]{};
      \draw [->] (12) -- (13) node[above=-1pt, midway, scale=.7125]{$\qquad\mu_{vu}$};
      \draw [->] (13) -- (14) node[above=-2pt,midway, scale=.7125]{$\via\>\phi_{\Du}^{-\<1}\qquad\quad $};

      \draw [->] (22) -- (23) node[above=-1pt, midway, scale=.7125]{$\qquad\mu_{vu}$};
      \draw [->] (32) -- (33) node[above=-1.5pt, midway, scale=.7125]{$\qquad\mu_{vu}$};
      \draw [->] (23) -- (24)  node[above=-2pt,midway, scale=.7125]{$\via\>\phi_{\Du}^{-\<1}\qquad\quad$};
      \draw [double distance=2pt] (41) -- (42);
      \draw [->] (42) -- (43) node[above=-1pt, midway, scale=.7125]{$\qquad\mu_{vu}$};

      \draw [->] (43) -- (44) node[right=16.5pt, above=-10.5pt,midway, scale=.7125]{$\via\>\phi_{\Du}^{-\<1}\quad$};
      \draw [double distance=2pt] (51) -- (43);
      \draw [->] (51) -- (52) node[left=1pt, midway, scale=.7125]{$\via\>\bar\phi$};
      \draw [double distance=2pt] (52) -- (54) ;
      \draw [double distance=2pt] (64) -- (94); 
      \draw [double distance=2pt] (93) -- (94) ;

      \draw [->] (11) -- (41) node[left=1pt,  midway, scale=.7125]{$\via\chi$};
      \draw [->] (41) -- (51) node[left=1pt,  midway, scale=.7125]{$3$};
      \draw [double distance=2pt] (12) -- (22) ;
      \draw [double distance=2pt] (13) -- (23) ;
      \draw [->] (22) -- (32) node[left=-2pt,  midway, scale=.7125]{$\via\>\>\nu_{\phantom{g}}$};
      \draw [double distance=2pt] (32)--(42);
      \draw [->] (23) -- (33) node[above=9pt,  midway, scale=.7125]{$\mkern54mu\via\>\>\nu$};
      \draw [->] (33) -- (43) node[above=11pt,  midway, scale=.7125]{$\mkern70mu\via\> \mu_{vu}$};
      \draw [->] (43) -- (54) node[above=1pt, midway, scale=.95]{$\lift.65,\mkern44mu\via\>\bar\phi,$};
      \draw [double distance=2pt] (14) -- (24);
      \draw [->] (24) -- (34) node[right=1pt, midway, scale=.7125]{$\via\> \mu_v$};
      \draw [->] (34) -- (44) node[right=1pt,  midway, scale=.7125]{$\via\>\>\nu$};
      \draw [->] (44) -- (54) node[right=1pt,  midway, scale=.7125]{$\via\>\phi_{\Dv}^{-\<1}$};
      \draw [->] (54) -- (64) node[right=1pt,  midway, scale=.7125]{$\via\>\>\nu^{-\<1}$};
      \draw [->] (52) -- (93) node[left=1pt,  midway, scale=.7125]{$\chi^{-\<1}$};
      \node (labelB2) at (intersection of 32--43 and 33--42) [scale=0.8]{$ \mathbf{B_2}$};
      \node (labelB3) at (intersection of 43--24 and 23--44) [scale=0.8  ]{$ \mathbf{B_3}$};
      \node (labelB4) at (intersection of 54--93 and 52--64) [scale=0.8]{$\mathbf{B_4}$};

      \end{tikzpicture}
\end{equation*}
 
 Commutativity of the unlabeled subdiagrams is easily verified. \va{-2}
That of $\mathbf{B_1}$ (without $\Gamma^{{\prime}*}\Gam{*}'$) is
essentially \va1
the definition of  the isomorphism 
$u^!v^!\overset{\ps^!}{=\!=}(vu)^!$,
see \cite[(5.7.5)]{AJL}; and similarly for $\mathbf{B_4}$ (without $\delta^*\<\delta_*$).

Commutativity of $ \mathbf{B_2}$ is contained in the next Lemma.

\begin{sublem}
Let\/ $X\overset{\Gamma}{\lto} W\overset{p}{\lto} X$ be qcqs maps
with 
$p\> \Gamma=\id_\sX,$ and let 
\[
\mu=\mu_{\Gamma\!,\,p}\colon  \Gamma^*  \Gamma_{\<\!*} (\Adot\otimes \Bdot) \iso \Adot \otimes  \Gamma^* \Gamma_{\<\!*} \Bdot \qquad (\Adot, \, \Bdot\in \D_\sX)
\]
be the functorial isomorphism defined as in\/
\textup{\ref{def-of-mu}}. 
Then for any\/ $\Adot,$ $\Bdot$ and\/~$\Cdot$ in\/ $\D_\sX$ the following diagram commutes:
\[
 \begin{tikzpicture}
     \draw[white] (0cm,2cm) -- +(0: \linewidth)
      node (H) [black, pos = 0.25] {$\Gamma^*\Gamma_{\<\!*}
      (\Adot\otimes (\Bdot\<\otimes \Cdot))$}
      node (J) [black, pos = 0.75] {$\Adot\otimes \Gamma^*\Gamma_{\<\!*} (\Bdot\<\otimes \Cdot)$};
      \draw[white] (0cm,0.5cm) -- +(0: \linewidth)
      node (E) [black, pos = 0.25] {$\Gamma^*\Gamma_{\<\!*} ((\Adot\otimes \Bdot) \otimes \Cdot)$}
      node (G) [black, pos = 0.75] {$\Adot\otimes \Bdot\<\otimes \Gamma^*\Gamma_{\<\!*} \Cdot$};
      \draw [double distance=2pt] (H) -- (E); 
      \draw [->] (J) -- (G) node[right=1pt, midway, scale=0.75]{$\via \mu$};
      \draw [->] (H) -- (J) node[above, midway, scale=0.75]{$\via \mu$};
      \draw [->] (E) -- (G) node[below=1pt, midway, scale=0.75]{$\via \mu$};
\end{tikzpicture}
\]
\end{sublem}
\begin{proof}
Referring to the definition of $\mu$, expand the diagram to the
following natural one, where the isomorphism $\ps^*$  is denoted by an equality. 
\begin{equation*}\mkern-3mu
    \begin{tikzpicture}[yscale=1.15]
      \draw[white] (0cm,9.5cm) -- +(0: \linewidth)
      node (11) [black, pos = 0.14 , scale=0.9] {$\Gamma^*\Gamma_{\<\!*} (\Adot\<\otimes\< (\Bdot\<\<\otimes\< \Cdot)\<)$}
      node (13) [black, pos = 0.5 , scale=0.9] {$\Gamma^*(p^*\<\<\Adot\<\otimes\< \Gamma_{\<\!*} (\Bdot\<\<\otimes\< \Cdot)\<)$}
      node (15) [black, pos = 0.86 , scale=0.9] {$\Adot\<\otimes\<
      \Gamma^*\Gamma_{\<\!*} (\Bdot\<\<\otimes\< \Cdot)$};

      \draw[white] (0cm,8.5cm) -- +(0: \linewidth)
      node (12) [black, pos = 0.32 , scale=0.9] {$\Gamma^*\Gamma_{\<\!*} (\Gamma^*\<p^*\<\<\Adot\<\otimes\< (\Bdot\<\<\otimes\< \Cdot)\<)$}
      node (14) [black, pos = 0.68 , scale=0.9]
      {$\Gamma^*\<p^*\<\<\Adot\<\otimes\< \Gamma^*\Gamma_{\<\!*}
      (\Bdot\<\<\otimes\< \Cdot)$};

      \draw[white] (0cm,7.5cm) -- +(0: \linewidth)
      node (21) [black, pos = 0.14 , scale=0.9] {$\Gamma^*\Gamma_{\<\!*} (\Adot\<\otimes\< (\Gamma^*\<p^*\<\< \Bdot\<\<\otimes\< \Cdot)\<)$}
      node (23) [black, pos = 0.5 , scale=0.9] {$\Gamma^*(p^*\<\<\Adot\<\otimes\< \Gamma_{\<\!*} (\Gamma^*\<p^*\<\<\Bdot\<\<\otimes\< \Cdot)\<)$}
      node (25) [black, pos = 0.86 , scale=0.9] {$\Adot\<\otimes\<
      \Gamma^*\Gamma_{\<\!*} (\Gamma^*\<p^*\<\< \Bdot\<\<\otimes\<
      \Cdot)$};

      \draw[white] (0cm,6.5cm) -- +(0: \linewidth)
      node (22) [black, pos = 0.32 , scale=0.9] {$\Gamma^*\Gamma_{\<\!*} (\Gamma^*\<p^*\<\<\Adot\<\otimes\< (\Gamma^*\<p^*\<\<\Bdot\<\<\otimes\< \Cdot)\<)$}
      node (24) [black, pos = 0.68 , scale=0.9] {$\Gamma^*\<p^*\<\<\Adot\<\otimes\< \Gamma^*\Gamma_{\<\!*} (\Gamma^*\<p^*\<\<\Bdot\<\<\otimes\< \Cdot)$};
      
      \draw[white] (0cm,5.5cm) -- +(0: \linewidth)
      node (31) [black, pos = 0.14 , scale=0.9] {$\Gamma^*\Gamma_{\<\!*} (\<(\Adot\<\otimes\< \Gamma^*\<p^*\<\< \Bdot) \<\otimes\< \Cdot)$}
      node (33) [black, pos = 0.5 , scale=0.9] {$\Gamma^*\<(p^*\<\<\Adot\<\otimes\< (p^*\<\<\Bdot\<\<\otimes\< \Gamma_{\<\!*} \Cdot)\<)$}
      node (35) [black, pos = 0.86 , scale=0.9] {$\Adot\<\otimes\<
      \Gamma^*(p^*\<\< \Bdot\<\<\otimes\< \Gamma_{\<\!*} \Cdot)$};

      \draw[white] (0cm,4.5cm) -- +(0: \linewidth)
      node (32) [black, pos = 0.32 , scale=0.9] {$\Gamma^*\Gamma_{\<\!*} (\<(\Gamma^*\<p^*\<\<\Adot\<\otimes\< \Gamma^*\<p^*\<\<\Bdot) \<\otimes\< \Cdot)$}
      node (34) [black, pos = 0.68 , scale=0.9]
      {$\Gamma^*\<p^*\<\<\Adot\<\otimes\<
      \Gamma^*(p^*\<\<\Bdot\<\<\otimes\< \Gamma_{\<\!*} \Cdot)$};
   
       \draw[white] (0cm,3.5cm) -- +(0: \linewidth)
        node (44') [black, pos = 0.68 , scale=0.9]
      {$\Gamma^*\<p^*\<\<\Adot\<\otimes\<
            \Gamma^*\<p^*\<\<\Bdot\<\<\otimes\< \Gamma^*\Gamma_{\<\!*}\Cdot$};
            
      \draw[white] (0cm,2.5cm) -- +(0: \linewidth)
      node (43) [black, pos = 0.5 , scale=0.9] {$\Gamma^*\<(\<(p^*\<\<\Adot\<\otimes\< p^*\<\<\Bdot) \<\otimes\< \Gamma_{\<\!*} \Cdot)$}
      node (45) [black, pos = 0.86 , scale=0.9] {$\Adot\<\otimes\<
      \Gamma^*\<p^*\<\< \Bdot\<\<\otimes\< \Gamma^*\Gamma_{\<\!*}\Cdot$};
  
     \draw[white] (0cm,1.5cm) -- +(0: \linewidth)      
     node (42) [black, pos = 0.32 , scale=0.9] {$\Gamma^*\Gamma_{\<\!*} (\Gamma^*\<(p^*\<\<\Adot\<\otimes\< p^*\<\<\Bdot) \<\otimes\< \Cdot)$}
      node(44)  [black, pos = 0.68 , scale=0.9]
      {$\Gamma^*\<(p^*\<\<\Adot\<\otimes\<
            p^*\<\<\Bdot)\<\<\otimes\< \Gamma^*\Gamma_{\<\!*}\Cdot$};

      \draw[white] (0cm,0.5cm) -- +(0: \linewidth)
      node (51) [black, pos = 0.14 , scale=0.9] {$\Gamma^*\Gamma_{\<\!*} (\<(\Adot\<\otimes\< \Bdot) \<\otimes\< \Cdot)$}
      node (53) [black, pos = 0.5 , scale=0.9] {$\Gamma^*\<(p^*\<(\Adot\<\otimes\< \Bdot) \<\otimes\< \Gamma_{\<\!*}  \Cdot)$}
      node (55) [black, pos = 0.86 , scale=0.9] {$\Adot\<\otimes\<
      \Bdot\<\<\otimes\< \Gamma^*\Gamma_{\<\!*} \Cdot$};

      \draw[white] (0cm,-0.5cm) -- +(0: \linewidth)
      node (52) [black, pos = 0.32 , scale=0.9] {$\Gamma^*\Gamma_{\<\!*} (\Gamma^*\<p^*\<(\Adot\<\otimes\< \Bdot) \<\otimes\< \Cdot)$}
      node (54) [black, pos = 0.68 , scale=0.9]
      {$\Gamma^*\<p^*\<(\Adot\<\otimes\< \Bdot) \<\otimes\<
      \Gamma^*\Gamma_{\<\!*}  \Cdot$};

      \node (B21) at (2.93cm,3cm) [scale=0.92] {$\mathbf{B_{21}}$};
      \node (B22) at (5.17cm,3.5cm) [scale=1.2]{$\lift.73,\,\mathbf{B_{22}},$};
      \node (B23) at (7.41cm,3cm) [scale=0.9]{$\mathbf{B_{23}}$};
      \node (B24) at (9.8cm,2cm) [scale=0.9]{$\mathbf{B_{24}}$};
      
      \draw [double distance=2pt] (11) -- (12);
      \draw [->] (12) -- (13);
      \draw [->] (13) -- (14);
      \draw [double distance=2pt] (14) -- (15);
      \draw [double distance=2pt] (21) -- (22);
      \draw [->] (22) -- (23);
      \draw [->] (23) -- (24);
      \draw [double distance=2pt] (24) -- (25);
      \draw [double distance=2pt] (31) -- (32);
      \draw [->] (33) -- (34);
      \draw [double distance=2pt] (34) -- (35);
      \draw [double distance=2pt] (44') -- (45);
      \draw [double distance=2pt] (51) -- (52); 
      \draw [->] (42) -- (43);
      \draw [->] (52) -- (53);
      \draw [->] (53) -- (54);
      \draw [double distance=2pt] (54) -- (55);

      \draw [double distance=2pt] (11) -- (21);
      \draw [->] (21) -- (31);
      \draw [double distance=2pt] (31) -- (51);
      \draw [double distance=2pt] (12) -- (22);
      \draw [->] (22) -- (32);
      \draw [->] (32) -- (42);
      \draw [->] (42) -- (52);
      \draw [double distance=2pt] (13) -- (23);
      \draw [->] (23) -- (33);
      \draw [->] (33) -- (43);
      \draw [->] (43) -- (53);
      \draw [->] (43) -- (44);
      \draw [double distance=2pt] (14) -- (24);
      \draw [->] (24) -- (34);
      \draw [->] (34) -- (44');
      \draw [->] (44) -- (54);
      \draw [->] (44') -- (44);
      \draw [double distance=2pt] (15) -- (25);
      \draw [->] (25) -- (35);
      \draw [->] (35) -- (45);
      \draw [double distance=2pt] (45) -- (55);
    \end{tikzpicture}
\end{equation*}

The commutativity of the unlabeled diagrams is obvious.

Commutativity of $\mathbf{B_{21}}$ and of $\mathbf{B_{24}}$ results directly from the fact that the contravariant pseudofunctor $(-)^*$ is \emph{monoidal} (see \cite[3.6.7(b)]{li}, a proof of which is outlined in \cite[(3.6.10)]{li}).

Commutativity of $\mathbf{B_{22}}$ is given by \cite[3.4.7(iv)]{li} (with $f=\Gamma$, $A=p^*\<\<E$, $B=p^*\<\<F$ and $C=G\>$).

Finally, $\mathbf{B_{23}}$ is \emph{dual} (see \cite[3.4.5]{li}) to the largest commutative diagram 
\mbox{in \cite[(3.4.2.2)]{li},} \emph{mutatis mutandis}, and so is itself commutative.
\end{proof}

To complete Step IIB, it remains
to show that subdiagram~$\mathbf{B_{3}}$ in~\eqref{expand B} commutes. For this, it suffices to
apply the next Lemma, with $\Adot\set v^! \CO_{\<Z}$ and  $\Bdot\set v^*\<G\ (G\in\D_{\<Z})$, to the diagram
  \[
   \begin{tikzpicture}[scale=.8]
      \draw[white] (0cm,3.4cm) -- +(0: \linewidth)
      node (12)  [black, pos = 0.3675] {$X$}
      node (13)  [black, pos = 0.6275] {$Y$};
      \draw[white] (0cm,1.7cm) -- +(0: \linewidth)
      node (22)  [black, pos = 0.3675] {$X\times Z$}
      node (23) [black, pos = 0.6275] {$Y\times Z$};
      \draw[white] (0cm,0cm) -- +(0: \linewidth)
      node (32)  [black, pos = 0.3675] {$X$}
      node (33)  [black, pos = 0.6275] {$Y$};
 
      \draw [->] (12)  -- (13) node[above, midway, scale=0.75]{$u$};
      \draw [->] (32)  -- (33) node[below=1pt, midway, scale=0.75]{$u$};
      \draw [->] (22)  -- (23) node[below, midway, scale=0.75]{$r\<\<:=\<u\<\<\times\<\< 1\,$};
      \draw [->] (22) -- (32) node[left=1pt, midway, scale=0.75]{$p^{\prime}$};
      \draw [->] (12) -- (22) node[left=1pt, midway, scale=0.75]{$\Gamma^{\prime}$};
      \draw [->] (23) -- (33) node[right=1pt, midway, scale=0.75]{$p$};
      \draw [->] (13) -- (23) node[right=1pt, midway, scale=0.75]{$\Gamma$};
   \end{tikzpicture}
  \]
where $p$ and $p^\prime$ are the natural projections.

\begin{sublem}
\label{muynu}
Let
  \[
   \begin{tikzpicture}[scale=.8]
      \draw[white] (0cm,3.4cm) -- +(0: \linewidth)
      node (12)  [black, pos = 0.3675] {$X$}
      node (13)  [black, pos = 0.6275] {$Y$};
      \draw[white] (0cm,1.7cm) -- +(0: \linewidth)
      node (22)  [black, pos = 0.3675] {$X'$}
      node (23) [black, pos = 0.6275] {$Y'$};
      \draw[white] (0cm,0cm) -- +(0: \linewidth)
      node (32)  [black, pos = 0.3675] {$X$}
      node (33)  [black, pos = 0.6275] {$Y$};
      
      \node at (intersection of 12--23 and 13--22)[scale=.8]{$\Dd$};
 
      \draw [->] (12)  -- (13) node[above, midway, scale=0.75]{$u$};
      \draw [->] (32)  -- (33) node[below=1pt, midway, scale=0.75]{$u$};
      \draw [->] (22)  -- (23) node[below=1pt, midway, scale=0.75]{$r$};
      \draw [->] (22) -- (32) node[left=1pt, midway, scale=0.75]{$p^{\prime}$};
      \draw [->] (12) -- (22) node[left=1pt, midway, scale=0.75]{$\Gamma^{\prime}$};
      \draw [->] (23) -- (33) node[right=1pt, midway, scale=0.75]{$p$};
      \draw [->] (13) -- (23) node[right=1pt, midway, scale=0.75]{$\Gamma$};
   \end{tikzpicture}
  \]
be a commutative diagram of qcqs maps such that $p'\<\smallcirc \Gamma'\<=\<1_\sX$ and\/ \mbox{$p\smallcirc \Gamma =\id_Y\<$}. Let\/ $\mu=\mu^{}_{\Gamma\!,\,p}$ and\/ 
$\mu'=\mu^{}_{\Gamma^{\prime}\!,\>\>p^{\prime}}$ be defined as in\/ \eqref{def-of-mu}$,$ $\nu$ as in the paragraph before~\eqref{expand B} and\/ 
$\defnu{}=\defnu{\Dd}$ as in\/ \eqref{def-of-phi}. Then the following diagram commutes for all\/ $\Adot$ and\/ $\Bdot$  in\/ $\D_Y$.  
\[
    \begin{tikzpicture}
      \draw[white] (0cm,3.5cm) -- +(0: \linewidth)
      node (11) [black, pos = 0.3] {$\Gamma^{\prime *}\Gam{*}' u^* (\Adot\otimes \Bdot)$}
      node (12) [black, pos = 0.7] {$u^*\Gamma^* \Gamma_{\<\!*} (\Adot\otimes \Bdot)$};
      \draw[white] (0cm,2cm) -- +(0: \linewidth)
      node (21) [black, pos = 0.3] {$\Gamma^{\prime *}\Gam{*}' (u^* \Adot\otimes u^* \Bdot)$}
      node (22) [black, pos = 0.7] {$u^*(\Adot\otimes \Gamma^* \Gamma_{\<\!*}  \Bdot)$};
      \draw[white] (0cm,0.5cm) -- +(0: \linewidth)
      node (31) [black, pos = 0.3] {$ u^* \Adot\otimes \Gamma^{\prime *}\Gam{*}' u^* \Bdot$}
      node (32) [black, pos = 0.7] {$u^*\Adot\otimes u^*\Gamma^* \Gamma_{\<\!*}  \Bdot$};
      \draw [->] (12) -- (11) node[above, midway, sloped, scale=0.75]{$\defnu{ }$};
      \draw [->] (32) -- (31) node[below, midway, sloped, scale=0.75]{$\id \otimes \,\defnu{ }$};
      \draw [->] (11) -- (21) node[left=1pt, midway, scale=0.75]{$\via \>\nu$};
      \draw [->] (21) -- (31) node[left=1pt, midway, scale=0.75]{$\mu'$};
      \draw [->] (12) -- (22) node[right=.5pt, midway, scale=0.75]{$\via \, \mu$};
      \draw [->] (22) -- (32) node[right=1pt, midway, scale=0.75]{$\nu$};
    \end{tikzpicture}
\]
\end{sublem}

\begin{proof}
The definitions of $\mu$ and $\phi$ lead to the following expansion of  the preceding diagram,
where ``$=\!=$" indicates the isomorphism~$\ps^*\<$, and the other maps are the obvious ones:
\[
\mkern-2mu
  \begin{tikzpicture}[xscale=.7,yscale=1.25]
      \draw[white] (0cm,12cm) -- +(0: \linewidth)
      node (01) [black, pos = 0]  {$\Gamma^{\prime *}\Gam{*}' u^*\< (\Adot\<\otimes\< \Bdot)$}
      node (02) [black, pos=0.29] {}
      node (03) [black, pos=0.63]  {$\Gamma^{\prime *}r^*\Gamma_{\<\!*} (\Adot\<\otimes\<\Bdot)$}
      node (04) [black, pos = 1]  {$u^*\Gamma^* \Gamma_{\<\!*} (\Adot\<\otimes\<\Bdot)$};
      
      \draw[white] (0cm,10.9cm) -- +(0: \linewidth)
      node (11) [black, pos = 0] {}
      node (12) [black, pos=0.29]  {$\Gamma^{\prime *}\Gam{*}' u^*\< (\Gamma^*\mkern-.5mu p^*\<\<\Adot\<\otimes\<\Bdot)$}
      node (13) [black, pos=0.63] {}
      node (14) [black, pos = 1]  {$u^*\Gamma^* \Gamma_{\<\!*} (\Gamma^*\mkern-.5mu p^*\<\<\Adot\<\otimes\<\Bdot)$};
      
      \draw[white] (0cm,9.8cm) -- +(0: \linewidth)
      node (21) [black, pos = 0]  {$\Gamma^{\prime *}\Gam{*}'  (u^*\<\<\Adot\<\otimes\< u^*\<\<\Bdot)$}
      node (22) [black, pos=0.29] {}
      node (23) [black, pos=0.63]  {$\Gamma^{\prime *}r^*\Gamma_{\<\!*} 
                                      (\Gamma^*\mkern-.5mu  p^*\<\<\Adot\<\otimes\<\Bdot)$}
      node (24) [black, pos = 1] {};
      
      \draw[white] (0cm,8.7cm) -- +(0: \linewidth)
      node (31) [black, pos = 0] {}
      node (32) [black, pos=0.29]  {$\Gamma^{\prime *}\Gam{*}' 
                               (u^* \Gamma^*\mkern-.5mu p^*\<\<\Adot\<\otimes\< u^*\<\<\Bdot)$}
      node (33) [black, pos=0.63] {}
      node (34) [black, pos = 1]  {$u^*\Gamma^*  (p^*\<\<\Adot\<\otimes\< \Gamma_{\<\!*}\Bdot)$};
      
      \draw[white] (0cm,7.6cm) -- +(0: \linewidth)
      node (41) [black, pos = 0]  {$\Gamma^{\prime *}\Gam{*}'  (\Gamma^{\prime *}p^{\prime *}u^*\<\<\Adot\<\otimes\< u^*\<\<\Bdot)$}
      node (43) [black, pos=0.63]  {$\Gamma^{\prime *}r^* (p^*\<\<\Adot\<\otimes\< \Gamma_{\<\!*}\Bdot)$}
      node (44) [black, pos = 1] {};
      \draw[white] (0cm,6.5cm) -- +(0: \linewidth)
      node (51) [black, pos = 0] {}
      node (52) [black, pos=0.29]  {$\Gamma^{\prime *}\Gam{*}' (\Gamma^{\prime *}r^{ *}p^*\<\<\Adot\<\otimes\< u^*\<\<\Bdot)$}
      node (53) [black, pos=0.63] {}
      node (54) [black, pos = 1]  {$u^*\<(\Gamma^*\mkern-.5mu  p^*  \<\<\Adot\<\otimes\< \Gamma^*  \Gamma_{\<\!*}\Bdot)$};
      
      \draw[white] (0cm,4.9cm) -- +(0: \linewidth)
      node (61) [black, pos = 0]  {$\Gamma^{\prime *}  (p^{\prime *}u^*\<\<\Adot\<\otimes\< \Gam{*}'u^*\<\<\Bdot)$}
      node (62) [black, pos=0.29] {}
      node (63) [black, pos=0.63]  {$\Gamma^{\prime *} (r^*p^*\<\<\Adot\<\otimes\< r^*\Gamma_{\<\!*}\Bdot)$}
      node (64) [black, pos = 1] {};
      
     \draw[white] (0cm,3.8cm) -- +(0: \linewidth)
      node (71) [black, pos = 0] {}
      node (72) [black, pos=0.29]  {$\Gamma^{\prime *}(r^*p^*\<\<\Adot\<\otimes\< \Gam{*}'u^*\<\<\Bdot)$}
      node (73) [black, pos=0.63] {}
      node (74) [black, pos = 1]  {$u^*\Gamma^*\mkern-.5mu  p^*  \<\<\Adot\<\otimes\< u^*\Gamma^*  \Gamma_{\<\!*}\Bdot$};
      
      \draw[white] (0cm,2.7cm) -- +(0: \linewidth)
      node (81) [black, pos = 0]  {$\Gamma^{\prime *}  p^{\prime *}u^*\<\<\Adot\<\otimes\< \Gamma^{\prime *}\Gam{*}'u^*\<\<\Bdot $}
      node (82) [black, pos=0.29] {}
      node (83) [black, pos=0.63]  {$\Gamma^{\prime *} r^*p^*\<\<\Adot\<\otimes\< \Gamma^{\prime *}r^*\Gamma_{\<\!*}\Bdot$}
      node (84) [black, pos = 1] {};
      \draw[white] (0cm,1.6cm) -- +(0: \linewidth)
      node (91) [black, pos = 0] {}
      node (92) [black, pos=0.29]  {$\Gamma^{\prime *}r^*p^*\<\<\Adot\<\otimes\< \Gamma^{\prime *}\Gam{*}'u^*\<\<\Bdot$}
      node (93) [black, pos=0.63] {};
      
     \draw[white] (0cm,0.5cm) -- +(0: \linewidth)
      node (94) [black, pos = 1]  {$u^*\<\<\Adot\<\otimes\< u^*\Gamma^*  \Gamma_{\<\!*}\Bdot$}
      node (X1) [black, pos = 0]  {$u^*\<\<\Adot\<\otimes\< \Gamma^{\prime *}\Gam{*}'u^*\<\<\Bdot $}
      node (X2) [black, pos=0.29] {}
      node (X3) [black, pos=0.63]   {$u^*\<\<\Adot\<\otimes\< \Gamma^{\prime *}r^*\Gamma_{\<\!*}\Bdot$}
      node (X4) [black, pos = 1] {};
      \draw [->] (01) -- (21)   ;
      \draw [double distance=2pt]
                 (21) -- (41)   ;
      \draw [->] (41) -- (61)   ;
      \draw [->] (61) -- (81)   ;
      \draw [double distance=2pt]
                 (81) -- (X1)   ;
      \draw [->] (12) -- (32)   ;
      \draw [double distance=2pt]
                 (32) -- (52)   ;
      \draw [->] (52) -- (72)   ;
      \draw [->] (72) -- (92)   ;
      \draw [double distance=2pt]
                 (03) -- (23)   ;
      \draw [->] (23) -- (43)   ;
      \draw [->] (43) -- (63)   ;
      \draw [->] (63) -- (83)   ;
      \draw [double distance=2pt]
                 (83) -- (X3)   ;
      \draw [double distance=2pt]
                 (04) -- (14)   ;
      \draw [->] (14) -- (34)   ;
      \draw [->] (34) -- (54)   ;
      \draw [->] (54) -- (74)   ;
      \draw [double distance=2pt]
                 (74) -- (94)   ;
      \draw [double distance=2pt] (04) -- (03)   ;
      \draw [->] (03) -- (01)   ;
      \draw [->] (83) -- (92)   ;
      \draw [double distance=2pt]
                 (92) -- (81)   ;
      \draw [->] (X3) -- (X1)   ;
      \draw [->] (23) -- (12)   ;
      \draw [double distance=2pt]
                 (01) -- (12)   ;
      \draw [double distance=2pt]
                 (14) -- (23)   ;
      \draw [double distance=2pt]
                 (32) -- (21)   ;
      \draw [double distance=2pt]
                 (52) -- (41)   ;
      \draw [double distance=2pt]
                 (72) -- (61)   ;
      \draw [->] (63) -- (72)   ;
      \draw [double distance=2pt]
                 (23)   ;
      \draw [double distance=2pt]
                 (34) -- (43)   ;
      \draw [double distance=2pt]
                 (74) -- (83)   ;
      \draw [double distance=2pt]
                 (92) -- (X1)   ;
      \draw [double distance=2pt]
                 (94) -- (X3)   ;
                 
      \node (1) at (intersection of 12--63 and 23--72) [scale=1.275]{$\lift-.9,\mathbf{B_{31}},\quad$};
      \node (2) at (intersection of 74--43 and 34--83) [scale=0.95]{$\mathbf{B_{32}}$};

  \end{tikzpicture}
\]
Here commutativity of the unlabeled subdiagrams is clear, by naturality of the transformations involved and by the transitivity property of pseudofunctoriality isomorphisms; that of 
$\mathbf{B_{31}}$ follows from \cite[Proposition~{3.7.3}]{li} with $(f,f'\!,\>g,g'\<\<,P,Q)\set(\Gamma,
\Gamma'\<\<,r,u,p^*\<\<E,F\>)$; and that of $\mathbf{B_{32}}$ results from the fact that the contravariant pseudofunctor $(-)^*$ is \emph{monoidal} (see \cite[(3.6.7)(b) and (3.6.10)]{li}).
\end{proof}

This completes the proof that subdiagram $\mathbf B$ in \eqref{expand sharps} commutes.
\end{cosa}

\begin{cosa}[\emph{Step IIA\kern2pt}]
We show next  that diagram $ \mathbf{A}$ commutes.

Recall the diagram formed by the last two rows of \eqref{pasoI}:
\begin{equation*}
\CD
@.X @>g>> X\<\times_Z Y @>k>> X\times Y @> p^{}_{Y}>> Y \\
@.@. @VrVV @V\id\<\times \>v VV @VVvV\\
@.@. X @>>\lift1.1,\Gam{vu},> X\times Z @>>\lift1,p^{}_{\<Z},> Z
\endCD
\end{equation*}
where  $k$ is the natural closed immersion; 
$g$ is the graph of $u$, i.e., the unique closed immersion such that $kg=\Gam{u}$; 
$r$ is projection onto the first factor; and $p^{}_{\>Y}$, 
$p^{}_{\<Z}$ are the projections onto the second factor, so that $p^{}_{\>Y}kg=u$ and 
$p^{}_{\<Z}\>\Gam{vu}=vu$. 
Recall also the isomorphisms $\bchadmirado{}$ (\S\ref{indsq}) and $\theta$ (\S\ref{thetaiso}).
Further, let $p\colon X\times Y\to X$ be the canonical projection, so that $k^*p^*=r^*$ and $\Gams{u}p^*=1\set\id_\sX$. 

Referring to the definitions of its constituent maps, expand $\mathbf{A}$ as follows,
\begin{figure}
\[  \mkern-6mu
    \begin{tikzpicture}[xscale=.85, yscale=1.05]
      
      \draw[white] (0cm,15.5cm) -- +(0: \linewidth)
      node (01) [black, pos = 0][scale=1] {$\Gams{u}\Gam{u*} u^!v^* $}
      node (02) [black, pos=.26] {}
      node (03) [black, pos=.64][scale=1] {$ u^!\CO_Y\<\otimes\mkern-.5mu \Gams{u}\Gam{u*}u^*\<v^*$}
      node (04) [white, pos = 1] {};
      
      \draw[white] (0cm,14.2cm) -- +(0: \linewidth)
      node (11) [black, pos = 0] {}
      node (12) [black, pos=.26][scale=1] {$\Gams{u}\Gam{u*} (u^!\CO_Y\<\otimes\< u^*\<v^*\<)$}
      node (13) [black, pos=.64] {}
      node (14) [black, pos = .9][scale=1] {$\mkern-11mu u^!\CO_Y\<\otimes\< u^*\delta_Y^*  \delta^{}_{Y\<*}v^*$};
      \draw[white] (0cm,13.5cm) -- +(0: \linewidth)
      node (21) [black, pos = 0][scale=1] {$\Gams{u} k_* g_*u^!v^* $}
      node (22) [black, pos=.26] {}
      node (23) [black, pos=.64] {}
      node (24) [black, pos = 1] {};
      
      \draw[white] (0cm,12.2cm) -- +(0: \linewidth)
      node (31) [black, pos = 0] {}
      node (32) [black, pos=.26][scale=1] {$\Gams{u} k_*  g_*  (u^!\CO_Y\<\otimes\< u^*\<v^*\<)$}
      node (33) [black, pos=.64][scale=1] {$\Gams{u} p^*\<u^!\CO_Y\<\otimes\mkern-.5mu \Gams{u} \Gam{u*} u^*\<v^*$}
      node (34) [black, pos = 1] {};
      
      \draw[white] (0cm,11.5cm) -- +(0: \linewidth)
      node (41) [black, pos = 0] [scale=1]{$\Gams{u} k_* g_*(p^{}_{\>Y} kg)^!v^* $}
      node (42) [black, pos=.26] {}
      node (43) [black, pos=.64] {}
      node (44) [black, pos = 1] {};
      
      \draw[white] (0cm,10.5cm) -- +(0: \linewidth)
      node (51) [black, pos = 0] {}
      node (52) [black, pos=.26][scale=1] {$\Gams{u} k_*  g_*  (g^* r^*\<u^!\CO_Y\<\otimes\< u^*\<v^*\<)$}
      node (53) [black, pos=.64] {}
      node (54) [black, pos = 1] {};
      
      \draw[white] (0cm,9.5cm) -- +(0: \linewidth)
      node (61) [black, pos = 0][scale=1] {$\Gams{u} k_* (p^{}_{\>Y} k)^!v^* $}
      node (62) [black, pos=.26] {}
      node (63) [black, pos=.64][scale=1] {$\Gams{u} p^*\<u^!\CO_Y\<\otimes\mkern-.5mu \Gams{u} k_* g_*u^*\<v^*$}
      node (64) [black, pos = 1] {};
      
     \draw[white] (0cm,9.07cm) -- +(0: \linewidth)
     node (A2) [black, pos=.13, scale=.85] {$\mathbf{A_2}$};
     
     \draw[white] (0cm,8.5cm) -- +(0: \linewidth)
      node (71) [black, pos = 0] {}
      node (72) [black, pos=.26][scale=1] {$\Gams{u} k_*  ( r^*\<u^!\CO_Y\<\otimes\< g_*u^*\<v^*\<)$}
      node (73) [black, pos=.64] {}
      node (74) [black, pos = 1] {};
      \draw[white] (0cm,7.5cm) -- +(0: \linewidth)
      node (81) [black, pos = 0][scale=1] {$\Gams{u} k_*  r^*(vu)^! $}
      node (82) [black, pos=.26] {}
      node (83) [black, pos=.64][scale=1] {$\Gams{u} p^*\<u^!\CO_Y\<\otimes\mkern-.5mu \Gams{u} k_* r^*\<u^*v^!$}
      node (84) [black, pos = 1] {};
      \draw[white] (0cm,6.5cm) -- +(0: \linewidth)
      node (91) [black, pos = 0] {}
      node (92) [black, pos=.26][scale=1] {$\Gams{u} k_*  ( r^*\<u^!\CO_Y\<\otimes\< r^*\<u^*v^!)$}
      node (93) [black, pos=.64] {}
      node (94) [black, pos = 1] {};
      
     \draw[white] (0cm,5.5cm) -- +(0: \linewidth)
      node (X1) [black, pos = 0][scale=1] {$\Gams{u} k_*  r^* u^! v^! $}
      node (X2) [black, pos=.26] {}
      node (X3) [black, pos=.64] {}
      node (X4) [black, pos = 1] {};
      
     \draw[white] (0cm,4.3cm) -- +(0: \linewidth)
      node (Y1) [black, pos = 0] {}
      node (Y2) [black, pos=.26][scale=1] {$\Gams{u} k_*  r^*( u^!\CO_Y\<\otimes\< u^*v^!)$}
      node (Y3) [black, pos=.64] {}
      node (Y4) [black, pos = 1] {};
      
     \draw[white] (0cm,3.5cm) -- +(0: \linewidth)
      node (Z1) [black, pos = 0][scale=1] {$\Gams{u} (1\<\<\times\<\< v\<)\<^*\Gam{vu*}u^!v^! $}
      node (Z2) [black, pos=.26] {}
      node (Z3) [black, pos=.64][scale=1] {$\<\Gams{u} p^*\<u^!\CO_Y\<\<\otimes\< 
                    \Gams{u} (1\<\<\times\<\< v\<)\<^* \Gam{vu*}u^{\<*}\<v^!$}
      node (Z4) [black, pos = 1] {};
      
     \draw[white] (0cm,2.1cm) -- +(0: \linewidth)
      node (T1) [black, pos = 0] {}
      node (T2) [black, pos=.26][scale=1] {$\Gams{u} (1\<\<\times\<\< v\<)\<^* {\Gamma_{vu}}_* ( u^!\CO_Y\<\otimes\< u^*v^!)$}
      node (T3) [black, pos=.64] {}
      node (T4) [black, pos = 1] {};
      
      \draw[white] (0cm,.7cm) -- +(0: \linewidth)
      node (V1) [black, pos = 0][scale=1] {$\Gams{vu}\Gam{vu*} u^!v^! $}
      node (V2) [black, pos=.26] {}
      node (V3) [black, pos=.64][scale=1] {$u^!\CO_Y\<\otimes\mkern-.5mu \Gams{vu}\Gam{vu*}u^*v^!$}
      node (V4) [black, pos = 1] {};
      
      \draw[white] (0cm,-0.7cm) -- +(0: \linewidth)
      node (W1) [black, pos = 0] {}
      node (W2) [black, pos=.26][scale=1] {$\Gams{vu} \Gam{vu*}( u^!\CO_Y\<\otimes\< u^*v^!)$}
      node (W3) [black, pos=.64] {}
      node (W4) [black, pos = .9][scale=1] {$\mkern-10mu u^!\CO_Y\<\otimes\< u^*\Gams{v}\Gam{v*}v^!$};
      
      \draw [double distance=2pt]
                 (01) -- (21) node[left=1pt, midway, scale=0.75]{$\via\ps_*$};
      \draw [double distance=2pt]
                 (21) -- (41) node[left=1pt, midway, scale=0.75]{$$};
      \draw [->] (41) -- (61) node[left, midway]{\scalebox{.75}{$\via$}$\,\smallint_{\mspace{-3mu}\lift.65,g,}^{\>p^{}_{\halfsize{$\sst Y$}}k}$};
      \draw [->] (61) -- (81) node[left=1pt, midway, scale=0.75]{$\via\bchadmirado{}^{-\<1}$};
      \draw [double distance=2pt]
                 (81) -- (X1) node[left=1pt, midway, scale=0.75]{$\via\ps^!$};
                       \draw [->] (X1) -- (Z1) node[left=1pt, midway, scale=0.75]{$\via\>\theta^{-\<1}$};
      \draw [double distance=2pt]
                 (Z1) -- (V1) node[left, midway, scale=0.75]{$\via\ps^*$};
      \draw [double distance=2pt]
                 (12) -- (32) node[right=1pt, midway, scale=0.75]{$\via\ps_*$};
      \draw [double distance=2pt]
                 (32) -- (52) node[right=1pt, midway, scale=0.75]{$\via\ps^*$};
      \draw [->] (52) -- (72) node[right=1pt, midway, scale=0.75]{\eqref{projection}};
      \draw [->] (72) -- (92) node[right=1pt, midway, scale=0.75]{$??$};
      \draw [->] (92) -- (Y2) node[right=1pt, midway, scale=0.75]{$\eqref{^* and tensor}$};
      \draw [->] (Y2) -- (T2) node[right, midway, scale=0.75]{$\via\theta^{-\<1}$};
      \draw [double distance=2pt]
                 (T2) -- (W2) node[right=1pt, midway, scale=0.75]{$\ps^*$};
      \draw [double distance=2pt]
                 (03) -- (33) node[left=1pt, midway, scale=0.75]{$\via\ps^*$};
      \draw [double distance=2pt]
                 (33) -- (63) node[left=1pt, midway, scale=0.75]{$\via\ps_*$};
      \draw [->] (63) -- (83) node[left=1pt, midway, scale=0.75]{$??$};
      \draw [->] (83) -- (Z3) node[left=1pt, midway, scale=0.75]{$\via\>\theta^{-\<1}$};
      \draw [double distance=2pt]
                 (Z3) -- (V3) node[left=1pt, midway, scale=0.75]{$\ps^*$};
      \draw [->] (14) -- (W4) node[right=1pt , midway, scale=0.75]{$\via\>\Da_v$};
      
      \draw [double distance=2pt] (01) -- (12) node[above, midway, sloped, scale=0.75, black]{$ $}
                              node[below, midway, sloped, scale=0.75, white]{$\sim$};
      \draw [->] (12) -- (03) node[above=-.5pt, midway, sloped, scale=0.75, black]{\rotatebox{5}{$\mu_{u}$}};
      \draw [->] (03) -- (14) node[above=-2pt, midway, sloped, scale=0.75, black]
                                                                  {\rotatebox{-4}{$\quad\,\> \via\phi^{-\<1}$}};
      \draw [double distance=2pt] (21) -- (32);
      \draw [double distance=2pt] (X1) -- (Y2);
      \draw [double distance=2pt] (Z1) -- (T2);
      \draw [double distance=2pt] (V1) -- (W2) ;
      \draw [->] (72) -- (63) node[below, midway, sloped, scale=0.75]{$???$};
      \draw [->] (92) -- (83) node[below, midway, sloped, scale=0.75]{$???$};
      \draw [->] (T2) -- (Z3) node[below, midway, sloped, scale=0.75]{$????$};
      \draw [->] (W2) -- (V3) node[below, midway, sloped, scale=0.75, black]{\rotatebox{5}{$\mu_{vu}$}};
      \draw [->] (V3) -- (W4) node[below=-1pt, midway, sloped, scale=0.75, black]{\rotatebox{-4.5}{$\mkern-20mu\via\>\phi^{-\<1}$}};
      
      \node (A1) at (intersection of 32--53 and 33--52) [scale=0.85]{$ \mathbf{A_{1}}$};
      \node (A3) at (intersection of 84--63 and 83--64) [scale=0.85]{$ \mathbf{A_{3}\mkern59mu}$};
      \node (A4) at (intersection of X2--Y3 and X3--Y2) [scale=0.85]{$ \mathbf{A_{4_{\mathstrut}}}$};
      \node (A5) at (intersection of T2--V3 and Z3--W2) [scale=0.85]{$ \mathbf{A_{5}^{\mathstrut}}$};
    \end{tikzpicture}
\]
\end{figure}
where the maps labeled ??\ are induced by a map
\(
\xi \colon g_*u^*\<v^*\lto r^*u^*v^!
\)
to be defined below (\ref{def-of-xi}); 
the ones labeled ???\ are induced by the  composition
(with $\id\set\id_{X\times_{\<Z}Y}$)
\begin{align*}
       \Gams{u} k_*  ( r^*\<u^!\CO_Y\<\otimes \id)
       &\overset{\via\ps^*_{\phantom{.}}}{=\!=} \Gams{u} k_*  ( k^*\<p^*\<u^!\CO_Y\<\otimes \id)      \\
       &\iso \Gams{u}  ( p^*\<u^!\CO_Y\<\otimes k_*)         \tag{via \eqref{projection}}\\
       &\underset{\lift1.25,\nu,}{\iso} \Gams{u} p^*\<u^!\CO_Y\<\otimes\Gams{u}\<k_* \tag{see \eqref{^* and tensor}};
\end{align*}
and with $q\colon X\times Z\lto X$ the canonical projection, so that
$p^*=(1\<\<\times\<\<v\<)^{\<*}q^*\<$, the map ????\ is 
the composite isomorphism
\begin{align*}
      \Gams{u} (1\<\<\times\<\<v\<)^{\<*}\>\Gam{vu*}(u^!\CO_Y\otimes u^* v^!)
       &\overset{\via\ps^*_{\phantom{.}}}{=\!=} \Gams{u}  (1\<\<\times\<\<v\<)^{\<*}\>\Gam{vu*}(\Gams{vu}q^*\<u^!\CO_Y\otimes u^* v^!) \\ 
       &\iso\Gams{u}  (1\<\<\times\<\<v\<)^{\<*}(q^*\<u^!\CO_Y\otimes \Gam{vu*}u^* v^!) \tag{see \eqref{projection}}\\
       &\underset{\nu}{\iso} \Gams{u} (1\<\<\times\<\<v\<)^{\<*}q^*\<u^!\CO_Y\otimes\Gams{u}(1\<\<\times\<\<v\<)^{\<*}\>\Gam{vu*}u^* v^!\tag{see \eqref{^* and tensor}} \\[-3pt]
       & \underset{\via\ps^*}{=\!=} \Gams{u}p^*u^!\CO_Y\otimes\Gams{u}(1\<\<\times\<\<v\<)^{\<*}\>\Gam{vu*}u^* v^! .
\end{align*}

\pagebreak
Commutativity of the unlabeled squares is transparent. 

Commutativity of $ \mathbf{A_5}$ becomes clear upon expansion of ????\
and\va{.6} $\mu^{}_{vu}$ according to their  definitions,\va{.4} and identification via the pseudofunctor $(-)^*$ of $\Gams{u}(1\<\times\<v)^*$ 
with $\Gams{vu}$. Details are left to the reader.

\pagebreak[3]
Commutativity of $ \mathbf{A_1}$ can be seen by expanding it as follows,
according to the definitions of the maps involved, with \mbox{$E\set u^!\OY$}, 
$F\set u^*\<v^*G\ \>(G\in\D_{\<Z})$,  and $\prj_{\bullet}$ denoting a projection isomorphism, 
see \eqref{projection}:
\[
\def\1{$\Gams{u}\Gam{u*}(E\otimes F\>)$}
\def\2{$E\otimes\Gams{u}\Gam{u*}F$}
\def\3{$\Gams{u}k_*g_*(E\otimes F\>)$}
\def\4{$\quad\Gams{u}p^*\<\<E\otimes\Gams{u}\Gam{u*}F$}
\def\5{$\Gams{u}k_*g_*(g^*r^*\<\<E\otimes F\>)$}
\def\6{$\Gams{u}p^*\<\<E\otimes\Gams{u}k_*g_*F$}
\def\7{$\Gams{u}k_*(r^*\<\<E\otimes g_*F\>)$}
\def\8{$\Gams{u}k_*g_*(g^*k^*p^*\<\<E\otimes F\>)$}
\def\9{$\Gams{u}\<(p^*\<\<E\<\otimes\< k_*g_*F\>)$}
\def\ten{$\Gams{u}k_*(k^*p^*\<\<E\otimes g_*F\>)$}
\def\lvn{$\Gams{u}\Gam{u*}(\Gams{u}p^*\<\<E\otimes F\>)$}
\def\twv{$\Gams{u}(p^*\<E\otimes\Gam{u*}F\>)$}
\def\thn{$\Gams{u}\k_8g_*(\Gams{u}p^*\<\<E\otimes F\>)$}
 \bpic[xscale=3.2, yscale=1.2]
  \node(11) at (1,-1)[scale=.95]{\1};
  \node(14) at (3.9,-1)[scale=.95]{\2};

  \node(21) at (1,-2)[scale=.95]{\3};
  \node(22) at (2.25,-2)[scale=.95]{\lvn};
  \node(23) at (3.4,-2)[scale=.95]{\twv};

  \node(31) at (1,-3)[scale=.95]{\5};
  \node(32) at (2.25,-3)[scale=.95]{\5};
  \node(34) at (3.9,-3)[scale=.95]{\4};

  \node(42) at (2.25,-4)[scale=.95]{\8};
  \node(43) at (3.4,-4)[scale=.95]{\9};
  
  \node(51) at (1,-5)[scale=.95]{\7};
  \node(52) at (2.25,-5)[scale=.95]{\ten};
  \node(54) at (3.9,-5)[scale=.95]{\6};
  
   \draw[->] (11)--(14) node[above, midway, scale=.75]{$\mu^{}_u$};
  
   \draw[->] (22)--(23) node[above, midway, scale=.75]{$\prj^{}_\Gamma$};
  
   \draw[double distance=2pt] (51)--(52) node[below=1pt, midway, scale=.75]{$\via\ps^*$};

   \draw[double distance=2pt] (11)--(21) node[left=1pt, midway, scale=.75]{$\via\ps_*$};
   \draw[double distance=2pt] (21)--(31) node[left=1pt, midway, scale=.75]{$\via\ps^*$};
   \draw[->] (31)--(51) node[left, midway, scale=.75]{$\prj^{}_g$};
  
   \draw[double distance=2pt] (22)--(32) node[left, midway, scale=.75]{$\via\ps_*$};
   \draw[double distance=2pt] (32)--(42) node[left, midway, scale=.75]{$\via\ps^*$};
   \draw[->] (42)--(52) node[left, midway, scale=.75]{$\prj^{}_g$};

   \draw[double distance=2pt] (23)--(43) node[left=1pt, midway, scale=.75]{$\via\ps_*$};
   
   \draw[double distance=2pt] (14)--(34) node[right=1pt, midway, scale=.75]{$\via\ps^*$};
   \draw[double distance=2pt] (34)--(54) node[right=1pt, midway, scale=.75]{$\via\ps_*$};
   
     \draw[double distance=2pt] (11)--(22) node[below=-3pt, midway, sloped, scale=.75]
     {\rotatebox{21.5}{$\via\ps^*\ $}};
     
     \draw[->] (23)--(34)node[above=-2.5pt, midway, scale=.75]{$\mkern25mu\nu$};
     \draw[double distance=2pt] (31)--(42) node[below=-4pt, midway, sloped, scale=.75]
     {\rotatebox{22}{$\via\ps^*\quad$}};

     \draw[->] (43)--(54) node[below=-2pt, midway, scale=.75]{$\nu\quad\ $};
       
     \draw[->] (52)--(43) node[below=-1.5pt, midway, scale=.75]{$\quad\ \prj^{}_k$};

    \node at (intersection of 23--52 and 32--43) [scale=.85] {$\mathbf{A_{11}}$};

 \epic
\]
Here, commutativity of the unlabeled subdiagrams is easily checked; 
and that of $\mathbf{A_{11}}$ results from transitivity of the projection isomorphism
with respect to the composition $\Gam{u}=kg$, cf.~\cite[Proposition 3.7.1]{li}.

 As for $ \mathbf{A_{4}}$, apply \cite[Proposition~{3.7.3}]{li} to 
\begin{equation*}
    \begin{tikzpicture}[yscale=.9]
      \draw[white] (0cm,2.5cm) -- +(0: \linewidth)
      node (G) [black, pos = 0.41] {$X\times_Z Y$}
      node (H) [black, pos = 0.59] {$X$};
      \draw[white] (0cm,0.5cm) -- +(0: \linewidth)
      node (E) [black, pos = 0.41] {$X\times Y$}
      node (F) [black, pos = 0.59] {$X\times Z$};
      \draw [->] (G) -- (H) node[above=1pt, midway, sloped, scale=0.75]{$r$};
      \draw [->] (E) -- (F) node[below=1pt, midway, sloped, scale=0.75]{$1\times v$};
      \draw [->] (G) -- (E) node[left=1pt,  midway, scale=0.75]{$k$};
      \draw [->] (H) -- (F) node[right=1pt, midway, scale=0.75]{$\Gamma_{v u}=\Gamma$};
    \end{tikzpicture}
  \end{equation*}
to obtain the commutativity of the following diagram, with $P\set q^*u^!\OY$, \mbox{$Q\set u^*v^!G\ (G\in\D_{\<Z}$),} $\nu_\bullet$ coming from \eqref{^* and tensor}, and $\prj_{\bullet}$ denoting a projection isomorphism, 
see \eqref{projection}---commutativity from which, with a bit of patience, one readily deduces commutativity of $ \mathbf{A_{4}}$ (details left to the reader):

\[\mkern-12mu
   \begin{tikzpicture}[xscale=1.1]
      \draw[white] (0cm,6.50cm) -- +(0: \linewidth)
      node (01) [black, pos = 0.17,  scale=.95] {$k_*(k^*(1\<\times\<v)^*\<\<P \otimes r^*Q)$}
      node (02) [black, pos = 0.485, scale=.95]  {$k_*(r^*\Gamma^*\<\<P \otimes r^*Q)$}
      node (03) [black, pos = 0.8, scale=.95] {$k_*r^*(\Gamma^*\<\<P \otimes Q)$};
      
      \draw[white] (0cm,5.25cm) -- +(0: \linewidth)
      node (11) [black, pos = 0.17,  scale=.95] {}
      node (12) [black, pos = 0.485, scale=.95]  {}
      node (13) [black, pos = 0.8, scale=.95] {$(1\<\times\<v)^*\Gamma_{\<\!*} (\Gamma^*\<\<P\otimes Q)$};
      
      \draw[white] (0cm,4cm) -- +(0: \linewidth)
      node (21) [black, pos = 0.17,  scale=.95] {$(1\<\times\<v)^*\<\<P\otimes k_*r^*Q$}
      node (22) [black, pos = 0.485,  scale=.95] {$(1\<\times\<v)^*\<\<P\otimes (1\<\times\<v)^*\Gamma_{\<\!*}Q$}
      node (23) [black, pos = 0.8, scale=.95] {$(1\<\times\<v)^*
                                                  (P\otimes \Gamma_{\<\!*}Q)$};
                                                  
      \draw[white] (0cm, 2.75cm) -- +(0: \linewidth)
      node (31) [black, pos = 0.17,  scale=.95] {$(1\<\times\<v)^*q^*u^!\CO_Y\<
                                                  \otimes k_*r^*Q$}
      node (32) [black, pos = 0.55, scale=.95] {$(1\<\times\<v)^*q^*u^!\CO_Y\<
                                                  \otimes (1\<\times\<v)^*\Gamma_{\<\!*}Q$};
                                                  
      \draw[white] (0cm,1.5cm) -- +(0: \linewidth)
      node (41) [black, pos = 0.17,  scale=.95] {$p^*u^!\CO_Y\<
                                                  \otimes k_*r^*Q$}
      node (42) [black, pos = 0.55, scale=.95] {$p^*u^!\CO_Y\<
                                                  \otimes (1\<\times\<v)^*\Gamma_{\<\!*}Q$};
                                                  
      \draw [->] (13) -- (03) node[right=1pt, midway, scale=0.75]{$\theta$};
      \draw [->] (23) -- (22) node[below=1pt, midway,  scale=0.75]{$\nu^{}_{1\<\times v}$};
      \draw [->] (22) -- (21) node[below=1pt, midway, scale=0.75]{$\via\>\theta$};
      \draw [->] (21) -- (01) node[left=1pt, midway, scale=0.75]{$\prj^{}_k$};
      \draw [->] (23) -- (13) node[right=1pt, midway, scale=0.75]{via $\prj^{}_\Gamma$};
      \draw [->] (03) -- (02) node[above=1pt, midway,  scale=0.75]{$k_*\nu^{}_r$};
      \draw [->] (03) -- (02) node[right, midway, scale=0.75]{};
      \draw [->] (32) -- (31) node[below=1pt, midway, scale=0.75]{$\via\>\theta$};
      \draw [->] (42) -- (41) node[below=1pt, midway, scale=0.75]{$\via\>\theta$};

      \draw [double distance=2pt]  (02) -- (01) node[above, midway, scale=0.75]{$\via \ps^*$};
      \draw [double distance=2pt] (31) -- (21) node[left, midway, scale=0.75]{};
      \draw [double distance=2pt] (32) -- (22) node[left, midway, scale=0.75]{};
      \draw [double distance=2pt] (41) -- (31) node[left=1pt, midway, scale=0.75]{via $\ps^*$};
      \draw [double distance=2pt] (42) -- (32) node[right=1pt, midway, scale=0.75]{via $\ps^*$};
  \end{tikzpicture}
\]

This leaves us with $ \mathbf{A_{2}}$ and $ \mathbf{A_{3}}$, for which we first need to define the above map $\xi$. Consider the fiber square diagram, with $1\set\id_Y$,
\[
 \CD
    \begin{tikzpicture}[xscale=1.1, yscale=1.1]
      \draw[white] (0cm,6.95cm) -- +(0: .5\linewidth)
      node (21) [black, pos = 0] {$X$}
      node (22) [black, pos = 0.5] {$Y$};
      \draw[white] (0cm,4.8cm) -- +(0: .5\linewidth)
      node (31) [black, pos = 0] {$X\times_Z Y$}
      node (32) [black, pos = 0.5] {$Y\times_Z Y$}
      node (33) [black, pos = 1] {$Y$};      
      \draw[white] (0cm,2.65cm) -- +(0: .5\linewidth)
      node (41) [black, pos = 0] {$X\times Y$}
      node (42) [black, pos = 0.5] {$Y\times Y$}
      node (43) [black, pos = 1] {$Y\times Z$};
      \draw[white] (0cm,0.5cm) -- +(0: .5\linewidth)
      node (52) [black, pos = 0.5] {$Y$}
      node (53) [black, pos = 1] {$Z$};
      \draw [->] (21) -- (31) node[left=1pt, midway, scale=0.75]{$g$};
      \draw [->] (31) -- (41) node[left=1pt, midway, scale=0.75]{$k$};
      \draw [->] (22) -- (32) node[right=1pt, midway, scale=0.75]{$\delta_v$};
      \draw [->] (32) -- (42) node[right=1pt, midway, scale=0.75]{$i^\prime$};
      \draw [->] (42) -- (52) node[right=1pt, midway, scale=0.75]{$p^{\prime}_Y$};
      \draw [->] (33) -- (43) node[right=1pt, midway, scale=0.75]{$\Gam{v}$};
      \draw [->] (43) -- (53) node[right=1pt, midway, scale=0.75]{$p'_{\<\<Z}$};
      \draw [->] (21) -- (22) node[above=1pt, midway, sloped, scale=0.75]{$u$};
      \draw [->] (31) -- (32) node[above=1pt, midway, sloped, scale=0.75]{$w\set u\<\times_{\<Z} 1\,$};
      \draw [->] (32) -- (33) node[above=1pt, midway, sloped, scale=0.75]{$t_1$};
      \draw [->] (41) -- (42) node[below=1pt, midway, scale=0.75]{$u\times 1$};
      \draw [->] (42) -- (43) node[below=1pt, midway, sloped, scale=0.75]{$1\times v$};
      \draw [->] (52) -- (53) node[below=1pt, midway, sloped, scale=0.75]{$v$};
      
      \node at (intersection of 21--32 and 22--31)[scale=.9]{$\De$} ;
      \node at (intersection of 31--42 and 32--41)[scale=.9]{$\Df$} ;
      \node at (intersection of 32--43 and 33--42)[scale=.9]{$\Dg$} ;
       
    \end{tikzpicture}
  \endCD
\]
Here $t_1$ is the projection onto the first factor, $k$ and $i'$ are the natural maps, $g$ and $\Gam{v}$ are graph maps (of $u$ and $v$ respectively), $\delta_v$ is the diagonal map, and $p^\prime_Y,$
$p'_{\<\<Z}$  are the projections onto the second factor, so that $p'_{\<\<Z}\smallcirc\Gam{v}=v$. Setting $t_2\set p_{\>Y}^{\prime}i^{\prime}$, one has then
 the composite functorial map
\begin{equation}
\label{paraxi}
\bar\lambda \colon {\delta_v}_* v^* = {\delta_v}_*(t_2\delta_v)^!v^* 
\overset{\raisebox{5pt}{$\smallint_{\mspace{-3mu}\lift.65,\delta_v,}^{\lift1.1,t_2,}\>$}}{\lto} t_2^!v^*
\xto{\<\<\bchadmirado{}^{-\<1}\<\<} t_1^{\<*} v^!
\end{equation}
that  shows up in a factorization of 
the map $\delta_{Y\<\<*}v^*\to (1\times v)^{\<*}\>\Gam{v*}v^!$ occurring in the definition of the map~
$\Da_v$ in~$\mathbf{A_3}$  (\cfr \eqref{def-lambda}), namely the map
\[
\delta_{Y\<\<*}v^*
\overset{\ps_*}{=\!=}
i_*'\delta_{v*}v^*
\xto{i_*'\bar\lambda\,}
i_*'t^*_1v^!
\xto{\theta^{-\<1}}
(1\times v)^{\<*}\>\Gam{v*}v^!.
\]

We define $\xi\colon g_*u^*\<v^* \to r^*u^*v^!$ to be the natural composition
\begin{equation}\label{def-of-xi}
\xi \colon g_*u^*\<v^* \xto{\!\via\> \bchasterisco{}^{-\<1}\!}w^*{\delta_v}_*v^*
                     \xto{\!\via \bar\lambda\>} w^*t_1^{\<*} v^!
                     \overset{\>\ps^*}{=\!=}r^*u^*v^!.
\end{equation}

To dispose of $ \mathbf{A_3}$ one sees, after  
expanding according to the definitions of the maps in play, that it's enough to show the next diagram commutes.
In that diagram, unlabeled arrows represent maps induced by isomorphisms of the type~$\bchasterisco{}^{-\<1}$. \begin{equation*}
    \begin{tikzpicture}[yscale=1.25]
      \draw[white] (0cm,15cm) -- +(0: \linewidth)
      node (11) [black, pos = 0.15][scale=1] {$\Gams{u}\Gam{u*} u^*\<v^*$}
      node (12) [black, pos = 0.5 ][scale=1]  {$\Gams{u}(u\<\times\< 1)^*\delta^{}_{Y\<*} v^*$}
      node (13) [black, pos = 0.85][scale=1] {$u^*\delta_Y^*\delta^{}_{Y\<*} v^*$};
      
      \draw[white] (0cm,14cm) -- +(0: \linewidth)
      node (21) [black, pos = 0.15][scale=1] {$\Gams{u}k_*g_*u^*\<v^*$}
      node (22) [black, pos = 0.5 ] {}
      node (23) [black, pos = 0.85] {};
      
      \draw[white] (0cm,13cm) -- +(0: \linewidth)
      node (31) [black, pos = 0.15][scale=1] {$\Gams{u}k_*w^*{\delta_v}_* v^*$}
      node (32) [black, pos = 0.5 ][scale=1] {$\Gams{u}(u\<\times\< 1)^*i^{\prime}_*{\delta_v}_* v^*$}
      node (33) [black, pos = 0.85][scale=1] {$u^*\delta_Y^*i^{\prime}_*{\delta_v}_* v^*$};
      
      \draw[white] (0cm,12cm) -- +(0: \linewidth)
      node (41) [black, pos = 0.15][scale=1] {$\Gams{u}k_*w^*t_1^{\<*}v^!$}
      node (42) [black, pos = 0.5 ][scale=1] {$\Gams{u}(u\<\times\< 1)^*i^{\prime}_*t_1^{\<*}v^!$}
      node (43) [black, pos = 0.85][scale=1] {$u^*\delta_Y^*i^{\prime}_*t_1^{\<*}v^!$};
      
      \draw[white] (0cm,11cm) -- +(0: \linewidth)
      node (51) [black, pos = 0.15][scale=1] {$\Gams{u}k_*r^*u^*v^!$}
      node (52) [black, pos = 0.5 ][scale=1] {$\Gams{u}(u\<\times\< 1)^*(1\<\times\< v)^{\<*}\>\Gam{v*}v^!$}
      node (53) [black, pos = 0.85][scale=1] {$u^*\delta_Y^*(1\<\times\< v)^{\<*}\>\Gam{v*}v^!$};
      
      \draw[white] (0cm,10cm) -- +(0: \linewidth)
      node (61) [black, pos = 0.15][scale=1] {$\Gams{u}(1\<\times\< v)^*\Gam{vu*}u^*v^!$}
      node (62) [black, pos = 0.5 ][scale=1] {$\Gams{u}(1\<\times\< v)^*(u\<\times\< 1)^*\Gamma_{v*}v^!$}
      node (63) [black, pos = 0.85] {};
      
     \draw[white] (0cm,9cm) -- +(0: \linewidth)
      node (71) [black, pos = 0.15] [scale=1]{$\Gams{vu}{\Gamma_{vu}}_*u^*v^!$}
      node (72) [black, pos = 0.5][scale=1] {$\Gams{vu}(u\<\times\< 1)^*\Gam{v*}v^!$}
      node (73) [black, pos = 0.85][scale=1] {$u^*\Gams{v}\Gam{v*}v^!$};
      
      \draw [double distance=2pt]
                 (11) -- (21) node[left=1pt, midway, scale=0.75]{$\via \ps_*$};
      \draw [->] (21) -- (31) node[left=1pt, midway, scale=0.75]{$\via\>\theta^{-\<1}$};
      \draw [->] (31) -- (41) node[left=1pt, midway, scale=0.75]{$\via \, \bar\lambda$};
      \draw [double distance=2pt]
                 (41) -- (51) node[left=1pt, midway, scale=0.75]{$\via \ps^*$};
      \draw [->] (51) -- (61) node[left=1pt, midway, scale=0.75]{$\via\>\theta^{-\<1}$};
      \draw [double distance=2pt]
                 (61) -- (71) node[left=1pt, midway, scale=0.75]{$\via \ps^*$};
                 
      \draw [double distance=2pt]
                 (12) -- (32) node[left=1pt, midway, scale=0.75]{$\via \ps_*$};
      \draw [->] (32) -- (42) node[left=1pt, midway, scale=0.75]{$\via \, \bar\lambda$};
      \draw [->] (42) -- (52) node[left=1pt, midway, scale=0.75]{$\via\>\theta^{-\<1}$};
      \draw [double distance=2pt]
                 (52) -- (62) node[left=1pt, midway, scale=0.75]{$\via \ps^*$};
      \draw [double distance=2pt]
                 (62) -- (72) node[left=1pt, midway, scale=0.75]{$\via \ps^*$};
                 
      \draw [double distance=2pt]
                 (13) -- (33) node[right=1pt, midway, scale=0.75]{$\via \ps_*$};
      \draw [->] (33) -- (43) node[right=1pt, midway, scale=0.75]{$\via \, \bar\lambda$};
      \draw [->] (43) -- (53) node[right=1pt, midway, scale=0.75]{$\via\>\theta^{-\<1}$};      
      \draw [double distance=2pt]
                 (53) -- (73) node[right=1pt, midway, scale=0.75]{$\via \ps^*$};
                 
      \draw [->] (11) -- (12);
      \draw [double distance=2pt]
                 (12) -- (13) node[above=1pt, midway, scale=0.75]{$\via \ps^*$};
                 
      \draw [->] (31) -- (32);
      \draw [double distance=2pt]
                 (32) -- (33) node[above=1pt, midway, scale=0.75]{$\via \ps^*$};
                 
      \draw [->] (41) -- (42);
      \draw [double distance=2pt]
                 (42) -- (43) node[above=1pt, midway, scale=0.75]{$\via \ps^*$};
                 
      \draw [double distance=2pt]
                 (52) -- (53) node[above=1pt, midway, scale=0.75]{$\via \ps^*$};
                 
      \draw [->] (61) -- (62) ;
      
      \draw [->] (71) -- (72) ;
      \draw [double distance=2pt]
                 (72) -- (73) node[below=1pt, midway, scale=0.75]{$\via \ps^*$};

      \node (A31) at (intersection of 11--32 and 31--12) [scale=.9]{$ \mathbf{A_{31}}$};
      \node (A32) at (intersection of 41--62 and 42--61) [scale=.9]{$ \mathbf{A_{32}\mkern50mu}$};
    \end{tikzpicture}
\end{equation*}

It is straightforward to see that the unlabeled subdiagrams commute.

Application of \cite[3.7.2(ii)]{li} to the composite diagram 
$\Df\smallcirc \De$ 
above\kf---for which $i'\delta_v=\delta_Y$ and $kg=\Gam{u}$---yields 
commutativity of  $\mathbf{A_{31}}$. 

As for $\mathbf{A_{32}}$,  we can ignore $\Gams{u}$ and $v^!$, and expand the rest as follows, where $\tilde
p=t_1w=ur$ is the projection from $X\times_ZY$ onto $Y$, and unlabeled
arrows represent maps induced by isomorphisms of type~$\bchasterisco{}^{-\<1}$:
\begin{equation}\label{expand A23}
\def\1{$k_*w^*t^*_1$}
\def\2{$k_*r^*\<u^*$}
\def\3{$(1\<\times\<v)\<^*\Gam{vu*}u^*$}
\def\4{$(u\<\times\<1)\<^*i'_*t^*_1$}
\def\5{$(u\<\times\<1)\<^*(1\<\times\<v)\<^*\Gam{v*}$}
\def\6{$(1\<\times\<v)\<^*(u\<\times\<1)\<^*\Gam{v*}$}
\def\7{$(u\<\times\<v)\<^*\Gam{v*}$}
\def\8{$k_*\tilde p\>^*$}
\CD
 \bpic[xscale=2.6, yscale=1.25]
  \node(11) at (1,-1) {\1};
  \node(12) at (2,-1) {\4};
  \node(13) at (3.37,-1) {\5};

  \node(21) at (1,-2) {\8};
  \node(23) at (3.37,-2) {\7};

  \node(31) at (1,-3) {\2};
  \node(32) at (2,-3) {\3};
  \node(33) at (3.37,-3) {\6};

  \draw[->] (11)--(12);
  \draw[->] (12)--(13);

  \draw[->] (21)--(23); 

  \draw[->] (31)--(32);
  \draw[->] (32)--(33);

  \draw[double distance=2pt] (11)--(21) node[left=1pt, midway, scale=.75]{$\via\ps^*$};
  \draw[double distance=2pt] (21)--(31) node[left=1pt, midway, scale=.75]{$\via\ps^*$};

  \draw[double distance=2pt] (13)--(23) node[right=1pt, midway, scale=.75]{$\ps^*$};
  \draw[double distance=2pt] (23)--(33) node[right=1pt, midway, scale=.75]{$\ps^*$};

 \epic
\endCD
\end{equation}
Application of \cite[3.7.2(iii)]{li} to the composite diagram
$\Dg\smallcirc\Df$ above and to
\[
 \bpic[xscale=2, yscale=1.5]
  
  \node(11) at (1,-1) {$X\times_ZY$};
  \node(12) at (2,-1) {$X$};
  \node(13) at (3,-1){$Y$} ;

  \node(21) at (1,-2) {$X\times Y$};
  \node(22) at (2,-2) {$X\times Z$};
  \node(23) at (3,-2) {$Y\times Z$};

  \draw[->] (11)--(12) node[above=1pt, midway, scale=.75]{$r$};
  \draw[->] (12)--(13) node[above=1pt, midway, scale=.75]{$u$};

  \draw[->] (21)--(22) node[below=1pt, midway, scale=.75]{$1\<\times\<v$}; 
  \draw[->] (22)--(23) node[below=1pt, midway, scale=.75]{$u\<\times\<1$};

  \draw[->] (11)--(21) node[left=1pt, midway, scale=.75]{$k$};

  \draw[->] (12)--(22) node[right=1pt, midway, scale=.75]{$\Gam{vu}$};

  \draw[->] (13)--(23) node[right=1pt, midway, scale=.75]{$\Gam{v}$};

 \epic
\]
gives commutativity of the top (respectively bottom) half of
\eqref{expand A23},
whence the commutativity of $\mathbf{A_{32}}$. 

Thus $\mathbf{A_3}$ commutes.

\medskip

It remains to consider $ \mathbf{A_2}$. 
Work with the fiber-square diagram
\begin{equation}\label{triplesquare}
 \CD    
   \begin{tikzpicture}[xscale=2.3]
      \draw[white] (0cm,5.1cm) -- +(0: .3\linewidth)
      node (21) [black, pos = 0.1] {$X$}
      node (22) [black, pos = 0.5] {$Y$};
      \draw[white] (0cm,2.8cm) -- +(0: .3\linewidth)
      node (31) [black, pos = 0.1] {$X\<\times_Z Y$}
      node (32) [black, pos = 0.5] {$Y\<\<\times_Z Y$}
      node (33) [black, pos = 0.9] {$Y$};      
      \draw[white] (0cm,0.5cm) -- +(0: .3\linewidth)
      node (51) [black, pos = 0.1] {$X$}
      node (52) [black, pos = 0.5] {$Y$}
      node (53) [black, pos = 0.9] {$Z$};
      \draw [->] (21) -- (31) node[left, midway, scale=0.75]{$g$};
      \draw [->] (31) -- (51) node[left, midway, scale=0.75]{$r$};
      \draw [->] (22) -- (32) node[right, midway, scale=0.75]{$\delta_v$};
      \draw [->] (32) -- (52) node[right, midway, scale=0.75]{$t_1$};
      \draw [->] (33) -- (53) node[right, midway, scale=0.75]{$v$};
      \draw [->] (21) -- (22) node[above, midway, sloped, scale=0.75]{$u$};
      \draw [->] (31) -- (32) node[above, midway, sloped, scale=0.75]{$w\set u\<\times_Z\id_Y$};
      \draw [->] (32) -- (33) node[above, midway, sloped, scale=0.75]{$t_2$};
      \draw [->] (51) -- (52) node[below, midway, sloped, scale=0.75]{$u$};
      \draw [->] (52) -- (53) node[below, midway, sloped, scale=0.75]{$v$};
    \end{tikzpicture}
 \endCD
  \end{equation}
where $r$ and $t_1$ are the canonical projections onto the first factor, and $t_2$ is the canonical projection onto the second factor. 
Set $\tau\set t_2 w$.

Using  the definition of \eqref{specialint}, and the equalities $u=\tau g$, $t_2\delta_v=\id$, one sees that for commutativity of $\mathbf{A_2}$ it suffices to prove commutativity of the following expanded diagram \eqref{DIAGRAMA A2}, in which $\CO\set\CO_{Y\<\times_{\<Z}Y}$, the map $\bar\lambda$ is as in~\eqref{paraxi}, the unlabeled maps are isomorphisms coming out of \eqref{projection} or~\eqref{thetaB} or~\eqref{def-of-bch-asterisco} (see \ref{thetaiso}), and the isomorphisms denoted 
by ``$\>=\!=$" are induced by $\ps^*$ or~$\ps^!$, or have other obvious interpretations.
\begin{figure}
\begin{equation}\label{DIAGRAMA A2}
    \mkern-10mu
    \begin{tikzpicture}[xscale=0.78,yscale=1.1]
    \node at (0cm,20.4cm){};
      \draw[white] (0cm,20cm) -- +(0: \linewidth)
      node (11) [black, pos = 0][ scale=0.89] {$g_*u^!\CO_Y\<\otimes\< \tau^*\<v^*$}
      node (12) [black, pos = 0.245][ scale=0.89] {}
      node (13) [black, pos = .625][ scale=0.89] {$g_*(u^!\CO_Y\<\otimes\< g^*\tau^*\<v^*\<)$}
      node (14) [black, pos = 1] [ scale=0.89]{$g_*(u^!\CO_Y\<\otimes\< u^*\<v^*\<)$};
      
      \draw[white] (0cm,18.75cm) -- +(0: \linewidth)
      node (21) [black, pos = 0] {}
      node (22) [black, pos = 0.245][ scale=0.89] {$g_*g^*r^*\<u^!\CO_Y\<\otimes\< \tau^*\<v^*$}
      node (23) [black, pos = .625][ scale=0.89] {$g_*(g^*r^*\<u^!\CO_Y\<\otimes\< g^*\tau^*\<v^*\<)$};
      
      \draw[white] (0cm,17.5cm) -- +(0: \linewidth)
      node (31) [black, pos = 0][ scale=0.89] {}
      node (32) [black, pos = 0.245][ scale=0.89] {$g_*(g^*r^*\<u^!\CO_Y\<\otimes\< \OX)\<\otimes\< \tau^*\<v^*$}
      node (33) [black, pos = .625][ scale=0.89] {$r^*\<u^!\CO_Y\<\otimes\< g_*g^* \tau^*\<v^*$}
      node (34) [black, pos = 1][ scale=0.89] {$\mkern-6mu g_*(g^*r^*\<u^!\CO_Y\<\otimes\< u^*\<v^*\<)$};
      
      \draw[white] (0cm,16.25cm) -- +(0: \linewidth)
      node (41) [black, pos = 0] {}
      node (42) [black, pos = 0.245][ scale=0.89] {\hbox to 85pt{\hss$r^*\<u^!\CO_Y\<\otimes\< g_* \OX \<\<\otimes\< \tau^*\<v^*\quad\ \ $\hss}}
      node (43) [black, pos = .625][ scale=0.89] {\hbox to 105pt{\hss\quad\ \ $r^*\<u^!\CO_Y\<\<\otimes\< g_*(\OX\<\<\otimes\< g^* \tau^*\<v^*\<)$\hss}}
      node (44) [black, pos = 1][scale=0.89] {$\ r^*\<u^!\CO_Y\<\<\otimes\< g_*u^*\<v^*$};
      
      \draw[white] (0cm,15cm) -- +(0: \linewidth)
      node (51) [black, pos = 0] {}
      node (52) [black, pos = 0.245][ scale=0.89] {$r^*\<u^!\CO_Y\<\otimes\< g_*u^*\CO_Y\< \<\otimes\< \tau^*\<v^*$}
      node (54) [black, pos = 1][ scale=0.89] {};
      
      \draw[white] (0cm,13.87cm) -- +(0: \linewidth)
      node (5a1) [black, pos = 0] {}
      node (5a2) [black, pos = 0.245][ scale=0.89] {$r^*\<u^!\CO_Y\<\otimes\< w^*\mkern-1.5mu\delta_{v*}\CO_Y\< \<\otimes\< \tau^*\<v^*$}
      node (5a5) [black, pos = .74][ scale=0.89] {$r^*\<u^!\CO_Y\<\otimes\< w^*(\delta_{v*}\CO_Y\< \<\otimes\< t_2^*\<v^*\<)$}
      node (5a3) [black, pos = 0.68][ scale=0.89] {}
      node (5a4) [black, pos = 1][ scale=0.89] {};
      
      \draw[white] (0cm,12.74cm) -- +(0: \linewidth)
      node (62) [black, pos = 0.245][ scale=0.89] {$w^!\CO \<\otimes\< w^*\mkern-1.5mu\delta_{v*}\CO_Y\< \<\otimes\< \tau^*\<v^*$}
       node (63) [black, pos = .74][ scale=0.89] {$r^*\<u^!\CO_Y\<\otimes\< w^*\mkern-1.5mu\delta_{v*}(\CO_Y\< \<\otimes\< \delta_v^*t_2^*v^*\<)$};
      
       \draw[white] (0cm,11.77cm) -- +(0: \linewidth)
      node (81) [black, pos = 0][ scale=0.89] {$\ \ g_*(\tau g)^!\CO_Y\< \<\otimes\< \tau^*\<v^*$}
      node (83)  [black, pos = .465]  [ scale=0.89] {$w^!\CO \<\otimes\<w^*(\delta_{v*}\CO_Y\< \<\otimes\< t_2^*\<v^*\<)$}
      node (H7) [black, pos = 1]  [ scale=0.89] {$r^*\<u^!\CO_Y\<\otimes\< w^*\mkern-1.5mu\delta_{v*}v^*$};
      
       \draw[white] (0cm,10.805cm) -- +(0: \linewidth)
       node (72) [black, pos = 0.245][ scale=0.89] {$w^!\<{\delta_v}_*\CO_Y\< \<\otimes\< \tau^*\<v^*$}
       node (4) [black, pos = .465][scale=.9]{\raisebox{-41pt}{\A4}}
       node (73) [black, pos = .695][ scale=0.89] {$w^!\CO \<\otimes\< w^*\mkern-1.5mu\delta_{v*}           (\CO_Y\< \<\otimes\< \delta_v^*t_2^*v^*\<)$} ;
       
      \draw[white] (0cm,9.65cm) -- +(0: \linewidth)
      node (80) [black, pos = 0][ scale=0.89] {$\ \ g_*g_\upl^!\tau^!\CO_Y\< \<\otimes\< \tau^*\<v^*$}
      node (82) [black, pos = 0.3][ scale=0.89] {$w^!t_1^*v^!\OZ\< \<\otimes\< \tau^*\<v^*$};

      \draw[white] (0cm,8.68cm) -- +(0: \linewidth)
      node (92) [black, pos = 0.13][ scale=0.89] {$\ w^!t_2^!\CO_Y\!\otimes\< \tau^*\<v^*$}
      node (S2) [black, pos = 0.43][ scale=0.89] {$w^!t_2^!\CO_Y\< \<\otimes\<\< w^*t_2^*v^*$}
      node (H3) [black, pos = .695][ scale=0.89] {$w^!\CO \<\<\otimes\<\< w^*\mkern-1.5mu\delta_{v*}v^*$};
  
      \draw[white] (0cm,7.5cm) -- +(0: \linewidth)
      node (R1) [black, pos = 0][ scale=0.89] {$\tau^!\CO_Y\< \<\otimes\< \tau^*\<v^*\ $}
      node (R2) [black, pos = .337][ scale=0.89] {$w^!\CO\<\<\otimes \<\<w^*(t_2^!\CO_Y\< \<\otimes\< t_2^*v^*\<)$};
      
      \draw[white] (0cm,6.9cm) -- +(0: \linewidth)
       node (H) [black, pos = 0.565] [ scale=0.89]{$w^!\<{\delta_v}_*v^*$};
  
      \draw[white] (0cm,6.25cm) -- +(0: \linewidth)
      node (01) [black, pos = 0][ scale=0.89] {$\tau^!v^*$}
      node (02) [black, pos = 0.245][ scale=0.89] {$\mkern-8mu w^!\<(\<t_2^!\CO_Y\!\otimes\<\< t_2^*v^*\<\<)\quad$};
         
      \draw[white] (0cm,5cm) -- +(0: \linewidth)
      node (T1) [black, pos = 0][ scale=0.89] {$(t_2w)^!v^*$}
      node (U1) [black, pos = .245][ scale=0.89] {$w^!t_2^!v^*$}
      node (T2) [black, pos = 0.418][ scale=0.89] {$w^!t_1^{\<*}v^!$}
      node (T3) [black, pos = .695][ scale=0.89] {$w^!\CO \<\otimes\< w^*t_1^{\<*}v^!$}
      node (T4) [black, pos = 1][ scale=0.89] {$r^*\<u^!\CO_Y\<\otimes\< w^*t_1^{\<*}v^!$};

     \draw[white] (0cm,3.75cm) -- +(0: \linewidth)
      node (X1) [black, pos = 0][ scale=0.89] {$r^*\<(vu)^!$}
      node (X2) [black, pos = 0.325][ scale=0.89] {$r^*\<u^!v^!$}
      node (X3) [black, pos = .695][ scale=0.89] {$r^*\<(u^!\CO_Y\<\otimes\< u^*v^!)$}
      node (X4) [black, pos = 1][ scale=0.89] {$r^*\<u^!\CO_Y\<\otimes\< r^*\<u^*v^!$};

      \draw [double distance=2pt] (11) -- (81) ;
      \draw [double distance=2pt] (81) --(80) ;    
      \draw [->] (80) -- (R1) node[left=1pt, midway, scale=0.7]{\eqref{! and otimes}(i)};
      \draw [double distance=2pt] (R1) -- (01) ;
      \draw [double distance=2pt] (01) -- (T1) ;
      \draw [->] (T1) -- (X1) ;      
      \draw [double distance=2pt] (22) -- (32);
      \draw [->] (32) -- (42) ;
      \draw [double distance=2pt] (42) -- (52) ;
      \draw [->] (52) -- (5a2) ; 
      \draw [->] (5a2) -- (62) ; 
      \draw [double distance=2pt] (62) -- (72) ;
      \draw [->] (72) -- (82) node[ above=-6.2pt, midway, scale=0.7]{$\via \bar\lambda\mkern60mu$};
      \draw [double distance=2pt] (92) -- (S2) ;
       \draw [->] (S2) -- (R2) ;
      \draw [double distance=2pt] (R2) -- (02) ;
      \draw [double distance=2pt] (02) -- (U1) ;
      \draw [->] (T2) -- (X2) ; 
      \draw [double distance=2pt] (13) -- (23) ;
      \draw [->] (23) -- (33) ;
      \draw  [double distance=2pt]  (43) -- (33);
      \draw[->] (5a5) -- (63);
      \draw [->] (63) -- (73)  ;
      \draw [double distance=2pt] (73) -- (H3) ; 
      \draw [->] (H3) -- (T3) node[right=1pt, midway,  scale=0.7]{$\via\bar\lambda$};
      \draw [double distance=2pt] (14) -- (34) node[right=1pt, midway, scale=0.7]{ };
      \draw [->] (34) -- (44) ;
      \draw [->] (44) -- (H7) ;
      \draw [->] (H7) -- (T4) node[right=1pt, midway,  scale=0.7]{$\via\bar\lambda$};
      \draw [double distance=2pt]  (T4) -- (X4);
      
      \draw [->] (13) -- (11) node[above, midway, sloped, scale=0.75]{ };
      \draw [double distance=2pt] (13) -- (14) ;
      \draw [->] (22) -- (23) ;
      \draw [double distance=2pt] (23) -- (34) ;
      \draw [->] (42) -- (43) ;
      \draw [double distance=2pt] (33) -- (44) ;
      \draw [->] (5a2) -- (5a5) ;
      \draw [double distance=2pt] (T2) -- (T3) ;
      \draw [->] (T3) -- (T4) ;

      \draw [double distance=2pt] (X2) -- (X3) ;
      \draw [->] (X3) -- (X4) ;

      \draw [double distance=2pt] (11) -- (22) node[above=-3pt, midway, scale=0.7]{$\mkern70mu\via \ps^*$};
      \draw [->] (62) -- (83) ;
      \draw[->] (83) -- (73) ;
      \draw [->] (82) -- (92);
      \draw [double distance=2pt] (63) -- (H7) ;
      \draw [double distance=2pt] (R1) -- (92) ;
      \draw [double distance=2pt] (T1) -- (U1) ;
      \draw [->] (T2) -- (U1) ;
      \draw [double distance=2pt] (X1) -- (X2) ;
      \draw [double distance=2pt] (H3) -- (H) ;
      \draw [->] (H) -- (T2) node[below=-8.5pt, midway,  scale=0.7]{$\mkern-55mu\via\bar\lambda$};

      \node (1) at (intersection of 22--33 and 32--23) [scale=.9]{ \A1};
      \node (2) at (intersection of 42--5a4 and 44--5a2) [scale=.9]{\A2};
      \node (3) at (intersection of 51--5a2 and 5a1--52)   [scale=.9]{\A{3}};
      \node (5) at (intersection of 02--R1 and S2--01) [scale=.9]{\raisebox{3pt}{\A{5}}};
      \node (6) at (intersection of T2--X1 and T1--X2) [scale=.9]{\raisebox{5pt}{\A{6}}};
      \node (7) at (intersection of T2--X4 and X2--T4) [scale=.9]{\raisebox{-10pt}{\A{7}\quad}};

    \end{tikzpicture}
\end{equation}
\end{figure}

The commutativity of the unlabeled subdiagrams is clear;\va{-2} that of \A{5} results from the definition of  the\va1 isomorphism $w^!t_2^!\overset{\ps^!}{=\!=}(t_2w)^!$,
see \cite[(5.7.5)]{AJL}; of \A{6} from the horizontal transitivity\va{.5} of $\bchadmirado{}$ (see \S\ref{indsq});  and of \A{7}  from the definition of ~$\bchadmirado{}$ \cite[5.8.4]{AJL}.

Commutativity of \A{1} is the same as commutativity of the following diagram of isomorphisms coming from \eqref{projection}, in which $\Adot=r^*u^!\CO_Y$ and $\Bdot=\tau^*\<v^* \Cdot\ (\Cdot\in \D_{\<Z})$.

\[
    \begin{tikzpicture}[yscale=1.65]
      \draw[white] (0cm,3.5cm) -- +(0: \linewidth)
      node (11) [black, pos = 0.15] {$g_*g^*\<\<\Adot\otimes \Bdot$}
      node (12) [black, pos = 0.515] {}
      node (13) [black, pos = 0.85] {$g_*(g^*\<\<\Adot\otimes g^*\<\<\Bdot)$};
      \draw[white] (0cm,2cm) -- +(0: \linewidth)
      node (21) [black, pos = 0.15] {$g_*(g^*\<\<\Adot\otimes  \OX)\otimes \Bdot$}
      node (22) [black, pos = 0.515] {$g_*(g^*\<\<\Adot\otimes  \OX \otimes g^*\<\<\Bdot)$}
      node (23) [black, pos = 0.85] {};
      \draw[white] (0cm,0.5cm) -- +(0: \linewidth)
      node (31) [black, pos = 0.15] {$\Adot\otimes g_* \OX \otimes \Bdot$}
      node (32) [black, pos = 0.515] {$\Adot\otimes g_*( \OX \otimes g^*\<\<\Bdot)$}
      node (33) [black, pos = 0.85] {$\Adot\otimes g_*g^*\<\<\Bdot$};
      \draw [->] (11) -- (13) node[above, midway, sloped, scale=0.75]{};
      \draw [double distance=2pt]
                 (33) -- (32) node[below, midway, sloped, scale=0.75]{};
      \draw [->] (31) -- (32) node[left, midway, scale=0.75]{};
      \draw [->] (21) -- (22) node[left, midway, scale=0.75]{};
      \draw [double distance=2pt]
                 (11) -- (21) node[right, midway, scale=0.75]{};
      \draw [->] (21) -- (31) node[right, midway, scale=0.75]{};
      \draw [double distance=2pt]
                 (13) -- (22) node[right, midway, scale=0.75]{};
      \draw [->] (22) -- (32) node[right, midway, scale=0.75]{};
      \draw [->] (13) -- (33) node[right, midway, scale=0.75]{};
      \node (12) at (intersection of 31--22 and 21--32) [scale=0.9]{\A{11}};
    \end{tikzpicture}
\]
 
Commutativity of the unlabeled subdiagrams results from functoriality of the isomorphisms in~\eqref{projection}.

Commutativity of \A{11} is a special case of that of the natural diagram, for any scheme\kf-map $f\colon X\to Y$ and $A\in\D(Y)$, $B\in\D(X)$,  $C\in\D(Y)$:
\[\mkern-3mu
\def\1{$\fst(f^*\<\<A\otimes_\sX \<\<B)\otimes_Y C$}
\def\2{$\fst\big(f^*\<\<A\otimes_\sX \<\<(B\otimes_\sX\<\< f^*C)\<\big)$}
\def\3{$A\otimes_Y\<\< \fst B \otimes_Y \<C$}
\def\4{$A\otimes_Y\< \fst(B \otimes_\sX \<f^*C)$}
\def\5{$\,C \otimes_Y\<\< \fst(f^*\<\<A\otimes_\sX\<\<B)$}
\def\6{$\fst(f^*C\otimes_\sX\! f^*\!A\otimes_\sX\<\<B)$}
\def\7{$\fst\big(f^*\<\<A\otimes_\sX\<\< (f^*C\otimes_\sX\<\<B)\<\big)$}
\def\8{$C\otimes_Y \<\<A\otimes_Y\<\fst B$}
\def\9{$\fst(f^*(C\otimes_Y \<\<A)\otimes_\sX\<\<B)$}
\def\ten{$\fst(f^*(A\otimes_Y C)\otimes_\sX\<\<B)$}
\def\lvn{$A\otimes_Y C\otimes_Y\<\fst  B$}
\def\twv{$A\otimes_Y\fst (f^*C\otimes_\sX B)$}
 \bpic[xscale=3, yscale=2]
 
  \node(11) at (1,-1){\1};
  \node(14) at (3.95,-1){\2};

  \node(21) at (1,-2){\5};
  \node(22) at (2.42,-2){\6};
  \node(24) at (3.95,-2){\7};

  \node(32) at (1.7,-3){\9};
  \node(33) at (3.22,-3){\ten};

  \node(41) at (1,-4){\8};
  \node(42) at (2.42,-4){\lvn};
  \node(44) at (3.95,-4){\twv};

  \node(51) at (1,-5){\3};
  \node(54) at (3.95,-5){\4};
  
   \draw[->] (11)--(14);

   \draw[->] (21)--(22);
   \draw[->] (22)--(24);

   \draw[->] (32)--(33);

   \draw[->] (41)--(42);
   \draw[->] (42)--(44);

   \draw[->] (51)--(54);
   
   \draw[->] (21)--(11);
   \draw[->] (41)--(21);
   \draw[->] (51)--(41);

   \draw[->] (32)--(22);

   \draw[->] (24)--(14);
   \draw[->] (44)--(24);
   \draw[->] (54)--(44);
   
   \draw[->] (22)--(14);  
  
   \draw[->] (33)--(24);
   
   \draw[->] (41)--(32);
   \draw[->] (42)--(33);

   \draw[->] (51)--(42);
   
   \node at (2.11, -1.5)[scale=.8]{\circled{\bf 1}};
   \node at (3.6, -1.6)[scale=.8]{\circled{\bf 2}};
   \node at (1.55, -2.52)[scale=.8]{\circled{\bf 3}};
   \node at (2.82, -2.52)[scale=.8]{\circled{\bf 4}};
   \node at (2.11, -3.54)[scale=.8]{\circled{\bf 5}};
   \node at (3.42, -3.54)[scale=.8]{\circled{\bf 6}};
   \node at (2.82, -4.55)[scale=.8]{\circled{\bf 8}};
   \node at (1.35, -4.45)[scale=.8]{\circled{\bf 7}};
   \node at (2.82, -4.55)[scale=.8]{\circled{\bf 8}};

 \epic
\]
Here  commutativity of subdiagram \circled{\bf5} results from functoriality of the projection isomorphisms \eqref{projection}, that of \circled{\bf 4} results from the \emph{dual} \cite[3.4.5]{li} of the second diagram  in \cite[(3.4.2.2)]{li}, that of \circled{\bf 3} and \circled{\bf 6} from \cite[3.4.7 (iv)]{li}, that of \circled{\bf1} and \circled{\bf 8} from \cite[3.4.6.1]{li}, and that of \circled{\bf2} and \circled{\bf 7} from the third diagram in~\cite[(3.4.1.1)]{li}.

Thus, \A{1} commutes.\va4

\pagebreak[3]
For subdiagram \A{2}\kf, it is enough to check commutativity of the following natural diagram, where 
$\delta\set\delta_v$ and $t\set t_2$ (so that $\tau=tw$ and $t\delta=\id_Y$):

\[
 \bpic[xscale=3.4, yscale=1.3]
 
  \node (41)  at (1,-1)  {$g_* \OX\<\<\<\otimes\< \tau^*$};
  \node (31)  at (2,-1) {$g_*( \OX\<\<\<\otimes\< g^*\<\tau^*\<)$};
  \node (11)  at (3.12,-1) {$g_*g^*\tau^*$};
  \node (13)  at (4,-1) {$g_*u^*$};

  \node (42)  at (1,-2)  {$g_*u^*\CO_Y\<\<\otimes\< \tau^*$};
  \node (32)  at (2,-2) {$\,g_*\<(\<u^*\CO_Y\<\<\otimes\< u^*\delta^*t^*\<)$};
  \node (22)  at (3.13,-2) {$g_*(\OX\<\<\otimes\< u^*\<\delta^*t^*\<)$};
  \node (21)  at (4,-2) {$g_*(\OX\<\<\otimes\< u^*)$};

  \node (43)  at (1,-3)  {$w^*\<\delta_*\CO_Y\<\<\otimes\< \tau^*$};
  \node (33)  at (2,-3)  {$g_*u^*(\CO_Y\<\<\otimes\< \delta^*t^*\<)$};
  \node (23)  at (3.13,-3)  {$g_*u^*\<\delta^*t^*$};
  \node (44)  at (4,-3) {$g_*u^*$}; 

  \node (45)  at (1,-4)  {$w^*\<(\delta_*\CO_Y\<\<\otimes\< t^*\<)$};
  \node (34)  at (2,-4) {$\!w^*\<\<\delta_*(\<\CO_Y\<\<\otimes\< \delta^*t^*\<)$};
  \node (24)  at (3.138,-4) {$w^*\<\delta_*\delta^*t^*$};
  \node (14)  at (4,-4)  {$w^*\<\delta_*$};
   
      \draw [->] (44) -- (14) ;
      \draw [->] (32) -- (33) ;
      \draw [->] (33) -- (34) ;
      \draw [->] (42) -- (43) ;
      \draw [->] (45) -- (34) ;
      \draw [->] (41) -- (31) ;
      \draw [<-] (45) -- (43) ;      

      \draw [double distance=2pt] (11) -- (13) ;
      \draw [double distance=2pt] (11) -- (31) ;
      \draw [double distance=2pt] (21) -- (44) ;
      \draw [double distance=2pt] (21) -- (22) ;
      \draw [double distance=2pt] (23) -- (44) ;
      \draw [double distance=2pt] (31) -- (32) ;
      \draw [double distance=2pt] (22) -- (32) ;
      \draw [double distance=2pt] (23) -- (33) ;
      \draw [double distance=2pt] (13) -- (21) ;
      \draw [double distance=2pt] (22) -- (23) ;
      \draw [double distance=2pt] (14) -- (24) ;
      \draw [double distance=2pt] (24) -- (34) ;
      \draw [double distance=2pt] (42) -- (41) ;

      \node (31) at (intersection of 32--43 and 33--42) [scale=0.93]{\A{21}};
      \node (31) at (intersection of 32--23 and 33--22) [scale=0.93]{\A{22}};
      
 \epic
\]
Commutativity of \A{21} follows from \cite[{3.7.3}]{li}, 
with $(f,g,f'\<,g'\<,P,Q)\set (\delta, w, g, u, t^*G,\OY)\ (G\in\D(Y))$, except that there the factors in the tensor products need to be switched, as do the two projection maps defined above  in~\eqref{projection}---all of which is made permissible by \cite[3.4.6.1]{li} and the \emph{dual}
(\cite[3.4.5]{li}) of the second diagram in \cite[(3.4.2.2)]{li}. Commutativity of \A{22} results 
from the dual of the first diagram in \cite[(3.4.2.2)]{li}.  Commutativity of the unlabeled diagrams is
easy to check.

Thus \A{2} commutes.\va2

Next we expand \A{4}---again dropping $v^*\<$, setting $\delta\set\delta_v$ and $t\set t_2\>$, and substituting $w^*t^*$ for $\tau^*$. The map 
\begin{equation}\label{lambda0}
\bar\lambda_0\colon\delta_*=\delta_*(t\delta)^!\overset{\raisebox{5pt}{$\smallint_{\mspace{-3mu}\lift.65,\delta_v,}^{\lift1.1,t,}\>$}}\lto t^!
\end{equation} 
is as in the definition \eqref{paraxi} of $\bar\lambda$.
\[
    \begin{tikzpicture}[yscale=1.85]
      \draw[white] (0cm,5.5cm) -- +(0: \linewidth)
      node (11) [black, pos = 0.15] {$w^!\CO \otimes w^*\<\delta_*\OY\<\<\otimes w^*t^*$}
      node (12) [black, pos = 0.493] {$w^!\CO\< \otimes\< w^*\<(\delta_*\OY\<\<\otimes t^*\<)$}
      node (13) [black, pos = 0.84] {$w^!\CO\< \otimes\< w^*\<\delta_*(\<\OY\<\<\otimes\<\delta^*t^*\<)$};
      
      \draw[white] (0cm,4.45cm) -- +(0: \linewidth)
      node (21) [black, pos = 0.15] {$\mkern-22mu w^!\delta_*\OY\<\<\otimes w^*t^*$}
      node (22) [black, pos = 0.493] {$\ w^!(\delta_*\OY\<\<\otimes t^*)$};

       \draw[white] (0cm,4cm) -- +(0: \linewidth)
       node (23) [black, pos = 0.84] {$w^!\CO \otimes w^*\<\delta_*$};
          
      \draw[white] (0cm,3.4cm) -- +(0: \linewidth)
      node (31) [black, pos = 0.15] {$\mkern-18mu w^!t^!\OY\<\<\otimes \<w^*t^*$}
      node (33) [black, pos = 0.68] {$w^!\delta_*(\OY\<\<\otimes\delta^*t^*)$};

      \draw[white] (0cm,2.95cm) -- +(0: \linewidth)
       node (N51) [black, pos = 0.597][scale=0.9]{\A{41}};

      \draw[white] (0cm,2.5cm) -- +(0: \linewidth)
      node (41) [black, pos = 0.15] {$w^!\CO\<\<\otimes
                                      w^*\<(t^!\OY\<\<\otimes t^*)$}      
      node (32) [black, pos = 0.493] {$w^!(t^!\OY\<\<\otimes t^*\<)$}
      node (42) [black, pos = 0.693] {$w^!t^!$}
      node (43) [black, pos = 0.84] {$w^!\delta_*$};
      
      \draw [double distance=2pt] (11) -- (21) ;
      \draw [->] (21) -- (31) node[left=1pt, midway, scale=0.75]{$\via\> \bar\lambda_0$};
      \draw [->] (31) -- (41);
      \draw [double distance=2pt] (12) -- (22) ;
      \draw [->] (22) -- (32) node[left=1pt, midway, scale=0.75]{$\via\> \bar\lambda_0$};
      \draw [double distance=2pt] (32) -- (42) ;
      \draw [double distance=2pt] (23) -- (43) ;
      \draw [double distance = 2pt] (13) -- (23) node[right, midway, scale=0.75]{};
      
      \draw [->] (22) -- (33) ;
      \draw [->] (12) -- (41) node[below=1pt, midway, scale=0.75]{$\mkern35mu\via\> \bar\lambda_0$};
      \draw [double distance=2pt] (13) -- (33) ;
      \draw [double distance=2pt] (33) -- (43);   
      \draw [->] (41) -- (32) ;


      \draw [->] (11) -- (12) ;
      \draw [->] (12) -- (13) ;
      \draw [->] (43) -- (42) node[below=1pt, midway, scale=0.75]{$\,\via\> \bar\lambda_0\!$};

 \end{tikzpicture}
\]

Commutativity of the unlabeled subdiagrams is easy to verify. 

For commutativity \A{41}, it is enough, by definition of the maps
involved, to verify commutativity of the
natural  diagram (in which the unlabeled maps are the obvious ones):
\[
\def\1{$\delta_*\CO_Y\<\<\otimes t^*$}
\def\2{$\delta_*(\CO_Y\<\<\otimes \delta^*\<t^*\<)$}
\def\3{$\delta_*$}
\def\4{$\delta_*(\delta^!t^!\CO_Y\<\<\otimes \delta^*\<t^*\OY\<)\<\otimes t^*$}
\def\5{$\delta_*(\delta^!t^!\CO_Y\<\<\otimes \delta^*\<t^*)$}
\def\6{$\delta_*\delta^!t^!\CO_Y\<\<\otimes t^*\OY\otimes t^*$}
\def\7{$\delta_*\delta^!t^!\CO_Y\<\<\otimes t^*$}
\def\8{$t^!\CO_Y\<\<\otimes t^*\OY\otimes t^*$}
\def\9{$t^!\CO_Y\<\<\otimes t^*$}
\def\ten{$\delta_*(t\delta)^!$}
\def\lvn{$\delta_*(t\delta)^!\CO_Y\<\<\otimes t^*$}
\def\twv{$\delta_*((t\delta)^!\CO_Y\<\<\otimes \delta^*\<t^*\<)$}
 \bpic[xscale=4.4, yscale=1.75]
 
   \node(10) at (1,-1){\1};
   \node(11) at (1.9,-1){\2};
   \node(12) at (2.515,-1){\3};
   \node(13) at (3,-1){\ten};

   \node(21) at (1,-2){\lvn};
   \node(22) at (1.9,-2){\twv};
   \node(23) at (3,-2){\5};

   \node(32) at (3,-3){\7};
   \node(33) at (3,-4){\9};

   \node(41) at (1,-3){\4};
   \node(42) at (1,-4){\6};
   \node(43) at (2.1,-4){\8};

   \draw[->] (10) -- (11) node[above=1pt, midway, scale=.75]{\eqref{projection}};
   \draw[double distance=2pt] (11) -- (12);
   \draw[double distance=2pt] (12) -- (13);

   \draw[->] (21) -- (22) node[above=1pt, midway, scale=.75]{\eqref{projection}};
   \draw[double distance=2pt] (22) -- (23);
   
   \draw[->] (32) -- (33);
   
   \draw[->] (41) --(42) node[left=1pt, midway, scale=.75]{via\;\eqref{projection}};
   \draw[->] (42) --(43);
   
   \draw[double distance=2pt] (10) -- (21) node[left=1pt, midway, scale=.75]{{$\via\>\ps^!$}};
   \draw[double distance=2pt] (21) -- (41)node[left, midway, scale=.75]{\raisebox{4pt}{$\via\>\ps^!\>$}}
                                    node[right=1pt, midway, scale=.75]{$\<\textup{and}\>\ps^*$};

   \draw[double distance=2pt] (11) -- (22) node[left=1pt, midway, scale=.75]{{$\via\>\ps^!$}};
                                      
   \draw[double distance=2pt] (32) -- (42) ;

   \draw[double distance=2pt] (13) -- (23) node[left, midway, scale=.75]{\raisebox{4pt}{$\via\>\ps^!\>$}}
                                    node[right=1pt, midway, scale=.75]{$\<\textup{and}\>\ps^*$};
   \draw[double distance=2pt] (43) -- (33);
   
  \draw[double distance=2pt] (21) -- (32);
  
   \draw[double distance=2pt] (32) -- (41);

   \draw[->] (23) -- (32) node[right=1pt, midway, scale=.75]{\eqref{projection}};

      \node  at (intersection of 11--23 and 13--22) [scale=0.9]{\A{411}};
      \node  at(1.63,-2.63)[scale=0.9]{\A{412}};
      \node  at(1.63,-3.33)[scale=0.9]{\A{413}};

 \epic
\]
It is evident that the unlabeled diagrams commute.

Subdiagram \A{412} (without $\otimes\>\> t^*$ and without $\delta_*$) expands as
\[
 \bpic[xscale=5,yscale=1.3]
 
   \node(11) at (1,-1) {$(t\delta)^!\OY$};
   \node(12) at (2,-1) {$\delta^!t^!\OY$};

   \node(21) at (1,-2) {$(t\delta)^!\OY\otimes(t\delta)^*\OY$};

   \node(31) at (1,-3) {$\delta^!t^!\OY\otimes\delta^*t^*\OY$};
   \node(32) at (2,-3) {$\delta^!t^!\OY\otimes(t\delta)^*\OY$};

  \draw[double distance=2pt] (11)--(12) node[above=1pt, midway, scale=.75]{$\ps^!$};
  \draw[double distance=2pt] (21)--(32) node[above=1pt, midway, scale=.75]{$\via\>\ps^!$};
  \draw[double distance=2pt] (31)--(32) node[below=1pt, midway, scale=.75]{$\via\>\ps^*$};

  \draw[double distance=2pt] (11)--(21) ;
  \draw[double distance=2pt] (21) -- (31) node[left, midway, scale=.75]{\raisebox{4pt}{$\via\>\ps^!\>$}}
                                    node[right=1pt, midway, scale=.75]{$\<\textup{and}\>\ps^*$};
  
  \draw[double distance=2pt] (12)--(32) ;

 \epic
\]
This expanded diagram is easily seen to commute.

Commutativity of \A{413} results from \cite[3.4.7(iii)]{li}.

Subdiagram \A{411} (without $\delta_*$) expands as follows (with $\id$ the identity functor on $\D_Y\<$):
\[
    \begin{tikzpicture}[yscale=1.1]
      \draw[white] (0cm,3.5cm) -- +(0: \linewidth)
      node (11) [black, pos = 0.15] {$\CO_Y\<\<\otimes \delta^*\<t^*$}
      node (12) [black, pos = 0.43 ] {$\CO_Y\<\<\otimes {\id}$}
      node (13) [black, pos = 0.6525 ] {$\CO_Y\<\<\otimes (t \delta)^!$}
      node (14) [black, pos = 0.85] {$(t \delta)^!$};
      
      \draw[white] (0cm,2cm) -- +(0: \linewidth)
      node (22) [black, pos = 0.43 ] {$(t \delta)^!\CO_Y\<\<\otimes \id$}
      node (23) [black, pos = 0.6525 ] { }
      node (24) [black, pos = 0.85] {$(t \delta)^!\CO_Y\<\<\otimes (t \delta)^*$};
      
      \draw[white] (0cm,0.5cm) -- +(0: \linewidth)
      node (21) [black, pos = 0.15] {$(t \delta)^!\CO_Y\<\<\otimes \delta^*\<t^*$}
      node (34) [black, pos = 0.85] {$\delta^!t^!\CO_Y\<\<\otimes \delta^*\<t^*$};
      
       \draw [double distance=2pt](14) -- (24) ;
      \draw [double distance=2pt]
                 (24) -- (34) node[left, midway, scale=.75]{\raisebox{4pt}{$\via\>\ps^!\>$}}
                                    node[right=1pt, midway, scale=.75]{$\<\textup{and}\>\ps^*$};
      \draw [double distance=2pt]
                 (21) -- (34) node[auto, swap, midway, scale=0.75]{$\via\ps^!$};
      \draw [double distance=2pt]
                 (11) -- (12) node[above=1pt, midway, scale=0.75]{$\via\ps^*$};
      \draw [double distance=2pt]
                 (13) -- (14) node[above, midway, scale=0.75]{};
      \draw [double distance=2pt]
                 (21) -- (22) node[above=-1pt, midway, scale=0.75]{$\via\ps^*\mkern35mu$};
      \draw [double distance=2pt]
                 (11) -- (21) node[auto, swap,  midway, scale=0.75]{};
      \draw [double distance=2pt]
                 (12) -- (22) node[left, midway, scale=0.75]{};
      \draw [double distance=2pt]
                 (12) -- (13) ;
      \draw [double distance=2pt]
                 (22) -- (24) node[below, midway, scale=0.75]{};
      \node (N47) at (intersection of 12--24 and 22--14) [scale=0.9]{\A{414}};
 \end{tikzpicture}
\]
Subdiagram \A{414} commutes because all its maps are  identity maps---see paragraph following \eqref{! and otimes1}. The rest is clear.

\penalty-1500
 It remains to show commutativity of \A{3}, for which  one can omit  
 ``$\otimes\tau^*v^*\<$." 
 
Before proceeding, recall from \S\ref{! and otimes} that for perfect maps the restriction of $(-)^!$ to $\Dqcpl$ is
\emph{pseudofunctorially} isomorphic to $(-)_\upl^!$. Moreover, this isomorphism ``respects
flat base change." More specifically, referring to \eqref{def-of-B} and \cite[Exercise 4.9.3(c)]{li}
one finds that the following diagram commutes:
\[\mkern27mu
\CD
r^*u^!\OY@>\mathsf B>> w^!t_1^*\OY=w^!\CO\mkern27mu\\
@V\simeq V\eqref{f! and f!+}V @V\eqref{f! and f!+}V\simeq V \\
r^*u_\upl^!\OY@>>\bar{\mathsf  B}> w_\upl^!t_1^*\OY=w_\upl^!\CO\mkern27mu
\endCD
\]

Next, the composed map 
$$
g_*u^!\OY\overset{\via\ps^*}{=\!=} g_*g^*r^*u^!\OY= g_*(g^*r^*u^!\OY\otimes\OX)
$$ 
at the top of \A{3} is the same as
$$
g_*u_\upl^!\OY=g_*(u_\upl^!\OY\otimes u^*\OY) \overset{\via\ps^*}{=\!=} g_*(g^*r^*u_\upl^!\OY\otimes\OX).
$$ 

Also, using that 
$\delta_*(\delta_\upl^!t^!\OY\<\<\otimes\delta^*t^*\OY\<)\to\delta_*\delta_\upl^!t^!\OY\<\<\otimes t^*\OY$
from \eqref{projection} is the identity map (see \cite[3.4.7(iii)]{li}, and the remarks following \eqref{! and otimes1}), one finds
that the map $\bar\lambda_0(\OY)$ (see \eqref{lambda0}), that forms part of the definition of the map $\bar\lambda(\OZ)$ near the bottom of \A{3}\kf, factors as
\[
\delta_*\OY=\delta_*(t\delta)_\upl^!\OY\xto{\delta_*\!\ps^!_{\<\dpll}\>}\delta_*\delta_\upl^!t_\upl^!\OY
\xto{\int_{\<\dpll}\>} t_\upl^!\OY,
\]
with $\int_{\<\dpl}$ the unit map for the adjunction 
$(-)_*\<\dashv(-)_\upl^!$ in~\ref{! and otimes}(i). \va1

Hence, with $\varpi$ the associated
unit map,\va{.6} 
$\delta\set\delta_v$, $t\set t_2\>$,  $\theta$ as in~\eqref{def-of-bch-asterisco},
and recalling that $f^!\OY=f_{\<\<\upl}^!\OY$ for any flat $\SS$\kf-map $f$, one can expand \A{3} as\looseness=-1  
\[
\def\1{$g_*u_\upl^!\OY$}
\def\2{$g_*(\<\tau g)_\upl^{\<!}\OY$}
\def\3{$g_*g_\upl^!\tau_\upl^!\OY$}
\def\4{$\mkern15mu\tau_\upl^!\OY$}
\def\5{$g_*u^!(\<\tau\delta)_\upl_\upl^!\OY$}
\def\6{$g_*u^!\delta_\upl^!\tau_\upl^!\OY$}
\def\7{$g_*u^!delta_\upl^!\delta_*\delta_\upl^!\tau_\upl^!\OY$}
\def\8{$g_*g_\upl^!w^!\delta_*\delta_\upl^!\tau_\upl^!\OY$}
\def\9{$g_*u_\upl^!\delta_\upl^!\delta_*\OY$}
\def\ten{$g_*g_\upl^!w_\upl^!\<\delta_*\OY$}
\def\lvn{$w_\upl^!\<\delta_*\OY$}
\def\thn{$w_\upl^!\<\delta_*\delta_\upl^!t_\upl^!\OY$}
\def\ffn{$g_*(u_\upl^!\OY\<\<\otimes u^*\OY\<)$}
\def\sxn{$g_*(g^*\<r^*u_\upl^!\OY\<\<\otimes u^*\OY\<)$}
\def\svn{$r^*u_\upl^!\OY\<\otimes g_*u^*\OY$}
\def\egn{$w_\upl^!\CO\otimes w^*\<\delta_*\OY$}
\def\ntn{$w_\upl^!t_\upl^!\OY$}
\def\twy{$r^*u_\upl^!\OY\<\otimes w^*\<\delta_*\OY$}
\def\twn{$g_*g_\upl^!w_\upl^!t_\upl^!\OY$}
\def\ttw{$g_*u_\upl^!\delta_\upl^!t_\upl^!\OY$}
\def\tth{$g_*u_\upl^!\delta_\upl^!\delta_*\delta_\upl^!t_\upl^!\OY$}
\def\tfr{$g_*g_\upl^!w_\upl^!\<\delta_*\delta_\upl^!t_\upl^!\OY$}
 \bpic[xscale=4.8,yscale=1.7]
 
  \node(11) at (1,-1){\2};
  \node(12) at (2.3,-1){\1};
  \node(13) at (3,-1){\ffn};
 
  \node(21) at (1,-2){\3};
  \node(22) at (2.3,-2){\9};
  \node(23) at (3,-2){\sxn};

  \node(31) at (1,-3.3){\4};
  \node(32) at (2.3,-3){\ten};
  \node(33) at (3,-3){\svn};

  \node(41) at (1,-5.15){\ntn};
  \node(43) at (3,-4){\twy};
  
  \node(51) at (1.65,-5.15){\thn};
  \node(52) at (2.3,-5.15){\lvn};
  \node(53) at (3,-5.15){\egn};
  
  \node(1a) at (1.65,-1.5){\twn};
  \node(2a) at (1.65,-2.4){\ttw};
  \node(2a-) at (1.63,-2.52){};
  \node(2a+) at (1.67,-2.52){};
  \node(3a-) at (1.63,-3.3){};
  \node(3a+) at (1.67,-3.3){};
  \node(3a) at (1.65,-3.4){\tth};
  \node(4a) at (1.65,-4.4){\tfr};

   \draw[double distance=2pt] (11) -- (12) ; 
   \draw[double distance=2pt] (12) -- (13) ;

   \draw[double distance=2pt] (51) -- (52) node[below=1pt, midway, scale=0.75]{$\via\ps_\upl^!$}; 
   \draw[double distance=2pt] (52) -- (53);
 
   \draw[->] (11) -- (21) node[left, midway, scale=0.75]{$g_*\<\<\ps_\upl^!$}; 
   \draw[->] (21) -- (31) node[left, midway, scale=0.75]{$\int_{\<\dpl}$}; 
   \draw[double distance=2pt] (31) -- (41) node[left=1pt, midway, scale=0.75]{$\ps_\upl^!$}; 
   \draw[->] (51) -- (41) node[below=1pt, midway, scale=0.75]{$w^!\!\!\int_{\<\dpl}$}; 
 
   \draw[->] (12) -- (22) node[right=1pt, midway, scale=0.75]{$\via\varpi$};   
   \draw[double distance=2pt] (22) -- (32) node[right=1pt, midway, scale=0.75]{$\via\ps_\upl^!$}; 
   \draw[->] (32) -- (52) node[right=1pt, midway, scale=0.75]{$\int_{\<\dpl}$}; 

   \draw[double distance=2pt] (1a) -- (2a) node[left=1pt, midway, scale=0.75]{$\via\ps_\upl^!$}; 
   \draw[->] (2a+) -- (3a+) node[right=-.5pt, midway, scale=0.75]{\raisebox{7pt}{$\via\varpi$}};   
   \draw[<-] (2a-) -- (3a-) node[left=-.5pt, midway, scale=0.75]{\raisebox{7pt}{$\via\<\int_{\<\dpl}$}};   
   \draw[double distance=2pt] (3a) -- (4a) node[left=1pt, midway, scale=0.75]{$\via\ps_\upl^!$}; 
   
   \draw[double distance=2pt] (13) -- (23) node[right=1pt, midway, scale=0.75]{$\via\ps^*$}; 
   \draw[->] (23) -- (33) node[right=1pt, midway, scale=0.75]{\eqref{projection}};
   \draw[->] (33) -- (43) node[right, midway, scale=0.75]{$\via\theta^{-1}$};
   \draw[->] (43) -- (53) node[right, midway, scale=0.75]{via\kern2.5pt\eqref{thetaB}}; 

  \draw[double distance=2pt] (21) -- (1a) node[above=-4pt, midway, sloped, scale=0.7]{\rotatebox{-23}{$\mkern10mu\via\ps_\upl^!$}};
;
  \draw[double distance=2pt] (12) -- (2a) node[below=-5pt, midway, sloped, scale=0.7]{\rotatebox{-25}{$\via\ps_\upl^!\mkern25mu$}};
  \draw[double distance=2pt] (22) -- (3a) node[below=-5pt, midway, sloped, scale=0.7]{\rotatebox{-25}{$\via\ps_\upl^!\mkern25mu$}};
  \draw[double distance=2pt] (32) -- (4a) node[below=-5pt, midway, sloped, scale=0.7]{\rotatebox{-25}{$\via\ps_\upl^!\mkern25mu$}};
  \draw[<-] (51) -- (4a) node[left=1pt, midway, scale=0.75]{$\int_{\<\dpl}$}; 
  \draw[->] (1.47,-4.2).. controls +(-.22,.2) and +(-.22,-.1) ..(1.485,-1.72)node[left=-.5pt, midway, scale=0.75]{\raisebox{7pt}{$\via\<\int_{\<\dpl}$}}; 
   
  \node at (2.65,-3.5) [scale=.9]{\A{31}};
 \epic
\]

Commutativity of the unlabeled subdiagrams is readily checked. Also,
the composition 
$$
\delta^!_\upl\xto{\via\varpi\>} \delta^!_\upl\delta_*\delta^!_\upl\xto{\int_{\<\dpll}}\delta^!_\upl
$$ is the identity map.
Diagram chasing\- shows then that we need ``only" check that \A{31} commutes.\footnote{For the authors, this was the most elusive point in the present proof of Theorem~\ref{trans fc}. } For this purpose we can even drop the
final $\OY$ at each vertex, regarding what's left as a diagram of functors defined on $\Dqcpl\>$.

Henceforth we will use the symbol ``$\prj$" to refer to either of the projection isomorphisms in \eqref{projection}, 
or their inverses.

One last preparatory remark: the isomorphism \eqref{f! and f!+} is a special case of
a canonical functorial map, defined in \cite[5.8]{Nk}\footnote{where $\qct$ should 
be $\qc$}\, \emph{for any\/ $\SS$\kf-map} $f\colon X\to Y$,
\[
\kappa(F)\colon f_{\<\<\upl}^!\OY\otimes f^*\<\<F \to f_{\<\<\upl}^!F\qquad \big(F\in\Dqcpl(Y)\big).
\]

If $f$ is proper, $\kappa(F)$ is adjoint to the composition
\[
f_*(f_{\<\<\upl}^!\OY\otimes f^*\<\<F)\xto{\prj} f_*f_{\<\<\upl}^!\OY\otimes F
\xto{\<\via\int_{\<\dpll}\>} \OY\otimes F=F.
\]

If $f$ is essentially \'etale, so that $f_{\<\<\upl}^!=f^*\<$, then $\kappa(F)$ is the identity map
of~$f^*\<\<F$.

One checks that $\kappa(\OY)$ is the identity map.\va3

It should now be clear that the next Proposition will complete the proof.

\begin{subprop}\label{lastprop} Not assuming\/ $u$ or\/ $w$  flat, consider any commutative\/ $\SS$\kf-diagram
\begin{equation*}\label{squares2}
  \mkern125mu
    \begin{tikzpicture}[xscale=1.32, yscale=.9]
      \draw[white] (0cm,5.1cm) -- +(0: \linewidth)
      node (21) [black, pos = 0.1] {$X$}
      node (22) [black, pos = 0.3] {$Y$};
      \draw[white] (0cm,2.8cm) -- +(0: \linewidth)
      node (31) [black, pos = 0.1] {$X'$}
      node (32) [black, pos = 0.3] {$Y'$};
      \draw[white] (0cm,0.5cm) -- +(0: \linewidth)
      node (51) [black, pos = 0.1] {$X$}
      node (52) [black, pos = 0.3] {$Y$};
      \draw [->] (21) -- (31) node[left, midway, scale=0.75]{$g$};
      \draw [->] (31) -- (51) node[left, midway, scale=0.75]{$r$};
      \draw [->] (22) -- (32) node[right, midway, scale=0.75]{$\delta$};
      \draw [->] (32) -- (52) node[right, midway, scale=0.75]{$t$};
      \draw [->] (21) -- (22) node[above, midway, sloped, scale=0.75]{$u$};
      \draw [->] (31) -- (32) node[above, midway, sloped, scale=0.75]{$w$};
      \draw [->] (51) -- (52) node[below, midway, sloped, scale=0.75]{$u$};
    
      \node  at (intersection of 21--32 and 22--31) [scale=0.8] {$\De$};
      \node   at (intersection of 31--52 and 32--51) [scale=0.8] {$\De'$};

      \end{tikzpicture}
  \end{equation*}
with $rg=\id_\sX,$  $t\delta=\id_Y,$ 
$\De'\ ($hence\/ $\De)$ a fiber square,  and\/ $t\ ($hence\/ $r)$ flat\/  $($cf.~\eqref{triplesquare}$)$.\va1 

The following diagram of\/ $\Dqcpl$-valued functors commutes.
\begin{equation*}\label{diag-lastprop}
\def\1{$g_*u_\upl^!$}
\def\2{$g_*(u_\upl^!\OY\<\otimes u^*\<)$}
\def\3{$g_*u_\upl^!\delta_\upl^!\delta_*$}
\def\4{$g_*(g^*\<r^*u_\upl^!\OY\<\<\otimes u^*\<)$}
\def\5{$g_*g_\upl^!w_\upl^!\<\delta_*$}
\def\6{$r^*u_\upl^!\OY\<\otimes g_*u^*$}
\def\7{$r^*u_\upl^!\OY\<\otimes w^*\<\delta_*$}
\def\8{$w_\upl^!\<\delta_*$}
\def\9{$w_\upl^!\CO_{Y'}\otimes w^*\<\delta_*$}
\CD
 \bpic[xscale=6,yscale=2]
 
  \node(12) at (2.3,-1){\1};
  \node(13) at (3,-1){\2};
 
  \node(22) at (2.3,-2){\3};
  \node(23) at (3,-2){\4};

  \node(32) at (2.3,-3){\5};
  \node(33) at (3,-3){\6};

  \node(43) at (3,-4){\7};
  
  \node(52) at (2.3,-5){\8};
  \node(53) at (3,-5){\9};
  
   \draw[<-] (12) -- (13) node[above=1pt, midway, scale=0.75]{$g_*\kappa$};
 
   \draw[<-] (52) -- (53) node[below=1pt, midway, scale=0.75]{$\kappa(\delta_*)$};
 
   \draw[->] (12) -- (22) node[left=1pt, midway, scale=0.75]{$\via\varpi$};   
   \draw[double distance=2pt] (22) -- (32) node[left=1pt, midway, scale=0.75]{$\via\ps_\upl^!$}; 
   \draw[->] (32) -- (52) node[left=1pt, midway, scale=0.75]{$\int_{\<\dpl}$}; 

   \draw[double distance=2pt] (13) -- (23) node[right=1pt, midway, scale=0.75]{$\via\ps^*$}; 
   \draw[<-] (23) -- (33) node[right=1pt, midway, scale=0.75]{$\prj$};
   \draw[<-] (33) -- (43) node[right, midway, scale=0.75]{$\via\theta$};
   \draw[->] (43) -- (53) node[left, midway, scale=0.75]{$\simeq$}
                                    node[right, midway, scale=0.75]{\textup{via}\kern2.5pt\eqref{thetaB}}; 

 \epic
\endCD\tag{\ref{lastprop}.1}
\end{equation*}

\end{subprop}

\begin{proof} 
We deal only with the pseudofunctor $(-)_\upl^!$, and not with $(-)^!$, so
to reduce notational clutter we will denote $f_{\<\<\upl}^!$ by $f^!$, for any $\SS$\kf-map $f$.
Likewise, we will denote $\int_{\<\dpl}$ simply by $\int$.

We will prove~\ref{lastprop} when $u$ (hence $w$)
is essentially \'etale (see \S\ref{efp}),  and then when $u$ (hence $w$) is proper. Then finally
we will use the fact that any $\SS$\kf-map is of the form (proper)$\smallcirc$(essentially \'etale)
\cite[4.1 and 2.7\kf]{Nk} to establish the general case.\va2

Let  us assume then, to begin, that $u$ and $w$ are essentially \'etale, so that 
$u^!=u^*$ and $w^!=w^*$. Note that since $g$ and $\delta$ have left inverses, they are closed immersions \cite[p.\,278, (5.2.4)]{EGA1}. 

In the next diagram, subrectangle \circled4 is as in \eqref{diag-lastprop}; the ``base change" 
map\looseness=-1 
$$
\bar{\mathsf B}\set\bar{\mathsf B}_{\De}\colon u^*\delta^!\to g^!w^*
$$ 
is defined to be adjoint to the natural composition 
$$
g_*u^*\<\delta^!\xto{\theta^{-1}} w^*\<\delta_*\delta^!\xto{\int} w^*;
$$
and the unlabeled maps are the obvious ones.  (In particular, three  maps in the rightmost column are \emph{identity maps,} see the last paragraph in \S\ref{adj-pseudo}.)
\[\mkern-1mu
\def\1{$\quad g_*u^!\quad$}
\def\2{$g_*(u^!\OY\<\otimes u^*\<)$}
\def\3{$g_*u^!\delta^!\delta_*$}
\def\4{$g_*(g^*\<r^*u^!\OY\<\<\otimes u^*\<)$}
\def\5{$g_*g^!w^!\<\delta_*$}
\def\6{$r^*u^!\OY\<\otimes g_*u^*$}
\def\7{$r^*u^!\OY\<\otimes w^*\<\delta_*$}
\def\8{$\quad w^!\<\delta_*\quad$}
\def\9{$w^!\CO_{Y'}\otimes w^*\<\delta_*$}
\def\lvn{$g_*u^*$}
\def\twv{$g_*(u^*\OY\<\otimes u^*\<)$}
\def\thn{$g_*u^*\<\delta^!\delta_*$}
\def\frn{$g_*(g^*\<r^*u^*\OY\<\<\otimes u^*\<)$}
\def\ffn{$g_*g^!w^*\<\delta_*$}
\def\sxn{$r^*u^*\OY\<\otimes g_*u^*$}
\def\svn{$r^*u^*\OY\<\otimes w^*\<\delta_*$}
\def\egn{$w^*\<\delta_*$}
\def\ntn{$w^*\CO_{Y'}\otimes w^*\<\delta_*$}
 \bpic[xscale=3.4,yscale=1.5]
 
  \node(12) at (1.83,-.95){\1};
  \node(12b) at (2,-.95){};
  \node(13) at (2.9,-.95){\2};
 
  \node(22) at (1.83,-2){\3};
  \node(23) at (2.9,-2){\4};

  \node(32) at (1.83,-3){\5};
  \node(33) at (2.9,-3){\6};

  \node(43) at (2.9,-4){\7};
    
  \node(52) at (1.83,-5.05){\8};
  \node(53) at (2.9,-5.05){\9};
  \node(52b) at (2,-5.05){};
  
  \node(12a) at (1,-.25){\lvn};
  \node(13a) at (4,-.25){\twv};
 
  \node(22a) at (1,-2){\thn};
  \node(23a) at (4,-2){\frn};

  \node(32a) at (1,-3){\ffn};
  \node(33a) at (4,-3){\sxn};

  \node(43a) at (4,-4){\svn};
  
  \node(52a) at (1,-5.75){\egn};
  \node(53a) at (4,-5.75){\ntn};

   \draw[<-] (12b) -- (13) node[below, midway, scale=0.75]{$g_*\kappa$};
   \draw[double distance=2pt] (12a) -- (13a) node[above=1pt, midway, scale=0.75]{$$};
 
   \draw[->] (53) -- (52b) node[above, midway, scale=0.75]{$\kappa(\delta_*)$};
   \draw[double distance=2pt] (52a) -- (53a) node[below=1pt, midway, scale=0.75]{$$};
 
   \draw[->] (12) -- (22) node[left=1pt, midway, scale=0.75]{};   
   \draw[double distance=2pt] (22) -- (32) node[right=1pt, midway, scale=0.75]{$\via\ps_\upl^!$}; 
   \draw[->] (32) -- (52) node[left=1pt, midway, scale=0.75]{}; 

   \draw[->] (12a) -- (22a) node[left=1pt, midway, scale=0.75]{$$};   
   \draw[->] (22a) -- (32a) node[left=1pt, midway, scale=0.75]{$\via \bar{\mathsf B}$}; 
   \draw[->] (32a) -- (52a) node[left=1pt, midway, scale=0.75]{$$}; 

   \draw[double distance=2pt] (13) -- (23) node[left=1pt, midway, scale=0.75]{$\via\ps^*$}; 
   \draw[<-] (23) -- (33) node[right=1pt, midway, scale=0.75]{};
   \draw[<-] (33) -- (43) node[right, midway, scale=0.75]{};
   \draw[->] (43) -- (53) node[left, midway, scale=0.75]{\textup{via}\kern2.5pt\eqref{thetaB}}; 

   \draw[double distance=2pt] (13a) -- (23a) node[right=1pt, midway, scale=0.75]{$$}; 
   \draw[double distance=2pt] (23a) -- (33a) node[right=1pt, midway, scale=0.75]{};
   \draw[<-] (33a) -- (43a) node[right, midway, scale=0.75]{$\theta$};
   \draw[double distance=2pt] (43a) -- (53a) node[right=1pt, midway, scale=0.75]{$$}; 
   
   \draw[double distance=2pt] (1.66, -.785) -- (1.11, -.45);
   \draw[double distance=2pt] (13) -- (13a);
   \draw[double distance=2pt] (22) -- (22a);
   \draw[double distance=2pt] (23) -- (23a);
   \draw[double distance=2pt] (32) -- (32a);
   \draw[double distance=2pt] (33) -- (33a);
   \draw[double distance=2pt] (43) -- (43a);
   \draw[double distance=2pt] (53) -- (53a);
   \draw[double distance=2pt] (1.66, -5.23) -- (1.11, -5.55);

   \node at (1.415,-2.5) {\circled1};
   \node at (2.27,-.65) {\circled2};
   \node at (2.28,-5.35) {\circled2${}'$};
   \node at (2.27,-3.05){\circled4};
   \node at (3.5,-4.8) {\circled3};

  \epic
\]

It is clear that the unlabeled subdiagrams commute. Subdiagram \circled1 commutes by \ref{! and otimes}(iii),  \circled2 and \circled2${}'$ commute by 
\cite[definition of 4.9.1.1]{li}, and \circled3 commutes, in view of \ref{! and otimes}(iii),  by the description of $r^*u^!\OY\to w^!\CO_{Y'}$ in~\ref{indsq} with $(f,u,g,v)\set(u,t,w,r)$---all valid for \emph{essentially} \'etale maps. 
So to achieve our goal of proving that \circled4 commutes, we need only do the same
for the outer one.

Proving commutativity of the outer rectangle means showing that the left column composes to
$\theta^{-1}$, that is, the next diagram commutes:

\[
 \bpic[xscale=3, yscale=1.5]
  
  \node(11) at (1,-1) {$g_*u^*\delta^!\delta_*$};
  \node(13) at (3,-1) {$g_*u^*$}; 
  
  \node(21) at (1,-2) {$g_*g^!w^*\delta_*$};
  \node(215) at (1.45,-2.05) {\circled5};
  \node(22) at (2,-2) {$w^*\delta_*\delta^!\delta_*$};
  
  \node(31) at (1,-3) {$w^*\delta_*$};
  \node(33) at (3,-3) {$w^*\delta_*$};
  
  \draw[<-] (11) -- (13)node[above=1pt, midway, scale=0.75] {$\via\varpi$};
  \draw[double distance=2pt] (31) -- (33);
  
  \draw[->] (11) -- (21) node[left=1pt, midway, scale=0.75]{$\via \bar{\mathsf B}$}; ;
  \draw[->] (21) -- (31);
  \draw[->] (13) -- (33) node[right=1pt, midway, scale=0.75]{$\theta^{-1}$}; 

  \draw[->] (11) -- (22) node[above=-1pt, midway, scale=0.75]{$\mkern30mu\theta^{-1}$};  
  \draw[->] (33) -- (22) node[below=-4pt, midway, scale=0.75]{$\via\varpi\mkern55mu$};   
  \draw[->] (22) -- (31) node[below=-5pt, midway, scale=0.75]{$\mkern50mu\via\<\<\int$};   
   
 \epic
\]
Here, subdiagram \circled5 commutes by the definition of $\bar{\mathsf B}$, and commutativity of the other
two subdiagrams is obvious. 

Thus Proposition~\ref{lastprop}  holds when $u$ and~$w$ are essentially \'etale.

\smallskip
Suppose next that $u$ and $w$ are \emph{proper}. 

It will suffice to show commutativity of the adjoint of  diagram \eqref{diag-lastprop}, namely subdiagram~\circled6  in
\[
\def\1{$w_*g_*u^!$}
\def\2{$w_*g_*(u^!\OY\<\otimes u^*\<)$}
\def\3{$w_*g_*u^!\delta^!\delta_*$}
\def\4{$w_*g_*(g^*\<r^*u^!\OY\<\<\otimes u^*\<)$}
\def\5{$w_*g_*g^!w^!\<\delta_*$}
\def\6{$w_*(r^*u^!\OY\<\otimes g_*u^*)$}
\def\7{$w_*(r^*u^!\OY\<\otimes w^*\<\delta_*)$}
\def\8{$\delta_*$}
\def\9{$w_*(w^!\CO_{Y'}\<\otimes w^*\<\delta_*)$}
\def\ten{$w_*w^!\<\delta_*$}
\def\lvn{$\delta_*u_*u^!$}
\def\twv{$w_*w^!\CO_{Y'}\<\otimes \delta_*$}
\def\thn{$\CO_{Y'}\<\otimes \delta_*$}
 \bpic[xscale=4.75,yscale=2]
 
  \node(11) at (1,-1){\lvn};
  \node(12) at (1.77,-1){\1};
  \node(13) at (3,-1){\2};
 
  \node(22) at (1.77,-2){\3};
  \node(23) at (3,-2){\4};

  \node(32) at (1.77,-3){\5};
  \node(325) at (2.3,-3){\circled6};
  \node(33) at (3,-3){\6};

  \node(42) at (1.77,-4){\ten};
  \node(43) at (3,-4){\7};
  
  \node(52) at (1.77,-5.3){\8};
  \node(521) at (2.1,-4.55){\thn};
  \node(527) at (2.7,-4.55){\twv};
  \node(53) at (3,-5.3){\9};
  
   \draw[<-] (12) -- (13) node[above=1pt, midway, scale=0.75]{$w_*g_*\kappa$};
   
   \draw[<-] (521) -- (527) node[below=1pt, midway, scale=0.75]{$\via\int$};
   \draw[<-] (527) -- (53) node[below=-1.5pt, midway, scale=0.75]{$\prj\mkern20mu$};
 
   \draw[->] (12) -- (22) node[right=1pt, midway, scale=0.75]{$\via\varpi$};   
   \draw[double distance=2pt] (22) -- (32) node[right=1pt, midway, scale=0.75]{$\via\ps_\upl^!$}; 
   \draw[->] (32) -- (42) node[right=1pt, midway, scale=0.75]{$\via\<\<\int$}; 
   \draw[->] (42) -- (52) node[right, midway, scale=0.75]{$\int$}; 
   
   \draw[double distance=2pt] (13) -- (23) node[right=1pt, midway, scale=0.75]{$\via\ps^*$}; 
   \draw[<-] (23) -- (33) node[right=1pt, midway, scale=0.75]{via\:$\prj$};
   \draw[<-] (33) -- (43) node[right, midway, scale=0.75]{$\via\theta$};
   \draw[->] (43) -- (53) node[right, midway, scale=0.75]{\textup{via}\kern2.5pt\eqref{thetaB}}; 

    \draw[double distance=2pt] (521) -- (52) ; 
   \draw[double distance=2pt] (12) -- (11) node[above=1pt, midway, scale=0.75]{$\ps_*$}; 
   \draw[->] (11) -- (52) node[above=-8pt, midway, scale=0.75]{$\via\<\<\int\mkern65mu$}; 

  \epic
\]

The subtriangle commutes by \cite[3.10.4(c)]{li}, applied to the map denoted there by
$\phi\colon g_*u^!\to u^!f_*$.  (Recall that over proper maps the pseudofunctor 
$(-)^!\set (-)_\upl^!$ is right-adjoint to $(-)_*$, and so may be identified with the pseudofunctor 
$(-)^\times$ in \cite{li}.) So it's enough to show  commutativity of the outer border.

Fill in that border as follows (with id the identity functor on $\Dqcpl(Y)$).
In this diagram,  the maps $\alpha$, $\beta$ and $\gamma$ are the respective composites
\begin{align*}
\alpha&\colon  
g^*w^*\delta_*\overset{\ps^*}{=\!=} u^*\delta^*\delta_*\xto{u^{\<*}\<\epsilon^{}_{\<\delta}} u^*.\\
\beta&\colon(\delta_*\<\<-\>\otimes\; \delta_* -)\xto{\prj\>} \delta_*(-\otimes\>\delta^*\delta_*-)
\xto{\via\epsilon^{}_{\<\delta}\>\>} \delta_*(-\otimes -).\\
\gamma&\colon r^*\xto{\,\eta^{}_g\,} g_*g^*r^*\overset{g_*\!\ps^*}{=\!=\!=}  g_*.
\end{align*}

\[
\mkern-2mu
\def\1{$w_*r^*\<u^!\OY\<\<\otimes\<\delta_*$}
\def\2{$w_*g_*u^!\OY\<\<\otimes\<\delta_*$}
\def\3{$w_*(g_*u^!\OY\<\<\otimes\< w^*\<\delta_*\<)$}
\def\4{$\!w_*g_*(u^!\OY\<\<\otimes\< g^*w^*\<\delta_*\<)$}
\def\5{$\delta_*u_*u^!\OY\<\<\otimes\<\delta_*$}
\def\6{$\delta_*\OY\<\<\otimes\<\delta_*$}
\def\7{$\,\delta_*(\OY\<\<\otimes\<\id)$}
\def\8{$\delta_*(\delta^*\CO_{Y'}\<\<\otimes\<\id)$}
\def\9{$\CO_{Y'}\<\<\otimes\<\delta_*$}
\def\ten{$w_*(w^!\CO_{Y'}\<\<\otimes\< w^*\<\delta_*\<)$}
\def\lvn{$w_*(r^*\<u^!\OY\<\<\otimes\< w^*\<\delta_*\<)$}
\def\twv{$w_*(r^*\<u^!\OY\<\<\otimes\< g_*u^*)$}
\def\thn{$\!\!w_*g_*(g^*r^*\<u^!\OY\<\<\otimes\< u^*\<)$}
\def\frn{$w_*g_*(u^!\OY\<\<\otimes\< u^*)$}
\def\ffn{$w_*g_*u^!$}
\def\sxn{$\delta_*u_*u^!$}
\def\svn{$\delta_*$}
\def\egn{$\delta_*u_*(u^!\OY\<\<\otimes\< u^*)$}
\def\ntn{$\delta_*(u_*u^!\OY\<\<\otimes\< \id)$}
\def\twy{$w_*w^!\CO_{Y'}\<\<\otimes\<\delta_*$}
 \bpic[xscale=3.4,yscale=1.8]

  \node(11) at (1,-1)[scale=.93]{\sxn}; 
  \node(12) at (2,-1)[scale=.93]{\ffn};
  \node(13) at (2.85,-1)[scale=.93]{\frn};
  \node(14) at (3.95,-1)[scale=.93]{\frn};
 
  \node(22) at (1.85,-2)[scale=.93]{\egn};
  \node(23) at (2.85,-2)[scale=.93]{\4};
  \node(24) at (3.95,-2)[scale=.93]{\thn};

  \node(32) at (1.85,-2.92)[scale=.93]{\ntn};
  \node(33) at (2.85,-2.92)[scale=.93]{\3};
  \node(34) at (3.95,-2.92)[scale=.93]{\twv};
  \node(51) at (1,-2.95)[scale=.93]{\svn};

  \node(42) at (1.85,-4.08)[scale=.93]{\5};
  \node(43) at (2.85,-4.08)[scale=.93]{\2};
   
  \node(52) at (1.85,-5)[scale=.93]{\6};
  \node(53) at (2.85,-5)[scale=.93]{\1};
  \node(54) at (3.95,-5)[scale=.93]{\lvn};
  
  \node(61) at (1,-5)[scale=.93]{\7};
  \node(62) at (1,-6)[scale=.93]{\9};
  \node(63) at (2.85,-6)[scale=.93]{\twy};
  \node(64) at (3.95,-6)[scale=.93]{\ten};

   \draw[double distance=2pt] (11) -- (12) node[above=1pt, midway, scale=.7]{$\ps_*$};  
   \draw[<-] (12) -- (13) node[above=1pt, midway, scale=.7]{$w_*g_*\kappa$};
   \draw[double distance=2pt] (13) -- (14) ;

   \draw[double distance=2pt] (42) -- (43) node[above=1pt, midway, scale=.7]{$\via\ps_*$}; 
  
   \draw[<-] (53) -- (54) node[below=1pt, midway, scale=.7]{$\prj$};

   \draw[double distance=2pt] (61) -- (62);   
   \draw[<-](62) -- (63) node[below, midway, scale=.7]{$\via\int$};
   \draw[->] (64) -- (63) node[below=1pt, midway, scale=.7]{$\prj$};

  \draw[->] (11) -- (51) node[left, midway, scale=.7]{$\via\!\int$}; 
  \draw[double distance=2pt] (51) -- (61) ; 

  \draw[<-] (22) -- (32) node[right=1pt, midway, scale=.7]{$\via\prj$};
  \draw[<-] (32) -- (42) node[right=1pt, midway, scale=.7]{$\via\beta$};
  \draw[->] (42) -- (52) node[right=1pt, midway, scale=.7]{$\via\<\int$}; 

  \draw[<-] (13) -- (23) node[left, midway, scale=.7]{$\via\alpha$}; 
  \draw[<-] (23) -- (33) node[left=1pt, midway, scale=.7]{$\via\prj$}; 
  \draw[<-] (33) -- (43) node[left=1pt, midway, scale=.7]{$\prj$};
  \draw[<-] (43) -- (53) node[left, midway, scale=.7]{$\via\gamma$};
  \draw[->] (53) -- (63) node[left, midway, scale=.7]{\textup{via}\kern2.5pt\eqref{thetaB}}; 

  \draw[double distance=2pt] (14) -- (24) node[right=1pt, midway, scale=.7]{$\via\ps^*$}; 
  \draw[<-] (24) -- (34) node[right=1pt, midway, scale=.7]{$\via\prj$};
  \draw[<-] (34) -- (54) node[right, midway, scale=.7]{$\via\theta$};
  \draw[->] (54) -- (64) node[right, midway, scale=.7]{\textup{via}\kern2.5pt\eqref{thetaB}}; 

  \draw[double distance=2pt] (13) -- (22) node[above=-2pt, midway, scale=.7]{$\ps_*\mkern40mu$};
  \draw[->] (22) -- (11) node[above=-2pt, midway, scale=.7]{$\mkern45mu\delta_*u_*\kappa$};
  \draw[->] (32) -- (61) node[above=2pt, midway, scale=.7]{$\via\<\int\mkern25mu$}; 
  \draw[->] (52) -- (61) node[below=1pt, midway, scale=.7]{$\via\beta$};
  \draw[->] (54) -- (33) node[below=-5pt, midway, scale=.7]{$\mkern65mu\via\gamma$};

  \node at (2.32,-2.73)[scale=.93] {\circled7};
  \node at (2.31,-5.05) [scale=.93]{\circled8};
  \node at (3.4,-2.55) [scale=.93]{\circled9};

 \epic
\]

Commutativity of the unlabeled subdiagrams is easily checked. (For the leftmost, see the preparatory remarks just before Proposition~\ref{lastprop}).

For showing commutativity of  \circled7, expand it as follows, with $A\set u^!\OY$:
\[\mkern-3mu
\def\1{$w_*g_*(\<A\otimes u^*)$}
\def\2{$\delta_*u_*(\<A\<\otimes\< u^*)$}
\def\3{$w_*g_*(\<A\<\otimes\< g^*w^*\<\delta_*\<)$}
\def\4{$\delta_*(u_*A\<\otimes\< \id)$}
\def\5{$w_*(g_*A\<\otimes\< w^*\<\delta_*\<)$}
\def\6{$\delta_*u_*A\<\otimes\<\delta_*$}
\def\7{$w_*g_*A\<\otimes\<\delta_*$}
\def\8{$(\delta u)_*A\<\otimes\<\delta_*$}
\def\9{$(wg)_*A\<\otimes\<\delta_*$}
\def\ten{$(\delta u)_*(\<A\<\otimes\<(\delta u)^*\delta_*\<)$}
\def\lvn{$(wg)_*(\<A\<\otimes\<(wg)^*\delta_*\<)$}
\def\twv{$\delta_* u_*(\<A\<\otimes\<u^*\delta^*\delta_*\<)$}
\def\thn{$w_* g_*(\<A\<\otimes\<u^*\delta^*\delta_*\<)$}
\def\frn{$\delta_*(u_* A\<\otimes\<\delta^*\delta_*\<)$}
 \bpic[xscale=3.3,yscale=1.7]

  \node(11) at (1,-1){\2};
  \node(14) at (4,-1){\1}; 
 
  \node(22) at (1.8,-2.2){\twv};
  \node(23) at (2.94,-2.2){\thn};
  \node(24) at (4,-2.2){\3};

  \node(31) at (1,-2.75){\frn};
  \node(32) at (1.8,-3.25){\ten};
  \node(33) at (2.94,-3.25){\lvn};

  \node(41) at (1,-1.85){\4};
  \node(42) at (1.8,-4.25){\8};
  \node(43) at (2.94,-4.25){\9};
  \node(44) at (4,-4.25){\5};

  \node(51) at (1,-5){\6};
  \node(54) at (4,-5){\7};


   \draw[double distance=2pt] (14) -- (11) node[above=1pt, midway, scale=.7]{$\ps_*$};

   \draw[double distance=2pt] (22) -- (23) node[above=1pt, midway, scale=.7]{$\ps_*$};
   \draw[double distance=2pt] (23) -- (24) node[below=1pt, midway, scale=.7]{$\ps^*$};

   \draw[double distance=2pt] (32) -- (33) ; 
 
   \draw[double distance=2pt] (42) -- (43) ;

   \draw[double distance=2pt] (51) -- (54) node[below=1pt, midway, scale=.7]{$\via\ps_*$}; 

  \draw[<-] (11) -- (41) node[left=1pt, midway, scale=.7]{$\via\prj$};
  \draw[<-] (31) -- (51) node[left=1pt, midway, scale=.7]{$\prj$};
  \draw[<-] (41) -- (31) node[left=1pt, midway, scale=0.75]{$\via\epsilon^{}_{\<\delta}$}; 

  \draw[double distance=2pt] (32) -- (22) node[left=1pt, midway, scale=0.75]{\raisebox{-8pt}{$\ps_*$} }
                                                                node[right=1pt, midway, scale=0.75]{$\ps^*$};
  \draw[->] (42) -- (32) node[left=1pt, midway, scale=.7]{$\prj$};
  
   \draw[double distance=2pt] (33) -- (23) node[left=1pt, midway, scale=0.75]{$\raisebox{-8pt}{$\ps_*$}$} 
                                                                node[right=1pt, midway, scale=0.75]{$\ps^*$};
   \draw[->] (43) -- (33) node[right=1pt, midway, scale=.7]{$\prj$};

   \draw[<-] (14) -- (24) node[right, midway, scale=.7]{$\via\alpha$};    
   \draw[<-] (24) -- (44) node[right=1pt, midway, scale=.7]{$\via\prj$}; 
   \draw[<-] (44) -- (54) node[right=1pt, midway, scale=.7]{$\prj$};

   \draw[<-] (11) -- (22) node[above=-4pt, midway, scale=0.75]{$\mkern57mu\via\epsilon^{}_{\<\delta}$};
   \draw[<-] (14) -- (23) node[above=-4pt, midway, scale=0.75]{$\via\epsilon^{}_{\<\delta}\mkern83mu$};
     \draw[<-] (22) -- (31) node[below=-2pt, midway, scale=.7]{$\mkern30mu\via\prj$};

   \draw[double distance=2pt] (42) -- (51) node[below=-2pt, midway, scale=.7]{$\mkern65mu\via\ps_*$};
   \draw[double distance=2pt] (43) -- (54) node[below=-2pt, midway, scale=.7]{$\via\ps_*\mkern50mu$};

%
   \node at (1.35,-3.75)[scale=0.9]{\circled7$_1$};
   \node at (3.47,-3.75)[scale=0.9]{\circled7$_2$};
 \epic
 \]

Commutativity of the unlabeled subdiagrams is clear. Subdiagrams~\circled7$_1$ 
and~\circled7$_2$ commute by \cite[3.7.1]{li}, \emph{mutatis mutandis.}

Thus \circled7 commutes.\va2

\smallskip
Expand  \circled 8 as follows, where  $\bar\gamma$ is the composition 
$t^*\xto{\eta^{}_{\<\delta}}\delta_*\delta^*t^*\xto{\delta_{\<*}\!\ps^*}\delta_*\>$:
\[
\def\0{$\delta_*$}
\def\1{$w_*w^!\CO_{Y'}\<\<\otimes\<\delta_*$}
\def\2{$\CO_{Y'}\<\<\otimes\<\delta_*$}
\def\3{$\,\delta_*(\OY\<\<\otimes\<\id)$}
\def\4{$\delta_*\OY\<\<\otimes\<\delta_*$}
\def\5{$\delta_*u_*u^!\OY\<\<\otimes\<\delta_*$}
\def\6{$w_*g_*u^!\OY\<\<\otimes\<\delta_*$}
\def\7{$w_*r^*\<u^!\OY\<\<\otimes\<\delta_*$}
\def\8{$t^*u_*u^!\OY\<\<\otimes\<\delta_*$}
\def\9{$t^*\<\OY\<\<\otimes\<\delta_*$}
 \bpic[xscale=4.7,yscale=1.85]
  \node(10) at (1,-1){\3};
  \node(11) at (1.667,-1){\4};
  \node(12) at (2.27,-1){\5};
  \node(13) at (3,-1){\6};

  \node(21) at (1,-2){\0};  
  \node(215) at (1.667,-2){\9};
  \node(22) at (2.27,-2){\8};
  \node(23) at (3,-2){\7};  

  \node(31) at (1,-3){\2};
  \node(33) at (3,-3){\1};
  
   \draw[->] (11) -- (10) node[above=1pt, midway, scale=.75]{$\via\beta$}; 
   \draw[<-] (11) -- (12) node[above=1pt, midway, scale=.75]{$\via\<\int$}; 
   \draw[double distance=2pt] (12) -- (13) node[above=1pt, midway, scale=.75]{$\via\ps_*$}; 

   \draw[<-](215) -- (22) node[above, midway, scale=.75]{$\via\int$};
   \draw[<-] (22) -- (23) node[above=1.5pt, midway, scale=.75]{$\<\via\theta^{-\<1}$}; 

   \draw[<-](31) -- (33) node[below, midway, scale=.75]{$\via\int$};
   
  \draw[double distance=2pt] (10) -- (21) ; 
  \draw[double distance=2pt] (21) -- (31) ; 

  \draw[<-] (13) -- (23) node[right=1pt, midway, scale=.7]{$\via\gamma$};
  \draw[->] (23) -- (33) node[right=1pt, midway, scale=.7]{\textup{via}\kern2.5pt\eqref{thetaB}}; 

  \draw[->] (215) -- (11) node[right=1pt, midway, scale=.75]{$\via\bar\gamma$}; 
  \draw[->] (22) -- (12) node[right=1pt, midway, scale=.75]{$\via\bar\gamma$}; 
  \draw[double distance=2pt] (215) -- (31);
  
  \node at (1.316,-1.75){\circled8$_1$};
  \node at (2.667,-1.5){\circled8$_2$};
  \node at (2.2,-2.55){\circled8$_3$};
 \epic
\]

Commutativity of the unlabeled subdiagram is obvious.

Next, commutativity of subdiagram \circled8$_1$ is equivalent to that of its adjoint, which, since
$\prj\colon\delta_*\OY\<\<\otimes\<\delta_*\to \delta_*(\OY\<\<\otimes\<\delta^*\delta_*)$ is adjoint to the composition 
\[
\delta^*(\delta_*\OY\<\<\otimes\<\delta_*)
\xto{\eqref{^* and tensor}\>}
\delta^*\delta_*\OY\<\<\otimes\<\delta^*\delta_*
\xto{\<\via\epsilon^{}_\delta\>}
\OY\<\<\otimes\<\delta^*\delta_*
\]
(cf.~\cite[3.4.6.2]{li}), is the outer border of
\[
\mkern-6mu
\def\1{$\delta^*\<(\CO_{Y'}\<\<\otimes\<\delta_*\<)$}
\def\2{$\delta^*\<(t^*\OY\<\<\otimes\<\delta_*\<)$}
\def\3{$\delta^*\<(\delta_*\delta^*t^*\OY\<\<\otimes\<\delta_*\<)$}
\def\4{$\delta^*\<(\delta_*\OY\<\<\otimes\<\delta_*\<)$}
\def\5{$\delta^*\delta_*\OY\<\<\otimes\<\delta^*\<\delta_*$}
\def\6{$\OY\<\<\otimes\<\delta^*\<\delta_*$}
\def\7{$\OY\<\<\otimes\<\id$}
\def\8{$\delta^*\CO_{Y'}\<\<\otimes\<\delta^*\delta_*$}
\def\9{$\delta^*t^*\OY\<\<\otimes\<\delta^*\delta_*$}
\def\ten{$\delta^*\delta_*\delta^*t^*\OY\<\<\otimes\<\delta^*\delta_*$}
\def\lvn{$\delta^*t^*\OY\<\<\otimes\<\delta^*\delta_*$}
\def\twv{$\id$}
\def\thn{$\delta^*\<\delta_*$}
 \bpic[xscale=3.4,yscale=1.85]

  \node(11) at (1,-1){\6}; 
  \node(13) at (2.87,-1){\5};
  \node(14) at (4,-1){\4};
 
  \node(22) at (1.8,-2){\lvn};
  \node(23) at (2.87,-2){\ten};
  \node(24) at (4,-2){\3};

  \node(31) at (1,-3){\6};
  \node(32) at (1.8,-3){\8};
  \node(33) at (2.87,-3){\9};

  \node(41) at (1,-4){\7};
  \node(43) at (2.87,-4){\1};
  \node(44) at (4,-4){\2};
 
  \node(51) at (1.6,-4){\twv};
  \node(52) at (2.11,-4){\thn};

   \draw[<-] (11) -- (13) node[above=1pt, midway, scale=.75]{$\via\epsilon^{}_\delta$}; 
   \draw[<-] (13) -- (14) node[above=1pt, midway, scale=.75]{\eqref{^* and tensor}};

  \draw[<-] (22) -- (23) node[above=1pt, midway, scale=.75]{$\<\via\epsilon^{}_\delta$}; 
  \draw[<-] (23) -- (24) node[above=1pt, midway, scale=.75]{\eqref{^* and tensor}};

  \draw[double distance=2pt] (31) -- (32) ;  
  \draw[double distance=2pt] (32) -- (33) ;

  \draw[double distance=2pt] (43) -- (44);
  
  \draw[<-] (51) -- (52) node[below=1pt, midway, scale=.75]{$\>\>\via\epsilon^{}_\delta$}; 

  \draw[double distance=2pt] (1,-1.25) -- (31) ; 
  \draw[->] (31) -- (41) node[left, midway, scale=.75]{$\via\epsilon^{}_\delta$}; 
  \draw[double distance=2pt] (41) -- (51);

  \draw[double distance=2pt] (22) -- (32) ; 

  \draw[double distance=2pt] (43) -- (52);

  \draw[double distance=2pt] (13) -- (23) node[right=1pt, midway, scale=.75]{$\via\ps^*$}; 
  \draw[<-] (23) -- (33) node[right=1pt, midway, scale=.75]{$\via\eta^{}_\delta$};

  \draw[double distance=2pt] (14) -- (24) node[right=1pt, midway, scale=.75]{$\via\ps^*$}; 
  \draw[<-] (24) -- (44) node[right=1pt, midway, scale=.75]{$\via\eta^{}_\delta$};

 
  \draw[double distance=2pt] (1.1,-1.2) -- (1.7,-1.82) node[above=-2pt, midway, scale=.75]
                                                                  {$\mkern60mu\via\ps^*$};
  \draw[double distance=2pt] (2.05,-2.17) -- (33);
  \draw[<-] (33)--(44) node[above=-2pt, midway, scale=.75]{$\mkern45mu$\eqref{^* and tensor}};                                                                 
  \draw[<-] (32)--(43) node[above=-2pt, midway, scale=.75]{$\mkern45mu$\eqref{^* and tensor}};                                                                 

  \node at (1.39,-2.45)[scale=.9] {\circled8$_{11}$};
  \node at (1.6,-3.55) [scale=.9]{\circled8$_{12}$};

 \epic
\]

For commutativity of \circled8$_{11}$, see \cite[3.6.7(b)]{li}.  
For commutativity of \circled8$_{12}$, see \cite[3.4.4(b)]{li}. 
Commutativity of the remaining subdiagrams is easily verified. Thus \circled8$_1$ commutes.\va2

Next, commutativity of \circled8$_2$ (without $u^!\OY\otimes\delta_*$) is implied by that of the following expanded diagram:
\[
\def\1{$\delta_*u_*$}
\def\2{$w_*g_*$}
\def\3{$\delta_*\delta^*t^*u_*$}
\def\4{$\delta_*\delta^*w_*r^*$}
\def\5{$\delta_*u_*g^*r^*$}
\def\6{$w_*g_*g^*r^*$}
\def\7{$t^*u_*$}
\def\8{$w_*r^*$}
 \bpic[xscale=2.8,yscale=1.85]

  \node(11) at (1,-1){\1}; 
  \node(13) at (3,-1){\1}; 
  \node(14) at (4,-1){\2};

  \node(21) at (1,-2){\3}; 
  \node(22) at (2,-2){\4};
  \node(23) at (3,-2){\5};
  \node(24) at (4,-2){\6};

  \node(31) at (1,-3){\7};
  \node(32) at (2,-3){\8};
  \node(34) at (4,-3){\8};

   \draw[double distance=2pt] (13) -- (11) ; 
   \draw[double distance=2pt] (13) -- (14) node[above=1pt, midway, scale=.75]{$\ps_*$}; 

   \draw[->] (21) -- (22) node[above=1pt, midway, scale=.75]{$\via\>\theta$}; 
   \draw[->] (22) -- (23) node[above=1pt, midway, scale=.75]{$\via\>\theta$}; 
   \draw[double distance=2pt] (23) -- (24) node[above=1pt, midway, scale=.75]{$\ps_*$}; 

   \draw[->] (31) -- (32) node[below=1pt, midway, scale=.75]{$\via\>\theta$}; 
   \draw[double distance=2pt] (32) -- (34) ; 
   
  \draw[double distance=2pt] (11) -- (21) node[left=1pt, midway, scale=.75]{$\via\ps^*$}; 
  \draw[->] (31) -- (21) node[left, midway, scale=.75]{$\eta^{}_\delta$}; 

  \draw[->] (32) -- (22) node[right=1pt, midway, scale=.75]{$\via\eta^{}_\delta$}; 

  \draw[double distance=2pt] (13) -- (23) node[right=1pt, midway, scale=.75]{$\via\ps^*$};

  \draw[double distance=2pt] (14) -- (24) node[right=1pt, midway, scale=.75]{$\via\ps^*$};
  \draw[->] (34) -- (24) node[right=1pt, midway, scale=.75]{$\via\eta^{}_g$};

  \node at (2,-1.5)[scale=.9] {\circled8$_{21}$};
  \node at (3,-2.5) [scale=.9]{\circled8$_{22}$};

 \epic
\]

Commutativity of \circled8$_{21}$ is \cite[3.7.2(ii)]{li}, applied to the composite diagram 
$\De'\<\smallcirc\De$ in Proposition~\ref{lastprop}.

Commutativity of \circled8$_{22}$\va{.6} results from the adjointness of $\theta\colon \delta^*w_*\to u_*g^*$ and the composite map $w_*\xto{\lift{.85},w_{\<*}\eta_g\>,} w_*g_*g^*\overset{\ps_*}{=\!=} \delta_*u_*g^*$, which holds by \cite[3.7.2(i)(b)]{li}, with $(f,g,f',g')\set (w,\delta,u,g)$.

Commutativity of the other two subdiagrams is clear. Thus \circled8$_2$ commutes.\va2

Next, since $u$ and $w$ are proper, commutativity of \circled8$_3$  is an immediate consequence of the definition of the 
base\kf-change map $\mathsf B$ in \eqref{thetaB} (with $(f,g,u,v)\set(u,w,t,r)$), 
see \cite[5.8.2, 5.8.5]{AJL}.

Thus \circled8 commutes.

\smallskip
Subdiagram \circled9 is the outer border of
\[
\def\2{$w_*(r^*\<u^!\OY\<\<\otimes\< w^*\<\delta_*\<)$}
\def\3{$w_*(g_*u^!\OY\<\<\otimes\< w^*\<\delta_*\<)$}
\def\4{$w_*g_*(u^!\OY\<\<\otimes\< g^*w^*\<\delta_*\<)$}
\def\5{$w_*g_*(u^!\OY\<\<\otimes\< u^*\delta^*\<\delta_*\<)$}
\def\6{$w_*g_*(u^!\OY\<\<\otimes\< g^*\<g_*u^*\<)$}
\def\7{$w_*(g_*u^!\OY\<\<\otimes\< g_*u^*\<)$}
\def\8{$w_*g_*(g^*\<g_*u^!\OY\<\<\otimes\< u^*\<)$}
\def\twv{$w_*(r^*\<u^!\OY\<\<\otimes\< g_*u^*)$}
\def\thn{$w_*g_*(g^*r^*\<u^!\OY\<\<\otimes\< u^*\<)$}
\def\frn{$w_*g_*(u^!\OY\<\<\otimes\< u^*)$}
 \bpic[xscale=4.4,yscale=1.85]

  \node(11) at (1,-1){\5}; 
  \node(13) at (3,-1){\frn};
 
  \node(21) at (1,-2){\4};
  \node(22) at (2,-2){\6};

  \node(31) at (1,-3){\3};
  \node(32) at (2,-3){\7};

  \node(325) at (2.5,-3.75){\8};
  
  \node(41) at (1,-4.5){\2};
  \node(42) at (2,-4.5){\twv};
  \node(43) at (3,-4.5){\thn};

   \draw[->] (11) -- (13) node[above=1pt, midway, scale=0.75]{$\via\epsilon^{}_{\<\delta}$};

   \draw[->] (21) -- (22) node[above=1pt, midway, scale=0.75]{$\!\via\theta$};

   \draw[->] (31) -- (32) node[above=1pt, midway, scale=0.75]{$\via\theta$};

   \draw[->] (41) -- (42) node[below=1pt, midway, scale=0.75]{$\via\theta$};
   \draw[->] (42) -- (43) node[below=1pt, midway, scale=.75]{$\via\prj$};

   \draw[double distance=2pt](11) -- (21) node[left=1pt, midway, scale=.75]{$\via\ps^*$}; 
   \draw[<-] (21) -- (31) node[left=1pt, midway, scale=.75]{$\via\prj$};
   \draw[<-] (31) -- (41) node[left=1pt, midway, scale=.75]{$\via\gamma$};
   
   \draw[<-] (22) -- (32) node[left=1pt, midway, scale=.75]{$\via\prj$};
   \draw[<-] (32) -- (42) node[left=1pt, midway, scale=.75]{$\via\gamma$};

   \draw[double distance=2pt] (3,-1.27) -- (43) node[right=1pt, midway, scale=.75]{\raisebox{36pt}{$\via\ps^*$}};
   
   \draw[->](22) -- (13) node[below=-3pt, midway, scale=0.75]{$\mkern50mu\via\epsilon^{}_g$};
   \draw[->](32) -- (325) node[above=-2.5pt, midway, scale=0.75]{\kern33pt $\via\prj$};
   \draw[<-](325) -- (43) node[above=-3pt, midway, scale=0.75]{$\mkern45mu\via\gamma$};
   \draw[->](325) -- (2.9,-1.25) node[below=-1pt, midway, scale=0.75]{$\mkern50mu\via\epsilon^{}_g$};

   \node at (1.8,-1.5){\circled9$_1$};
   \node at (2.4,-2.53){\circled9$_2$};
   \node at (2.82,-3){\circled9$_3$};
 \epic
 \]
 
Subdiagram \circled9$_1$ commutes because $\theta$ is, by definition, $g^*\!\<\dashv\< g_*$-adjoint to $\alpha$.

Commutativity of \circled9$_3$ results from the obvious commutativity of 
\[
\def\1{$g^*r^*$}
\def\2{$g^*g_*\>g^*r^*$}
\def\3{$g^*r^*$}
\def\4{$g^*g_*$}
\def\5{$\id$}
 \bpic[xscale=3.3,yscale=2.2]

  \node(11) at (1,-2){\1}; 
  \node(12) at (2,-1.5){\2}; 
  \node(13) at (3,-2){\3};
  
  \node(21) at (1,-1){\4}; 
  \node(23) at (3,-1){\5}; 
  \draw[double distance=2pt] (11) -- (13) ;  
  
  \draw[->] (21) -- (23) node[above, midway, scale=0.75]{$\via\epsilon^{}_g$};  

  \draw[->] (11) -- (21) node[left=1pt, midway, scale=0.75]{$g^*\gamma$};  
  
  \draw[double distance=2pt] (13) -- (23) node[right=1pt, midway, scale=0.75]{$\ps^*$};

   \draw[->] (11) -- (12) node[above=-1pt, midway, scale=0.75]{$g^*\<\eta_g\mkern40mu$};  
   \draw[double distance=2pt] (12) -- (21) node[below=-2.5pt, midway, scale=0.75]{$g^*\<\<g_*\<\<\ps^*\mkern85mu$};  

   \draw[->] (12) -- (13) node[above=-2pt, midway, scale=0.75]{$\mkern20mu\epsilon^{}_g$};  

 \epic
\]

As for commutativity of \circled9$_2$, after dropping $w_*$  and setting $A\set u^!\OY$, $B\set u^*$, one need only show commutativity of 
\begin{equation}\label{diamond}
\def\1{$g_*A\otimes g_*B$}
\def\2{$g_*(A\otimes g^*\<g_*B)$}
\def\3{$g_*(g^*g_*A\otimes B)$}
\def\4{$g_*(A\otimes B)$}
\CD
 \bpic[xscale=1.75,yscale=1.4]

  \node(12) at (2,-1){\1}; 
  \node(21) at (1,-2){\2}; 
  \node(23) at (3,-2){\3};
  \node(32) at (2,-3){\4};
  
  \draw[->] (12) -- (21) node[above=-2pt, midway, scale=0.75]{$\prj\mkern30mu$};  
  \draw[->] (12) -- (23) node[above=-2pt, midway, scale=0.75]{$\mkern30mu\prj$};  
  
  \draw[->] (21) -- (32) node[below=-4pt, midway, scale=0.75]{$\via\epsilon^{}_g\mkern40mu$};  
  \draw[->] (23) -- (32) node[below=-4pt, midway, scale=0.75]{$\mkern55mu\via\epsilon^{}_g$};    

 \epic
\endCD
\end{equation}

In the following diagram
\[
\def\1{$g^*(g_*A\otimes g_*B)$}
\def\2{$g^*\<g_*(A\otimes g^*\<g_*B)$}
\def\3{$g^*\<g_*(g^*\<g_*A\otimes B)$}
\def\4{$g^*\<g_*(A\otimes B)$}
\def\5{$g^*\<g_*A\otimes g^*\<g_*B$}
\def\6{$A\otimes g^*\<g_*B$}
\def\7{$g^*\<g_*A\otimes B$}
\def\8{$A\otimes B$}
 \bpic[xscale=1.75,yscale=1.4]

  \node(12) at (2,-1){\1}; 
  \node(21) at (1,-2){\2}; 
  \node(23) at (3,-2){\3};
  \node(32) at (2,-3){\4};

  \node(12a) at (2,.7){\5}; 
  \node(21a) at (-.7,-2){\6}; 
  \node(23a) at (4.7,-2){\7};
  \node(32a) at (2,-4.7){\8};
 
  \draw[->] (12) -- (21) node[above=-2pt, midway, scale=0.75]{$g^*\prj\mkern30mu$};  
  \draw[->] (12) -- (23) node[above=-2pt, midway, scale=0.75]{$\mkern30mu g^*\prj$};  
  
  \draw[->] (21) -- (32) node[below=-4pt, midway, scale=0.75]{$\via\epsilon^{}_g\mkern40mu$};  
  \draw[->] (23) -- (32) node[below=-4pt, midway, scale=0.75]{$\mkern55mu\via\epsilon^{}_g$};

  \draw[->] (12a) -- (21a) node[above, midway, scale=0.75]{$\via\epsilon^{}_g\mkern45mu$};  
  \draw[->] (12a) -- (23a) node[above, midway, scale=0.75]{$\mkern50mu\via\epsilon^{}_g$};  
  
  \draw[->] (21a) -- (32a) node[below, midway, scale=0.75]{$\via\epsilon^{}_g\mkern40mu$};  
  \draw[->] (23a) -- (32a) node[below, midway, scale=0.75]{$\mkern55mu\via\epsilon^{}_g$};    
  
 \draw[->] (21) -- (21a) node[above, midway, scale=0.75]{$\epsilon^{}_g$};  
 \draw[->] (23) -- (23a) node[above, midway, scale=0.75]{$\epsilon^{}_g$};  

    \draw[->] (12a) -- (12)  node[right=1pt, midway, scale=0.75]{$\simeq$} 
                                         node[left=1pt, midway, scale=0.75]{\eqref{^* and tensor}};  

    \draw[->] (32) -- (32a) node[left=.5pt, midway, scale=0.75]{$\epsilon^{}_g$};  

  \node at (.85,-1)[scale=.9]{\circled9$_{21}$}; 
  \node at (3.15,-1)[scale=.9]{\circled9$_{22}$}; 
  \node at (.85,-3)[scale=.9]{\circled9$_{23}$}; 
  \node at (3.15,-3)[scale=.9]{\circled9$_{24}$};  
 \epic
\]
commutativity of the outer border is clear, as is that of subdiagrams~\circled9$_{23}$ and~\circled9$_{24}$; and commutativity of~\circled9$_{21}$ and~\circled9$_{22}$ results from \cite[3.4.6.2]{li}. Looking inside the diagram one sees then that the $g^*\!\dashv\<g_*$-adjoint of \eqref{diamond}---hence
\eqref{diamond} itself---commutes.\looseness=-1

Thus \circled9$_2$---and finally \circled9 itself---commutes.\va1

This completes the proof of~
Proposition~\ref{lastprop}  in case  $u$ and~$w$ are proper.

\smallskip

 In the general case, $u$ factors as $X\xto{u_2\>\>}Z\xto{u_1\>}Y$ where $u_1$ is proper and $u_2$ is essentially \'etale 
\cite[4.1 and 2.7]{Nk}. It follows that the diagram $\De'\<\smallcirc\De$ in~Proposition~\ref{lastprop}
expands as
\begin{equation*}\label{squares3}
  \mkern125mu
    \begin{tikzpicture}[xscale=1.75,yscale=.9]
      \draw[white] (0cm,5.1cm) -- +(0: \linewidth)
      node (21) [black, pos = 0.1] {$X$}
      node (215) [black, pos = 0.2] {$Z$}
      node (22) [black, pos = 0.3] {$Y$};
      \draw[white] (0cm,2.8cm) -- +(0: \linewidth)
      node (31) [black, pos = 0.1] {$X'$}
      node (315) [black, pos = 0.2] {$Z\times_YY'$}
      node (32) [black, pos = 0.3] {$Y'$};
      \draw[white] (0cm,0.5cm) -- +(0: \linewidth)
      node (51) [black, pos = 0.1] {$X$}
      node (515) [black, pos = 0.2] {$Z$}
      node (52) [black, pos = 0.3] {$Y$};
      \draw [->] (21) -- (31) node[left, midway, scale=0.75]{$g$};
      \draw [->] (31) -- (51) node[left, midway, scale=0.75]{$r$};
      \draw [->] (215) -- (315) node[left, midway, scale=0.75]{$h$};
      \draw [->] (315) -- (515) node[left, midway, scale=0.75]{$s$};
      \draw [->] (22) -- (32) node[right, midway, scale=0.75]{$\delta$};
      \draw [->] (32) -- (52) node[right, midway, scale=0.75]{$t$};
      \draw [->] (21) -- (215) node[above, midway,  scale=0.75]{$u^{}_2$};
      \draw [->] (215) -- (22) node[above, midway,  scale=0.75]{$u^{}_1$};
      \draw [->] (31) -- (315) node[above, midway,  scale=0.75]{$w^{}_2$};
      \draw [->] (315) -- (32) node[above, midway,  scale=0.75]{$w^{}_1$};
      \draw [->] (51) -- (515) node[below=1pt, midway,  scale=0.75]{$u^{}_2$};
      \draw [->] (515) -- (52) node[below=1pt, midway,  scale=0.75]{$u^{}_1$};
    
      \end{tikzpicture}
  \end{equation*}
where $w_1$ and $s$ are the natural projections; $h$ is the unique map such that $w_1h=\delta u_1$ and $sh=\id_Z$; and $w_2$ is the unique map such that $sw_2=u_2r$ and $w_1w_2=w$. One checks  that all the subsquares are fiber squares; so $w_1$ is proper and $w_2$ is essentially \'etale (see second-last paragraph in \S\ref{efp}). Since $u_2$ is flat, the map $\kappa\colon u_2^!\OZ\otimes u_2^*\to u_2^!$ is an isomorphism on $\Dqcpl(Z)$; and likewise for $w_2$.

Straightforward use of the isomorphisms 
$$
u^!\overset{\ps^!}{=\!=}u_2^!u_1^!,\quad
u^*\overset{\ps^*}{=\!=}u_2^*u_1^*,\quad
w^!\overset{\ps^!}{=\!=}w_2^!w_1^!,\quad
w^*\overset{\ps^*}{=\!=}w_2^*w_1^*
$$
transforms the assertion in
Proposition~\ref{lastprop} to that of commutativity of the  border of the next diagram \eqref{DIAG}, in which $\CO\set\OY$, $\CO'\set\CO_{Y'}$, $\CO''\set\CO_{Z\times_YY'}$, 
and the unlabeled maps are the obvious ones:
\begin{figure}
\begin{equation}\label{DIAG}
\def\0{$\,g_*\<u_2^!(u_1^!\<\CO\<\otimes\< u_1^*)$}
\def\1{$r^*\<u_2^!u_1^!\CO\<\otimes\< g_*u_2^*u_1^*$}
\def\one{$r^*\<u_2^!u_1^!\CO\<\otimes\< w_2^*w_1^*\delta_*$}
\def\2{$g_*(g^*\<r^*\<u_2^!u_1^!\CO\<\otimes\< u_2^*u_1^*)$}
\def\3{$g_*(u_2^!u_1^!\CO\<\otimes\< u_2^*u_1^*)$}
\def\4{$g_*(u_2^!\OZ\<\otimes\< u_2^*u_1^!\CO\<\otimes\< u_2^* u_1^*)$}
\def\four{$g_*(u^!\CO\<\otimes\< u^*\<)$}
\def\5{$g_*u_2^!u_1^!=g_*u^!\quad$}
\def\6{$g_*u_2^!u_1^!\delta^!\delta_*$}
\def\7{$g_*g^!w_2^!w_1^!\delta_*$}
\def\8{$w_2^!w_1^!\delta_*=w^!\delta_*$}
\def\9{$w_2^!w_1^!t^*\CO\<\otimes\< w_2^*w_1^*\delta_*$}
\def\nine{$w^!\CO'\<\<\otimes\< w^*\delta_*$}
\def\ten{$g_*u_2^!h^!h_*u_1^!$}
\def\lvn{$g_*g^!w_2^!h_*u_1^!$}
\def\twv{$w_2^!h_*u_1^!$}
\def\thn{$w_2^!h_*u_1^!\delta^!\delta_*$}
\def\frn{$\!\!g_*(g^*r^*\<u_2^!\OZ\<\otimes\<\< u_2^*(u_1^!\CO\<\otimes\< u_1^*)\<)$}
\def\ffn{$\,r^*\<u_2^!\OZ\<\otimes\< g_*u_2^*(u_1^!\CO\<\otimes\<\< u_1^*)$}
\def\sxn{$w_2^!\CO''\<\<\otimes\< w_2^*h_*(u_1^!\CO\<\otimes\< u_1^*)$}
\def\svn{$w_2^!h_*(u_1^!\CO\<\otimes\< u_1^*)$}
\def\egn{$w_2^!h_*(h^*\<s^*u_1^!\CO\<\otimes\< u_1^*)$}
\def\ntn{$w_2^!(s^*u_1^!\CO\<\otimes\< h_*u_1^*)$}
\def\twy{$w_2^!h_*h^!w_1^!\delta_*$}
\def\twn{$w_2^!(w_1^!t^*\CO\<\otimes\< w_1^*\delta_*)$}
\def\ttw{$w_2^!\CO''\<\<\otimes\< w_2^*w_1^!t^*\CO\<\<\otimes\< w_2^*w_1^*\delta_*\quad $}
\def\tth{$w_2^!\CO''\<\<\otimes\< w_2^*(s^*\<u_1^!\CO\<\otimes\< h_*u_1^*\<)\ $}
\def\tfr{$w_2^!\CO''\<\<\otimes\< w_2^*h_*(h^{\<*}\<\<s^*\<u_1^!\CO\<\otimes\< u_1^*\<)\quad$}
 \bpic[xscale=3.2,yscale=1.8]
  \node at (2.6,.45){};
  \node(11) at (2.6,.1){\5};
  \node(13) at (2.6,-.9){\0};
  \node(14) at (4,.1){\four};
 
  \node(23) at (2.6,-1.9){\4};
  \node(24) at (4,-1.9){\3};

  \node(31) at (1,.1){\6};
  \node(32) at (1.45,-1.9){\ten }; 
  \node(33) at (2.6,-3){\frn };
  \node(34) at (4,-3){\2};

  \node(42) at (1.45,-4.1){\lvn };
  \node(43) at (2.6,-4.1){\ffn };
  \node(44) at (4,-4.1){\1};
  
  \node(53) at (2.6,-5){\sxn };
  \node(54) at (4,-5){\one};

  \node(62) at (2.05,-5.95){\svn};
  
  \node(72) at (1.45,-7){\twv };
  \node(73) at (2.3,-7){\egn };
 
  \node(82) at (1.45,-8){\thn};
  \node(83) at (2.3,-8){\ntn };

  \node(92) at (1.45,-9.1){\twy };
  \node(93) at (2.3,-9.1){\twn };
  \node(94) at (4,-9.1){\9};
  
  \node(01) at (1,-10){\7};
  \node(02) at (2.075,-10){\8};
  \node(04) at (4,-10){\nine};
  
  \node(74) at (3.3,-6.4){\tfr};
  \node(84) at (3.3,-7.4){\tth};
  \node(914) at (3.3,-8.4){\ttw};
  
   \draw[->] (11) -- (31) ;   
   \draw[<-] (2.94,.1) -- (14) ;
 
   \draw[<-] (23) -- (24) node[above, midway, scale=0.75]{$\!\via$}
                                     node[below=-.7pt, midway, scale=0.75]{$\kappa^{-\<1}$};

   \draw[<-] (02) -- (01) ; 
   \draw[<-] (02) -- (04) ;
 
   \draw[double distance=2pt] (31) -- (01) ; 
   
   \draw[double distance=2pt] (32) -- (42) ; 
   \draw[->] (42) -- (72) ;
   \draw[->] (72) -- (82) ;
   \draw[double distance=2pt] (82) -- (92) ; 

   \draw[<-] (23) -- (33) ;
   \draw[<-] (33) -- (43) ;
   \draw[<-] (43) -- (53) ; 
   \draw[->] (73) -- (83) node[right=1pt, midway, scale=0.75]{$\simeq$};
   \draw[->] (83) -- (93) node[right=1pt, midway, scale=0.75]{$\simeq$}; 

   \draw[->] (84) -- (74) node[right=1pt, midway, scale=0.75]{$\simeq$};   
   \draw[->] (914) -- (84) node[right=1pt, midway, scale=0.75]{$\simeq$};

   \draw[double distance=2pt] (14) -- (24) ; 
   \draw[double distance=2pt] (24) -- (34) ; 
   \draw[<-] (34) -- (44) ;
   \draw[<-] (44) -- (54) ;
   \draw[<-] (54) -- (94) ;
   \draw[double distance=2pt] (94) -- (04) ; 
   
   \draw[<-] (11) -- (13) ;
   \draw[<-] (13) -- (23) node[right=1pt, midway, scale=0.75]{$\via\kappa$};
   \draw[->] (11) -- (32) ;
   \draw[double distance=2pt] (74) -- (53) ; 
   \draw[->] (53) -- (62) node[below=-1.5pt, midway, scale=0.7]{$\mkern15mu\kappa$};
   \draw[->] (62) -- (72) ; 
   \draw[double distance=2pt] (62) -- (73) ; 
   \draw[->] (74) -- (73) node[below=-.5pt, midway, scale=0.7]{$\mkern5mu\kappa$}; 
   \draw[->] (84) -- (83) node[below=-.5pt, midway, scale=0.7]{$\mkern5mu\kappa$}; 
   \draw[->] (92) -- (02) ; 
   \draw[->] (93) -- (02) ; 
   \draw[->] (94) -- (914) node[above=-8.5pt, midway, scale=0.7]{$\via\>\kappa^{-\<1}\mkern70mu$}; 
   \draw[->] (914) -- (93) node[below=-.5pt, midway, scale=0.7]{$\mkern5mu\kappa$}; 
   
   \node at (3.3,-.95) {\circled a} ;
   \node at (1.225,-5) {\circled b} ;
   \node at (2.05,-3.6) {\circled c} ;
   \node at (3.3,-3.6) {\circled d} ;
   \node at (1.875,-7.53) {\circled e} ;
   \node at (3.13,-9.4) {\circled f} ;

 \epic
\end{equation}
\end{figure}

Diagram chasing shows it suffices now to prove commutativity of all the subdiagrams.
\goodbreak

Commutativity of the unlabeled subdiagrams is clear.\va1

Commutativity of \circledd c follows easily from the essentially \'etale case of~\ref{lastprop}, applied to the diagram
\[
  \mkern160mu
    \begin{tikzpicture}[xscale=1.75,yscale=.9]
      \draw[white] (0cm,5.1cm) -- +(0: \linewidth)
      node (21) [black, pos = 0.1] {$X$}
      node (215) [black, pos = 0.2] {$Z$};
 
      \draw[white] (0cm,2.8cm) -- +(0: \linewidth)
      node (31) [black, pos = 0.1] {$X'$}
      node (315) [black, pos = 0.2] {$Z\times_YY'$};
      
      \draw[white] (0cm,0.5cm) -- +(0: \linewidth)
      node (51) [black, pos = 0.1] {$X$}
      node (515) [black, pos = 0.2] {$Z$};
   
      \draw [->] (21) -- (31) node[left, midway, scale=0.75]{$g$};
      \draw [->] (31) -- (51) node[left, midway, scale=0.75]{$r$};
      \draw [->] (215) -- (315) node[right, midway, scale=0.75]{$h$};
      \draw [->] (315) -- (515) node[right, midway, scale=0.75]{$s$};
      \draw [->] (21) -- (215) node[above, midway,  scale=0.75]{$u^{}_2$};
 
      \draw [->] (31) -- (315) node[above, midway,  scale=0.75]{$w^{}_2$};
      \draw [->] (51) -- (515) node[below=1pt, midway,  scale=0.75]{$u^{}_2$};
    
      \end{tikzpicture}
\]
Commutativity of \circledd e results from the proper case of~\ref{lastprop}, applied to   
 \[
  \mkern71mu
    \begin{tikzpicture}[xscale=1.75, yscale=.9]
      \draw[white] (0cm,5.1cm) -- +(0: \linewidth)
      node (215) [black, pos = 0.2] {$Z$}
      node (22) [black, pos = 0.3] {$Y$};
      \draw[white] (0cm,2.8cm) -- +(0: \linewidth)
      node (315) [black, pos = 0.2] {$Z\times_YY'$}
      node (32) [black, pos = 0.3] {$Y'$};
      \draw[white] (0cm,0.5cm) -- +(0: \linewidth)
      node (515) [black, pos = 0.2] {$Z$}
      node (52) [black, pos = 0.3] {$Y$};
      \draw [->] (215) -- (315) node[left, midway, scale=0.75]{$h$};
      \draw [->] (315) -- (515) node[left, midway, scale=0.75]{$s$};
      \draw [->] (22) -- (32) node[right, midway, scale=0.75]{$\delta$};
      \draw [->] (32) -- (52) node[right, midway, scale=0.75]{$t$};
      \draw [->] (215) -- (22) node[above, midway,  scale=0.75]{$u^{}_1$};
      \draw [->] (315) -- (32) node[above, midway,  scale=0.75]{$w^{}_1$};
      \draw [->] (515) -- (52) node[below=1pt, midway,  scale=0.75]{$u^{}_1$};
    
      \end{tikzpicture}
\] 

Subdiagram \circled b  has the following, clearly commutative, expansion (where the maps are the obvious ones):

\[
\def\1{$g_*u_2^!u_1^!$}
\def\2{$g_*u_2^!u_1^!\delta^!\delta_{\<*}$}
\def\3{$g_*g^!w_2^!w_1^!\delta_{\<*}$}
\def\4{$w_2^!w_1^!\delta_{\<*}$}
\def\5{$w_2^!h_*h^!w_1^!\delta_{\<*}$}
\def\6{$w_2^!h_*u_1^!\delta^!\delta_{\<*}$}
\def\7{$w_2^!h_*u_1^!$}
\def\8{$g_*g^!w_2^!h_*u_1^!$}
\def\9{$g_*u_2^!h^!h_*u_1^!$}
\def\ten{$g_*u_2^!h^!h_*u_1^!\delta^!\delta_{\<*}$}
\def\lvn{$g_*u_2^!h^!w_1^!\delta_{\<*}$}
\def\twv{$\ g_*u_2^!h^!h_*h^!w_1^!\delta_{\<*}$}
\def\thn{$g_*g^!w_2^!h_*u_1^!\delta^!\delta_{\<*}$}
\def\frn{$g_*g^!w_2^!h_*h^!w_1^!\delta_{\<*}$}
 \bpic[xscale=3.2,yscale=2.1]
 
  \node(11) at (1,-1){\2};
  \node(13) at (3.045,-1){\2};
  \node(14) at (4,-1){\1};

  \node(22) at (1.975,-1.95){\lvn}; 
  \node(23) at (3.045,-1.95){\ten};
  \node(24) at (4,-1.95){\9};

  \node(32) at (1.975,-3){\twv};

  \node(41) at (1,-4){\3};
  \node(42) at (1.975,-4){\frn };
  \node(43) at (3.045,-4){\thn };
  \node(44) at (4,-4){\8};

  \node(51) at (1,-4.95){\4};
  \node(52) at (1.975,-4.95){\5};  
  \node(53) at (3.045,-4.95){\6 };
  \node(54) at (4,-4.95){\7};

   \draw[double distance=2pt] (11) -- (13) ;   
   \draw[<-] (13) -- (14) ;
 
   \draw[<-] (23) -- (24) ;

   \draw[<-] (41) -- (42) ; 
   \draw[double distance=2pt] (42) -- (43) ;   
   \draw[<-] (43) -- (44) ;

    \draw[<-] (51) -- (52) ; 
    \draw[double distance=2pt] (52) -- (53) ;    
    \draw[<-] (53) -- (54) ;

   \draw[double distance=2pt] (11) -- (41) ; 
   \draw[->] (41) -- (51) ;
   
   \draw[->] (22) -- (32) ; 
   \draw[double distance=2pt] (32) -- (42) ;
   \draw[->] (42) -- (52) ;

   \draw[->] (13) -- (23) ;
   \draw[double distance=2pt] (23) -- (43) ;
   \draw[->] (43) -- (53) ;

   \draw[->] (14) -- (24) ; 
   \draw[double distance=2pt] (24) -- (44) ; 
   \draw[->] (44) -- (54) ;
   
    \draw[double distance=2pt] (41) -- (22) ; 
    \draw[double distance=2pt] (22) -- (13) ; 
    \draw[double distance=2pt] (23) -- (32) ; 
   
 \epic
\]

Commutativity of subdiagram \circledd a follows from \cite[4.9.3(d)]{li} as regards \cite[4.7.3.4(d)]{li}
with $(f,g,E)\set(u_2,u_1,\CO)$---in view of \cite[4.9.3(d)]{li} as regards \cite[4.7.3.4(a)]{li} 
with $(f,E,F,G)\set(u_2,\OZ,u_1^!\CO, u_1^*)$, which gives that 
$(\<\via\kappa)\<\smallcirc\<(\<\via\:\kappa^{-1})$ in \circledd a is  the map 
$g_*\chi^{u_2}_{u_1^!\<\CO\<, \>u_1^*-}$ coming from \cite[(4.9.1.1)]{li}, as extended to $\SS$\kf-maps in the manner of \cite[5.8]{Nk}. A similar argument shows that \circled{\kern-.5pt f} commutes.
Details are left to the reader.\va1

Subdiagram \circled d expands as follows, with the map $\prj$ coming from~\eqref{projection}. (Recall: $u_2^!=u_2^*$,  $w_2^!=w_2^*$.)
\[\mkern-10mu
\def\1{$r^*\<u_2^!\OZ\<\otimes\< g_*u_2^*(u_1^!\CO\<\otimes\<\< u_1^*)$}
\def\2{$\ \ g_*\<(g^*r^*\<u_2^!\OZ\<\otimes\< u_2^*(u_1^!\CO\<\otimes\<\< u_1^*))$}
\def\3{$w_2^!\CO''\<\<\otimes\< w_2^*h_*(u_1^!\CO\<\otimes\< u_1^*)$}
\def\4{$\mkern-15mug_*(u_2^!\OZ\<\otimes\<\< u_2^*u_1^!\CO\<\<\otimes\<\< u_2^*u_1^*\<)$}
\def\5{$w_2^!\CO''\<\<\otimes\< w_2^*h_*(h^{\<*}\<\<s^*\<u_1^!\CO\<\otimes\< u_1^*\<)$}
\def\6{$\mkern10mu g_*(u_2^!u_1^!\CO\<\otimes\< u_2^*u_1^*)$}
\def\7{$w_2^!\CO''\<\<\otimes\< w_2^*s^*\<u_1^!\CO\<\otimes\< w_2^*h_*\<u_1^*$}
\def\sev{$w_2^!\CO''\<\<\otimes\< w_2^*(s^*\<u_1^!\CO\<\otimes\< h_*\<u_1^*)$}
\def\8{$g_*(g^*\<r^*\<u_2^!u_1^!\CO\<\otimes\< u_2^*u_1^*)$}
\def\9{$w_2^!\CO''\<\<\otimes\< w_2^*w_1^!\CO'\<\<\otimes\< w_2^*w_1^*\delta_*$}
\def\ten{$r^*\<u_2^!u_1^!\CO\<\otimes\< g_*u_2^*u_1^*$}
\def\lvn{$r^*\<u_2^!u_1^!\CO\<\otimes\< w_2^*w_1^*\delta_*$}
\def\twv{$w_2^!w_1^!t^*\CO\<\otimes\< w_2^*w_1^*\delta_*$}
\def\thn{$r^*\<u_2^!\OZ\<\otimes\< g_*\<(\<u_2^*u_1^!\CO\<\otimes\<\< u_2^*\<u_1^*)$}
\def\frn{$\ \ r^*\<u_2^!\OZ\<\otimes\< g_*(g^*r^*u_2^*u_1^!\CO\<\otimes\<\< u_2^*\<u_1^*)\quad\ $}
\def\ffn{$r^*\<u_2^!\OZ\<\otimes\< r^*u_2^*u_1^!\CO\<\otimes\<\< g_*u_2^*\<u_1^*$}
\def\sxn{$g_*(g^*r^*\<u_2^!\OZ\<\otimes\< g^*r^*\<u_2^*u_1^!\CO\<\otimes\<\< u_2^*\<u_1^*\<)$}
\def\svn{$w_2^!s^*\<u_1^!\CO\<\<\otimes\<w_2^*h_*\<u_1^*$}
\def\egn{$g_*(g^*r^*\<u_2^!\OZ\<\otimes\<u_2^*u_1^!\CO\<\otimes\<\< u_2^*\<u_1^*\<)$}
 \bpic[xscale=4.15,yscale=2.2]
 
  \node(11) at (1,-1)[scale=.95]{\2};
  \node(13) at (3,-1)[scale=.95]{\4};
  \node(115) at (1,-1.7)[scale=.95]{\1};
  \node(125) at (2.25,-1.7)[scale=.95]{\egn};
 
  \node(215) at (1,-2.7)[scale=.95]{\3};
  \node(22) at (2,-2.3)[scale=.95]{\thn};
  \node(23) at (3,-2.7)[scale=.95]{\6};

  \node(31) at (2,-3.25)[scale=.95]{\frn};
  \node(32) at (2.28,-4.25)[scale=.95]{\sxn };

  \node(41) at (1,-3.75)[scale=.95]{\5 };
  \node(42) at (1.9,-5.26)[scale=.95]{\ffn };
  \node(43) at (3,-4.8)[scale=.95]{\8};
  
  \node(515) at (1,-4.8)[scale=.95]{\sev};
  \node(51) at (1,-6)[scale=.95]{\7 };
  \node(52) at (2.08,-6)[scale=.95]{\svn};
  \node(53) at (3,-6)[scale=.95]{\ten};

  \node(61) at (1,-6.7)[scale=.95]{\9 };  
  \node(62) at (2.08,-6.7)[scale=.95]{\twv };
  \node(63) at (3,-6.7)[scale=.95]{\lvn };
   
   \draw[<-] (1.63,-1) -- (2.43,-1) ; 
   
   \draw[double distance=2pt] (51) -- (52) ;
   \draw[->] (52) -- (53) node[above, midway, scale=0.7]{$\Iso$};

   \draw[<-] (61) -- (62) ; 
   \draw[->] (62) -- (63) node[above, midway, scale=0.7]{$\Iso$};
 
   \draw[double distance=2pt] (215) -- (41) ; 
   \draw[<-] (11) -- (115) ;
   \draw[<-] (115) -- (215) ;
   \draw[double distance=2pt] (215) -- (41) ; 
   \draw[<-] (41) -- (515) ; 
   \draw[<-] (515) -- (51) ; 
   \draw[<-] (51) -- (61) ; 

   \draw[double distance=2pt] (22) -- (31) ; 
   \draw[<-] (31) -- (32) node[below=-1pt, midway, scale=0.7]{$\prj\mkern20mu$};    
   \draw[<-] (32) -- (42) ;    
   \draw[<-] (1.575,-3.44) -- (1.575,-5.07) ;    
   \draw[<-] (52) -- (62) ;    
   
   \draw[double distance=2pt] (2.626,-1.9) -- (2.626,-4.06) ;    

   \draw[<-] (13) -- (23) ;    
   \draw[double distance=2pt] (23) -- (43) ;
   \draw[<-] (43) -- (53) ; 
   \draw[<-] (53) -- (63) ;
       
   \draw[<-] (115) -- (22) ;
   \draw[double distance=2pt] (2.65,-1.5) -- (2.85, -1.2) ;
   \draw[<-] (11) -- (125) ;
   \draw[<-] (2.65,-4.45) -- (2.85, -4.63) ;
   \draw[double distance=2pt] (2.29,-5.5) -- (53) ;   
   \draw[<-] (1.54,-5.5) -- (51) ;    

    \node at (1.2,-3.25) [scale=0.9] {\circled d$_1$};
    \node at (1.85,-4.8) [scale=0.9] {\circled d$_2$};
    \node  at (intersection of 52--63 and 62--53) [scale=0.9] {\raisebox{-5pt}{\circled d$_3$}} ;

 \epic
\]

Diagram chasing shows that to prove commutativity of the border it will suffice to prove commutativity of all the subdiagrams.

Commutativity of the unlabeled subdiagrams is easily verified.

Commutativity of \circled d$_3$ results from transitivity of \eqref{thetaB} and of 
\eqref{def-of-bch-asterisco}. (See \cite[3.7.2(iii)]{li}, having in mind that $u_2$ and  $w_2$, as well as $r$, $s$ and $t$,  are flat.)

Commutativity of \circled d$_2$ results from \cite[3.4.7(iii)]{li}, with $(f,A,B,C)\set(g,r^*u_2^!\OZ,r^*u_2^*u_1^!\CO, u_2^*u_1^*-)$.

Last, in the next
diagram of \emph{isomorphisms,} with $A,B \in \Dqcpl(Z)$, 
the border commutes by \cite[3.7.3]{li} with
$(f,f'\<,g,g'\<,P,Q)\set(h,g,w_2,u_2,s^*A,B)$, and commutativity of
the unlabeled subdiagrams 
is easy to check (the one at~the bottom by pseudofunctoriality of
$(-)^*$), and 
hence \circled d$_1'$ commutes. Setting $A\set u_1^!\CO$,   $B\set u_1^*-,$ one obtains commutativity\- of~\circled d$_1$ from that of {\circled{d}}${}^{\prime}_1$.

\[\mkern-4mu
\def\1{$w_2^*h_*(\<A\<\otimes\< B)$}
\def\2{$g_*u_2^*(\<A\<\otimes\< B)$}
\def\3{$w_2^*h_*(h^*\<\<s^*\!A\<\otimes\< B)$}
\def\4{$g_*(\<u_2^*A\<\otimes\< u_2^*B)$}
\def\5{$w_2^*(s^*\!A\<\otimes\< h_*B)$}
\def\6{$g_*(g^*r^*u_2^*A\<\otimes\< u_2^*B)$}
\def\7{$r^*u_2^*A\<\otimes\< g_*u_2^*B$}
\def\8{$g_*u_2^*(h^*\<\<s^*\!A\<\otimes\< B)$}
\def\9{$g_*(\<u_2^*h^*\<\<s^*\!A\<\otimes\< u_2^*B)$}
\def\ten{$\,g_*(g^*w_2^*s^*\!A\<\otimes\< u_2^*B)$}
\def\lvn{$w_2^*s^*\!A\<\otimes\< g_*u_2^*B$}
\def\twv{$w_2^*s^*\!A\<\otimes\< w_2^*h_*B$}
 \bpic[xscale=3.6,yscale=1.8]

  \node(12) at (2.1,-1){\5};
  \node(13) at (3.1,-1){\3};
  \node(14) at (3.8,-1.6){\8};

  \node(21) at (2.1,-2.1){\twv};
  \node(23) at (3.1,-2.1){\1};
   
  \node(31) at (1.2,-3){\lvn};
  \node(32) at (2.1,-3){\7};
  \node(33) at (3.1,-3){\2};

  \node(42) at (2.1,-3.9){\6 };
  \node(43) at (3.1,-3.9){\4};
  
  \node(51) at (1.2,-4.8){\ten };
  \node(54) at (3.8,-4.8){\9};

   \draw[->] (21) -- (12) ; 
   \draw[->] (12) -- (13) ; 
   \draw[->] (13) -- (14) ;
  
   \draw[->] (31) -- (32) ;
   
   \draw[double distance=2pt] (42) -- (43) ;
 
   \draw[double distance=2pt] (51) -- (54) node[below=1pt, midway, scale=0.75]{$\via\ps^*$};

   \draw[->] (21) -- (31) ; 
   \draw[->] (31) -- (51) ;
   
   \draw[->] (21) -- (32) ;   
   \draw[->] (32) -- (42) ;
 
   \draw[double distance=2pt] (13) -- (23) ; 
   \draw[->] (23) -- (33) ; 
   \draw[<-] (33) -- (43) ; 

   \draw[->] (14) -- (54) ;
       
   \draw[->] (42) -- (51) ; 
   \draw[double distance=2pt] (43) -- (54) ;
  
     \node at (2.65,-2.6) {\circled d$_1'$};

 \epic
\]

\medskip
 
\noindent  With this, Proposition~\ref{lastprop},   Step IIA and  Theorem~\ref{trans fc}, are proved.
 \end{proof}
\end{cosa}

\end{document}